\documentclass[10pt, reqno]{amsart} 
\usepackage{mathrsfs,esvect}
\usepackage{extarrows}
\usepackage{tikz}
\usetikzlibrary{calc, decorations.pathreplacing, matrix, patterns, arrows.meta}
\usepackage{pgfplots}
\pgfplotsset{compat=1.18, width=12cm, height=8cm}
\usepackage[margin=1in]{geometry}
\definecolor{deepblue}{RGB}{0,74,155}
\definecolor{crimson}{RGB}{178,34,52}
\usepackage[colorlinks=true, linkcolor=blue, citecolor=red, urlcolor=red, backref=page]{hyperref}
\newcommand{\BMO}{\mathrm{BMO}}
\newcommand{\Dini}{\mathrm{Dini}}

\usepackage{verbatim}
\usepackage{amssymb,amscd,amsfonts,amsbsy}
\usepackage{latexsym}
\usepackage{exscale}
\usepackage{amsmath,amsthm,amsfonts}
\usepackage{mathrsfs}
\usepackage{xcolor} 
\usepackage{esint} 
\usepackage{amssymb} 
\usepackage{stmaryrd}
\usepackage{cite} 
\usepackage{tikz}
\usepackage{enumerate}
\calclayout

\allowdisplaybreaks[3] 

\newtheorem{theorem}{Theorem}[section]
\newtheorem{proposition}[theorem]{Proposition}
\newtheorem{lemma}[theorem]{Lemma}
\newtheorem{corollary}[theorem]{Corollary}
\newtheorem{definition}[theorem]{Definition}
\newtheorem{remark}[theorem]{Remark}

\numberwithin{equation}{section}

\makeatletter
\@addtoreset{equation}{section}
\makeatother

\def\osc{\operatorname{osc}}
\def\V{{\mathbb V}} 
\def\rn{{\mathbb R^n}}

\def\cc{{\mathbb C}}
\def\BMO{{\rm BMO}}

\def\B{{\mathscr B}}
\def\M{{\mathscr M}}

\def\T{{\mathscr T}}
\def\B{{\mathcal B}}

\def\S{{\mathscr S}}
\def\P{{\mathscr P}}

\def\G{{\mathcal{G}}}
\def\H{{\mathcal H}}
\def\F{{\mathscr F}}

\def\calA{\mathcal{A}}

\DeclareMathAlphabet
{\mathpzc}{OT1}{pzc}{m}{it}

\newcommand{\dada}{\beta}

\newcommand{\lla}{\left\langle}
\newcommand{\rra}{\right\rangle}

\newcommand{\laa}{\big\langle}
\newcommand{\raa}{\big\rangle}
\newcommand{\laaa}{\Big\langle}
\newcommand{\raaa}{\Big\rangle}

\def\m{\mathcal{M}}

\def\Xint1{\mathchoice
   {\XXint\displaystyle\textstyle{1}}%
   {\XXint\textstyle\scriptstyle{1}}%
   {\XXint\scriptstyle\scriptscriptstyle{1}}%
   {\XXint\scriptscriptstyle\scriptscriptstyle{1}}%
   \!\int}
\def\XXint123{{\setbox0=\hbox{$1{23}{\int}$}
     \vcenter{\hbox{$23$}}\kern-.5\wd0}}

\DeclareMathOperator*{\essinf}{ess\,inf}
\DeclareMathOperator*{\esssup}{ess\,sup}

\newcommand{\aver}[1]{-\hskip-0.46cm\int_{1}}
\newcommand{\textaver}[1]{-\hskip-0.40cm\int_{1}}

\def\calB{\mathcal{B}}

\def\d{\mathcal{D}}

\def\F{\mathcal{F}}
\def\N{\mathbb{N}}
\def\Z{\mathbb{Z}}

\def\R{\mathbb{R}}
\def\Rn{\mathbb{R}^n}
\def\S{\mathcal{S}}

\def\w{{\omega}}

\def\supp{\operatorname{supp}}

\def\BMO{\operatorname{BMO}}

\newcommand{\C}{\boldsymbol{C}}
\newcommand{\D}{\mathcal{D}}

\newcommand{\calS}{\mathcal{S}}

\newcommand{\norm}[1]{\left\lVert1\right\rVert}
\newcommand\restr[2]{\ensuremath{\left.1\right|_{2}}}

\def\avint_#1{\mathchoice{\mathop{\kern 0.2em\vrule width 0.6em height 0.69678ex depth -0.58065ex \kern -0.8em \intop}\nolimits_{\kern -0.4em#1}}{\mathop{\kern 0.1em\vrule width 0.5em height 0.69678ex depth -0.60387ex \kern -0.6em \intop}\nolimits_{#1}} {\mathop{\kern 0.1em\vrule width 0.5em height 0.69678ex depth -0.60387ex \kern -0.6em \intop}\nolimits_{#1}} {\mathop{\kern 0.1em\vrule width 0.5em height 0.69678ex depth -0.60387ex \kern -0.6em \intop}\nolimits_{#1}}}

\allowdisplaybreaks

\begin{document}
\title[The multilinear fractional bounded mean oscillation operator theory...]
{\bf The multilinear fractional bounded mean oscillation operator theory I: sparse domination, sparse $T1$ theorem, off-diagonal extrapolation, quantitative weighted estimate---for generalized commutators}

\author[X. Cen]{Xi Cen*}
\address{Xi Cen\\
School of Science\\
China University of Mining and
Technology (Beijing)\\
Beijing 100083 \\
People's Republic of China}\email{xicenmath@gmail.com}

\author[Z. Song]{Zichen Song}
\address{Zichen Song\\
School of Mathematics and Stastics\\
Xinyang Normal University\\
Xinyang 464000\\
People's Republic of China}\email{zcsongmath@gmail.com}

\date{\today}
	
	
	
\thanks{$^{*}$ Corresponding author, Email: xicenmath@gmail.com}

\begin{abstract} 
This paper introduces and studies a class of multilinear fractional bounded mean oscillation operators (denoted {\rm $m$-FBMOOs}) defined on ball-basis measure spaces $(X, \mu, \mathcal{B})$. These operators serve as a generalization of canonical classes, such as the multilinear fractional maximal operators, the multilinear fractional Ahlfors-Beurling operators, the multilinear pseudo-differential operators with multi-parameter Hörmander symbol, and some multilinear operators admitting \textbf{\textsf{$\mathbb{V}$-valued $m$-linear fractional Dini-type Calderón-Zygmund kernel representation}}. Crucially, the definition utilized here, incorporating the notion of "bounded mean oscillation," provides greater generality compared to those in Karagulyan \cite{K2019} (2019) and Cao et al. \cite{Cao23} (2023). Our investigation systematically examines the properties of these operators and their generalized commutators through the lens of modern harmonic analysis, focusing on two principal directions:

\begin{itemize}
\item {\emph{\textbf{\textsf Karagulyan-type sparse domination}}} \& {\emph{\textbf{\textsf Multilinear fractional sparse $T1$ theorem.}}}
We establish Karagulyan-type sparse domination for the generalized commutators. Subsequently, by developing bespoke dyadic representation theorems pertinent to this setting, we prove a corresponding multilinear fractional sparse $T1$ theorem for these generalized commutators.

\item
{\emph{\textbf{\textsf Multiparameter multilinear fractional weight class}}}  \& {\emph{\textbf{\textsf Multilinear off-diagonal extrapolation theorem.}}}
With sparse bounds established, we obtain weighted estimates in multiple complementary methods:

(1) Under a novel class of multilinear fractional weights, we prove four types of weighted inequalities: sharp-type, Bloom-type, mixed weak-type, and local decay-type.

(2) We develop a multilinear non-diagonal extrapolation framework for these weights, which transfers boundedness flexibly among weighted spaces and establishes corresponding vector-valued inequalities.
\end{itemize}

In conclusion, this research offers a comprehensive analysis of this specific class of $m$-FBMOOs and their generalized commutators. Beyond the results for the operators themselves, this work advances core methodologies within Fourier analysis on measure spaces, significantly extending techniques related to
{\emph{{\textbf{\textsf "multilinear fractional theory"}},{\textbf{\textsf "dyadic representation theory"}}, \textbf{\textsf "sparse domination theory"}, \textbf{\textsf "sparse $T1$ theory"}, \textbf{\textsf "quantitative weighted estimates theory
"}, \textbf{\textsf "multilinear off-diagonal extrapolation theory"},}} and {\emph{\textbf{\textsf "higher-order commutator theory".}}}
\end{abstract}

	\maketitle
	\tableofcontents

\section{\bf Introduction}\label{Introduction}
\subsection{Some ideas for multilinear dyadic harmonic analysis}
~~

In recent years, problems within modern Fourier analysis have been investigated from various theoretical viewpoints. Within this domain, research concerning operator theory is broadly classified into the study of boundedness and compactness properties. The theory of the boundedness of operators, in turn, is generally subdivided into diagonal estimates and off-diagonal estimates. This paper primarily addresses the theory of off-diagonal estimates for a specific class of operators known as multilinear fractional bounded mean oscillation operators ($m$-FBMOOs, cf. Definition \ref{def:FBMOO}), along with their associated generalized commutators. Our investigation is structured around the following core aspects, delineated sequentially:

(1) \textbf{\textsf Sparse domination theory:} 

\emph{
The concept of sparse domination for an operator involves establishing pointwise or norm bounds for the operator via simpler, positive operators known as sparse operators, or through associated sparse forms. A substantial body of classical work exists on establishing sparse bounds, e.g. \cite{Lerner2016,Lerner2018,Lerner19,Li2018,Cao2018,Yang2019,CenSong2412}. Lerner-type sparse control was initially formulated in \cite{Lerner2016}. The resolution of such problems often hinges upon techniques involving sharp maximal truncation operators $M_{T}$, which provide an effective means to manage the singularity in the operators under consideration.}

\emph{
Specifically, Karagulyan (2019, 2023) \cite{K2019,K2023} investigated a class of linear bounded oscillation operators, herein termed Karagulyan-type operators, which serve as a direct prototype for the operators central to this study. Subsequently, Cao et al. (2023) \cite{Cao23} examined and generalized the multilinear version of these operators, which we designate as multilinear Karagulyan-type operators. The definition of Karagulyan-type operators, analogous to sharp maximal truncation operators, rigorously characterizes the behavior of the truncation difference of the operator via the \textbf{\textsf{full uniform upper bound}} conditions (T1) and (T2). This structural property effectively circumvents limitations typically imposed by operator singularity. 
Building upon these foundational definitions, the $m$-FBMOOs are introduced through our observation that the condition on the truncation difference can be relaxed to a \textbf{\textsf{mean uniform upper bound}} behavior in \eqref{M_1.condi.} and \eqref{M_2.condi.}. Concurrently, these operators incorporate fractionalization to generalize their properties. Our conceptualization of this operator class is also partly informed by Lerner's idea (cf. \cite[p.1012]{Lerner19}) regarding local weak-type estimates of the operators, which permits a highly refined analysis of operator behavior localized on the balls, as discussed in Definition \ref{def:multilinear_W}.
}

(2) \textbf{\textsf Sparse $T1$ theory and dyadic representation theory:} 

\emph{
The sparse $T1$ theorem which is first introduced by \cite{Lacey16} represents a contemporary manifestation of the classical $T1$ theorem that is constructed by David and Journé (cf. \cite{David84}) within the framework of modern Fourier analysis, marking a significant milestone in the advancement of dyadic analysis. In essence, the sparse $T1$ theorem establishes that an operator admits sparse bounds provided certain $T1$-type testing conditions are satisfied (often involving the operator's action on characteristic functions $\chi_Q$ or specific "test" functions $\phi$ with $|\phi| \le 1$). The realization of such theorems typically relies on a sophisticated dyadic representation theory, implying that the integral representation of our operator can be effectively captured or dominated by sparse forms constructed from elementary dyadic operators (e.g., dyadic shifts, Haar multipliers, paraproducts).}

\emph{
This line of inquiry was pioneered by Hytönen (2012) \cite{Hytonen12}, who simultaneously resolved the celebrated $A_2$ conjecture for general Calderón-Zygmund operators. This paradigm has been subsequently extended and refined by researchers including Hytönen, Lacey, Li, Martikainen, Ou, et al. \cite{Li2018,Li19,Li25,Hytonen2017,Martikainen11,Ou17,Li2018_2,Li2018_3,Hytonen16}. Furthermore, these dyadic representation techniques can be adapted to address problems of weighted compactness for operators, leading to what is known as compact dyadic representation theory cf. \cite{Cao23,Cao2307,Cao2410,Cao24,Cao2404}. This paper will leverage this conceptual framework to establish a \emph{ multilinear fractional sparse T1 theorem} for the general commutators associated with the {\rm $m$-FBMOOs}, cf. Theorem \ref{thm:sparse-T1}.
}

Within harmonic analysis, the control strategies of operator are broadly divided into maximal control and sparse control. The significance of these strategies lies in the fact that both maximal operators and sparse operators often admit robust, and frequently sharp, quantitative weighted norm inequalities. This provides a direct link to the development and application of multilinear weighted theory.

(3) \textbf{\textsf Development of multilinear weighted theory:} 
~~

\emph{
Seminal contributions that shaped modern multilinear weighted theory include the works of Perez et al. \cite{Perez09}, Moen \cite{Moen09}, and Xue et al. \cite{Xue09}. Their emergence significantly advanced classical weighted theory beyond the single-weight setting. They enabled the study of multilinear operators satisfying weighted boundedness criteria involving multiple, potentially distinct, weight functions, often integrated into a unified multilinear weight class structure (e.g., consolidating $\prod\limits_{i = 1}^m {{A_{{p_i}}}}$ into a joint $A_{\vec{p}}$ class). Furthermore, Moen (2009) \cite{Moen09} and Xue (2009) \cite{Xue09} introduced a fractional-type integration condition for these multilinear weights, pioneering the concept of multilinear fractional weights (e.g., $A_{\vec{p}, q}$ classes). More recently, Cruz-Uribe et al. (2012, 2020, 2022) \cite{Cruz12,Cruz20,Cruz22} have explored extensions to the variable exponent setting. Building on this, we in 2024 proposed \cite{Cen2310,Cen2404,Cen2408} a fractional variable exponent generalization of the work by Cruz-Uribe et al. over space of homogenous type.
}

Crucially, another significant direction in the evolution of multilinear weights has been propelled by research on multilinear extrapolation theory. This line of research spurred the development of multi-weight theory towards sophisticated frameworks designed to handle multiple independent parameters associated with the weights and operators. A primary goal of the present work is to investigate non-trivial fractional weight extensions precisely within this multiparameter setting, cf. Definition \ref{def:weight} and Fig. \ref{figure1}.

\textbf{\textsf (4) Multilinear extrapolation theory:} 

\emph{
Beyond direct sparse control methods, establishing quantitative weighted estimates for operators can also be achieved by constructing suitable extrapolation frameworks. Key recent developments in multilinear extrapolation theory feature the contributions of Cao, Hytönen, Li, Martell, Nieraeth, Torres, et al. (2018--2022) \cite{Cao22,Martell2020,Li21,Nie2018,Tor2019}. This body of work provides powerful machinery that permits a weighted norm inequality, initially established perhaps only for a fixed tuple of Lebesgue exponents $(p_1, \dots, p_m, \tilde{p})$, to be systematically generalized ('liberalized') to hold for a wide range of arbitrary exponent tuples $(q_1, \dots, q_m, \tilde{q})$. However, a comprehensive off-diagonal multilinear extrapolation theory specifically grounded in multilinear fractional weights involving multiple independent parameters has, to our knowledge, remained underdeveloped. Addressing this theoretical gap constitutes a central objective of this paper.}

\emph{
It is pertinent to emphasize that constructing such an extrapolation framework holds particular significance: it is precisely the tool needed to achieve the desired flexibility in exponent selection (i.e., exponent liberalization) for the quantitative weighted norm inequalities governing sparse operators, operating exactly within the novel theoretical architecture of multilinear fractional weights with independent parameters developed herein, cf. Theorems \ref{thm:Ex_1} and \ref{op.to.jh}.}

The successful realization of these interconnected objectives--advancing sparse control for $m$-FBMOOs and their generalized commutators, establishing a corresponding sparse $T1$ theorem, developing the requisite multiparameter fractional multiple weight theory, and constructing the associated  multilinear off-diagonal extrapolation framework--will forge a cohesive synthesis of these contemporary theoretical strands in multilinear harmonic analysis. This synthesis represents the principal motivation underpinning this research endeavor.

\begin{definition}\label{def:basis}
Let $(X, \mu)$ be a $\sigma$--finite measure space. A collection $\mathcal{B}$ of measurable subsets is called a {\bf ball-basis} if it satisfies:
	\begin{list}{\rm (\theenumi)}{\usecounter{enumi}\leftmargin=1.2cm \labelwidth=1cm \itemsep=0.2cm \topsep=.2cm \renewcommand{\theenumi}{B\arabic{enumi}}}
	
	\item\label{list:B1} $\mathcal{B}$ is a basis, meaning that $0 < \mu(B) < \infty$ for every $B \in \mathcal{B}$.
	
	\item\label{list:B2} For any two points $x, y \in X$, there exists a ball $B \in \mathcal{B}$ such that $x, y \in B$. 
	
	\item\label{list:B3} For every $\varepsilon > 0$ and measurable set $E \subset X$, there exists a sequence $\{B_k\}_{k=1}^\infty \subseteq \mathcal{B}$ satisfying $\mu\left(E \Delta \bigcup_{k=1}^\infty B_k\right) < \varepsilon$. 
	
	\item\label{list:B4} For every $B \in \B$, there exists $B^\dagger \in \B$ such that $\mu(B^\dagger) \le \boldsymbol{C}_0 \mu(B)$, 
	\begin{align*}
	\bigcup_{B' \in \B: B' \cap B \neq \emptyset \atop \mu(B') \le 2\mu(B)} B' \subseteq B^\dagger, 
	\quad\text{ and }\quad 
	A \subseteq B \Longrightarrow A^\dagger \subseteq B^\dagger, 
	\end{align*}
where $\C_0 \ge 1$ is a universal constant.
	\end{list}
In this case, $(X,\mu,\mathcal{B})$ is called ball-basis measure space.
	\end{definition}

\begin{remark}\label{example_0}
One can verify that the following are all ball-bases.

\begin{enumerate}
\item[$\spadesuit$]
Let \(\mathcal{D}^1, \ldots, \mathcal{D}^N\) be dyadic lattices in the space of homogeneous type $(X,d,\mu)$ (cf. subsection \ref{dyadic}), then \(\mathcal{D} = \bigcup_{j=1}^N \mathcal{D}^j\) is a ball-basis.

\item[$\spadesuit$] Let
$\mathcal{D}=\mathop  \cup \limits_{i = 1}^{3^n} {\mathcal{D}_i}$ with $\{\mathcal{D}_i\}_{i=1}^{3^n}$ are $3^n$ fixed dyadic lattices, $\mathscr{Q}$ is the set of all cubes, and 
$\mathcal{Q}$ is the set of all cubes with sides parallel to the coordinate axes.
Then $\D$, $\mathscr{Q}$, and $\mathcal{Q}$ are three ball-bases of $(\Rn, dx)$. In this case, we denote by $\mathcal{D}(Q)$ the collection of all dyadic descendants of $Q$. 

\item[$\spadesuit$]

If \(\mathscr{B}\) is a set of all balls in \(\mathbb{R}^n\). Then $\mathscr{B}$ is a ball-basis of \((\mathbb{R}^n, dx)\). In this case, we can take \(\C_0 = \kappa^n\) with \(\kappa \geq 1 + 2^{1 + \frac{1}{n}}\) and defining \(B^\dagger = \kappa B\) for each \(B \in \mathcal{B}\).


\item[$\spadesuit$] Let $(X, \mu)$ be a measure space with a martingale basis $\B$, defined as follows:
\begin{itemize}
\item \(\mathcal{B} = \bigcup_{j \in \mathbb{Z}} \mathcal{B}_j\) forms the \(\sigma\)-algebra of all measurable sets in \(X\). 
\item Each \(\mathcal{B}_j\) is a finite or countable decomposition of \(X\).  
\item For every \(j\) and \(B \in \mathcal{B}_j\), \(B\) is a union of elements from \(\mathcal{B}_{j+1}\).  
\item Any two points $x, y \in X$ are contained in some set $B \in \mathcal{B}$.
\end{itemize}
For \(B \in \mathcal{B}\), let \(\mathscr{F}(B) \in \mathcal{B}\) denote the smallest set such that $B \subsetneq \mathscr{F}(B)$.
Define \(\mathscr{F}^1 = \mathscr{F}\) and \(\mathscr{F}^{k+1} = \mathscr{F}(\mathscr{F}^k)\) for \(k \geq 1\). Then:  
\[
B^* = 
\begin{cases} 
B & \text{if } \mu(\mathscr{F}(B)) > 2 \mu(B), \\ 
\mathscr{F}^k(B) & \text{if } \mu(\mathscr{F}^k(B)) \leq 2 \mu(B) < \mu(\mathscr{F}^{k+1}(B)). 
\end{cases}
\]  
It can be readily shown that \(\mathcal{B}\) is a ball-basis of $(X, \mu)$ with the constant $\C_0=2$. In this case,  $\mu$ can be a non-doubling measure.

\end{enumerate}

The following collections are not ball-bases. 
\begin{itemize}
\item[$\clubsuit$] The collection of all standard dyadic cubes in $\Rn$ does not form a ball-basis of $(\Rn, dx)$ since it does not satisfy the property \eqref{list:B2}. 

\item[$\clubsuit$] 
In a space of homogeneous type $(X,\mu,d)$, a fixed dyadic partition $\D_k$ (cf. Definition \ref{dyadic}) is not a ball-basis because it fails to satisfy condition \eqref{list:B2}.

\item[$\clubsuit$]
The family of all rectangles in \(\mathbb{R}^n\) with sides parallel to the coordinate axes is not a ball-basis of \((\mathbb{R}^n, dx)\). For \(n=2\), let
\[
B = [0,1)^2, \quad B_k = [0,2^{k+1}) \times [0,2^{-k}) \quad \text{for } k = 0, 1, \ldots.
\]
Then \(B_k \cap B \neq \emptyset\) and \(|B_k| = 2|B|\) for each \(k\), but there is no rectangle containing \(\bigcup_{k=0}^\infty B_k\).


\item[$\clubsuit$] 
Let \(\mathcal{B}\) be the collection of Zygmund rectangles in \((\mathbb{R}^3, dx)\) with sides parallel to the coordinate axes and lengths \(s\), \(t\), and \(st\) for \(s, t > 0\). Then \(\mathcal{B}\) is not a ball-basis of \((\mathbb{R}^3, dx)\). For example, let
\[
B = [0,1)^3, \quad B_k = [0,2^{k + \frac{1}{2}}) \times [0,2^{-k}) \times [0,2^{\frac{1}{2}}) \quad \text{for } k = 0, 1, \ldots.
\]
Then \(B_k \cap B \neq \emptyset\) and \(|B_k| = 2|B|\) for each \(k\), but there is no Zygmund rectangle containing \(\bigcup_{k=0}^\infty B_k\).
\end{itemize} 
\end{remark}

\subsection{Notation}~~

Next we introduce some notation that will be used throughout the paper.

\begin{itemize}
\item For some positive constant \(C\) independent of the main parameters, we denote \(A \lesssim B\) to mean that \(A \leq CB\), and \(A \approx B\) to indicate that \(A \lesssim B\) and \(B \lesssim A\). Additionally, \(A \lesssim_{\alpha, \beta} B\) signifies that \(A \leq C_{\alpha,\beta} B\), where \(C_{\alpha,\beta}\) depends on \(\alpha\) and \(\beta\).
\item 
We say \(\omega: X \rightarrow [0, \infty]\) is a weight, if $\omega$ is a locally integrable function satisfying \(0 < \omega(x) < \infty\) almost everywhere. We define the measure \(d\omega(x) = \omega(x)d\mu(x)\) and \(\omega(E) = \int_E \omega \,d\mu\) for measurable \(E \subseteq X\).
\item 
Let \(\vec{r} = (r_i)_{i=1}^m\) and \(\vec{p} = (p_i)_{i=1}^m\). We write \(\vec{r} \leq \vec{p}\) if \(r_j \leq p_j\) and \(\vec{r} < \vec{p}\) if \(r_j < p_j\) for all \(j = 1, \dots, m\). We write \((\vec{r}, s) \stackrel{}{\prec} (\vec{p}, q)\) if \(\vec{r} < \vec{p}\) and \(q < s\); we write \((\vec{r}, s) \stackrel{}{\prec^*} (\vec{p}, q)\) if \(\vec{r} \leq \vec{p}\) and \(q < s\), and we write \((\vec{r}, s) \preceq^{} (\vec{p}, q)\) if \(\vec{r} \leq \vec{p}\) and \(q \leq s\).
\item
We denote $\vec{f}=\left(f_1, \ldots, f_m\right)$ and $\vec{f} \chi_E=\left(f_1 \chi_E, \ldots, f_m \chi_E\right)$ for any measurable set $E$. Let $r \in[1, \infty)$, $0 \leq \eta <1$ such that $0\leq\eta r <1$. For any $B \in \mathcal{B}$, 
the fractional average bump of $f$ is defined by 
$$\langle f\rangle_{\eta,r,B}:=\left(\frac{1}{\mu(B)^{1-\eta r}}\int_Bf^r d \mu\right)^{\frac{1}{r}};$$
the fractional maximal average bump of $f$ is set by
\begin{align}\label{avggg}
{{{\left\langle {\left\langle {{f}} \right\rangle } \right\rangle }_{\eta ,r,B}}}:=\sup _{B^{\prime} \in \mathcal{B}: B^{\prime} \supseteq B}\langle |f|\rangle_{\eta,r,B^{\prime}}.
\end{align}
When $r=1$, we always denote ${{{\left\langle {\left\langle {{f_i}} \right\rangle } \right\rangle }_{\eta,B}}}:={{{\left\langle {\left\langle {{f_i}} \right\rangle } \right\rangle }_{\eta ,r,B}}}$ and $\langle f\rangle_{\eta,B}:=\langle f\rangle_{\eta,r,B}$.
When $\eta=0$, we denote them by ${{{\left\langle {\left\langle {{f_i}} \right\rangle } \right\rangle }_{r,B}}}$ and $\langle f\rangle_{r,B}$ respectively.
\item  A subset \(E \subseteq X\) is termed \textbf{bounded} if there exists some \(B \in \B\) satisfying  $ E \subseteq B$.
When \(E\) and \(F\) are measurable, we say that \(E\) is \textbf{almost surely} contained in \(F\), written $E \subseteq F \ \text{a.s.}$
precisely when $  \mu\bigl(E \setminus F\bigr) \;=\; 0.$
For any finite set \(A\), its cardinality is denoted by \(\#A\).  
\end{itemize}

\subsection{Some conceptions}~~

\begin{definition}

Let $(X,\mu,\mathcal{B})$ is a ball-basis measure space. Set $\delta \in (0,1)$ and we say $\mathcal{S} \subseteq \mathcal{B}$ is a $\delta$-sparse family, if for each $Q \in \mathcal{S}$, there exists a measurable subset $E_Q \subseteq Q$ with $\mu(E_Q) \geq \delta \mu(Q)$, and $\{E_Q\}_{Q \in \mathcal{S}}$ is a pairwise disjoint family.
    
\end{definition}

\begin{definition}
Let $m \in \mathbb{N}$, $ r_1,\ldots,r_m \in [1, \infty)$, $\eta_i \in [0,1)$ with \( 0 \le \eta_i r_i \leq 1\), for $i=1,\cdots,m$, and the fractional total order $\eta := \sum\limits_{i=1}^{m}\eta_i \in [0,m)$. Suppose that ${\bf t,k }\in \N_0^m$ with \(\mathbf{0} \le \mathbf{t} \leq \mathbf{k}\), and multi-symbol $\mathbf{b}=\left(b_1, \ldots, b_m\right) \in\left(L_{l o c}^1(X)\right)^m$. For a sparse family $\mathcal{S} \subseteq \mathcal{B}$, we define the \textbf{higher-order multi-symbol $m$-linear fractional sparse operator} 
$\mathscr{A}_{\vec{\eta},\mathcal{S},\vec{r}}^{\mathbf{b},\mathbf{k},\mathbf{t}}$   defined by
\begin{align*}
{\mathscr A}_{\vec{\eta},\S,{\vec r}}^\mathbf{b,k,t}(\vec{f})(x)=&\sum_{B \in \mathcal{S}} \mu(B)^{\eta} \prod_{i = 1}^m\left|b_i(x) - b_{i,B}\right|^{k_i - t_i} \left\langle\left|f_i (b_i - b_{i,B})^{t_i}\right|\right\rangle_{r_i,B}  \chi_B(x)\\
=&\sum_{B \in \mathcal{S}}  \prod_{i = 1}^m\left|b_i(x) - b_{i,B}\right|^{k_i - t_i} \left\langle\left|f_i (b_i - b_{i,B})^{t_i}\right|\right\rangle_{\eta_i,r_i,B}  \chi_B(x).
\end{align*}
We tend to denote ${\mathscr A}_{\vec{\eta},\S,{\vec r}}^\mathbf{b,k,t}$ by ${\mathscr A}_{\eta,\S,{\vec r}}^\mathbf{b,k,t}$.
When ${\bf k}=(0,\cdots,k_i,\cdots,0)$, we denote it by ${\mathscr A}_{{\eta},\S,{\vec r}}^{b_i,k_i,t_i}$. When $\bf k=0$, we denote it by ${\mathscr A}_{{\eta},\S,{\vec r}}$.
\end{definition}

\begin{definition}\label{def.form}
Let $m \in \mathbb{N}$, $ r_1,\ldots,r_m,s' \in [1, \infty)$, and the fractional total order $\eta \in [0,m)$. Suppose that ${\bf t,k }\in \N_0^m$ with \(\mathbf{0} \le \mathbf{t} \leq \mathbf{k}\), and multi-symbol $\mathbf{b}=\left(b_1, \ldots, b_m\right) \in\left(L_{l o c}^1(X)\right)^m$. For a sparse family $\mathcal{S} \subseteq \B$, the \textbf{higher-order multi-symbol $(m+1)$-linear fractional sparse form} is defined as
   \begin{align*}
{\mathcal A}_{\eta,\S,{\vec r},s'}^\mathbf{b,k,t} (\vec f,\psi) &=\sum_{B \in \mathcal{S}} \mu(B)^{\eta + 1} \prod_{i=1}^m\left\langle\left(\left|f_i (b_i - b_{i,B})^{t_i}\right|\right)\right\rangle_{r_i,B}\cdot\left\langle\left( \prod_{i=1}^m  \left|b_i(x) - b_{i,B}\right|^{k_i - t_i}\right)\psi \right\rangle_{s',B}.
     \end{align*}
When $\bf k=0$, we denote it by ${\mathcal{A}}_{{\eta},\S,{\vec r},s'}$.
\end{definition}
We will establish the sharp weighted estimates of these operators, cf. subsection \ref{sharp}.


Let ${\bf b} \in {\left( {L_{loc}^1\left( X \right)} \right)^m}$. Given a $m$-sublinear operator $T$ and a multi-index \(\mathbf{k} = (k_1, \ldots, k_{m}) \in (\N_0)^m \), its generalized commutator can be defined by 
\begin{align*}
T^{{\bf b, k}}(\vec{f})(x):=T\left((b_1(x) - b_1)^{k_1}f_1,\ldots,(b_m(x) - b_m)^{k_m}f_m\right)(x),
\end{align*}

All of our main results pertain to the study of the generalized commutators $T^{\bf b,k}$ of order $|\bf k|$. Here, we need to clarify:
\begin{itemize}
\item When $|{\bf k}|=0$, $T^{\bf b,k}=T$ called '0-order commutator'.
\item When $|{\bf k}|=1$, $T^{\bf b,k}=T^{b_i,1}$ for $i=1,\cdots,m$, called "1-order commutator". 
The classcial multilinear commutator $T^{\sum \bf b}:=\sum\limits_{i = 1}^m {T^{{b_i},1}}$ is initially introduced in \cite{Perez2009} for multilinear Calder\'{o}n--Zygmund operators and multilinear maximal operators.

\item When $|{\bf k}| \ge 2$, $T^{\bf b,k}$ is called "higher-order commutator" studied in \cite{Xue2021} for multilinear Calder\'{o}n--Zygmund operators and multilinear Littlewood-Paley square operators.

\item When $|{\bf k}|=m$ with $k_1=\cdots=k_m$=1, $T^{\bf b,k}$ is called "multilinear iterated commutators" initially introduced in \cite{Perez2013,Xue2013} for multilinear Calder\'{o}n--Zygmund operators, multilinear fractional maximal operators, and multilinear fractional integral operators.
\end{itemize}
Over the past 15 years, researchers have often considered multilinear singular integral operators with (fractional) Calder\'{o}n--Zygmund kernels. Our main results below center on generalized commutators $T_{\vec{\eta}}^{\bf b,k}$, which cover all of the above cases and extend their work.  Because of the different applicable spaces of the theorems, we always set the ball-basis measure space $(X,\mu,\mathcal{B})$, space of homogeneous type $(X,d,\mu,\mathcal{D})$, and classical Euclidean space $(\rn,|\cdot|,dx,\mathcal{D})$ as the base space we need at the beginning, where $\mathcal{D}$ is a fixed dyadic grid of the base spaces and $|\cdot|$ is the Euclidean distance.  These base spaces form successively specialized cases.

\begin{definition}\label{def:FBMOO}
Let \(r_i \in [1, \infty)\), \(s_1, s_2,s_3\in [1, \infty]\), $\eta_i \in [0,1)$ with \( 0 \le \eta_i r_i \leq 1\), for $i=1,\cdots,m$, and the fractional total order $\eta := \sum\limits_{i=1}^{m}\eta_i \in [0,m)$. Set \(\mathbb{V}\) is a quasi-Banach space and \((X, \mu, \mathcal{B})\) is a ball-basis measure space. We say \({T}_{\vec\eta}\) is a \(\mathbb{V}\)-valued  {\bf $m$-(sub)linear  fractional bounded mean oscillation operator} {\rm ($m$-FBMOO)} on \((X, \mu, \mathcal{B})\), if there exists  uniform contants \({C}_{1}:={C}_{1}({T}_{\vec \eta})>0, {C}_{2}:={C}_{2}({T}_{\vec \eta}) > 0\),
such that for all nonnegative $\vec f \in \prod\limits_{i = 1}^m {{L^{{r_i}}}(X)}$,
\begin{list}{\rm (\theenumi)}{\usecounter{enumi}\leftmargin=1.2cm \labelwidth=1cm \itemsep=0.2cm \topsep=.2cm \renewcommand{\theenumi}{m-FBMOO-\Roman{enumi}}}
    \item\label{M_1.condi.} 
    For every $B_0 \in \B$ with $B_0^{\dagger} \subsetneq X$, there exists $B \in \B$ with $B \supsetneq B_0$ such that 
    \begin{align*} 
{\left\langle\|T_{\vec\eta}(\vec{f} \chi_{B^\dagger})(\cdot)-T_{\vec\eta}(\vec{f} \chi_{B^\dagger_0})(\cdot)\|_{\mathbb{V}}\right\rangle_{s_1,B_0}} \leq {{C}_{1}}{\prod\limits_{i=1}^m\left\langle f_i\right\rangle_{\eta_i,r_i,B^\dagger}},
	\end{align*}

	\item\label{M_2.condi.} For every $B \in \B$,  
	\begin{align*}
{\left \langle \left\langle \left\|\left(T_{\vec\eta}(\vec{f})(x)-T_{\vec\eta}(\vec{f}\chi_{B^\dagger})(x) \right)- \left(T_{\vec\eta}(\vec{f})(x') -T_{\vec\eta}(\vec{f}\chi_{B^\dagger})(x')\right)\right\|_\mathbb{V}\right\rangle_{s_2,B \ni x'} \right\rangle_{s_3,B \ni x}} \leq {{C}_{2}} \prod\limits_{i = 1}^m {{{\left\langle {\left\langle {{f_i}} \right\rangle } \right\rangle }_{{\eta _i},{r_i},B}}}.
	\end{align*}
	\end{list}
In the absence of ambiguity, we omit the "$\mathbb{V}$-valued" from the narrative and denote \( T_{\vec{\eta}} \) as \( T_{\eta} \).
	\end{definition}
    

\begin{remark}\label{rem:key1}
The liberalization of {\rm $m$-FBMOOs} is reflected in the significant reduction of the oscillation requirements. In fact, if $s_1=s_2=s_3=\infty$ and at the same time setting $\vec{\eta}=0$, Definition \ref{def:FBMOO} can return to \cite[Definition 1.3]{K2019} and \cite[Definition 1.4]{Cao23}. Here, we tend to describe their definition as \textbf{\textsf{full uniform upper bound}}, while \eqref{M_1.condi.} and \eqref{M_2.condi.} are more general and called, \textbf{\textsf{mean uniform upper bound}}. This consideration extends the diversity and complexity of operators.
\end{remark}

We introduce our key precondition, called local weak-type estimate, as follows. This idea was partly inspired by Lerner and Ombrosi, cf. \cite[p.1012]{Lerner19}, although it is indeed somewhat different from theirs.

\begin{definition}\label{def:multilinear_W}
Let \((X, \mu, \mathcal{B})\) be a ball-basis  measure space. Let \(m \in \mathbb{N}\), $\eta \in [0,m)$, $r_i \in (1,\infty)$ for \(i =1,\cdots,m\), and $\frac{1}{\tilde{r}}:=\sum\limits_{i = 1}^m {\frac{1}{{{r_i}}}}  - \eta>0$. Set $\mathbb{V}$ be a quasi-Banach spaces, and $G$ be a $\mathbb{V}$-valued $m$-sublinear operator. We write $G \in W_{\vec{r},\tilde{r}}$, if there exists a function $\Phi_{G,\vec{r}, \tilde{r}}:\left( {0,1} \right) \to \left( {0,\infty } \right) $, such that for every $Q \in \B$, any $R \in \B$ with $ R  \supseteq Q$, and any $f_i \in L^{r_i}(R)$, $1 \leq i \leq m$,
\[
 \mathop {\sup }\limits_{\lambda  \in \left( {0,1} \right)} \lambda^{-1} \mu\left(\left\{ x \in Q : 
 \|G(f_1\chi_R, \dots, f_m\chi_R)(x)\|_{\mathbb{V}} > \Phi_{G,\vec{p}, q}(\lambda) \mu(Q)^{\eta} \prod_{i=1}^m \langle f_i\rangle_{r_i, Q} \right\}\right) \leq \mu(Q).
\]
When $G:{L^{{r_1}}} \times  \cdots  \times {L^{{r_m}}} \to {L^{\tilde r,\infty }(\mathbb{V})}$ is bounded, then $G \in W_{\vec{r},\tilde{r}}$ with $$\Phi_{G,\vec{r},\tilde{r}}(\lambda) = \|G\|_{L^{r_1}(X) \times \cdots \times L^{r_m}(X) \to L^{\tilde{r},\infty}(X,\mathbb{V})} \cdot \lambda^{-1/\tilde{r}}.$$
\end{definition}

\section{\bf Main results}\label{Main}

\subsection{Karagulyan-type sparse domination theorem}~~

Now, we establish pointwise sparse bounds for the generalized commutators.

{\begin{theorem}\label{S.d.main}
Let \((X, \mu, \mathcal{B})\) be a ball-basis  measure space. Let \(r_i \in [1, \infty)\), \(s_1, s_2,s_3\in [1, \infty)\) with $s_1 \leq s_3$, $\eta_i \in [0,1)$ with \( 0 \le \eta_i r_i \leq 1\), for $i=1,\cdots,m$, and the fractional total order $\eta := \sum\limits_{i=1}^{m}\eta_i \in [0,m)$. Set $T_{\vec \eta}$ be a $\mathbb{V}$-valued {\rm $m$-FBMOO} on \((X, \mu, \mathcal{B})\) for the exponents above. Suppose that ${\bf t,k }\in \N_0^m$ with \(\mathbf{0} \le \mathbf{t} \leq \mathbf{k}\), and multi-symbol $\mathbf{b}=\left(b_1, \ldots, b_m\right) \in\left(L_{l o c}^1(X)\right)^m$. If $\lambda > 3\C_0^6$ and ${T}_{\vec \eta} \in W_{\vec{r},\tilde{r}}$ with $\frac{1}{\tilde{r}} :=\sum\limits_{i=1}^{m}\frac{1}{r_i} -\eta>0$, then there is a $\frac{1}{2 \boldsymbol{C}_0^3}$-sparse family $\S \subseteq \B$ such that for any nonnegative functions $f_i \in L_b^{r_i}(X)$, $i=1,\cdots,m$,
\begin{align}\label{zhang:th1.6}
\left\|{T}_{\vec \eta}^{{\bf b, k}}(\vec{f})(x)\right\|_{\mathbb{V}} \lesssim_{\bf k}  \Phi_{\varLambda,\vec{r},\tilde{r}}\left(\frac{1}{\boldsymbol{C}^3_0 \lambda}\right) \cdot \sum_{\mathbf{0} \le \mathbf{t} \leq \mathbf{k}} {\mathscr A}_{\eta ,\S,{\vec r}}^\mathbf{b,k,t}(\vec{f})(x), \text{ a.e. $x \in X$,}
  \end{align}
where $\C_0$ is defined in \eqref{list:B4}, $\Phi_{T_{\vec{\eta}},\vec{r},\tilde{r}} \colon (0,1) \to (0,\infty)$, and the notation $\sum\limits_{0 \leq \mathbf{t} \leq \mathbf{k}}$ means $\sum\limits_{i=1}^m \sum\limits_{t_i=0}^{k_i}$.
\end{theorem}}

\begin{remark}\label{rem:key2}
We need to point out that ${T}_{\vec \eta} \in W_{\vec{r},\tilde{r}}$ can be implied by ${T}_{\vec \eta}:{L^{{r_1}}} \times  \cdots  \times {L^{{r_m}}} \to {L^{\tilde r,\infty}} \text{ boundedly.}$ Authors often chose the latter option for consideration in past historical literature, e.g. \cite[Theorem 1.1]{K2019} and \cite[Theorem 1.5]{Cao23}. In fact, $W_{\vec{r},\tilde{r}}$ is the core.
\end{remark}

\subsection{Multilinear fractional sparse $T1$ theorem}
~~

Next, since operators in this subsection always assumes a definite kernel representation, the operator only depends on the fractional total order $\eta$ of the kernel. For convenience, we replace $T_{\vec \eta}$ with $T_{\eta}$.

Under the setting of $(\rn, |\cdot|, dx,\d)$, cf. Remark \ref{example_0}. The multilinear fractional sparse $T1$ theorem for the generalized commutators of {\rm $m$-FBMOO} is established as follows.

\begin{theorem}\label{thm:sparse-T1}
Let \( \eta \in [0,m) \) and \( \mathbf{k} \in \mathbb{N}_0^m \). Let \( T_\eta \) be a \( \mathbb{V} \)-valued \textit{$m$-FBMOO} on \( (\mathbb{R}^n, |\cdot|, dx, \mathcal{D}) \), which admits a \textbf{\textsf{$\mathbb{V}$-valued $m$-linear fractional Calderón-Zygmund kernel representation}} (cf. Definition \ref{def:full}).  Suppose that
\begin{align}
\mathcal{N} := {\left\| T_{\eta}^{\bf{b,k}} \right\|_{{WBP}_{\eta}^{\bf b,k}(\mathbb{V})}} + \sum\limits_{j = 0}^m {{{\left\| {T_{\eta}^{{\bf{b,k}},j*}(1, \cdots ,1)} \right\|}_{\BMO_{2,\eta}(\mathbb{V})}}}. \label{cond.N}
\end{align}
If $\vec b \in (\BMO)^m$ and $\mathcal{N} < \infty$, then there is a sparse family $\calS \subseteq \d$, such that for any nonnegative $f_i \in L_b^\infty$, $i=1,\cdots,m$, and nonnegative $\psi \in L_b^\infty$,
\begin{align*}
\left\langle {\Big\|T_{\eta }^{\bf{b,k}} (\vec f)\Big\|_{\mathbb{V}},\psi} \right\rangle { \lesssim _{\eta ,n,\delta ,\mathcal{N} }} \prod\limits_{i = 1}^m {\left\| {{b_i}} \right\|_{\BMO}^{{k_i}}} \cdot{{\calA}_{\eta ,\calS,\vec 1,1}}\left( {\vec f,\psi} \right).
\end{align*}

where
$
\calA_{\eta,\mathcal{S},\vec{1},1}\left(\vec{f}, \psi\right):=\sum_{Q \in \mathcal{S}}|Q|^{\eta+1}\prod_{i=1}^m\langle | f_i| \rangle_Q\langle |\psi| \rangle_Q .
$
\end{theorem} 

\begin{remark}\label{rem:key3}
We note that once an {\rm $m$-FBMOO} has a fractional Calderón-Zygmund kernel representation, we can replace the condition $W_{\vec{r},\tilde{r}}$ by the condition $\mathcal{N}$, which is indeed more interesting than the previous sparse bounds results.
\end{remark}

Theorem \ref{thm:sparse-T1} indeed can be deduced from the dyadic representation result--Theorem \ref{thm:DT} and the dyadic sparse bounds result--Theorem \ref{thm:d.s.d.} for dyadic model operators. 
Thus, it suffices to prove the two theorems in Sect. \ref{T1}.

\begin{theorem}\label{thm:DT} 
Let \( \eta \in [0,m) \) and \( \mathbf{k} \in \mathbb{N}_0^m \). Fix a random dyadic grid \( \mathcal{D}_\omega \), and let \( T_\eta \) be a \( \mathbb{C} \)-valued \textit{$m$-FBMOO} on \( (\mathbb{R}^n, |\cdot|, dx, \mathcal{D}_\omega) \) as in Definition \ref{def:FBMOO}, which admits a \textbf{\textsf{$\mathbb{C}$-valued $m$-linear Calderón-Zygmund kernel representation}}. Suppose that
\[\mathcal{N} = {\left\| T_{\eta}^{\bf{b,k}} \right\|_{{{WBP}_{{\eta}}^{\bf b, k}}(\mathbb{C})}} + \sum\limits_{j = 0}^m {{{\left\| {T_{\eta}^{{\bf{b,k}},j*}(1, \cdots ,1)} \right\|}_{\BMO_{2,\eta}(\mathbb{C})}}}.\]
If $\vec b \in (\BMO)^m$ and $\mathcal{N} < \infty$, then 
$T_{\eta}^{\bf{b,k}}$ admits an $m$-linear dyadic  representation. That is, there is a constant $C_{n,\delta}<\infty$, such that for every $f_i \in L_b^\infty$, $i=1,\cdots,m$, and $\psi \in L_b^\infty$, we have
\begin{align*}
\left\langle {T_{\eta }^{\bf{b,k}}(\vec{f}),\psi} \right\rangle  = C_{\eta,n,\delta,\mathcal{N},C_{K_\eta}}\cdot{\mathbb{E}_\omega }\sum\limits_{l = 0}^\infty  {{2^{ - \delta k/2}}\left\langle {U_{\eta, \omega }^{{\bf{b,k}},l}\left( {\vec f} \right),\psi} \right\rangle }  
\end{align*}
where $U_{\eta, \omega}^{{\bf{b,k}},l}$ is a finite sum of $U_{\eta, \omega}^{{\bf{b,k}},i, l}$ that denotes some fractional dyadic model operators and their adjoints of such operators, cf. Sect. \ref{T1}.
\end{theorem}

\subsection{Multiparameter multilinear fractional weight class and multilinear off-diagonal extrapolation theorem}~~

In this subsection, we first introduce a class of multiparameter multilinear fractional weights.

\begin{definition}\label{def:weight}
Let $(X, \mu,\B)$ be a ball-basis measure space. 
Let $\eta \in [0,m)$, $\vec{r}=(r_i)_{i=1}^m$ with $0<r_1, \ldots, r_m <\infty$, $\vec{p}=(p_i)_{i=1}^m$ with $0 <p_1, \ldots, p_m,s \leq \infty$, and $\frac{1}{\tilde{p}}=\sum\limits_{i=1}^m \frac{1}{p_i} - \eta $ such that $(\vec{r}, s) \preceq (\vec{p},\tilde{p})$.  

We say that $\vec{\omega} \in A^{\B}_{(\vec{p},\tilde{p}),(\vec{r}, s)}$, if
$$
[\vec{\omega}]_{A^{\B}_{(\vec{p},\tilde{p}),(\vec{r}, s)}}:=\sup _{B \in \B}\left(\prod_{j=1}^m\left\langle \omega_j^{-1}\right\rangle_{{\frac{1}{\frac{1}{r_j}-\frac{1}{p_j}}},B}\right)\langle \omega \rangle_{\frac{1}{\frac{1}{\tilde{p}}-\frac{1}{s}},B}<\infty,
$$
where $\omega:=\prod\limits_{i = 1}^m {{\omega _i}}$.
We say that $\vec{\omega} \in A^{\star,\B}_{(\vec{p},\tilde{p}),(\vec{r}, s)}$, if
\begin{align*}
[\vec{\omega}]_{A^{\star,\B}_{(\vec{p},\tilde{p}),(\vec{r}, s)}}&:=[\omega^{\frac{1}{\tilde{p}}},\omega_1^{\frac{1}{p_1}},\ldots,\omega_m^{\frac{1}{p_m}}]^{\tilde{p}}_{A^{\B}_{(\vec{p},\tilde{p}),(\vec{r}, s)}}\\
	&=\sup _{B \in \B}\left(\prod\limits_{j=1}^m\left\langle \omega_j^{-\frac{1}{p_j}}\right\rangle^{\tilde{p}}_{{\frac{1}{\frac{1}{r_j}-\frac{1}{p_j}}},B}\right)\langle \omega^{\frac{1}{\tilde{p}}} \rangle^{\tilde{p}}_{\frac{1}{\frac{1}{\tilde{p}}-\frac{1}{s}},B}<\infty,
\end{align*}
where $\omega:=\prod\limits_{i=1}^{m}\omega_i^{\frac{\tilde{p}}{p_i}}$. 

For the sake of convenience, we denote $A^{\B}_{(\vec{p},\tilde{p}),(\vec{r}, s)}$ by $A_{(\vec{p},\tilde{p}),(\vec{r}, s)}$ and denote $A^{\star,\B}_{(\vec{p},\tilde{p}),(\vec{r}, s)}$ by $A^{\star}_{(\vec{p},\tilde{p}),(\vec{r}, s)}$.

	
\end{definition}

\begin{remark}\label{transform.}
Based this definition, set $r_{m+1} = s'$, $p_{m+1} = \tilde{p}'$, and $\omega_{m+1}=\omega^{-1}$, and \(\vv{\boldsymbol{p}} = (p_1, \dots, p_{m+1})\), and then 
\begin{align*}
[\vec{\omega}]_{A^{\B}_{(\vec{p},\tilde{p}),(\vec{r}, s)}} & =\sup _{B \in \mathcal{B}} \prod_{j=1}^{m+1}\left\langle \omega_j^{-1}\right\rangle_{{\frac{1}{\frac{1}{r_j}-\frac{1}{p_j}}},B} = \left[\left(\omega_1, \ldots, \omega_{m+1}\right)\right]_{\left(\left(p_1, \ldots, p_{m+1}\right),\tilde{p}\right),\left(\left(r_1, \ldots, r_{m+1}\right), \infty\right)}\\
& =:[\vv{\boldsymbol{\omega}}]_{(\vv{\boldsymbol{p}},\tilde{p}),(\vv{\boldsymbol{r}},\infty)}.
\end{align*}
\end{remark}

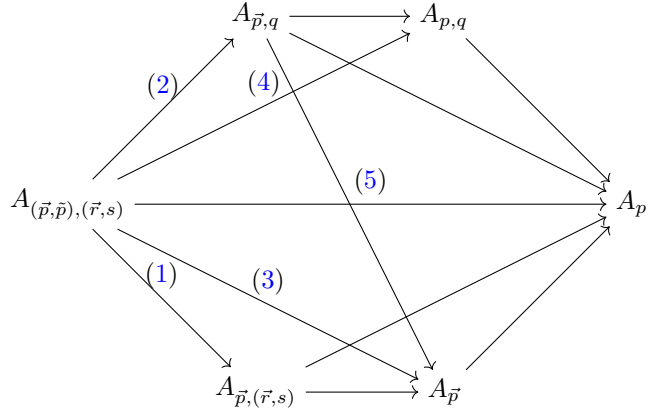
\begin{figure}[!h]
	\begin{center}
		\begin{tikzpicture}
		\node (1) at(0,0) {$A^{}_{(\vec{p},\tilde{p}),(\vec{r}, s)}$};
		\node (2) at(2.5,2.5) {$A_{\vec p, q}$};
      \node (3) at(5,2.5) {$A_{p, q}$};
		\node (4) at(2.5,-2.5) {$A^{}_{\vec{p},(\vec{r}, s)}$};
		\node (5) at(5,-2.5) {$A_{\vec{p}}$};
		\node (6) at(7.5,0) {$A_{p}$};
		\draw[->] (1)--(2) node[midway, above] {\eqref{we_2}};
      \draw[->] (1)--(3) node[midway, above] {\eqref{we_5}};
      \draw[->] (1)--(5) node[midway, above] {\eqref{we_3}};
      \draw[->] (1)--(6) node[midway, above] {\eqref{we_6}};
		\draw[->] (1)--(4) node[midway, above] {\eqref{we_1}};
		\draw[->] (2)--(3);
		\draw[->] (2)--(5);
		\draw[->] (4)--(5);
      \draw[->] (4)--(6);
		\draw[->] (3)--(6);
		\draw[->] (5)--(6);
      \draw[->] (2)--(6);
		\end{tikzpicture}
	\end{center}
	\caption{The relationships between weights}\label{figure1}
\end{figure}   
\begin{remark}
This definition is a very broad extension. $ A^{}_{(\vec{p},\tilde{p}),(\vec{r}, s)}$ can be reduced to the classical weight classes, see Fig.\ref{figure1}. We set $X=\mathbb{R}^n$ in the following cases.

\begin{list}{\rm (\theenumi)}{\usecounter{enumi}\leftmargin=1.2cm \labelwidth=1cm \itemsep=0.2cm \topsep=.2cm \renewcommand{\theenumi}{\arabic{enumi}}}
\item \label{we_1} When $p=\tilde{p}$, then $A^{}_{(\vec{p},\tilde{p}),(\vec{r}, s)}(X) = A^{}_{\vec{p},(\vec{r}, s)}(\rn)$ introduced in \cite{Nie2018}.

\item \label{we_2} When $\vec r = \{r_i\}_{i=1}^m= \{1\}_{i=1}^m$ and $s = \infty$, then
$A^{}_{(\vec{p},\tilde{p}),(\vec{r}, s)}(X) = A^{}_{\vec{p},q}(\rn)$ introduced in \cite{Moen09}.

\item  \label{we_3} When $\vec r = \{r_i\}_{i=1}^m= \{1\}_{i=1}^m$, $\tilde{p}=p$, and $s = \infty$, then
$A^{}_{(\vec{p},\tilde{p}),(\vec{r}, s)} (X)= A^{}_{\vec{p}}(\rn)$ introduced in \cite{Perez2009}.

\item \label{we_5} When $\vec p = \{p_i\}_{i=1}^m= \{p\}_{i=1}^m$, $\vec r = \{r_i\}_{i=1}^m= \{1\}_{i=1}^m$, and $s = \infty$, then
$A^{}_{(\vec{p},\tilde{p}),(\vec{r}, s)}(X) = A^{}_{{p},q}(\rn)$ introduced in \cite{Muc1974}.

\item \label{we_6} When $\vec p = \{p_i\}_{i=1}^m= \{p\}_{i=1}^m$, $\vec r = \{r_i\}_{i=1}^m= \{1\}_{i=1}^m$, $p=q$, and $s = \infty$, then
$A^{}_{(\vec{p},\tilde{p}),(\vec{r}, s)}(\rn) = A^{}_{{p}}(\rn)$ introduced in \cite{Muc1972}.
\end{list}
\end{remark}

The following remark is important for the characterization of weights in quantitative estimation.
\begin{remark}
Let $m \in \N$, $\eta \in [0,m)$, $ r_1,\cdots,r_m ,s' \in [1,\infty)$, $p_1,\dots,p_m \in (1,\infty)$, and $\frac{1}{{{\tilde{p}}}}:=\sum\limits_{i = 1}^m {\frac{1}{{{p_i}}}}- \eta \in (0,1)$. Moreover, let $\omega_1,\ldots,\omega_m$ are weights and define $\omega:=\prod_{i=1}^{m}\omega_i$, $v:=\omega^{\frac{s'\tilde{p}'}{\tilde{p}'-s'}}$ ,and ${v_j}: = \omega _j^{\frac{{{r_j}{p_j}}}{{{r_j} - {p_j}}}}$ for $j= 1,\ldots,m$. Then 
\begin{itemize}
\item[(1)]  $\vec{\omega} \in A^{\B}_{(\vec{p},\tilde{p}),(\vec{r}, s)}$ if and only if $v_1, \ldots, v$ are locally integrable and there is a constant $c>0$ such that for all $B \in \B$
\begin{align}\label{def.of.weighted.condi.}
\left\langle v\right\rangle_{1, B}^{\frac{1}{s'}}\left(\prod_{j=1}^{m}\left\langle v_j\right\rangle_{1,B}^{\frac{1}{r_j}}\right)\mu(B)^{1 + \eta} \leq [\vec{\omega}]_{A^{\B}_{(\vec{p},\tilde{p}),(\vec{r}, s)}} v(B)^{\frac{1}{\tilde{p}'}} \prod_{j=1}^{m} v_j(B)^{\frac{1}{p_j}}.
   \end{align}
   \item[(2)] Let $B \in \B$ and $E_B \subseteq B$ such that $\mu(B) \leq \delta\mu(E_B)$ with $\delta \in (0,1)$. Then
\begin{align}\label{def.of.weighted.condi._2}
\left\langle v\right\rangle_{1, B}^{\frac{1}{s'}}\left(\prod_{j=1}^{m}\left\langle v_j\right\rangle_{1,B}^{\frac{1}{r_j}}\right)\mu(B)^{1 + \eta} \leq C_{\delta} [\vec{\omega}]_{A^{\B}_{(\vec{p},\tilde{p}),(\vec{r}, s)}}^{\varTheta} v(B)^{\frac{1}{\tilde{p}'}} \prod_{j=1}^{m} v_j(E_B)^{\frac{1}{p_j}},
   \end{align}
   where $C_\delta = \delta^{(\varTheta - 1)\left[\sum\limits_{i=1}^m\left(\frac{p_i - r_i}{r_ip_i}\right)+ \frac{\tilde{p}'-s'}{\tilde{p}'s'}\right]}$ and $\varTheta =\max \left\{ {\frac{{{p_1}}}{{{p_1} - {r_1}}}, \cdots ,\frac{{{p_m}}}{{{p_m} - {r_m}}},\frac{{\tilde{p}'}}{{\tilde{p}' - s'}}} \right\}$.
\end{itemize}
\end{remark}

Next, we establish multilinear off-diagonal extrapolation theorem for multiparameter multilinear fractional weight class $A_{(\vec{p},\tilde{p}), (\vec{r},s)}$.
\begin{theorem}\label{thm:Ex_1}
Let $(X,d,\mu,\mathcal{D})$ be a space of homogeneous type. Let $\eta \in [0,m)$, $\vec{r} = (r_i)_{i=1}^m \in [1, \infty)^m$, and \(s' \in [1, \infty)\). Let $\mathcal{G}$ be a collection of $(m+1)$--tuples of non-negative functions. Given $\vec{p} := (p_i)_{i=1}^m$, $\frac{1}{\tilde{p}} := \sum\limits_{i=1}^{m}\frac{1}{p_i} - \eta>0$ satisfying \((\vec{r}, s)  \preceq (\vec{p}, \tilde{p})\), assume that for all $\vec{\omega}= (\omega_i)_{i=1}^m \in A_{(\vec{p},\tilde{p}), (\vec{r},s)}$, 
\begin{align}\label{eq:iioo:1.2}
	\|g\|_{L^{\tilde{p}}(\omega^{\tilde{p}})} \le C\left([\vec \omega]_{A_{(\vec{p},\tilde{p}),(\vec r,s)}}\right) \prod_{i=1}^m \|f_i\|_{L^{p_i}(\omega_i^{p_i})}, \quad \text{for every } (g, f_1, \dots, f_m) \in \mathcal{G}, 
	\end{align}
	where $\omega=\prod\limits_{i=1}^m \omega_i$. Then, for all $\vec{q}:= (q_i)_{i=1}^m$, $\frac{1}{\tilde{q}} := \sum\limits_{i=1}^{m}\frac{1}{q_i} - \eta>0$ with $(\vec{r},s) \prec (\vec{q},\tilde{q})$ and for all $\vec{v}=(v_i)_{i=1}^m \in A_{(\vec{q},\tilde{q}), (\vec{r},s)}$, 
	\begin{align}\label{eq:iioo:1.3}
	\|g\|_{L^{\tilde{q}}(v^{\tilde{q}})} \le C\left([\vec v]_{{(\vec{q},\tilde{q}), (\vec{r},s)}}\right) \prod_{i=1}^m \|f_i\|_{L^{q_i}(v_i^{q_i})}, \quad \text{for every }(g, f_1, \dots, f_m) \in \mathcal{G}, 
	\end{align}
	where $v=\prod\limits_{i=1}^m v_i$. 
Moreover, for the same family of exponents and weights, and for all exponents $\vec{t}= \left(t_1, \ldots, t_m\right)$ with $(\vec{r},s) \prec (\vec{t},\tilde{t})$
\begin{align}\label{eq:iioo:1.4}
\left\|\left(\sum_j\left(g^j\right)^{\tilde{t}}\right)^{\frac{1}{\tilde{t}}}\right\|_{L^{\tilde{q}}(v^{\tilde{q}})}  
\leq C\left([\vec{v}]_{A_{(\vec{q},\tilde{q}), (\vec{r},s)}}\right) 
\prod_{i=1}^m\left\|\left(\sum_j\left(f_i^j\right)^{t_i}\right)^{\frac{1}{t_i}}\right\|_{L^{q_i}\left(v_i^{q_i}\right)},
\end{align}
for all $\left\{\left(g^j, f_1^j, \ldots, f_m^j\right)\right\}_j \subseteq \mathcal{G}$ and where $\frac{1}{\tilde{t}}:=\frac{1}{t_1}+\cdots+\frac{1}{t_m}-\eta>0$.

\end{theorem}

\begin{remark}\label{li:remark:1.8}
Fixed $i \in \{1,\dotsc,m\}$, if we let $p_i = r_i$ in Theorem~\ref{thm:Ex_1}, then we can replace $q_i > r_i$ with $q_i \geq r_i$ to ensure that our conclusion is also valid. 

To justify this relaxation, we proceed in two steps:
\begin{itemize}
    \item On the one hand, for $q_i = p_i = r_i$, the result follows immediately from \eqref{eq:iioo:1.2} combined with  $\vec{r} \preceq \vec{q}$.
    \item On the other hand, for $q_i > p_i = r_i$, the result will be obtained by applying the extrapolation argument specifically to the $i$-th component.
\end{itemize}
\end{remark}

\begin{remark}\label{R:1}
Similar to the classical vector-valued extropolation theory, we can conclude that \eqref{eq:iioo:1.3} is equivalent to \eqref{eq:iioo:1.4}. We just show that \eqref{eq:iioo:1.3} implies \eqref{eq:iioo:1.4} since \eqref{eq:iioo:1.3} is obviously deduced from \eqref{eq:iioo:1.4}.

In fact, assume that \eqref{eq:iioo:1.3} is valid and fix $\vec{t}=\left(t_1, \ldots, t_m\right)$ with $(\vec{r},s) \prec (\vec{t},\tilde{t})$ where $\frac{1}{\tilde{t}}:=\sum_{i=1}^m \frac{1}{t_i} - \eta>0$. Define a new family $\mathcal{F}_{\vec{t}}$ consisting on the $m+1$-tuples of the form
\[
\left(G, F_1, \ldots, F_m\right)=\left(\left(\sum_j\left(g^j\right)^{\tilde{t}}\right)^{\frac{1}{\tilde{t}}},\left(\sum_j\left(f_1^j\right)^{t_1}\right)^{\frac{1}{t_1}}, \ldots,\left(\sum_j\left(f_m^j\right)^{t_m}\right)^{\frac{1}{t_m}}\right)
\]
where $\left\{\left(g^j, f_1^j, \ldots, f_m^j\right)\right\}_j \subseteq \mathcal{F}$. 
By monotone convergence theorem, it suffices to consider that all of the sums are finite. For all $\vec{v} \in A_{(\vec{t},\tilde{t}), (\vec{r},s)}$ with $v:=\prod_{i=1}^m v_i$ and every $\left(G, F_1, \ldots, F_m\right) \in \mathcal{F}_{\vec{t}}$, it follows from \eqref{eq:iioo:1.3} with $\vec{t}$ in place of $\vec{q}$ and Hölder's inequality that
\begin{align}\notag
\|G\|_{L^{\tilde{t}}(v^{\tilde{t}})}&=\left(\sum_j\left\|g^j\right\|_{L^{\tilde{t}}(v^{\tilde{t}})}^{\tilde{t}}\right)^{\frac{1}{\tilde{t}}} \lesssim\left(\sum_j \prod_{i=1}^m\left\|f_i^j\right\|_{L^{\tilde{t_i}}\left({v_i}^{t_i}\right)}^{\tilde{t}}\right)^{\frac{1}{\tilde{t}}} \\ \label{eq:iioo:4.8}
& \leq \prod_{i=1}^m\left(\sum_j\left\|f_i^j\right\|_{L^{t_i}\left({v_i}^{t_i}\right)}^{\tilde{t_i}}\right)^{\frac{1}{\tilde{t_i}}}
\le \prod\limits_{i = 1}^m {{{\left( {\sum\limits_j {\left\| {f_i^j} \right\|_{{L^{{t_i}}}\left( {{v_i}^{t_i}} \right)}^{{t_i}}} } \right)}^{\frac{1}{{{t_i}}}}}}
=\prod_{i=1}^m\left\|F_i\right\|_{L^{t_i}\left({{v_i}^{t_i}}\right)},
\end{align}
where $\frac{1}{{{\tilde t_i}}}: = \frac{1}{{{t_i}}} - \frac{\eta }{m} \le \frac{1}{{{t_i}}}.$

Thus, \eqref{eq:iioo:1.3} is valid for $\mathcal{F}_{\vec{t}}$ and then \eqref{eq:iioo:1.4} derives from applying \eqref{eq:iioo:1.3} again for all $\vec{v} \in A_{(\vec{q},\tilde{q}), (\vec{r},s)}$ with $v:=\prod_{i=1}^m v_i$.
\end{remark}

Based on Remark \ref{R:1}, we further have
\begin{remark}
Under the same assumption of Theorem \ref{thm:Ex_1}, for the same family of exponents and weights, and for all exponents $\vec{R}= \left(R_1, \ldots, R_m\right)$ with $(\vec{r},s) \prec (\vec{R},\tilde{R})$, it follows from \eqref{eq:iioo:1.4} that
\begin{align}\label{eq:iioo:1.5}
{\left\| {{{\left( {\sum\limits_k {{{\left( {\sum\limits_j {{{\left( {{g^{j,k}}} \right)}^{\tilde t}}} } \right)}^{\frac{{\tilde R}}{{\tilde t}}}}} } \right)}^{\frac{1}{{\tilde R}}}}} \right\|_{{L^{\tilde q}}({v^{\tilde q}})}} \le C\left( {{{[\vec v]}_{{A_{(\vec q,\tilde q),(\vec r,s)}}}}} \right)\prod\limits_{i = 1}^m {{{\left\| {{{\left( {\sum\limits_k {{{\left( {\sum\limits_j {{{\left( {f_i^{j,k}} \right)}^{{t_i}}}} } \right)}^{\frac{{{R_i}}}{{{t_i}}}}}} } \right)}^{\frac{1}{{{R_i}}}}}} \right\|}_{{L^{{q_i}}}\left( {v_i^{{q_i}}} \right)}}},
\end{align}
for all $\left\{\left(g^{j,k}, f_1^{j,k}, \ldots, f_m^{j,k}\right)\right\}_{j,k} \subseteq \mathcal{G}$ and where $\frac{1}{\tilde{R}}:=\frac{1}{R_1}+\cdots+\frac{1}{R_m}-\eta>0$.
\end{remark}

We can apply Theorems \ref{est:form} and \ref{thm:Ex_1} to obtain
\begin{corollary}\label{cor:key1}
Let $(X,d,\mu,\mathcal{D})$ be a space of homogeneous type. Let $m \in \N$, $\eta \in [0,m)$, $r_1,\cdots,r_m ,s' \in [1,\infty)$, $p_1,\dots,p_m \in (1,\infty)$, and $\frac{1}{{{\tilde{p}}}}:=\sum\limits_{i = 1}^m {\frac{1}{{{p_i}}}}- \eta >0$. 
Assume that $T_\eta$ is an $m$-sublinear operator satisfying for any nonnegative functions $f_1,\cdots,f_m, \psi \in L_b^\infty(X)$,  
\[\left| {\left\langle {T_{\eta}(\vec f),\psi } \right\rangle } \right| \lesssim \sup_{\mathcal{S} \subseteq \mathcal{D}} \mathcal{A}_{\eta, \mathcal{S},\vec{r}, s'}(\vec{f},\psi).\]
If $(\vec{r}, s) \prec (\vec{p},\tilde{p})$ and $\vec{\omega} \in A_{(\vec{p},\tilde{p}),(\vec{r}, s)}$, then 
\begin{align*}
{\left\| T_{\eta} \right\|_{\prod\limits_{i = 1}^m {{L^{{p_i}}}(\omega _i^{{p_i}}) \to } {L^{\tilde{p}}}({\omega ^{\tilde{p}}})}}
&\lesssim 
\begin{cases}
[\vec{\omega}]^{\varTheta}_{A_{(\vec{p},\tilde{p}),(\vec{r}, s)}} & \text{if } \tilde{p}>1;\\
C\left( [\vec{\omega}]_{A_{(\vec{p},\tilde{p}),(\vec{r}, s)}}\right) & \text{if } 0<\tilde{p} \le 1,
\end{cases}
\end{align*}
where $\varTheta:=\max \left\{ {\frac{{{p_1}}}{{{p_1} - {r_1}}}, \cdots ,\frac{{{p_m}}}{{{p_m} - {r_m}}},\frac{{\tilde{p}'}}{{\tilde{p}' - s'}}} \right\}$.

Further, for the same family of exponents and weights, and for all exponents $\vec{t}= \left(t_1, \ldots, t_m\right)$ with $(\vec{r},s) \prec (\vec{t},\tilde{t})$,
\begin{align*}
\left\| \left\{ T_{\eta}(f_1^j, \cdots , f_m^j) \right\}_{\ell^{\tilde t}} \right\|_{L^{\tilde{p}}(\omega^{\tilde{p}})} 
&\lesssim \prod_{i=1}^m \left\| \left\{ f_i^j \right\}_{\ell^{t_i}} \right\|_{L^{p_i}(\omega_i^{p_i})}
\end{align*}

\end{corollary}

The quantitative multilinear off-diagonal extrapolation for the generalized commutators is given as the end of this part.
\begin{theorem}\label{op.to.jh}
Let $m \in \N$, $\eta \in [0,m)$, $r_1,\cdots,r_m ,s' \in [1,\infty)$, $p_1,\dots,p_m \in [1, \infty)$ and $\frac{1}{{\tilde{p}}}:=\sum\limits_{i = 1}^m {\frac{1}{{{p_i}}}} - \eta >0$, such that $(\vec{r}, s) \prec (\vec{p}, \tilde{p})$. Let $T_{\vec \eta}$ be an $m$--sublinear operator. 
Moreover, suppose that there exists an increasing functions $\phi:[1, \infty) \rightarrow[0, \infty)$. If for all weights $\vec{\omega} \in A_{(\vec{p},\tilde{p}),(\vec{r},s)}$ with $\omega:=\prod\limits_{i=1}^m \omega_i$, then 
\begin{align}\label{1t1}
	\left\|{T}_{\vec{\eta}}(\vec{f})\right\|_{L^{\tilde{p}}(\omega^{\tilde{p}})} \lesssim \phi\left([\vec \omega]_{A_{(\vec{p},\tilde{p}),(\vec{r},s)}}\right) \prod_{i=1}^m\left\|f_i\right\|_{L^{p_i}(\omega_i^{p_i})},
\end{align}
If ${\bf k }\in \N_0^m$ and ${\bf b}=\left(b_1, \ldots, b_m\right) \in \mathrm{BMO}^m$, then
\begin{align}\label{1tb}
\left\|{T}_{\vec{\eta}}^{{\bf b, k}}(\vec{f})\right\|_{L^{\tilde{p}}(\omega^{\tilde{p}})} \lesssim {\bf k}!{\phi}\left(\boldsymbol{c}_0[\vec \omega]_{A_{(\vec{p},\tilde{p}),(\vec{r},s)}}\right) [\vec{\omega}]^{{|\bf k}|\Xi}_{A_{(\vec{p},\tilde{p}),(\vec{r}, s)}}\prod_{i=1}^m\left\|b_i\right\|_{\mathrm{BMO}}^{k_i}\left\|f_i\right\|_{L^{p_i}(\omega_i^{p_i})},
\end{align}
where $\boldsymbol{c}_0=4^{ \sum_{i=1}^m \min \left\{p_i / \delta_i, p/ \delta_{m+1}\right\} / p_i}$ with $\frac{1}{p}:=\sum\limits_{i = 1}^m {\frac{1}{{{p_i}}}}$, $\Xi=\max \left\{ {\frac{{{p_1r_1}}}{{{p_1} - {r_1}}}, \cdots ,\frac{{{p_mr_m}}}{{{p_m} - {r_m}}},\frac{{\tilde{p}'s'}}{{\tilde{p}' - s'}}} \right\}$, and the notation is defined below Sect. \ref{Sec.extro}. 

Moreover, set $\vec{t}=(t_1,\cdots,t_m)$ and $\frac{1}{\tilde{t}}:=\sum_{i=1}^{m}\frac{1}{t_i}-\eta>0$ satisfying  $(\vec{r},s) \prec (\vec{t},\tilde{t})$. Then for all $(\vec{r}, s) \preceq (\vec{q}, \tilde{q})$ and for all $\vec{v} \in A_{(\vec{q},\tilde{q}),(\vec{r},s)}$ with $v:=\prod\limits_{i=1}^m v_i$, it follows from Theorem \ref{thm:Ex_1} that
\begin{align}\label{1tbv}
\left\|\left\{{T}_{\vec{\eta}}^{{\bf b, k}}\left(f_1^j, f_2^j, \ldots, f_m^j\right)\right\}_{\ell^{\tilde t}}\right\|_{L^{\tilde{q}}(v^{\tilde{q}})}  
\lesssim  \prod_{i=1}^m \|b_i\|_{\mathrm{BMO}}^{k_i} \cdot 
\prod_{i=1}^m\left\|\left\{f_i^j\right\}_{\ell^{\tilde t}}\right\|_{L^{q_i}\left(v_i^{q_i}\right)}.
\end{align}
\end{theorem}


\subsection{Quantitative weighted estimates theorem}\label{weighted-estimates}
~~

\subsubsection{Sharp-type estimates}
~

We can realize the sharp quantitative weighted estimates by following ways.
The first theorem derives from $$\text{Theorem \ref{S.d.main} + Theorem \ref{thm:sparse-dom_1} + Theorem \ref{op.to.jh} + Theorem \ref{thm:Ex_1}}.$$
\begin{theorem}\label{thm:weighted1}
Let $(X,d,\mu,\mathcal{D})$ be a space of homogeneous type. Let $m \in \N$, $\eta \in [0,m)$, $r_1,\cdots,r_m ,s' \in [1,\infty)$, $p_1,\dots,p_m \in (1,\infty)$, and $\frac{1}{{{\tilde{p}}}}:=\sum\limits_{i = 1}^m {\frac{1}{{{p_i}}}}- \eta > 0$.
Set $T_{\vec \eta}$ be a $\mathbb{V}$-valued {\rm $m$-FBMOO} on $(X,d,\mu,\mathcal{D})$ as in Theorem \ref{S.d.main}, and ${T}_{\vec \eta} \in W_{\vec{r},\tilde{r}}$ with $\frac{1}{\tilde{r}} :=\sum\limits_{i=1}^{m}\frac{1}{r_i} -\eta>0$. 
If multi-symbol \(\mathbf{b} = (b_1, \ldots, b_{m}) \in (\BMO(X))^m\), $(\vec{r}, s) \prec (\vec{p},\tilde{p})$ and $\vec{\omega} \in A_{(\vec{p},\tilde{p}),(\vec{r}, s)}$, then  
\begin{align*}
{\left\| T_{\vec{\eta}}^{{\bf b, k}} \right\|_{\prod\limits_{i = 1}^m {{L^{{p_i}}}(\omega _i^{{p_i}}) \to } {L^{\tilde{p}}}({\omega ^{\tilde{p}}},\mathbb{V})}}
&\lesssim 
\prod_{i=1}^m \|b_i\|_{\mathrm{BMO}}^{k_i} \cdot [\vec{\omega}]^{\varTheta+{|\bf k}|\Xi}_{A_{(\vec{p},\tilde{p}),(\vec{r}, s)}},
\end{align*}
where $\varTheta:=\max \left\{ {\frac{{{p_1}}}{{{p_1} - {r_1}}}, \cdots ,\frac{{{p_m}}}{{{p_m} - {r_m}}},\frac{{\tilde{p}'}}{{\tilde{p}' - s'}}} \right\}$, $\Xi=\max \left\{ {\frac{{{p_1r_1}}}{{{p_1} - {r_1}}}, \cdots ,\frac{{{p_mr_m}}}{{{p_m} - {r_m}}},\frac{{\tilde{p}'s'}}{{\tilde{p}' - s'}}} \right\}$, and $\varTheta$ is sharp, i.e. it cannot be reduced.
For all exponents $\vec{t}= \left(t_1, \ldots, t_m\right)$ with $(\vec{r},s) \prec (\vec{t},\tilde{t})$,
\begin{align*}
{\left\| {{{\left\{ {T_{\vec{\eta}}^{{\bf b, k}}(f_1^j, \cdots ,f_m^j)} \right\}}_{{l^{\tilde t}}}}} \right\|_{{L^{\tilde{p}}}\left( {{\omega ^{\tilde{p}}}} ,\mathbb{V}\right)}} \lesssim \prod_{i=1}^m \|b_i\|_{\mathrm{BMO}}^{k_i} \cdot \prod\limits_{i = 1}^m {{{\left\| {{{\left\{ {f_i^j} \right\}}_{{l^{{t_i}}}}}} \right\|}_{{L^{{p_i}}}(\omega _i^{{p_i}})}}}.
\end{align*}
\end{theorem}

The sharpness of $\varTheta$ can be proved as follows. When $(X,d,\mu,\mathcal{D})=(\rn,|\cdot|,dx,\mathcal{D})$ and ${\bf k=0}$, since the $m$-linear fractional integral operator $I_\alpha$ is a $\mathbb{C}$-valued {\rm $m$-FBMOO}, cf. Theorem \ref{FCZO}, the classical sparse operators theory manifests the fact that ${I_\alpha }(\vec f) \approx \sum\limits_{i = 1}^{{3^n}} {{{\mathscr{A}}_{\frac{\alpha}{n} ,\mathcal{S}_i,\vec 1}}(\vec f)}.$  This implies that the sharpness of $\varTheta$ for $I_\alpha$ is the same as that of Theorem \ref{thm:sparse-dom_1}.

The second can be deduced from 
$$\text{Theorem \ref{thm:sparse-T1}+ Corollary \ref{cor:key1} (Theorem \ref{thm:sparse-T1} + Theorem \ref{est:form} +Theorem \ref{thm:Ex_1})}.$$

\begin{theorem}\label{thm:weighted3}
Let \( m \in \mathbb{N} \), \( \eta \in [0,m) \), \( p_1, \dots, p_m \in (1, \infty) \), and \( \frac{1}{{\tilde{p}}} := \sum\limits_{i = 1}^m \frac{1}{{p_i}} - \eta > 0 \). Let \( T_\eta \) be a \( \mathbb{R} \)-valued \rm{$m$-FBMOO} on \( (\mathbb{R}^n, |\cdot|, dx, \mathcal{D}_\omega) \), which admits a \textbf{\textsf{$\mathbb{V}$-valued $m$-linear fractional Calderón-Zygmund kernel representation}}. Suppose that \( \mathbf{k} \in \mathbb{N}_0^m \) and
\[\mathcal{N} := {\left\| T_{\eta}^{\bf{b,k}} \right\|_{{WBP}_{\eta}^{\bf b,k}(\mathbb{V})}} + \sum\limits_{j = 0}^m {{{\left\| {T_{\eta}^{{\bf{b,k}},j*}(1, \cdots ,1)} \right\|}_{BMO_{2,\eta}}}} < \infty.\]
If multi-symbols \(\mathbf{b} = (b_1, \ldots, b_{m}) \in (\BMO)^m\) and $\vec{\omega} \in A_{\vec{p},\tilde{p}}$, then 
\begin{align*}
{\left\| T_{\eta}^{{\bf b, k}} \right\|_{\prod\limits_{i = 1}^m {{L^{{p_i}}}(\omega _i^{{p_i}}) \to } {L^{\tilde{p}}}({\omega ^{\tilde{p}}},\mathbb{V})}}
&\lesssim 
\begin{cases}
\prod_{i=1}^m \|b_i\|_{\mathrm{BMO}}^{k_i} \cdot [\vec{\omega}]^{\varTheta}_{A_{(\vec{p},\tilde{p})}} & \text{if } \tilde{p}\ge1;\\
\prod_{i=1}^m \|b_i\|_{\mathrm{BMO}}^{k_i} \cdot C\left( [\vec{\omega}]_{A_{(\vec{p},\tilde{p})}}\right) & \text{if } 0<\tilde{p} < 1,
\end{cases}
\end{align*}
where $\varTheta:= \max \left\{ {{p_1'}, \cdots ,{p_m'},\tilde p} \right\}$. Indeed, $\varTheta$ is still sharp.

Further, for the same family of exponents and weights, and for all exponents $1<t_1,\dots,t_m<\infty$, and $\frac{1}{{{\tilde{t}}}}:=\sum\limits_{i = 1}^m {\frac{1}{{{t_i}}}}- \eta >0$,
\begin{align*}
\left\| \left\{ T_{\eta}^{{\bf b, k}}(f_1^j, \cdots , f_m^j) \right\}_{l^{\tilde{t}}} \right\|_{L^{\tilde{p}}(\omega^{\tilde{p}},\mathbb{V})} 
&\lesssim \prod_{i=1}^m \| b_i \|_{\mathrm{BMO}}^{k_i} \cdot \prod_{i=1}^m \left\| \left\{ f_i^j \right\}_{l^{t_i}} \right\|_{L^{p_i}(\omega_i^{p_i})}
\end{align*}

\end{theorem}

\begin{remark}
In obtaining Theorem \ref{thm:weighted3}, we used the fact that ${A_{(\vec p,\tilde p),(\vec 1,\infty )}} = {A_{\vec p,\tilde p}}$ introduced by Moen \cite{Moen09} (2009).
\end{remark}

\begin{remark}
We roughly make this judgment: considering sharp weighted bounds of some appropriate operators and their generalized commutators, we have that
\begin{align*}
&\text{\bf\small The sharp weighted bounds for the generalized commutators} \\
\approx &\text{\bf\small (1) pointwise sparse bounds + sharp weighted estimates for sparse operators} \\
&\text{\bf\small + extrapolation}\\
\approx &\text{\bf\small (2) form sparse bounds + sharp weighted estimates for sparse forms} \\
&\text{\bf\small + extrapolation}
\end{align*}
 It is exactly idea (2) that has significantly expedited the research on form sparse bounds over the past 10 years, since many special operators (e.g. some Fourier multipliers\cite{Chen2018}, Bochner-Risez multipliers \cite{Lacey2019a}, Hilbert transforms along curves\cite{Cladek2018}, spherical maximal operators \cite{Lacey2019b,Borges2023}) may not be $m$-FBMOOs, and thus the acquisition of pointwise sparse bounds is more difficult than form sparse bounds.
\end{remark}

\subsubsection{Bloom-type estimates}~~

We introduce the weighted $\BMO$ space as follows.
\begin{definition}\label{MWbmo}
Let $(X,d,\mu,\mathcal{D})$ be a space of homogeneous type, $\omega$ is a weight. Let $b \in L^1_{\text{loc}}(X)$, we say $b \in {BMO}_\omega(X)$  if 
  \[
  \|b\|_{{ BMO}_\omega(X)} := \sup_{B \subseteq X} \frac{1}{\omega(B)} \int_B |b(x) - b_B| \, d\mu(x) < \infty,
  \]
  where
  \[
  b_B := \frac{1}{\mu(B)} \int_B b(x) \, d\mu(x).
  \]
    \end{definition}

Our Bloom-type estimates is deduced from the Theorems \ref{S.d.main} and \ref{main.bloom} as follows.

\begin{theorem}\label{thm:bloom}
 Let $(X,d,\mu,\mathcal{D})$ be a space of homogeneous type. Let $m \in \N$, $\eta \in [0,m)$, $r_1, \cdots, r_m, s^{\prime} \in [1,\infty)$, $p_1,\dots,p_m \in (1,\infty)$, and $\frac{1}{{{\tilde{p}}}}:=\sum\limits_{i = 1}^m {\frac{1}{{{p_i}}}}- \eta \in (0,1]$.
Set $T_{\vec \eta}$ be a $\mathbb{V}$-valued {\rm $m$-FBMOO} on $(X,d,\mu,\mathcal{D})$ as in Theorem \ref{S.d.main}, and ${T}_{\vec \eta} \in W_{\vec{r},\tilde{r}}$ with $\frac{1}{\tilde{r}} :=\sum\limits_{i=1}^{m}\frac{1}{r_i} -\eta>0$. 
Set \(\mathbf{t}\) and \(\mathbf{k}\) be multi-indices satisfying \(\vec{1} \le \mathbf{k}\), and
$R := \sum\limits_{i=1}^{m} k_i  - 1$. Suppose that \(\mu_0^{-\tilde{p}'}, \omega^{-\tilde{p}'} \in A_{\tilde{p}'/s'}(X)\), \(\mu_i^{p_i}, \omega_i^{p_i} \in A_{p_i/r_i}(X)\), and the bloom type weight
$\varphi_i := \left(\omega_i/\mu_i\right)^{\frac{1}{k_i}}$, for \(i = 1,\ldots,m\).
If multi-symbol \(\mathbf{b} = (b_1, \ldots, b_{m}) \in (\BMO(X))^m\), 
then there exists a sparse family $\tilde{\mathcal{S}}$ satisfying ${\mathcal{S}} \subseteq \tilde{\mathcal{S}} \subseteq \d$, such that
\begin{align}\label{max.weight_}
    &\quad {\left\| T_{\vec{\eta}}^{{\bf b, k}} \right\|_{\prod\limits_{i = 1}^m {{L^{{p_i}}}(\omega _i^{{p_i}}) \to } {L^{\tilde{p}}}({\omega ^{\tilde{p}}},\mathbb{V})}} \lesssim    \prod_{i = 1}^m \|b_i\|_{\mathrm{BMO}_{\varphi_i}(X)}^{k_i}\cdot \mathcal{N}_{0},
    \end{align}
    where $\varpi_i:= \begin{cases}\mu_i, & i \in \mathcal{J}\\ \omega_i, & i \in \mathcal{J}^c,\end{cases}$
 and
  \begin{align*}
    \mathcal{N}_{0}
:=&\left\|\mathcal{A}_{\eta,\S,\vec{r},s'}\right\|_{\prod\limits_{i=1}^{m} L^{p_i}(\varpi^{p_i}_i)\times L^{\tilde{p}'}(\mu_0^{-\tilde{p}'})\rightarrow \mathbb{R}} \cdot \left([\omega^{-\tilde{p}'}]^{\frac{(R+1)s'+1}{2}}_{A_{\tilde{p}/s'}(X)} 
    [\mu_0^{-\tilde{p}'}]^{\frac{(R+1)s'-1}{2}}_{A_{\tilde{p}/s'}(X)}\right)^{\max\left\{ \frac{1}{s'},\frac{1}{\tilde{p}'-s'}\right\} } \\
    &\cdot \prod\limits_{i=1}^m \left([\omega_i^{p_i}]^{\frac{k_ir_i+1}{2}}_{A_{p_i/r_i}(X)} \left[\mu_i^{p_i}\right]^{\frac{k_ir_i-1}{2}}_{A_{p_i/r_i}(X)}\right)^{\max \left\{\frac{1}{r_i}, \frac{1}{p_i-r_i}\right\}}.
  \end{align*}
If we more assume that $\vec{\varpi}:=\left(\varpi_1,\cdots,\varpi_m\right) \in A_{(\vec{p},\tilde{p}),(\vec{r}, s)}$ with $(\vec{r}, s) \prec (\vec{p},\tilde{p})$, then
\begin{align*}
 {\left\| T_{\vec{\eta}}^{{\bf b, k}} \right\|_{\prod\limits_{i = 1}^m {{L^{{p_i}}}(\omega _i^{{p_i}}) \to } {L^{\tilde{p}}}({\omega ^{\tilde{p}}},\mathbb{V})}}\lesssim& \prod_{i = 1}^m \|b_i\|_{\mathrm{BMO}_{\varphi_i}(X)}^{k_i} \cdot [\vec{\varpi}]^{\varTheta}_{A_{(p,\tilde{p}),(\vec r,s)}(X)}
 \cdot \left([\omega^{-\tilde{p}'}]^{\frac{(R+1)s'+1}{2}}_{A_{\tilde{p}/s'}(X)} 
    [\mu_0^{-\tilde{p}'}]^{\frac{(R+1)s'-1}{2}}_{A_{\tilde{p}/s'}(X)}\right)^{\max\left\{ \frac{1}{s'},\frac{1}{\tilde{p}'-s'}\right\} } \\
    &\cdot \prod\limits_{i=1}^m \left([\omega_i^{p_i}]^{\frac{k_ir_i+1}{2}}_{A_{p_i/r_i}(X)} \left[\mu_i^{p_i}\right]^{\frac{k_ir_i-1}{2}}_{A_{p_i/r_i}(X)}\right)^{\max \left\{\frac{1}{r_i}, \frac{1}{p_i-r_i}\right\}},
\end{align*}
where $\varTheta:=\max \left\{ {\frac{{{p_1}}}{{{p_1} - {r_1}}}, \cdots ,\frac{{{p_m}}}{{{p_m} - {r_m}}},\frac{{\tilde{p}'}}{{\tilde{p}' - s'}}} \right\}$.
\end{theorem}

\subsubsection{Local decay-type estimates}~~

For endpoint weak-type estimates, we give the following two estimates.

\begin{theorem}\label{thm:local}
Let $(X,\mu,\mathcal{B})$ be a ball-basis measure space. Let $r_i\in [1,\infty)$, $\eta_i \in [0,1)$ with $0 \le \eta_ir_i \le 1$ for $i=1,\cdots,m$, and the fractional total order $\eta {\rm{ = }}\sum\limits_{i = 1}^m {{\eta _i}} \in [0,m)$.
Set $T_{\vec \eta}$ be a $\mathbb{V}$-valued {\rm $m$-FBMOO} on $(X,\mu,\mathcal{B})$ as in Theorem \ref{S.d.main}, and ${T}_{\vec \eta} \in W_{\vec{r},\tilde{r}}$ with $\frac{1}{\tilde{r}} :=\sum\limits_{i=1}^{m}\frac{1}{r_i} -\eta>0$. Then there exists $\varDelta > 0$ such that for all $B \in \B$, for any nonnegative functions $f_i \in L_b^{r_i}(X)$, $i=1,\cdots,m$, 
\begin{align}
	\mu \left(\left\{x \in B:  \left\|T_{\vec \eta}^{{}}(\vec{f})(x)\right\|_{\mathbb{V}} > t \, \mathcal{M}^{\B}_{\vec \eta,\vec r}(\vec{f^{}})(x)  \right\}\right) \lesssim e^{- \varDelta t}  \mu(B),  
	\end{align}
	for all $t>0$.
If $A_{\infty, \B}$ satisfies the sharp reverse H\"{o}lder property, then
		\begin{align}
	\mu \left(\left\{x \in B:  \left\|T_{\vec \eta}^{{\bf b, k}}(\vec{f})(x)\right\|_{\mathbb{V}} > t \, \mathcal{M}^{\B}_{\vec \eta,\vec r}(\vec{f^{\star}})(x)  \right\}\right)
	\lesssim e^{-\left(\frac{\varDelta t}{\|\boldsymbol{b}\|^{\bf k}}\right)^{\frac{1}{\kappa+1}}} \mu(B),  
	\end{align}
for all $t>0$, where $\kappa = \sum \limits_{i=1}^{m}k_i -t_i$, $\|\boldsymbol{b}\|^{\bf k} = \prod\limits_{i =1}^m\|b_i\|_{\rm{BMO}}^{k_i}$, and $f_i^{\star} = M_{\mathcal{B}}^{\lfloor r_i \rfloor}(|f_i|^{r_i})^{1/r_i}$ for each $i = 1,\ldots,m$, where $\lfloor r_i \rfloor$ denotes the smallest integer greater than or equal to $r_i$.
\end{theorem}

\subsubsection{Mixed weak-type estimates}~~

\begin{theorem}\label{thm:weak}
Let $(X,\mu,\mathcal{B})$ be a ball-basis measure space such that $A_{1, \B} \subseteq \bigcup_{l>1} RH_{l, \B}$.
Let $r_i\in [1,\infty)$, $\eta_i \in [0,1)$ with $0 \le \eta_ir_i \le 1$ for $i=1,\cdots,m$, and the fractional total order $\eta {\rm{ = }}\sum\limits_{i = 1}^m {{\eta _i}} \in [0,m)$.
Set $T_{\vec \eta}$ be a $\mathbb{V}$-valued {\rm $m$-FBMOO} on $(X,\mu,\mathcal{B})$ as in Theorem \ref{S.d.main}, and ${T}_{\vec \eta} \in W_{\vec{r},\tilde{r}}$ with $\frac{1}{\tilde{r}} :=\sum\limits_{i=1}^{m}\frac{1}{r_i} -\eta>0$. If $\omega \in A_{1, \B}$ and $v^{\tilde{r}} \in A_{\infty, \B}$, 
	then 
	\begin{align}
	\bigg\|\frac{T_{\vec \eta}(\vec{f})}{v}\bigg\|_{L^{\tilde{r},\infty}(X, \, \omega v^{\tilde{r}})}
	&\lesssim \bigg\|\frac{\mathcal{M}^{\B}_{\vec \eta, \vec r}(\vec{f})}{v}\bigg\|_{L^{\tilde{r},\infty}(X, \, \omega v^{\tilde{r}})}.
	\end{align}
\end{theorem}

\subsection{Summary for novelty}~~

{\bf Firstly}, it is needful to emphasize that when \( \vec{\eta} = 0 \), all of our key results---\emph{sparse bounds, multiple weights, multilinear fractional $T1$ theorem theorem, multilinear off-diagonal extropolation theorem, and weighted estimates,} will be back to the diagonal theory. Moreover, Definition \ref{def:FBMOO} can be straightforwardly considered as a generalization of \cite[Definition 1.3]{K2019}, \cite[Definition 1.4]{K2023}, and \cite[Definition 1.4]{Cao23}, cf. Remark \ref{rem:key1}.  This might well be the advantage of \emph{\rm $m$-FBMOOs} for \emph{\rm $m$-BOOs} (\emph{$m$-linear bounded oscillation operator}).
At the same time, we also found that the $W_{\vec{r},\tilde{r}}$ condition significantly reduces the requirements we have for the operators, cf. Remark \ref{rem:key2}. 

Theorem \ref{S.d.main} generalizes the previous the results for sparse domination, e.g. \cite[Theorem 1.1]{K2019}, \cite[Theorem 1.5]{Cao23}, \cite[Theorem 1.1]{Lerner2017}, \cite[Theorems 1.1 and 1.2]{Li2018}, \cite[Theorem 2.1]{Rivera2017}, \cite[Theorem 1.5]{Cao2018}, and \cite[Theorem 1.14]{CenSong2412}.

{\bf Secondly}, the multilinear fractional sparse $T1$ theorem and multilinear dyadic representation theorem for the generalized commutators realize the study of generalized commutators from another perspective. The condition $\mathcal{N}$, cf. \eqref{cond.N}, is undoubtedly another way to establish form sparse bounds.

Theorems \ref{thm:sparse-T1} and \ref{thm:DT} generalize the previous results for the fractional and generalized commutators cases, e.g. \cite[Theorem 1.1 and Corollary 1.2]{Li2018_2} and \cite[Theorems 3.2 and 3.3]{Hytonen2017}.

{\bf Thirdly}, the class of multiparameter multilinear fractional weights $A^{\B}_{(\vec{p},\tilde{p}),(\vec{r}, s)}$ defined in Definition \ref{def:weight} generalizes all the previous multilinear weights, in order to study the relationship between the multilinear weights and {\rm $m$-FBMOOs}, we initially consider to establish the corresponding multilinear off-diagonal extrapolation theorem, cf. Theorem \ref{thm:Ex_1}.
This generalizes the previous results for the fractional cases, e.g. \cite[Theorem 2.2]{Nie2018}, \cite[Theorem 1.1]{Martell2020}, and \cite[Theorem 2.1]{Li2020}. Furthermore, Theorem \ref{op.to.jh} generalizes \cite[Theorems 4.13 and A.2]{Tor2019} and \cite[Theorem 2.22]{Martell2020}.

{\bf Fourthly}, the subsection \ref{weighted-estimates}--Quantitative weighted estimates-- provides a wealth of weighted estimates. In terms of sharp-type estimates, cf. Theorems \ref{thm:weighted1}--\ref{thm:weighted3}, we have extended the previous results, e.g. \cite[Theorem 1.1]{Perez2013}, \cite[Theorem 1.1]{Perez2014}, \cite[Theorem 1.4]{Moen2014}. In terms of Bloom-type estimates, cf. Theorems \ref{thm:bloom}, we also extend those past works, e.g. \cite[Theorem 1.4]{Lerner2017  }, \cite[Theorem 1.1]{Lerner2018}, \cite[Theorem 1.3]{Yang2019}, \cite[Theorem 1.9]{Zhang2023}, and \cite[Theorem 1.28]{CenSong2412}. Theorems \ref{thm:weak} and \ref{thm:local} implement a weak-type estimates of \emph{\rm $m$-FBMOOs}, which are different from the previous strong-type estimates.

{\bf Fifthly,} we give the results on some operators associated with {\rm $m$-FBMOOs}.

*{\bf Maximally modulated {\rm $m$-FBMOOs}}~~

Let ${\left\{ {{T_t }} \right\}_{t  \in \Lambda }}$ is a family of some multilinear operators, the maximal operator of this family is defined by
\[{T_\Lambda }(\vec f)(x): = \mathop {\sup }\limits_{t  \in \Lambda } \left| {{T_t }(\vec f)(x)} \right|.\]

We can deduce immediately from Definition \ref{def:FBMOO} that
\begin{theorem}\label{MAX-FBMOO} 
Let \((X, \mu, \mathcal{B})\) be a ball-basis measure space. 
Suppose that ${\left\{ {{T_{\vec{\eta},t} }} \right\}_{t  \in \Lambda }}$ is a family of $\mathbb{V}$-valued {\rm $m$-FBMOOs} with the uniform constants ${{C_1}({T_{\vec \eta , t }})}<\infty$ and ${{C_2}({T_{\vec \eta , t }})}<\infty$, for any $t \in \Lambda$. If $C_1^\Lambda : = {\left\| {{C_1}({T_{\vec \eta , \cdot }})} \right\|_{{L^\infty }\left( \Lambda  \right)}}<\infty$, and $C_2^\Lambda : = {\left\| {{C_2}({T_{\vec \eta , \cdot }})} \right\|_{{L^\infty }\left( \Lambda  \right)}}<\infty$, then ${T_{\vec{\eta},\Lambda} }$ is a $\V$-valued {\rm $m$-FBMOO} on $(X,\mu,\mathcal{B})$.
\end{theorem}

*{\bf Square modulated {\rm $m$-FBMOOs}}~~

Let ${\left\{ {{T_t }} \right\}_{t  \in \Lambda }}$ is a family of some multilinear operators, the square function of this family is defined by
\[{T_\Lambda }(\vec f)(x): = {\left\| {{T_{( \cdot )}}(\vec f)} \right\|_{{L^2}(\Lambda ,dv)}}.\]

We can also deduce that
\begin{theorem}\label{square-FBMOO} 
Let \((X, \mu, \mathcal{B})\) be a ball-basis measure space. 
Suppose that ${\left\{ {{T_{\vec{\eta},t} }} \right\}_{t  \in \Lambda }}$ is a family of $\V$-valued {\rm $m$-FBMOO} with the uniform constants ${{C_1}({T_{\vec \eta , t }})}<\infty$ and ${{C_2}({T_{\vec \eta , t }})}<\infty$, for any $t \in \Lambda$. If $C_1^\Lambda : = {\left\| {{C_1}({T_{\vec \eta , \cdot }})} \right\|_{{L^2 }\left( \Lambda,dv  \right)}}<\infty$, and $C_2^\Lambda : = {\left\| {{C_2}({T_{\vec \eta , \cdot }})} \right\|_{{L^2 }\left( \Lambda,dv  \right)}}<\infty$, then ${T_{\vec{\eta},\Lambda} }$ is a $\V$-valued {\rm $m$-FBMOO} on $(X,\mu,\mathcal{B})$.
\end{theorem}

{\bf Finally,} we emphasize that the strength of this work lies in its coverage of essential operators overlooked by previous studies--operators that consistently play vital roles in modern Fourier analysis. Furthermore, we establish significant conclusions on distinct base spaces that were either missed or inadequately addressed in prior literature. This work vividly demonstrates how classical harmonic analysis has progressively embraced abstract spaces.


\subsection{Structure of this paper}~~

\begin{itemize}
\item 
In Sect. \ref{example}, we provide the representative examples of {\rm $m$-FBMOOs} to demonstrate the significance and interest of our main results.
\item
In Sect. \ref{Pre}, the geometry of ball-basis measure spaces and spaces of homogeneous type and Dyadic cubes are further discussed.
\item
In Sect. \ref{sec.sparse}, we prove Theorem~\ref{S.d.main} to build Karagulyan-type sparse domination. 
\item 
In Sect. \ref{T1}, to prove Theorem \ref{thm:sparse-T1}, it suffices for us to prove the dyadic representation---Theorem \ref{thm:DT}.
\item 
In Sect. \ref{Sec.extro}, we prove Theorem \ref{thm:Ex_1} and establish the entire multilinear off-diagonal extrapolation theory.
\item
In Sect. \ref{sec.qwe}, we set up sharp-type estimates, Bloom-type estimates, local decay-type estimates, and mixed weak-type estimates. 
To prove Theorem \ref{thm:weighted1}, it is ample to prove Theorem \ref{thm:sparse-dom_1}.       
To prove Theorem \ref{thm:bloom}, we only need to prove Theorem \ref{main.bloom}.
Last, Theorems \ref{thm:local} and \ref{thm:weak} are verified.
\item 
In Sect. \ref{Appendix A}, based on ball-basis measure spaces, we introduce the classical Muckenhoupt weights, Young functions, Orlicz spaces, and their corresponding lemmas. 
\end{itemize}

\section{\bf Representative examples of {\rm $m$-FBMOOs}}\label{example}

In this section, we apply the main results to several classes of operators. Given the abundance of established results, it suffices to demonstrate that these operators can be characterized as $m$-FBMOOs satisfying $W_{\vec{r},\tilde{r}}$. The desired conclusions thus follow directly by combining the main results with the subsequent findings. Prior to this, we define the following spaces, all of which are ball-based measure spaces.
\emph{\begin{itemize}
\item We denote a space of homogeneous type by $(X,d,\mu,\mathcal{D})$, where $\mathcal{D}=\mathop  \cup \limits_{i = 1}^N {\mathcal{D}_i}$, and $\{\mathcal{D}_i\}_{i=1}^N$ are $N$ fixed dyadic lattices. 
\item We denote the classical Euclidean space by  $(\rn,|\cdot|,dx, \mathscr{B})$, where $\mathscr{B}$ is the set of all balls in $\rn$.
We sometimes replace $\mathscr{B}$ with $\mathcal{D}$, $\mathcal{Q}$, or $\mathscr{Q}$, where $\mathcal{D}=\mathop  \cup \limits_{i = 1}^{3^n} {\mathcal{D}_i}$ with $\{\mathcal{D}_i\}_{i=1}^{3^n}$ are $3^n$ fixed dyadic lattices, $\mathscr{Q}$ is the set of all cubes, and 
$\mathcal{Q}$ is the set of all cubes with sides parallel to the coordinate axes.
\item  We denote the complex plane by $(\mathbb{C},|\cdot|, dA, \mathscr{B})$, where $dA$ denotes the Lebesgue measure on $\mathbb{C}$ and $\mathscr{B}$ denotes the set of all disks in $\mathbb{C}$.
\end{itemize}
}

\subsection{$\mathbb{V}$-valued multilinear fractional Calderón--Zygmund operators}
~~

In this subsection, $(X,d,\mu,\mathcal{D})$ is the base space.

\begin{definition}\label{def:full}
Let $m \in \N$, $\eta \in [0,m)$, $\V$ be a quasi-Banach space, and $B(\cc,\V)$ be the space of all bounded operators from $\cc$ to $\V$. Set an operator-valued function $K_{\eta}:(X^{m+1} \backslash \Delta ) \to B(\cc,\V),$ where $\Delta = \{ (x,\vec y) \in X^{m+1} :x = {y_1} =  \cdots  = {y_m}\}$. An $m$-linear operator $T$ admits a
\textbf{\textsf{$\mathbb{V}$-valued $m$-linear fractional Dini-type Calderón-Zygmund kernel representation}}, if for any \( \vec{f} \in (L_b^\infty(X))^m \) and each \( x \notin \bigcap_{i = 1}^m \text{supp}(f_i) \),

	\begin{align*}
T_{\eta}(\vec{f})(x):=\int_{X^m} K_{\eta}\left(x, \vec y\right) \left(\prod_{j=1}^{m}f_j\left(y_j\right)\right) d\mu (\vec y),
	\end{align*}
	
where $\Dini$ type kernel $K_{\eta}$ satisfies that for 

\begin{enumerate}
 \item \textbf{Size Condition}: 
\begin{align}
{\left\| {{K_{\eta} }\left( {x,\vec y} \right)} \right\|_{B(\cc,\V)}} \le C_{K_\eta} {\left( {\sum\limits_{i = 1}^m {\mu (B(x,d(x,{y_i})))} } \right)^{{\eta}  - m}}  \label{S1} ,
\end{align}
\item \textbf{Hölder Smoothness Condition}: 
    
For each $j \in \left\{ {0, \cdots ,m} \right\}$, whenever $d\left( {{{y_j'}},{y_j}} \right) \le \frac{1}{2}\max \left\{ {d\left( {{y_0},{y_i}} \right):i = 1, \cdots ,m} \right\}$,
\begin{align}
&{\left\| {{K_{\eta} }\left( {{y_0}, \cdots ,{{y_j'}}, \cdots ,{y_m}} \right) - {K_{\eta} }\left( {{y_0}, \cdots ,{y_j}, \cdots ,{y_m}} \right)} \right\|_{B(\cc,\V)}}, \notag \\
\le C_{K_\eta} & {\left( {\sum\limits_{i = 1}^m {\mu (B({y_0},d({y_0},{y_i})))} } \right)^{{\eta}  - m}}w\left( {\frac{{d\left( {{{y'}_j},{y_j}} \right)}}{{\sum\limits_{i = 1}^m {d\left( {{y_0},{y_j}} \right)} }}} \right), \label{S2}
\end{align}
\end{enumerate}
where $w$ is increasing, $w(0)=0$, and ${\left[ w \right]_{\Dini}} = \int_0^1 {\frac{{w\left( t \right)}}{t}dt}  < \infty$.
Once $w(t)=t^{\delta}$ for some $\delta \in (0,1]$, then we call that $T$ admits a \textbf{\textsf{$\mathbb{V}$-valued $m$-linear fractional Calderón-Zygmund kernel representation}}.

Suppose that \( T \) admits a \textbf{\textsf{$\mathbb{V}$-valued $m$-linear fractional (Dini-type) Calderón-Zygmund kernel representation}}. \( T_{\eta} \) is called a \textbf{\textsf{$\mathbb{V}$-valued multilinear fractional (Dini-type) Calderón-Zygmund operator}} ({\rm $m$-F(D)CZO}), if for some \( r_1, \dots, r_m \geq 1 \), with
$\frac{1}{\tilde r}:=\sum\limits_{i = 1}^m \frac{1}{{{r_i}}}-\eta >0$, such that $T_{\eta}$ is bounded from $\prod\limits_{i = 1}^m {{L^{{r_i}}}(X)}$ to ${L^{\tilde r,\infty }}(X,\V)$.
\end{definition}

Let $\alpha \in (0,mn)$, (or take $\eta:=\alpha/n \in (0,m)$), and we define the $m$-linear fractional integral operator by
\[{I_\alpha }(\vec f)(x): = \int_{{\mathbb{R}^{mn}}} {\frac{{\prod\limits_{i = 1}^m {{f_i}({y_i})} }}{{{{\left( {\sum\limits_{i = 1}^m {\left| {x - y_i} \right|} } \right)}^{mn - \alpha }}}}d\vec y}.\]
This is a special case of {\rm $m$-FDCZO} on $(\rn, |\cdot|,dx, \mathcal{B})$ where $\mathcal{B}$ is the set of all balls or a dyadic lattice of $\rn$. 
In addition, as another special case of {\rm $m$-FDCZO}, we can define the first kind of $m$-linear fractional Littlewood--Paley square operator by  
\[{\T_\eta^{(1)} }(\vec f)(x): = \left\| {{T_{\eta}}(\vec f)}(x) \right\|_{L^2((0,\infty),\frac{dt}{t})},\]
where ${T_{\eta}}$ is a ${L^2((0,\infty),\frac{dt}{t})}$-valued {\rm $m$-FDCZO}.

The second kind of $m$-linear fractional Littlewood--Paley square operator is defined by 
\[{\T_\eta^{(2)} }(\vec f)(x): = \left\| {{T_{\eta,z,t}}(\vec f)}(x) \right\|_{L^2((\R^n \times (0,\infty),\frac{{dzdt}}{{{t^{n + 1}}}})},\]
where ${T_{\eta,z,t}}$ is a ${L^2((\R^n \times (0,\infty),\frac{{dzdt}}{{{t^{n + 1}}}})}$-valued {\rm $m$-FDCZO}, for any $(z,t) \in \R^n \times (0,\infty)$.

From the past literature (e.g. \cite{Xue2015,Xue2021}), we can provides specific definitions of the kernel functions of $m$-linear fractional Littlewood--Paley square operator. However, from the current perspective, these definitions can be seen as derivable from the statement above, regardless of how the kernel is defined.

\begin{theorem}\label{FCZO}
Let the fractional total order ${\eta} \in [0,m)$. 
Then $\V$-valued {\rm $m$-FDCZO} $T_{\eta}$ is a $\V$-valued {\rm $m$-FBMOO} on $(X,d,\mu,\mathcal{D})$ satisfying ${W_{\vec r,\tilde r}}$ for some $r_1,\cdots,r_m \ge 1$, with 
$\frac{1}{\tilde r}:=\sum\limits_{i = 1}^m \frac{1}{{{r_i}}}-\eta >0$. 
In particular, $m$-linear fractional integral operator ${I_\alpha }$ is a $\mathbb{C}$-valued {\rm $m$-FBMOO} on $(\rn,|\cdot|,dx, \mathcal{B})$, satisfying ${W_{\vec r,\tilde r}}$, for any $r_1,\cdots,r_m \ge 1$, with 
$\frac{1}{\tilde r}:=\sum\limits_{i = 1}^m \frac{1}{{{r_i}}}-\eta >0$, where $\eta:=\alpha/n \in (0,m)$.
In addition, the $m$-linear fractional Littlewood--Paley square operators ${\T_\eta^{(1)}},{\T_\eta^{(2)}}$ are both $\R$-valued {\rm $m$-FBMOOs} on $(\rn,|\cdot|,dx, \mathcal{B})$ and also satisfy ${W_{\vec r,\tilde r}}$, for some $r_1,\cdots,r_m \ge 1$, with 
$\frac{1}{\tilde r}:=\sum\limits_{i = 1}^m \frac{1}{{{r_i}}}-\eta >0$, where $\eta:=\alpha/n \in [0,m)$.
\end{theorem}

\begin{remark}
In this subsection, all operators can be viewed as \( \mathbb{V} \)-valued {\rm $m$-FBMOOs} on the spaces \( (X,d,\mu,\mathcal{D}) \) and \( (\mathbb{R}^n, |\cdot|, dx, \mathcal{B}) \), satisfying condition \( W_{\vec{r},\tilde{r}} \). All results in Sect. \ref{Main} apply to them, except for Theorems \ref{thm:sparse-T1} and \ref{thm:DT}. However, if these operators additionally admit a \textbf{\textsf{$\mathbb{V}$-valued $m$-linear fractional Calderón-Zygmund kernel representation}}, then Theorems \ref{thm:sparse-T1} and \ref{thm:DT} also hold for them.
\end{remark}

\begin{proof}
It suffices to prove for $\V$-valued {\rm $m$-FDCZO} $T_{\eta}$. Due to its definition, ${W_{\vec r,\tilde r}}$ follows from the weak-type boundedness. Next, we present why $T_\eta$ can be regarded as a $\V$-valued {\rm $m$-FBMOO} on $(X,d,\mu,\mathcal{D})$.

Once \eqref{M_1.condi.} and \eqref{M_2.condi.} are valid on $(X,d,\mu,\mathcal{B})$, where $\mathcal{B}$ is the set of all balls in $(X,d,\mu)$, we can transform them to $(X,d,\mu,\mathcal{D})$ to make them effective. This is possible due to the inclusion equivalence between $\mathcal{D}:=\cup_{j=1}^N \mathcal{D}^j$ and $\mathcal{B}$ in $(X,d,\mu)$; see subsection \ref{dyadic}, Lemma \ref{cubeeq}, and Lemma \ref{lem.covering}.
Thus, it is enough to verify \eqref{M_1.condi.} and \eqref{M_2.condi.} on $(X,d,\mu,\mathcal{B})$.

From \cite[Theorem 7.1]{K2019}, $\B^{\dagger}$ is a ball-basis satisfying doubling property, cf. subsection \ref{more.information}. For any $B_0=B(x_0,r_0) \in \B$, $B_0^\dagger=B(x_0, R_0)$ with $2r_0 \le R_0 \le \infty$. Set 
$
B := B(x_0, 2R),
\text{ where }
R := \sup\{r \ge R_0:\, B(x_0, r)=B(x_0, R_0)\}. 
$
Since $B_0^\dagger=B(x_0, R_0) \subsetneq X$, we have $R_0<\infty$ and 
\begin{align}\label{BB}
B_0^{\dagger}=B(x_0, R_0) \subsetneq B(x_0, 2R)=B.
\end{align}
For any $\vec{y} \in (B^{\dagger})^m \setminus (B_0^{\dagger})^m$, 
we firstly can obtain that there is a $j \in \{1,\cdots,m\}$, such that $y_j \in B^\dagger \backslash B_0^\dagger$.

For any $x \in B_0$ and $\vec{y} \in (B^{\dagger})^m \setminus (B_0^{\dagger})^m$, we have that $d(x,y_j)>r_0$.
Note that the fact that $B(x_0,r_0) \subseteq B(x,2r_0)$ and $B(x,r_0) \subseteq B(x_0,2r_0)$.
It follows from \eqref{S1}, this fact, and the doubling property of $\mu$ that 
\begin{align}
\|K_\eta(x, \vec{y})\|_{B(\cc,\V)}  
\lesssim \frac{C_{K_\eta}}{\mu(B(x, d(x, y_j)))^{m-\eta}} 
\lesssim \frac{C_{K_\eta}}{\mu(B_0)^{m-\eta}}, \label{KB}
\end{align}
Thus, \eqref{M_1.condi.} with $C_1 \lesssim C_{K_\eta}$ will derive instantly from the following statement.

Collecting \eqref{BB} and \eqref{KB}, for any $x\in B_0$, 
\begin{align*}
&\left\|T_\eta(\vec{f}\mathbf{1}_{{B^\dagger}\backslash {B_0^\dagger}})(x)\right\|_{\mathbb{V}} 
\le\int_{(B^\dagger)^m \setminus (B_0^\dagger)^m} \left\|K(x, \vec{y})\right\|_{B(\cc,\V)} \left(\prod_{i=1}^m |f_i(y_i)|\right) d\mu(\vec{y})\\
&\lesssim C_{K_\eta} \prod_{i=1}^m \frac{1}{\mu(B_0)^{1-\eta_i}} \int_{B^\dagger}|f_i(y_i)| \, d\mu(y_i) 
\lesssim C_{K_\eta} \prod_{i=1}^m \frac{1}{\mu(B^\dagger)^{1-\eta_i}}\int_{B^\dagger}|f_i(y_i)| \, d\mu(y_i)\\
&\le C_{K_\eta} {\prod\limits_{i=1}^m\left\langle f_i\right\rangle_{\eta_i,r_i,B^\dagger}}
\end{align*}
where we used the doubling property of $\mu$ and H\"older's inequality. 

Analogously, \eqref{M_2.condi.} with $C_2 \lesssim [\omega]_{{\rm Dini}}$ will be deduced from the following statement.

For any $B \in \B$ with $B \subsetneq X$ and for any $x,x_0 \in B$, it is derived from \eqref{S2} and H\"older's inequality that 
\begin{align*}
&{\left\| {{T_\eta }(\vec f{{\bf{1}}_{{{({B^\dagger })}^c}}})(x) - {T_\eta }(\vec f{{\bf{1}}_{{{({B^\dagger })}^c}}})({x_0})} \right\|_\mathbb{V}}
\\ 
& = {\left\|\int_{X^m \setminus (B^{\dagger})^m} 
(K_\eta(x, \vec{y}) - K_\eta(x_0, \vec{y}) ) \left(\prod_{i=1}^m f_i(y_i)\right) d\mu(\vec{y}) \right\|_\mathbb{V}}
\\ 
& \leq\sum_{k=0}^{\infty} \int_{(2^{k+1}B^{\dagger})^m \setminus (2^k B^{\dagger})^m} 
{\left\|K_\eta(x, \vec{y}) - K_\eta(x_0, \vec{y})\right\|_{B(\cc,\V)}} \left(\prod_{i=1}^m |f_i(y_i)|\right) d\mu(\vec{y}) 
\\
& \lesssim \sum_{k=0}^{\infty} \int_{(2^{k+1}B^{\dagger})^m \setminus (2^k B^{\dagger})^m} 
\frac{w\big(\frac{d(x, x_0)}{\max\limits_{1 \le i \le m} d(x_0, y_i)}\big) 
\left(\prod_{i=1}^m |f_i(y_i)|\right) d\mu(y_i)}{\left(\sum_{i=1}^m \mu(B(x_0, d(x_0, y_i)))\right)^{m-\eta}} 
\\
& \lesssim \sum_{k=0}^{\infty} \omega(2^{-k-1}) 
\prod_{i=1}^m {\left\langle {{f_i}} \right\rangle _{{\eta _i},1,{2^{k + 1}}B^\dagger}}
\lesssim [\omega]_{{\rm Dini}} \prod_{i=1}^m {\left\langle {\left\langle {{f_i}} \right\rangle } \right\rangle _{{\eta _i},{r_i},B}}, 
\end{align*}
where we used the definition of ${\left\langle {\left\langle {{f_i}} \right\rangle } \right\rangle _{{\eta _i},{r_i},B}}$ and the fact that
\[
\sum_{k=1}^{\infty} \w(2^{-k}) 
\approx \int_0^1 \w(t) \frac{dt}{t}
= \|\omega\|_{{\rm Dini}} <\infty. \qedhere
\] 
\end{proof}

\subsection{Variation operators associated with fractional Calderón--Zygmund operators}~~

In this subsection, the base space is $(\rn,|\cdot|,dx, \mathscr{Q})$. In this space, for any cube $Q \in \mathscr{Q}$,  $Q^\dagger:=5Q$. 

Let $\mathbb{V}$ is a quasi-Banach space, $q \ge 1$, and a family of real numbers $a := \{a_{\eta}\}_{\eta>0}$, the $\rho$-variation norm of $a$ is defined by
\[{\left\| a \right\|_{{V_\rho}(\mathbb{V})}}: = \mathop {\sup }\limits_{\left\{ {{\varepsilon _i}} \right\} \downarrow 0} {\left( {\sum\limits_{i= 0}^\infty  {{{\left\| {{a_{{\varepsilon _i}}} - {a_{{\varepsilon _{i + 1}}}}} \right\|_{\mathbb{V}}}^\rho}} } \right)^{\frac{1}{\rho}}},\]
where the supremum takes over all number sequences $\{\eta_j\}$ decreasing to 0. 
It is obviously to see that ${\left\| \cdot \right\|_{{V_\rho}}}$ is only a seminorm.

Let $T_\eta$ be a $\mathbb{V}$-valued {\rm $1$-FDCZO} on $(\rn, |\cdot|,dx, \mathscr{Q})$, cf. Definition \ref{def:full}. 
Let $\epsilon>0$, we define the truncated operator of $T_\eta$ by 
\begin{equation*}
T_{\eta,\epsilon}f(x) :=\int_{|x-y| >\epsilon} K_\eta (x, y) f(y) \, dy.
\end{equation*}
Given $b \in \BMO$ and $k \in \N_0$, we denote
\begin{align*}
&\T_\eta f(x) := \big\{T_{\eta,\epsilon}f(x) \big\}_{\epsilon > 0}, \quad 
\T_{\eta}^{b,k} f(x) := \big\{T_{\eta,\epsilon}^{b,k}f(x) \big\}_{\epsilon > 0}, 
\\
&\text{and} \quad
\mathscr{K}_\eta(x, y) :=\big\{K_{\eta,\epsilon}(x, y) := K_\eta(x, y) \chi_{\{|x-y|>{\epsilon}\}}\big\}_{{\epsilon}>0}.
\end{align*}
For $1 \le \rho < \infty$, the $\rho$-variation operator of $T_\eta$ and its gerneralized commutator are defined as
\begin{align*}
(\V_\rho\circ \T_\eta)(f)(x) :=\|\T_\eta f(x) \|_{V_\rho(\mathbb{V})} \text{ and }
(\V_\rho\circ \T_\eta)^{b,k} (f)(x) :=\| \T_\eta^{b,k} f(x) \|_{V_\rho(\mathbb{V})}=(\V_\rho\circ \T_\eta^{b,k}) (f)(x).
\end{align*}

\begin{theorem}\label{thm:Var}
Let the fractional total order ${\eta} \in [0,1)$ and $r \in (1,\rho)$ with $\frac{1}{{\tilde r}}: = \frac{1}{r} - \eta  > 0$. If $\rho>2$, then $\V_\rho \circ \T_\eta$ is a $\R$-valued {\rm $m$-FBMOO} on $(\rn,|\cdot|,dx, \mathscr{Q})$ with constants $C_1(\V_\rho \circ \T_\eta) \lesssim C_{K_\eta}$ and $C_2(\V_\rho \circ \T_\eta) \lesssim 1+[\w]_{\mathrm{Dini}}$.
\end{theorem}

\begin{remark}
In this subsection, $\V_\rho\circ \T_\eta$ can be regarded as a $\V$-valued {\rm $m$-FBMOO} on the space $(\rn, |\cdot|,dx, \mathscr{Q})$.
If we assume that $\V_\rho\circ \T_\eta \in {W_{r,\tilde r}}$ where $r \in (1,\rho)$ as in Theorem \ref{thm:Var}, with 
$\frac{1}{\tilde r}:=\frac{1}{r}-\eta >0$.
All results in Sect. \ref{Main} apply to it, except for Theorems \ref{thm:sparse-T1} and \ref{thm:DT}. However, if it additionally admits a \textbf{\textsf{$\mathbb{V}$-valued $m$-linear fractional Calderón-Zygmund kernel representation}}, then Theorems \ref{thm:sparse-T1} and \ref{thm:DT} also hold for it.
\end{remark}

\begin{proof}
Let $Q_0 \in \mathscr{Q}$, $x \in Q_0$, and take $Q:= 2Q_0$. Observe that 
\begin{align*}
\|\mathscr{K}_\eta (x, y)\|_{V_\rho({B(\cc,\V)})}
&=\sup_{\{\epsilon_j\} \downarrow 0} \bigg(\sum_{j=0}^{\infty}
\left\|K_{\eta,\epsilon_j}(x, y) - K_{\eta,\epsilon_{j+1}}(x, y) \right\|_{{B(\cc,\V)}}^\rho \bigg)^{\frac{1}{\rho}}
\\
&\le \left\|K_\eta(x, y)\right\|_{B(\cc,\V)} \sup_{\{\epsilon_j\}\downarrow 0} \sum_{j=0}^{\infty} 
\mathbf{1}_{\{\eta_{j+1} < |x-y| \le \eta_j \}} 
\le \left\|K_\eta(x, y)\right\|_{B(\cc,\V)}.
\end{align*}
\eqref{M_1.condi.} follows from \eqref{S1} the fact that
\begin{align*}
\|&\T_{\eta}(f \chi_{Q^\dagger})(x) - \T_{\eta}(f \chi_{Q_0^\dagger})(x)\|_{V_\rho(\mathbb{V})}
=\bigg\|\int_{Q^\dagger \setminus Q_0^\dagger} \mathscr{K}_\eta(x, y) f(y) \, dy \bigg\|_{V_\rho(\V)}
\\ 
&\le \int_{Q^\dagger \setminus Q_0^\dagger} \|\mathscr{K}_\eta(x, y)\|_{V_\rho(B(\mathbb{C},\mathbb{V}))} |f(y)| \, dy 
\le C_{K_\eta} \int_{Q^\dagger \setminus Q_0^\dagger} \frac{|f(y)|}{|x-y|^{n(1-\eta)}} \, dy \\
&\lesssim C_{K_\eta} \langle f \rangle_{\eta, 1,Q^\dagger} \le C_{K_\eta} \langle f \rangle_{\eta,r,Q^\dagger},
\end{align*}
which implies that $C_1 \lesssim C_{K_\eta}$.

Last, we verify \eqref{M_2.condi.} as follows.
Let $Q \in \mathscr{Q}$ and $x, x' \in Q$. 
\begin{align*}
\|&(\T_\eta(f \chi_{(Q^\dagger)^c}))(x) - (\T_\eta(f \mathbf{1}_{(Q^\dagger)^c}))(x')\|_{V_\rho(\mathbb{V})}
\\
&\le \sum_{k=1}^{\infty} \bigg\|\int_{2^k Q \setminus 2^{k-1} Q} (\mathscr{K}_\eta(x, y) - \mathscr{K}_\eta(x', y)) f(y) dy \bigg\|_{V_q(\mathbb{V})}
=: \sum_{k=1}^{\infty} \mathcal{T}_k,  
\end{align*}
Set $f_k := f \chi_{2^k Q \setminus 2^{k-1} Q}$ and $R_j=(\epsilon_{j+1}, \epsilon_j]$. Note that 
$$\mathcal{T}_k= \sup_{\{\epsilon_j\} \downarrow 0}  \left(\sum\limits_{j = 1}^\infty \left| {\int_{{2^k}Q \setminus {2^{k - 1}}Q} \left\|{\left( {{{\mathcal{K}}_\eta }(x,y){\chi _{{R_j}}}(x - y) - {{\mathcal{K}}_\eta }(x',y){\chi _{{R_j}}}(|x' - y|)} \right)}\right\|_{B(\cc,\V)}f(y)dy }\right|^\rho\right)^{\frac{1}{\rho}},$$
and
\begin{align}\label{GGkk}
\mathcal{T}_k 
\le \mathcal{T}_k^1 +\mathcal{T}_k^2,   
\end{align}
where 
\begin{align*}
\mathcal{T}_k^1 
&:=\sup_{\{\epsilon_j\} \downarrow 0} \bigg(\sum_{j=0}^{\infty} \left|\int_{\Rn}
\left\|{K}_\eta(x, y) - {K}_\eta(x', y)) \right\|_{B(\cc,\V)}\chi_{R_j}(|x-y|) f_k(y) dy\right|^\rho \bigg)^{\frac{1}{\rho}},
\end{align*}
and 
\begin{align*}
\mathcal{T}_k^2 
&:=\sup_{\{\epsilon_j\} \downarrow 0} \bigg(\sum_{j=0}^{\infty} \left|\int_{\Rn}
\left\|{K}_\eta(x', y) \right\|_{B(\cc,\V)}\left(\chi_{\{\epsilon_{j+1}<|x-y| \le \epsilon_j\}} - \chi_{\{\epsilon_{j+1}<|x'-y| \le \epsilon_j\}}\right) f_k(y) \, dy\right|^\rho \bigg)^{\frac{1}{\rho}}. 
\end{align*}
It follows immediately from \eqref{S2} that 
\begin{align}\label{GGkk-1}
\mathcal{T}_k^1 
\le \int_{\Rn} \left\|K_\eta(x, y) - K_\eta(x', y)\right\|_{B(\cc,\V)} |f_k(y)| \, dy
\lesssim \w(2^{-k}) \langle f \rangle_{\eta,1,2^k Q} \le \w(2^{-k}) \langle\langle f \rangle\rangle_{\eta,r,Q}. 
\end{align}
It derives from the methods of \cite[p.10]{DLY} and \cite[p.401]{MTX} that
\begin{align}\label{GGkk-2}
\mathcal{T}_k^2 
&\lesssim \ell(Q)^{\frac{1}{r'}} \bigg(\int_{\Rn} \frac{|f_k(y)|^r}{|x-y|^{r+n(1-\eta)-1}} dy \bigg)^{\frac1r}
\nonumber \\ 
&\approx \ell(Q)^{\frac{1}{r'}} \bigg(\frac{1}{(2^{k-1}\ell(Q))^{r+n(1-\eta)-1}} \int_{2^k Q} |f(y)|^r dy \bigg)^{\frac1r}
\nonumber \\
&\lesssim 2^{-k/r'} \langle f \rangle_{\eta,r,2^k Q} \le 2^{-k/r'} \langle\langle f \rangle\rangle_{\eta,r,Q}.
\end{align}
Consequently, these estimates above implies that 
\eqref{M_2.condi.} is valid with $C_2(\V_\rho \circ \T_\eta) \lesssim 1 + [\w]_{\mathrm{Dini}}$. 
\end{proof}

\subsection{Multilinear fractional Ahlfors-Beurling operators}
~~

In this subsection, $(\mathbb{C},|\cdot|, dA, \mathscr{B})$ is the base space. In this space, for any cube $B \in \mathscr{B}$,  $B^\dagger:=5B$. 

\begin{definition}
 Let ${\eta _1}, \cdots ,{\eta _m} \in [0,1)$, and $\eta:= \sum\limits_{i = 1}^m {{\eta _i}} \in [0,m)$. Let $r_1,\cdots,r_m \ge1$ and $\frac{1}{{\tilde r}}: = \sum\limits_{i = 1}^m {\frac{1}{{{r_i}}}}  - \eta>0.$ Define $m$-linear fractional Ahlfors-Beurling operator $\mathscr{C}_{\vec{\eta}}:C_b^\infty(\mathbb{C})\times\cdots\times C_b^\infty(\mathbb{C})\to (C_{b}^\infty(\mathbb{C}))'$ by 
\[
\mathscr{C}_{\vec{\eta}}(\vec{f})(z) = \int_{\mathbb{C}^m} K_{\vec{\eta}}(z, \zeta_1, \dots, \zeta_m) \prod_{i=1}^{m} f_i(\zeta_i) dA(\zeta_1) \dots dA(\zeta_m).
\]
where $z \notin \mathop  \cap \limits_{i = 1}^m \supp{f_i}$.

The kernel $K_{\vec{\eta}}: \mathbb{C}^{m+1} \setminus \Delta \to \mathbb{C}$, where $\Delta : = \left\{ {\left( {z,{\xi _1} \cdots ,{\xi _m}} \right) \in {\mathbb{C}^{m + 1}}:z = {\xi _i}}, \text{ for some } i \in \{1,\cdots,m\} \right\}$ satisfying

\begin{enumerate}
    \item \textbf{Size Condition}: 
    \[
    |K_{\vec{\eta}}(z, \zeta_1, \dots, \zeta_m)| \lesssim {\prod_{i=1}^{m} |z - \zeta_i|^{2\eta_i-2}},
    \]

    \item \textbf{Hölder Smoothness Condition}: there exists a $\delta > 0$, such that for any $z, z' \in \mathbb{C}$, with $|z - z'| \leq \frac{1}{2} \max_i |z - \zeta_i|$, 
    \[
    |K_{\vec{\eta}}(z, \zeta_1, \dots, \zeta_m) - K_{\vec{\eta}}(z', \zeta_1, \dots, \zeta_m)| \lesssim \frac{|z - z'|^\delta}{\prod_{i=1}^{m} |z - \zeta_i|^{2-2\eta_i + \delta}},
    \]
\end{enumerate}
\end{definition}

\begin{theorem}
Let ${\eta _1}, \cdots ,{\eta _m} \in [0,1)$, and the fractional total order $\eta:= \sum\limits_{i = 1}^m {{\eta _i}}$. Then  $\mathscr{C_{\vec{\eta}}}$ is a $\mathbb{C}$-valued {\rm $m$-FBMOO} on $(\mathbb{C},|\cdot|, dA, \mathscr{B})$.
\end{theorem}

This proof is similar to that of Theorem \ref{FCZO}. We omit it here.

\begin{remark}
In this subsection, $\mathscr{C_{\vec{\eta}}}$ can be regarded as a $\mathbb{C}$-valued {\rm $m$-FBMOO} on the space $(\mathbb{C},dA, \mathcal{B})$.
If we assume that $\mathscr{C_{\vec{\eta}}} \in {W_{\vec r,\tilde r}}$ for some $r_1,\cdots,r_m \ge 1$, with $
\frac{1}{\tilde{r}} := \sum_{i = 1}^m \frac{1}{r_i} - \eta > 0,$
then all results in Sect. \ref{Main} apply to it, except for Theorems \ref{thm:sparse-T1} and \ref{thm:DT}, since it does not admit a \textbf{\textsf{$\mathbb{C}$-valued $m$-linear fractional Calderón-Zygmund kernel representation}}.
\end{remark}

\subsection{Multilinear pseudo-differential operators with multi-parameter Hörmander symbol}
~~

In this subsection, $(\rn,|\cdot|,dx, \mathcal{Q})$ is the base space. In this space, for any cube $Q \in \mathcal{Q}$,  $Q^\dagger:=5Q$.

\begin{definition}\label{MFPDO}
Given a function $a: \R^{n(m+1)} \to \mathbb{C}$, for all $f_i \in \mathscr{S}(\Rn)$, $i=1,\ldots,m$, the $m$-linear pseudo-differential operator $T_a$ is defined as
\begin{align*}
T_a(\vec{f})(x) := \int_{\R^{nm}} a(x, \vec{\xi})
\widehat{f}_1(\xi_1) \cdots \widehat{f}_m(\xi_m) e^{2\pi i x \cdot (\xi_1+\cdots+\xi_m)}  \, d\vec{\xi},
\end{align*}
where $\widehat{f}=\F(f)$ is the Fourier transform of $f$, and we always set $a(x, \vec{\xi})= \mathcal{F}\left( {{K_{a}}(x, \cdot )} \right)(\vec \xi )$.
If $a(x, \vec{\xi}) \equiv \sigma(\vec{\xi})$ for all $x \in \Rn$, $T_\sigma$ is called $m$-linear Fourier multiplier. 
\end{definition}

In 2024, Mondal et al. introduce the following multi-parameter Hörmander symbol class.
\begin{definition}[\cite{Mondal2024}, Definition 2.5]
Letting $(\tau_1,\cdots,\tau_m) \in \R^m$ and $\rho, \delta \in [0, 1]$, we write $a \in S_{\rho,\delta}^{\tau_1,\cdots,\tau_m}$, if for every multi-indices $\gamma, \beta_1,\ldots,\beta_m$, 
\begin{align*}
{\left[ a \right]_{S_{\rho,\delta}^{\tau_1,\cdots,\tau_m}(\gamma,\vec{\beta})}}: = \mathop {\sup }\limits_{x,{\xi _1}, \cdots ,{\xi _m} \in \rn} \prod\limits_{i = 1}^m {{{\left( {1 + \left| {{\xi _i}} \right|} \right)}^{ - \left( {{\tau _i} - \rho \left| {{\beta _i}} \right| + \delta |\gamma |} \right)}}}|\partial _x^\gamma \partial _{{\xi _1}}^{{\beta _1}} \cdots \partial _{{\xi _m}}^{{\beta _m}}a(x,\vec \xi )| <\infty.
\end{align*}

\end{definition}

This aroused our interest and we can obtain that
\begin{theorem}\label{thm:Ta}
Let $\rho, \delta \in [0, 1]$, and $a \in S_{\rho,\delta}^{\tau_1,\cdots,\tau_m}$. If for some $\vec{\beta} \in \N_0^{mn}$ with $|\vec{\beta}|=mn$, ${\tau _i} - \rho \left| {{\beta _i}} \right| + n < 0$, $i=1,\cdots,m$, then $m$-linear pseudo-differential operator $T_{a}$ is a $\mathbb{R}$-valued {\rm $m$-FBMOO} on $(\rn,|\cdot|,dx, \mathcal{Q})$ with $T_a \in W_{\vec{1},1/m}$, and the fractional total order $\eta = 0$.
\end{theorem}

\begin{proof} 

By Definition \ref{MFPDO} and the property of Fourier transform, for any multi-index $\vec{\beta} \in \N_0^{mn}$, we have
\begin{equation*}
{K_{a}}(x,\vec y) = C\frac{{{\mathcal{F}^{ - 1}}\left( {\partial _ \cdot ^{\vec \beta }a(x, \cdot )} \right)\left( {\vec y} \right)}}{{{{ {\vec y} }^{\vec{\beta}} }}}.
\end{equation*}
Taking $|\vec{\beta}|=mn$ as in our condition, we have ${\tau _i} - \rho \left| {{\beta _i}} \right| + n < 0$. It follows that
\begin{align*}
|{K_a}(x,\vec{x}-\vec y)|  \lesssim 
\frac{{\prod\limits_{i = 1}^m {\int_{{\rn}} {{{(1 + \left| {{\xi _i}} \right|)}^{\left( {{\tau _i} - \rho \left| {{\beta _i}} \right|} \right)}}} {\mkern 1mu} d{\xi _i}} }}{{{{(\sum\limits_{i = 1}^m | x - {y_i}|)}^{mn }}}} 
\approx \frac{{\prod\limits_{i = 1}^m {\int_1^\infty  {{r^{{\tau _i} - \rho \left| {{\beta _i}} \right| + n - 1}}dr} } }}{{{{(\sum\limits_{i = 1}^m | x - {y_i}|)}^{mn }}}}:=\frac{C}{{{{(\sum\limits_{i = 1}^m | x - {y_i}|)}^{mn}}}},
\end{align*}
where $\vec x - \vec y: = \left( {x - {y_1}, \cdots ,x - {y_m}} \right).$

Given $Q_0 \in \mathcal{Q}$ and $x \in Q_0$, choosing $Q=5Q_0$, it derives that
\begin{align}\label{Tsig}
\big|&T_{a}(\vec{f} \mathbf{1}_{Q^\dagger})(x) - T_{a}(\vec{f} \mathbf{1}_{Q_0^\dagger})(x)\big|
\nonumber \\
&=\bigg|\int_{(Q^\dagger)^m \setminus (Q_0^\dagger)^m} {K_{a}}(x, \vec{x} -\vec{y}) \prod_{i=1}^m f_i(y_i) \, d\vec{y}\bigg|
\nonumber \\
&\lesssim \int_{(Q^\dagger)^m \setminus (Q_0^\dagger)^m} \frac{\prod_{i=1}^m |f_i(y_i)|}{(\sum_{i=1}^m |x-y_i|)^{mn}} \, d\vec{y} 
\lesssim {\prod\limits_{i=1}^m\left\langle f_i\right\rangle_{r_i,Q^\dagger}}.
\end{align}
This implies that $T_{a}$ satisfies the condition \eqref{M_1.condi.}. 

From the proof of \cite[(4.2)]{Cao2020}, we indeed can get that 
\[\left| {{T_{a}}(\vec f{\chi _{{{\left( {5Q} \right)}^c}}})(z) - {T_{a}}(\vec f{\chi _{{{\left( {5Q} \right)}^c}}})(z')} \right| \lesssim \prod\limits_{i = 1}^m {{{\left\langle {\left\langle {{f_i}} \right\rangle } \right\rangle }_{{r_i},Q}}},\]
and \eqref{M_2.condi.} follows instantly from this.

Finally, it is easy to see $T_a \in W_{\vec{1},\frac{1}{m}}$ (in fact, ${T_{{a}}}:{L^1} \times  \cdots  \times {L^1} \to {L^{\frac{1}{m},\infty }}$ boundedly), and one can refer to the proof of \cite[Theorem 3.3]{Cao2020}.
\end{proof}

\subsection{Multilinear fractional maximal operators}
~~

In this subsection, ball-basis measure space $(X,\mu,\mathcal{B})$ is the base space. 

Let $(X,\mu,\mathcal{B})$ be a ball-basis measure space. Given $1 \leq r_1,\cdots,r_m<\infty$, $\eta=\sum\limits_{i=1}^m \eta_i$, we define multilinear fractional maximal operators by that if $x \in \bigcup_{B \in \B} B$, 
$$
\mathcal{M}^{\B}_{\vec{\eta}, \vec{r}}(\vec{f})(x):=\sup _{B \in \B} \prod_{i=1}^m\left\langle |f_i|\right\rangle_{\eta_i,r_i,B} \cdot \chi_B(x)=\sup _{B \in \B} \mu(B)^\eta \prod_{i=1}^m\left\langle |f_i|\right\rangle_{r_i,B} \cdot \chi_B(x),
$$ 
$$
\mathcal{M}^{\B,\otimes}_{\vec \eta,\vec{r}}(\vec{f})(x):=\left( {\mathop  \otimes \limits_{i = 1}^m M_{{\eta _i},{r_i}}^B({f_i})} \right)\left( {x, \cdots ,x} \right)=\prod_{i=1}^m M^{\B}_{\eta_i, r_i} (f_i)(x).
$$
Otherwise, we define $\mathcal{M}^{\B}_{\vec \eta,\vec{r}}(\vec{f})(x)=\mathcal{M}^{\B,{\otimes}}_{\vec \eta,\vec{r}}(\vec{f})(x)=0$. When $\vec{r} = \vec{1}$, we use the simplified notation $\mathcal{M}^{\B}_{\vec{\eta}} := \mathcal{M}^{\B}_{\vec{\eta},\vec{1}}$. For the case $m=1$, we write $M^{\B}_{\eta,r} := \mathcal{M}^{\B}_{\vec \eta,\vec r}$.
Also, we will abuse the symbols $\mathcal{M}^{\B}_{\vec{\eta}, \vec{r}}$ and $\mathcal{M}^{\B}_{{\eta}, \vec{r}}$ in this paper to facilitate emphasis.



\begin{proposition}\label{lem:M_0}
Let $(X,\mu,\mathcal{B})$ be a ball-basis measure space. Let $\eta \in [0,1)$, $r \in [1,\infty)$ with $\frac{1}{\tilde{r}}=\frac{1}{r} - \eta$, then
\begin{align}
\|M^{\B}_{\eta,r}\|_{L^r(X, \mu) \to L^{\tilde{r}, \infty}(X, \mu)} \le \C_0^{\frac{1}{\tilde{r}}}\label{Mr-1} 
\end{align}
\end{proposition}


\begin{proof}
For a fixed \( \lambda > 0 \), define the level set: $E = \{ x \in X : M_{\eta,r}^{\mathcal{B}} f(x) > \lambda \}.$
For each \( x \in E \), there exists a ball \( B(x) \in \B \) such that \( x \in B(x) \) and
\[
\int_{B(x)} |f(t)|^r \, d\mu(t) > \lambda^r \mu(B(x))^{1 - \eta r}.
\]
Thus, \( E = \bigcup_{x \in E} B(x) \). Given \( B \in \mathcal{B} \) with \( \mu(B) < \infty \), \cite[Lemma 3.1]{K2019} applied to \( \{ B(x) : x \in E \cap B \} \) yields disjoint \( \{ B_k \}_k \) with \( E \cap B \subseteq \bigcup_k B_k^1 \). The union \( Q(B) = \bigcup_k B_k^1 \) is measurable. Therefore, it follows that
\begin{align*}
    \mu(Q(B)) \leq \sum_k \mu(B_k^{1}) \leq \C_0 \sum_k \mu(B_k) < \C_0 \sum_k \frac{1}{\lambda^{\tilde{r}}} \left( \int_{B_k} |f(t)|^r \, d\mu(t) \right)^{\tilde{r} / r}.
\end{align*}
Considering \( \tilde{r} / r > 1 \) and the \( B_k \) are disjoint, we have
\[
\sum_k \left( \int_{B_k} |f|^r \, d\mu \right)^{\tilde{r} / r} \leq \left( \sum_k \int_{B_k} |f|^r \, d\mu \right)^{\tilde{r} / r} = \left( \int_{\bigcup_k B_k} |f|^r \, d\mu \right)^{\tilde{r} / r} \leq \| f \|_{L^r}^{\tilde{r}},
\]
Thus, $\mu(Q(B)) < \frac{\C_0}{\lambda^{\tilde{r}}} \| f \|_{L^r}^{\tilde{r}}$.

Next, we extend the local construction to $X$ via \cite[Lemma 3.2]{K2019}, which provides a sequence $\{G_k\}$ satisfying 
$X = \bigcup_{n=1}^\infty \bigcap_{k=n}^\infty G_k,$
and consequently $E \subseteq \bigcup_{n=1}^\infty \bigcap_{k=n}^\infty Q(G_k)$
and
\[
\mu^*(E)=\mu\left(\bigcup_{n \geq 1} \bigcap_{k \geq n} Q\left(G_k\right)\right) \leq \frac{\C_0}{\lambda^{\tilde{r}}} \| f \|_{L^r}^{\tilde{r}}.
\]

\end{proof}

\begin{proposition}\label{lem:M}
Let $(X,\mu,\mathcal{B})$ be a ball-basis measure space. Let $r_i\in [1,\infty)$, $\eta_i \in [0,1)$ with $0 \le \eta_ir_i \le 1$ for $i=1,\cdots,m$, the fractional total order $\eta {\rm{ = }}\sum\limits_{i = 1}^m {{\eta _i}} \in [0,m)$, and $\frac{1}{\tilde{r}}=\sum_{i=1}^m \frac{1}{r_i} - \eta>0$. Then
\begin{align}
\label{Mr-33} & \|\mathcal{M}^{\B,\otimes}_{\vec \eta, \vec r}\|_{L^{r_1}(X, \mu) \times \cdots \times L^{r_m}(X, \mu) \to L^{\tilde{r}, \infty}(X, \mu)} 
\le C_m, 
\\ 
\label{Mr-3} & \|\mathcal{M}^{\B}_{\vec \eta, \vec r}\|_{L^{r_1}(X, \mu) \times \cdots \times L^{r_m}(X, \mu) \to L^{\tilde{r}, \infty}(X, \mu)} 
\le C_m. 
\end{align}
\end{proposition}
\begin{proof}
To establish \eqref{Mr-33} and \eqref{Mr-3}, recall the weak-type H\"older inequality from \cite[p.~16]{249}, for all $\frac1p=\sum_{i=1}^m \frac{1}{p_i}$ with $0<p_1, \ldots, p_m<\infty$, we have
\begin{align}\label{Hol-weak}
\|f_1 \cdots f_m\|_{L^{p, \infty}(X, \mu)} 
\le p^{-\frac1p} \prod_{i=1}^m p_i^{\frac{1}{p_i}} \|f_i\|_{L^{p_i, \infty}(X, \mu)}.
\end{align}
Set $\frac{1}{\tilde{r_i}}= \frac{1}{r_i} - \eta_i$, it follows from \eqref{Hol-weak} and \eqref{Mr-1} that 
\begin{align*}
\left\|\mathcal{M}^{\B}_{\vec \eta, \vec r}(\vec{f})\right\|_{L^{\tilde{r}, \infty}(X, \mu)} 
\le \left\|\mathcal{M}^{\B,\otimes}_{\vec \eta,\vec r}(\vec{f})\right\|_{L^{\tilde{r}, \infty}(X, \mu)} 
= \bigg\|\prod_{i=1}^m M^{\B}_{\eta_i,r_i} (f_i) \bigg\|_{L^{\tilde{r}, \infty}(X, \mu)} 
\\
\le (\tilde{r})^{-\frac{1}{\tilde{r}}} \prod_{i=1}^m \tilde{r_i}^{\frac{1}{\tilde{r_i}}} \left\|M^{\B}_{\eta_i,r_i}(f_i)\right\|_{L^{\tilde{r_i}, \infty}(X, \mu)}
\le \left(\frac{\C_0}{\tilde{r}}\right)^{\frac{1}{\tilde{r}}} \prod_{i=1}^m \tilde{r_i}^{\frac{1}{\tilde{r_i}}} \|f_i\|_{L^r(X, \mu)}. 
\end{align*}
The proof is complete. 
\end{proof}

\begin{theorem}\label{ca:2.1}
Let $(X,\mu,\mathcal{B})$ be a ball-basis measure space. Let $\eta_i \in [0,1)$ with the fractional total order $\eta:=\sum\limits_{i=1}^m\eta_i$, and $1 \le r_1,\cdots,r_m < \infty$ with 
$\frac{1}{\tilde r}:=\sum\limits_{i = 1}^m \frac{1}{{{r_i}}}-\eta >0$. Then multilinear fractional maximal operator \(\mathcal{M}^{\B}_{\vec{\eta},\vec{r}}\) is a $\mathbb{R}$-valued {\rm $m$-FBMOO} satisfying $W_{\vec{r},\tilde{r}}$.
\end{theorem}

\begin{remark}
In this subsection, \(\mathcal{M}^{\B}_{\vec{\eta},\vec{r}}\) can be regarded as a $\R$-valued {\rm $m$-FBMOO} on  ball-basis measure $(X,\mu,\mathcal{B})$, satisfying condition ${W_{\vec r,\tilde r}}$ with $
\frac{1}{\tilde{r}} := \sum_{i = 1}^m \frac{1}{r_i} - \eta > 0,$
then all results in Sect. \ref{Main} apply to them, except for Theorems \ref{thm:sparse-T1} and \ref{thm:DT}, since they do not admit a \textbf{\textsf{$\mathbb{R}$-valued $m$-linear fractional Calderón-Zygmund kernel representation.}}

\end{remark}

\begin{proof}

To consider \eqref{M_1.condi.}, we claim that
for any fixed $B_0 \in \B$, there is a \( B \) that satisfies $B \supsetneq B_0$ as in  \eqref{M_1.condi.}.

Indeed, we set
\begin{align*}
	\B^{\prime}:=\left\{B \in \B: B \cap B_0 \neq \emptyset, \mu(B)>\mu\left(B_0\right)\right\}.
\end{align*}
Based on the construction of $\B^{\prime}$, by \eqref{list:B4}, we can find a ball $B_1 \in \B^{\prime}$ such that
\begin{align*}
b \leq \mu\left(B_1\right)<2 b \text{\quad and \quad}
	B_0 \subseteq \bigcup_{B' \in \B: B' \cap B_1 \neq \emptyset \atop \mu(B') \le 2\mu(B_1)} B' \subseteq B_1^{\dagger},
	\end{align*}
 where $b:=\inf\limits_{B \in \B^{\prime}} \mu(B)$. Since $ B_0 \subsetneq B_1 $, then we take $B:=B_1^{\dagger}$.

For any $\varepsilon > 0$ and any nonnegative $f_i \in L^{r_i}(X, \mu)$, $i=1,\cdots,m$, there exists $B_2 \in \B$ with $B_2 \ni x$, such that
 \begin{align}\label{eq:ca:2.2:setpre}
\mathcal{M}^{\B}_{\vec{\eta},\vec{r}}\left(\vec{f} \chi_{B^\dagger}\right)(x) -\varepsilon  \leq \prod_{i=1}^m\left\langle f_i \chi_{B^\dagger}\right\rangle_{\eta_i,r_i,B_2}.
 \end{align}
 
\emph{{\textbf{\textsf Case 1: $\mu\left(B_2\right) \leq \mu\left(B_0\right)$}}}

It follows from \eqref{list:B4} that 
$B_2 \subseteq B_0^{\dagger} \subseteq B^\dagger,$
and then we have
\begin{align*}
0 \le & \mathcal{M}^{\B}_{\vec{\eta},\vec{r}}\left(\vec{f} \chi_{B^\dagger}\right)(x)-\mathcal{M}^{\B}_{\vec{\eta},\vec{r}}\left(\vec{f} \chi_{B_0^\dagger}\right)(x) \le\varepsilon+\prod_{i=1}^m\left\langle f_i \chi_{B^\dagger}\right\rangle_{\eta_i,r_i,B_2}-\mathcal{M}^{\B}_{\vec{\eta},\vec{r}}\left(\vec{f} \chi_{B_0^\dagger}\right)(x)\\
\le& \varepsilon+\prod_{i=1}^m\left\langle f_i \chi_{B^\dagger}\right\rangle_{\eta_i,r_i,B_2}-\prod_{i=1}^m\left\langle f_i \chi_{B_0^\dagger}\right\rangle_{\eta_i,r_i,B_2}= \varepsilon.
\end{align*}
Thus, letting  $\varepsilon<\prod\limits_{i=1}^m\left\langle f_i\right\rangle_{\eta_i,r_i,B^\dagger}$, the desired is obtained.

\emph{{\textbf{\textsf Case 2: $\mu\left(B_2\right) > \mu\left(B_0\right)$}}}

We can find that $B_2 \in \mathcal{B}'$, then $$\mu\left(B^\dagger\right) \ge \mu\left(B_2\right) \geq b \ge \frac{1}{2}\mu\left(B_1\right) \ge \frac{1}{2C_0} \mu\left(B_1^\dagger\right) \ge \frac{1}{2C_0^2} \mu\left(B_1^2\right)= \frac{1}{2C_0^2} \mu\left({B^\dagger} \right).$$
It derives from \eqref{eq:ca:2.2:setpre} that
\begin{align*}
0 \le& \mathcal{M}^{\B}_{\vec{\eta},\vec{r}}\left(\vec{f} \chi_{B^\dagger}\right)(x)-\mathcal{M}^{\B}_{\vec{\eta},\vec{r}}\left(\vec{f} \chi_{B_0^\dagger}\right)(x) \le\varepsilon+\prod_{i=1}^m\left\langle f_i \chi_{B^\dagger}\right\rangle_{\eta_i,r_i,B_2}\\
\lesssim &\varepsilon+\prod_{i=1}^m\left\langle f_i \chi_{B^\dagger}\right\rangle_{\eta_i,r_i,{B^\dagger}} \lesssim \prod_{i=1}^m\left\langle f_i \right\rangle_{\eta_i,r_i,{B^\dagger}},
\end{align*}
where $\varepsilon<\prod\limits_{i=1}^m\left\langle f_i\right\rangle_{\eta_i,r_i,B^\dagger}$.

Therefore, \eqref{M_1.condi.} is valid by collecting these estimates. 

Next, we consider \eqref{M_2.condi.} as follows.

Let $B \in \mathcal{B}$ and $x,x' \in B$. For any nonnegative $f_i \in L^{r_i}$, $i=1,\cdots,m$., there is a $B_3 \in \mathcal{B}$ with $B_3 \ni x$, such that
\begin{align}\label{eq:ca:2.5:setpre}
\mathcal{M}^{\B}_{\vec{\eta},\vec{r}}(\vec{f})(x)-\prod_{i=1}^m\left\langle\langle f_i\right\rangle\rangle_{\eta_i,r_i,B}\leq \prod_{i=1}^m\left\langle f_i\right\rangle_{\eta_i,r_i,B_3}.
\end{align}

{\emph{{\textbf{\textsf Case 3: $\mu\left(B_3\right) \leq \mu\left(B_0\right)$}}}}

Similar to Case 1, \eqref{list:B4} implies that 
$B_3\subseteq B^\dagger,$ and then 
$$\mathcal{M}^{\B}_{\vec{\eta},\vec{r}}(\vec{f}\chi_{B^\dagger})(x) \ge \prod_{i=1}^m\left\langle f_i\right\rangle_{\eta_i,r_i,B_3}.$$
Thus, 
\begin{align}
0 \le \mathcal{M}^{\B}_{\eta,\vec{r}}(\vec{f})(x)-\mathcal{M}^{\B}_{\eta,\vec{r}}\left(\vec{f} \chi_{B^\dagger}\right)(x) \le& \prod_{i=1}^m\left\langle\langle f_i\right\rangle\rangle_{\eta_i,r_i,B}+ \prod_{i=1}^m\left\langle f_i\right\rangle_{\eta_i,r_i,B_3}-\mathcal{M}^{\B}_{\eta,\vec{r}}\left(\vec{f} \chi_{B^\dagger}\right)(x) \notag\\ 
\le& \prod_{i=1}^m\left\langle\langle f_i\right\rangle\rangle_{\eta_i,r_i,B}.\label{eq:n:T2sat}
\end{align}

{\emph{{\textbf{\textsf Case 4: $\mu\left(B_3\right) > \mu\left(B_0\right)$}}}}

\eqref{list:B4} implies that $B \subseteq B_3^{\dagger}$ and 
$$\prod_{i=1}^m\left\langle f_i\right\rangle_{\eta_i,r_i,B_3} \lesssim \prod_{i=1}^m\left\langle f_i\right\rangle_{\eta_i,r_i,B_3^\dagger} \le \prod_{i=1}^m\left\langle\langle f_i\right\rangle\rangle_{\eta_i,r_i,B},$$
and \eqref{eq:ca:2.5:setpre} gives that 
\begin{align}
0 \le &\mathcal{M}_{\vec{\eta},\vec{r}}^\B(\vec{f})(x)-\mathcal{M}_{\vec{\eta},\vec{r}}^\B\left(\vec{f} \chi_{B^\dagger}\right)(x)
  \leq \prod_{i=1}^m\left\langle f_i\right\rangle_{\eta_i,r_i,B_3}+\prod_{i=1}^m\left\langle\langle f_i\right\rangle\rangle_{\eta_i,r_i,B}
\lesssim \prod_{i=1}^m\left\langle\langle f_i\right\rangle\rangle_{\eta_i,r_i,B}. \label{eq:II}
\end{align}

Hence, combining Case 3 and Case 4, \eqref{M_2.condi.} holds by the following that 
\begin{align*}
\left|\left(\mathcal{M}_{\vec{\eta},\vec{r}}^\B(\vec{f})(x)-\mathcal{M}_{\vec{\eta},\vec{r}}^\B(\vec{f}\chi_{B^\dagger})(x) \right)- \left(\mathcal{M}_{\vec{\eta},\vec{r}}^\B(\vec{f})(x') -\mathcal{M}_{\vec{\eta},\vec{r}}^\B(\vec{f}\chi_{B^\dagger})(x')\right)\right| \lesssim \prod_{i=1}^m\left\langle\langle f_i\right\rangle\rangle_{\eta_i,r_i,B}. 
\end{align*}

Finally, by Proposition \ref{lem:M}, we have that $\mathcal{M}_{\vec{\eta},\vec{r}}^\B \in W_{\vec{r},\tilde{r}}$. 
\end{proof}

\section{\bf Preliminaries}\label{Pre}

\subsection{Geometry of measure spaces}~~

Let \((X, \mu, \mathcal{B})\) be a ball-basis  measure space, where $\B$ a basis satisfying property \eqref{list:B4}. This property implies that for any $A, B \in \B$,
\begin{align}\label{ball} 
A \cap B \neq \emptyset \text{ and } \mu(A) \le 2\mu(B) \implies A \subseteq B^\dagger.
\end{align}
For any $B \in \B$, define recursively:
\[
B^0 := B, \quad 
B^{k+1} := (B^k)^\dagger \quad (k \ge 0).
\]
Then \eqref{list:B4} yields the following estimates:
\begin{align}\label{BkBk}
\mu(B^{k+1}) \le \C_0 \mu(B^k) \quad \text{and} \quad \mu(B^k) \le \C_0^k \mu(B) \quad (k \ge 0).
\end{align}

Below we offer a concise account of the geometric features of measure spaces, following \cite{K2019}.

\begin{lemma}\label{lem:BG}
	\noindent Let $(X,\mu,\B)$ be a measure space with ball--basis~$\B$. Suppose there is a sequence of balls $\{G_k\}\subseteq\B$ and a ball $B\in\B$ such that $G_k\cap B\neq\emptyset$ for every $k$ and $\mu(G_k)\to r:=\sup_{A\in\B}\mu(A)$ as $k\to\infty$. Then one shows that $X\subseteq\bigcup_{k=1}^\infty G_k^*$, further, for each ball $A\in\B$ there is a $k_0\in\mathbb{N}$ so that $A\subseteq G_k$ for all $k\ge k_0$.  
\end{lemma}

\begin{lemma}\label{lem:aooa:3.3}
Let \((X, \mu, \mathcal{B})\) is a measure space with a ball-basis. Then the following hold: 
\begin{list}{\rm (\theenumi)}{\usecounter{enumi}\leftmargin=1.2cm \labelwidth=1cm \itemsep=0.2cm \topsep=.2cm \renewcommand{\theenumi}{\alph{enumi}}}

\item\label{list-1} 
Let \(E\subseteq X\) be bounded and satisfy \(E\subseteq\bigcup_{G\in\mathcal{G}}G\) for some \(\mathcal{G}\subseteq\mathcal{B}\).  Then there is a pairwise disjoint subfamily \(\mathcal{G}'\subseteq\mathcal{G}\) (finite or infinite) with $  E\subseteq\bigcup_{G\in\mathcal{G}'}G^*.$

\item\label{list-2} Let \(A\in\mathcal{B}\) and let \(\mathcal{G}\subseteq\mathcal{B}\) be a pairwise disjoint family for which 
\[
A\cap G^{\dagger}\neq\emptyset
\quad\text{and}\quad
0<c_{1}\le\mu(G)\le c_{2}<\infty
\quad\forall G\in\mathcal{G}.
\]
Then one obtains the bound
\[
\#\mathcal{G}\;\le\;\frac{1}{c_{1}}\min\bigl\{\,c_{2}\,\mathbf{C}_{0}^{3},\;\mathbf{C}_{0}\,\mu(A)\bigr\}.
\] 

\item\label{list-4}
If \(\B\) also satisfies property~\eqref{list:B3}, then for every bounded measurable set \(E\) with \(\mu(E)>0\) and each \(\varepsilon>0\), one can select a sequence \((B_k)\subseteq\B\) so that  
\[
\mu\bigl(\bigcup_k B_k\setminus E\bigr)<\varepsilon,
\qquad
\mu\bigl(E\setminus\bigcup_k B_k\bigr)<\alpha\,\mu(E),
\]
where \(\alpha\in(0,1)\) is a fixed constant.

\item\label{list-5} Moreover, if \(\B\) satisfies property~\eqref{list:B3}, then for every bounded measurable set \(E\) one can find a sequence \((B_k)\subseteq\B\) such that
\[
E\subseteq\bigcup_k B_k\quad\text{a.s.}
\quad\text{and}\quad
\sum_k\mu(B_k)\le2\,\mathbf C_0\,\mu(E).
\]

\item\label{list-3} The density condition is exactly equivalent to property~\eqref{list:B3}.
 
\end{list} 
\end{lemma}

{
\subsection{Spaces of homogeneous type }
~~

\begin{definition}\label{quasi-metric}
	Let $d: X \times X \rightarrow [0, \infty)$ be a positive function and $X$ be a set, then the quasi-metric space $(X, d)$ satisfies the following conditions:
	\begin{enumerate}
		\item  When $x=y$, $d(x, y)=0$.
		\item $d(x, y)=d(y, x)$ for all $x, y \in X$.
		\item  For all $x, y, z \in X$, there is a constant $A_0 \geq 1$ such that $d(x, y) \leq A_0(d(x, z)+d(z, y))$.
	\end{enumerate}
\end{definition}

\begin{definition}\label{def:doubling}
Let $\mu$ be a measure of a space $X$. For a quasi-metric ball $B(x,  r)$ and any $r>0$, if $\mu$ satisfies doubling condition, then there exists a doubling constant $C_\mu \geq 1$, such that  
	\begin{align}\label{def_hom}
    0<\mu(B(x, 2 r)) \leq C_\mu \mu(B(x, r))<\infty.
  \end{align}
\end{definition}

\begin{definition}
	For a non-empty set $X$ with a qusi-metric $d$, a triple $(X, d, \mu)$ is said to be a space of homogeneous type if $\mu$ is a regular measure which satisfies doubling condition on the $\sigma$-algebra, generated by open sets and quasi-metric balls.
\end{definition}

\subsection{Dyadic analysis}\label{dyadic}~~

We will use the following definition of a dyadic lattice in $(X,d,\mu)$.


\begin{definition}
Let $0 < c_0 \leq C_0 < \infty$, and $0 < \delta < 1$. Let $J_k$ be a index set. For each $k$, let $\d_k = \{Q_j^k\}_{j \in J_k}$ be a collection of measurable sets and  $\{z_j^k\}_{j \in J_k}$ be a set of points. $\d := \bigcup_{k \in \mathbb{Z}} \d_k$ is called a \emph{dyadic lattice} with parameters $c_0$, $C_0$, and $\delta$ if it satisfies the following conditions:
\begin{list}{\rm (\theenumi)}{\usecounter{enumi}\leftmargin=1cm \labelwidth=1cm \itemsep=0.2cm \topsep=.2cm \renewcommand{\theenumi}{\roman{enumi}}}
\item For all $k \in \mathbb{Z}$ we have
\begin{align*}
X=\bigcup_{j \in J_k} Q_j^k;
\end{align*}
\item For $k \geq l, Q \in \d_k$ and $Q^{\prime} \in \d_l$ we either have $Q \cap Q^{\prime}=\emptyset$ or $Q \subseteq Q^{\prime}$;
\item For each $k \in \mathbb{Z}$ and $j \in J_k$ we have
\begin{align}\label{eq:contain}
B\left(z_j^k, c_0 \delta^k\right) \subseteq Q_j^k \subseteq B\left(z_j^k, C_0 \delta^k\right),
\end{align}
where $z_j^k$ and $\delta^k$ are called the center and side length of a cube $Q_j^k \in \d_k$, respectively;
\item 
for $l \ge k$, if $Q_{j'}^l \subseteq Q_j^k$, then $B(z_{j'}^l;C_0\delta^k) \subseteq B(z_{j}^l;C_0\delta^k)$.
\end{list} 
\end{definition}

We also define the restricted dyadic lattice $\d(Q):=\{P \in \d: P \subseteq Q\}$. 

We will say that an estimate depends on $\d$ if it depends on the parameters $c_0, C_0$ and $\delta$. We more consider the dilations of such cubes, 
in Euclidean space, the expansion of a cube can be easily determined using volume calculations. However, in a space of homogeneous type, we need to redefine the dilations $\dada P$ for $\dada \geq 1$ as
\begin{align*}
\dada P:=B\left(z, \dada \cdot C_0 \delta^k\right).
\end{align*}

In $\R^n$, we all know a fact that any ball is contained in a cube of comparable size from one of these dyadic lattices, cf. \cite[ Lemma 3.2.26]{hombook}, called three lattice lemma. We introduce the following to achieve the same effect on spaces of homogeneous type, cf. \cite[Proposition 2.1.1]{Lorist2021} and \cite{HK}.

\begin{proposition}[\cite{Lorist2021}, Proposition 2.1.1]\label{cubeeq}
  Let \((X, d, \mu)\) be a space of homogeneous type. There exist constants \(0 < c_0 \leq C_0 < \infty\), \(\gamma \geq 1\), \(0 < \delta < 1\), and an integer \(m \in \mathbb{N}\) such that there exist dyadic lattices \(\d^1, \ldots, \d^N\) with parameters \(c_0, C_0\), and \(\delta\), satisfying the following property: for each point \(x \in X\) and radius \(r> 0\), there exists an index \(j \in \{1, \ldots, N\}\) and a cube \(Q \in \d^j\) such that

\[
B(x, r) \subseteq Q, \quad \text{and} \quad \operatorname{diam}(Q) \leq \gamma r.
\]
\end{proposition}


We also observe that
\begin{lemma}[\cite{Lorist2021}, Lemma 2.1.3]\label{lem.covering}
  Let \((X, d, \mu)\) be a space of homogeneous type and \(\d\) a dyadic system with parameters \(c_0, C_0\), and \(\delta\). Suppose that \(\operatorname{diam}(X) = \infty\), take \(\dada \geq \frac{3 A_0^2}{\delta}\), and set \(E \subseteq X\) satisfy \(0 < \operatorname{diam}(E) < \infty\). Then there exists a partition \(\mathcal{D}' \subseteq \d\) of \(X\) such that \(E \subseteq \dada P\) for all \(Q \in \mathcal{D}'\).
\end{lemma}

\section{\bf Sparse domination for $m$-FBMOOs}\label{sec.sparse}

\subsection{Proof of Theorem~\ref{S.d.main}}~~

In this section, we present a direct proof of Theorem~\ref{S.d.main}, relying on key facts that will be established in the Section \ref{more.information}.
First, for clarity in the proof, we introduce the following operator:
	\[
	\varLambda(\vec{f})(x) := \max\big\{|T_{\vec \eta}(\vec{f})(x)|, \, T^{\star}_{\vec \eta}(\vec{f})(x), \, 
	\C \mathcal{M}_{\vec \eta, \vec r}^{\B,\otimes}(\vec{f})(x) \big\}, 
	\]
    where \[
\boldsymbol{C} := C_1 + C_2 + \Phi_{{T}_{\vec{\eta}},\vec{r},\tilde{r}}(\lambda),
\]
with $\Phi_{{T}_{\vec{\eta}},\vec{r},\tilde{r}}:(0,1)\to(0,\infty)$.
Then it follows from Proposition \ref{lem:M} and Theorem \ref{thm:aooa:4.1} below that $\varLambda \in W_{\vec{r}, \tilde{r}}$.


 \begin{proof}[Proof of Theorem~\ref{S.d.main}]~~

For any $B_0 \in \mathcal{B}$ with $\cup_{i=1}^m\supp(f_i) \subseteq B_0$. Due to the arbitrariness of $B_0$, it suffices to prove that there exists a $\frac{1}{2\C_0^3}$--sparse family $\mathcal{S} \subseteq \mathcal{B}$, such that for almost every $x \in B_0$,
\begin{align}\label{main0609}
   \left\|T_{\vec \eta}^{{\bf b, k}}(\vec f\chi_{B^3_0})(x)\right\|_{\mathbb{V}} 
 \lesssim \Phi_{\varLambda,\vec{r},\tilde{r}}\left(\frac{1}{\boldsymbol{C}^3_0 \lambda}\right) \sum\limits_{0 \le {\bf t} \le {\bf k}}\sum_{B \in \S}\prod_{i = 1}^m\big|b_i(x)-b_{i,B^3}\big|^{k_i-t_i} 
    \langle\big|\big(b_i-b_{i,B^3}\big)^{t_i} f_i\big|\rangle_{\eta_{i},r_{i},B^3},
\end{align}

For the above $B_0$, when $\lambda$ is large enough, Lemma~\ref{ppremain} below provides two $\frac{1}{2}$--sparse collections $\mathcal{Q}_1$ and $\mathcal{Q}_2$. Furthermore, setting $\mathcal{Q} = \mathcal{Q}_1 \cup \mathcal{Q}_2$ and $B_k = B$, we can find a sequence $\{B_j\}_{j=0}^k \subseteq \mathcal{Q}$ with $B_{j+1} \in \F(B_j)$ for $j=0,\ldots,k-1$, and for a.e. $x \in B_0$, there exists  cubes $\widehat{B}_j$ and points $\xi_j \in \widehat{B}_{j+1}$ satisfying \eqref{BBT-1}--\eqref{BBT-3}. 
Hence, we obtain
\begin{align}\notag
    \left\|T_{\vec \eta}^{{\bf b, k}}(\vec f\chi_{B^3_0})(x)\right\|_{\mathbb{V}}&\leq  \left\|T_{\vec \eta}^{{\bf b, k}}(\vec f\chi_{B^3})(x)\right\|_{\mathbb{V}} + \sum_{j=0}^{k-1}\left\|T_{\vec \eta}^{{\bf b, k}}(\vec f\chi_{B^3_j\backslash B^3_{j+1}})(x)\right\|_{\mathbb{V}}\\ \label{mainseq}
    &:= \mathbb{D}_1 + \mathbb{D}_2.
\end{align}
Additionally, in view of Lemma \ref{ppremain} below, it can be deduced that
\begin{align}\label{SS}
\text{$\S_i :=\{B^{3}: B \in \mathcal{Q}_i\}$ is $\frac{1}{2\C_0^3}$-sparse}, \quad i=1, 2. 
\end{align}
Therefore, to demonstrate \eqref{main0609}, it is sufficient to estimate $\mathbb{D}_1$ and $\mathbb{D}_2$.

Notice that 
\begin{align}\label{zhuyi1}
  &\quad \prod_{i = 1}^m (b_i(x) - b_i(y_i))^{k_i} = \sum\limits_{0 \le {\bf t} \le {\bf k}} \left( \prod_{i = 1}^m C_{k_i}^{t_i} (-1)^{t_i} [b_i(x) - b_{i, B^3}]^{k_i - t_i} [b_i(y_i) - b_{i, B^3}]^{t_i} \right),
\end{align}
where $\sum\limits_{0 \le {\bf t} \le {\bf k}} :=\sum\limits_{i{\rm{ = }}1}^m {\sum\limits_{{t_i} = 0}^{{k_i}} {} }$.

For $\mathbb{D}_1$, from the definition of $ T_{\vec \eta}^{{\bf b, k}}$ and \eqref{eq5.30} below, it follows that
\begin{align}\notag
	&\quad  \left\|T_{\vec \eta}^{{\bf b, k}}(\vec f\chi_{B^3})(x)\right\|_{\mathbb{V}}\\ \notag
    & \lesssim \sum\limits_{0 \le {\bf t} \le {\bf k}}\prod_{i = 1}^m\big|b_i(x)-b_{i,B^3}\big|^{k_i-t_i}   \left\|T_{\vec \eta}^{}\left(\left(b_i-b_{i,B^3}\right)^{t_i}f_i\chi_{B^3}\right)(x)\right\|_{\mathbb{V}} \\ \notag
	&\lesssim \Phi_{\varLambda,\vec{r},\tilde{r}}\left(\frac{1}{\boldsymbol{C}^3_0 \lambda}\right) \sum\limits_{0 \le {\bf t} \le {\bf k}}\prod_{i = 1}^m\big|b_i(x)-b_{i,B^3}\big|^{k_i-t_i} \varLambda\left(\left(b_i-b_{i,B^3}\right)^{t_i}f_i\chi_{B^3}\right)(x) \\ \label{eq5.35}
	&\lesssim \Phi_{\varLambda,\vec{r},\tilde{r}}\left(\frac{1}{\boldsymbol{C}^3_0 \lambda}\right) \sum\limits_{0 \le {\bf t} \le {\bf k}}\prod_{i = 1}^m\big|b_i(x)-b_{i,B^3}\big|^{k_i-t_i} 
 \langle\big|\big(b_i-b_{i,B^3}\big)^{t_i} f_i\big|\rangle_{\eta_{i},r_{i},B^3}
\end{align}
For $\mathbb{D}_2$, from Lemma \ref{est.d2}, for all $x \in B_k$ and $j=0, 1, \ldots, k-1$, it derives that
\begin{align}\notag
    &\left\|T_{\vec \eta}^{{\bf b, k}}(\vec{f} \chi_{B^3_j \backslash B^3_{j+1}})(x)  \right\|_{\mathbb{V}} 
 \\ \label{lem:TFA}
& \leq \Phi_{\varLambda,\vec{r},\tilde{r}}\left(\frac{1}{\boldsymbol{C}^3_0 \lambda}\right)  \sum\limits_{0 \le {\bf t} \le {\bf k}}\prod_{i = 1}^m\big|b_i(x)-b_{i_1,B^3_j}\big|^{k_i-t_i} 
\langle\big|\big(b_i-b_{i_1,B^3_j}\big)^{t_i} f_i\big|\rangle_{\eta_{i},r_{i},B^3_j} 
\end{align}
Setting $\mathcal{S} = \mathcal{S}_1 \cup \mathcal{S}_2$,  \eqref{main0609} follows from \eqref{mainseq}, \eqref{eq5.35},
 and \eqref{lem:TFA}.

\end{proof}

\subsection{More information about $m$-FBMOOs}\label{more.information}
~~


We define its localized maximal operator by
\begin{align}\label{txing}
    {T}^{\star}_{\vec \eta,s}(\vec{f})(x):=\sup _{B \in \B}\left\langle\left\|{{T}_{\vec\eta}}\left(\vec{f} \chi_{(B^\dagger)^c}\right)(\cdot)\right\|_{\mathbb{V}}\right\rangle_{s,B}\cdot\chi_B(x),
\end{align}
where $s = \min\{s_1,s_2\}$,  and $s_1,s_2$ referred to Definition \ref{def:FBMOO}.

We say $\mathcal{B}$ is \emph{doubling} if there exists a constant $\gamma > 1$ such that for every ball $B \in \mathcal{B}$ with $B^\dagger \subsetneq X$, there exists $B' \in \mathcal{B}$ containing $B$ with $\mu(B') \leq \gamma \mu(B)$.

\begin{definition}
Let \( A_0 \in \mathcal{B} \). For any \( A, B \in \mathcal{B} \) with \( A \subseteq B \), and the condition \eqref{M_1.condi.}, we define
\begin{align}\label{Deltaaa}
\left\bracevert A, B \right\bracevert_{s_1}:=\sup _{\substack{f_i \in L_b^{r_i}(X) \, i=1, \ldots, m}}  
\frac{\left\langle \left\|{{T}_\eta}\left(\vec{f} \chi_{B^\dagger\backslash A^{\dagger}}\right)(\cdot)\right\|_{\mathbb{V}}\right\rangle_{s_1,A}}{\prod\limits_{i=1}^m\left\langle f_i\right\rangle_{\eta_i,r_i,B^\dagger}}.
\end{align}
\end{definition}

\begin{theorem}\label{thm:aooa:4.1}
Let \((X, \mu, \mathcal{B})\) be a ball-basis  measure space. Assume that ${T}_{\vec \eta}$ is a $\mathbb{V}$-valued multilinear operator satisfying the condition \eqref{M_2.condi.} and admitting ${T}_{\vec \eta} \in W_{\vec{r},\tilde{r}}$. Then,
\begin{list}{\rm (\theenumi)}{\usecounter{enumi}\leftmargin=1cm \labelwidth=1cm \itemsep=0.2cm \topsep=.2cm \renewcommand{\theenumi}{\roman{enumi}}}
\item ${{T}}^{\star}_{\vec \eta,s} \in W_{\vec{r},\tilde{r}}$.
\item When $\mathcal{B}$ is doubling, ${T}_{\vec \eta}$ satisfies the condition \eqref{M_1.condi.}, i.e. ${T}_{\vec \eta}$ is an {\rm $m$-FBMOO}.

\end{list}
\end{theorem}

\begin{proof}[Proof of Theorem \ref{thm:aooa:4.1}]~~

\emph{{\textbf{\textsf (i):}}}
From Definition \ref{def:multilinear_W}, given $\lambda_1 \in (0,1)$, and setting $\lambda= \frac{2\boldsymbol{C}_0 \lambda_1}{1 - \lambda_1} \in (0,1)$, we define
\begin{align}\label{asetting1}
\Phi_{{T}^{\star}_{\vec \eta,s},\vec{r},\tilde{r}}(\lambda):= \Phi_{\mathcal{M}_{\eta,\vec{r}}^{\B,\otimes},\vec{r},\tilde{r}}\left(\lambda_{1}\right)\left( 2\Phi_{T_{\vec \eta},\vec{r},\tilde{r}}\left(\lambda_1\right)+ {C}_{2}\right)+1,
\end{align}
where $C_2$ is from Definition \ref{def:FBMOO}.


To demonstrate that \( {T}^{\star}_{\vec \eta,s} \in W_{\vec{r},\tilde{r}} \), by Definition \ref{def:multilinear_W}, it suffices to prove that for any $\lambda \in (0,1)$ and every $R \in \B$ with $B_0 \subseteq R$, 
\begin{align}
\mu(E(B_0)) \leq \lambda \mu(B_0)
\label{keyest.2}
\end{align}
where $E(B_0):=\left\{x \in B_0: {T}^{\star}_{\vec \eta,s}(\vec{f})(x)>\Phi_{{T}^{\star}_{\vec \eta,s},\vec{r},\tilde{r}}\left(\lambda\right)\prod_{i=1}^m\left\langle f_i\right\rangle_{\eta_i,r_i,B_0}\right\}$, for every $f_i \in L^{r_i}$ with $\supp f_i \subseteq R$, $i=1\cdots,m$. 

On the one hand, from the construction of ${T}^{\star}_{\vec \eta,s}$, see \eqref{txing}, if for any \( x \in E(B_0) \), for any $B \in \B$ containing \( x \), the reverse inequality of \eqref{keyest.1} blow is valid, then $E(B_0)=\emptyset$. This implies \eqref{keyest.2} is valid. 

On the other hand, if for any \( x \in E(B_0) \), there exists \( B_x \in \B \) containing \( x \) such that
\begin{align}
\left\langle\left\|T_{\vec \eta}\left(\vec{f} \chi_{(B_x^{\dagger})^c}\right)\right\|_{\mathbb{V}}\right\rangle_{s_2,B_x^{\dagger}} \geq \Phi_{{T}^{\star}_{\vec \eta,s},\vec{r},\tilde{r}}\left(\lambda\right) \prod_{i=1}^m\left\langle f_i\right\rangle_{\eta_i,r_i,B_0},\label{keyest.1}
\end{align}
then $E\left( {{B_0}} \right) \subseteq \mathop  \cup \limits_{x \in E\left( {{B_0}} \right)} {B_x}.$ Further, by Lemma \ref{lem:aooa:3.3} \eqref{list-1}, there is a disjoint subfamily \( \{B_k\} \) such that \( E(B_0) \subseteq \bigcup_k B_k^{\dagger} \).

For each \( k \) and any $\lambda_1,\lambda_2 \in (0,1)$, we define the sets
\begin{align*}
B_{k,0}&:=\left\{x \in B_k \cap B_0:\left\|T_{\vec \eta}\left(\vec{f} \chi_{B_k^{\dagger}}\right)(x)\right\|_{\mathbb{V}} \leq \Phi_{T_{\vec \eta},\vec{r},\tilde{r}}\left(\lambda_1\right)\prod_{i=1}^m\left\langle\left\langle f_i\right\rangle\right\rangle_{\eta_i,r_i,B_k}\right\}, \\
B_{1}&:=\left\{x \in B_0:\left\|T_{\vec \eta}(\vec{f})(x)\right\|_{\mathbb{V}} > \Phi_{T^{}_{\vec\eta},\vec{r},\tilde{r}}\left(\lambda_{1}\right)\prod_{i=1}^m\left\langle f_i\right\rangle_{\eta_i,r_i,B_0}\right\}, \\
B_{2}&:=\left\{x \in B_0: \mathcal{M}^{\B,{\otimes}}_{\vec{\eta},\vec{r}}(\vec{f})(x)> \Phi_{\mathcal{M}_{\vec{\eta},\vec{r}}^{\B,\otimes},\vec{r},\tilde{r}}\left(\lambda_{2}\right)\prod_{i=1}^m\left\langle f_i\right\rangle_{\eta_i,r_i,B_0}\right\}.
\end{align*}
We now claim that 
\begin{align}
\bigcup_k B_{k,0} \subseteq B_{1} \cup B_{2}.\label{keyest.3}
\end{align}
Indeed, for every $k$, for any \( x \in B_{k,0} \setminus B_{1} \), it derives that
\begin{align*}
&\Phi_{{T}^{\star}_{\vec \eta,s},\vec{r},\tilde{r}}\left(\lambda\right) \prod_{i=1}^m\left\langle f_i\right\rangle_{\eta_i,r_i,B_0}\\
\le & \left\langle\left\langle\left\|T_{\vec \eta}\left(\vec{f} \chi_{(B_k^{\dagger})^c}\right)(\cdot)\right\|_{\mathbb{V}}\right\rangle_{s_2,B_k^{\dagger}}\right\rangle_{s_3,B^{\dagger}_k \ni x}\\ \notag
 \leq& \left\langle \left\|T_{\vec \eta}(\vec{f})\right\|_{\mathbb{V}}\right\rangle_{s_2,B_k^{\dagger}}+ \left\langle \left\|T_{\vec \eta}\left(\vec{f} \chi_{B_k^{\dagger}}\right)\right\|_{\mathbb{V}}\right\rangle_{s_2,B_k^{\dagger}} + \left\langle\left\langle\left\|T_{\vec \eta}\left(\vec{f} \chi_{(B_k^{\dagger})^c}\right)(x)-T_{\vec \eta}\left(\vec{f} \chi_{(B_k^{\dagger})^c}\right)\left(\cdot\right)\right\|_{\mathbb{V}}\right\rangle_{s_2,B_k^{\dagger}} \right\rangle_{s_3,B_k^{\dagger} \ni x}\\ \notag
= :& I_1 + I_2 + I_3.
\end{align*}
For $I_1$, it is derived from the structure of $B_{1}$ that
\begin{align*}
    I_1 & \leq \Phi_{T^{}_{\vec\eta},\vec{r},\tilde{r}}\left(\lambda_{1}\right)\prod_{i=1}^m\left\langle f_i\right\rangle_{\eta_i,r_i,B_0}\leq \Phi_{T^{}_{\vec\eta},\vec{r},\tilde{r}}\left(\lambda_{1}\right) \inf_{x \in B_0}{\mathcal{M}_{\eta,\vec{r}}^{\B,\otimes}}(\vec{f})(x)\\
    &\leq \Phi_{T_{\vec \eta},\vec{r},\tilde{r}}\left(\lambda_1\right) {\mathcal{M}_{\eta,\vec{r}}^{\B,\otimes}}(\vec{f})(x).
\end{align*}
For $I_2$, it follows from the structure of $B_{k,0}$ that
\begin{align*}
I_2 &\leq \Phi_{T_{\vec \eta},\vec{r},\tilde{r}}\left(\lambda_1\right)\prod_{i=1}^m\left\langle\left\langle f_i\right\rangle\right\rangle_{\eta_i,r_i,B_k} \leq \Phi_{T_{\vec \eta},\vec{r},\tilde{r}}\left(\lambda_1\right) {\mathcal{M}_{\eta,\vec{r}}^{\B,\otimes}}(\vec{f})(x).
\end{align*}
For $I_3$, it can be deduced from \eqref{M_2.condi.} that 
\begin{align*}
I_3 \le {C}_{2} \mathcal{M}^{\B,{\otimes}}_{\eta,\vec{r}}(\vec{f})(x).
\end{align*}

Collecting these estimates, we have that 
\begin{align*}
    \mathcal{M}^{\B,{\otimes}}_{\vec{\eta},\vec{r}}(\vec{f})(x)> \Phi_{\mathcal{M}_{\vec{\eta},\vec{r}}^{\B,\otimes},\vec{r},\tilde{r}}\left(\lambda_{2}\right)\prod_{i=1}^m\left\langle f_i\right\rangle_{\eta_i,r_i,B_0},
\end{align*}
that is, \eqref{keyest.3} is valid.

Next, from the fact that \( T_{\vec \eta} \in W_{\vec{r},\tilde{r}} \), we get
\begin{align*}
   \mu((B_k \cap B_0) \setminus B_{k,0}) \leq \lambda_1 \mu(B_k) 
\end{align*}
Hence, we have \( \mu(B_{k,0}) \geq (1-\lambda_1)\mu(B_k) \).  This implies
\begin{align*}
\mu(E(B_0)) & \leq \sum_k \mu\left(B_k^{\dagger}\right) \leq \boldsymbol{C}_0\sum_k \mu\left(B_k\right) \leq \frac{\boldsymbol{C}_0}{1 - \lambda_1}\sum_k \mu\left(B_{k,0}\right).
\end{align*}

Therefore, for any $\lambda_1,\lambda_2 \in (0,1),$ since $T_{\vec{\eta}},\mathcal{M}^{\B,{\otimes}}_{\vec{\eta},\vec{r}} \in W_{\vec{r},\tilde{r}}$, it follows that
\begin{align*}
\mu(E(B_0))& \leq \frac{\boldsymbol{C}_0}{1 - \lambda_1} \mu\left(\bigcup_k B_{k,0}\right)\leq  \frac{\boldsymbol{C}_0}{1 - \lambda_1} \left(\mu(B_{1}) + \mu(B_{2})\right)\leq \frac{\boldsymbol{C}_0}{1 - \lambda_1}(\lambda_1 + \lambda_2) \mu(B_0),
\end{align*}
where we used the disjointness of \( \{B_k\} \).
We can take $\lambda_2=\lambda_1$ and we have 
\begin{align*}
\mu(E(B_0))\leq \frac{2\boldsymbol{C}_0\lambda_1}{1 - \lambda_1} \mu(B_0):=\lambda \mu(B_0).
\end{align*}

Last, \eqref{keyest.2} derives from the fact that
if for any $\lambda \in(0,1)$, then there is $\lambda_1 = \frac{\lambda }{{2\boldsymbol{C}_0 + \lambda }}\in (0,1)$, such that 
$\lambda= \frac{2\boldsymbol{C}_0\lambda_1}{1 - \lambda_1}$. 


\emph{{\textbf{\textsf (ii):}}}
Finally, since $\mathcal{B}$ is doubling, for \( A \in \B \) with \( A^{\dagger} \subsetneq X \),  there is a set \( B \supseteq A \) satisfying \( \mu(B) \leq \gamma \mu(A) \). By Lemma \ref{lem:aooa:4.3} \eqref{ABC-1} below, 
\begin{align*}
\left\bracevert A, B \right\bracevert_{s_1} \lesssim \gamma^{\frac{1}{r}-\eta} \left(C_2+\Phi_{{T}_{\vec \eta},\vec{r},\tilde{r}}(\lambda_1)\right).
\end{align*}
This proves \eqref{M_1.condi.} for \( T_{\vec \eta} \). \qedhere
\end{proof}

We provide the auxiliary lemmas required for the above.

 \begin{lemma}\label{lem:aooa:4.2}
	Let $1 \leq r_1,\ldots,r_m < \infty$, $0 \leq \eta_1,\ldots,\eta_m < 1$,
with \( \eta = \sum_{i=1}^{m} \eta_i \), \( 0 \leq \eta_i r_i < 1 \), and \( s_1 \in [0, \infty) \). Let $A, B ,C\in \B$ with $ A \subseteq B \subseteq C$, the following results hold
\begin{list}{ $(\theenumi)$}{\usecounter{enumi}\leftmargin=1.2cm \labelwidth=1cm \itemsep=0.2cm \topsep=.2cm \renewcommand{\theenumi}{\alph{enumi}}}

\item\label{ABF-1}  $ \left\bracevert   A, B \right\bracevert_{s_1} \leq \left\bracevert   A, C \right\bracevert_{s_1}$.

\item\label{ABF-2}For any fixed $i \in \{1,\ldots,m\}$, 
 $
\left\langle \left\langle f_i \chi_{B^\dagger}\right\rangle\right\rangle_{\eta_i,r_i,A} \lesssim \left(\frac{\mu(B)}{\mu(A)}\right)^{\frac{1}{r_i}-\eta_i} \langle f\rangle_{\eta_i,r_i,B^\dagger}.
$
\end{list}
\end{lemma}

\begin{proof}

{\bf \eqref{ABF-1}.}
Let $A,B,C \in \B$ with $A \subseteq B \subseteq C$, $g_i := f_i \chi_{B^{\dagger}}$, and $f_i \in L^{r_i}(\mu)$ for $i = 1,\ldots,m$.
It follows that
\begin{align*}
	& \quad \left\langle\left\|{{T}_{\vec \eta}}\left(\vec{f} \chi_{B^\dagger \backslash A^{\dagger}}\right)\right\|_{\mathbb{V}}\right\rangle_{s_1,A}=\left\langle\left\|{{T}_{\vec \eta}}\left(\vec{f}\chi_{B^\dagger}\cdot \chi_{C^{\dagger} \backslash A^{\dagger}}\right)\right\|_{\mathbb{V}}\right\rangle_{s_1,A}\\
& \leq \left\bracevert   A, C \right\bracevert_{s_1} \prod_{i=1}^m\left\langle g_i\right\rangle_{\eta_i,r_i,C^{\dagger}}=\left\bracevert   A, C \right\bracevert_{s_1} \prod_{i=1}^m\left[\mu\left(B^\dagger\right) / \mu\left(C^{\dagger}\right)\right]^{\frac{1}{r_i}-\eta_i}\left\langle f_i\right\rangle_{\eta_i,r_i,B^\dagger} \\
& \leq \left\bracevert   A, C \right\bracevert_{s_1} \prod_{i=1}^m\left\langle f_i\right\rangle_{\eta_i,r_i,B^\dagger} .
\end{align*}
Then part \eqref{ABF-1} can be implied by \eqref{Deltaaa}.

{\bf \eqref{ABF-2}.}
Let $\tilde{A}, A, B \in \B$ with $A \subseteq \tilde{A} \cap B$. On the one hand, if $\mu(\tilde{A}) \leq \mu\left(B^\dagger\right)$, $\tilde{A} \subseteq B^2$ can be deduced from property \eqref{list:B4} and $\tilde{A} \cap B^\dagger \neq \emptyset$. It is derived that
	\begin{align}\notag
		\left\langle f_i \chi_{B^\dagger}\right\rangle_{\eta_i,r_i,\tilde{A}} &\leq\left[\mu\left(B^2\right) / \mu\left(\tilde{A}\right)\right]^{\frac{1}{r_i}-\eta_i}\left\langle f_i \chi_{B^\dagger}\right\rangle_{\eta_i,r_i,B^{2}} \\ \label{0604_1}
&\lesssim_{\boldsymbol{C}_0,\eta_i,r_i} [\mu(B) / \mu(A)]^{\frac{1}{r_i}-\eta_i}\langle f_i\rangle_{\eta_i,r_i,B^\dagger}.
	\end{align}
On the other hand, if $\mu(\tilde{A})>\mu\left(B^\dagger\right)$,  from property \eqref{list:B4} and $\tilde{A} \cap B^\dagger \neq \emptyset$, it can be deduced that $B^\dagger \subseteq \tilde{A}^{\dagger}$.  Then we obtain
\begin{align}\notag
	\left\langle f_i \chi_{B^\dagger}\right\rangle_{\eta_i,r_i,\tilde{A}} & \leq\left[\mu\left(\tilde{A}^{\dagger}\right) / \mu\left(\tilde{A}\right)\right]^{\frac{1}{r_i}-\eta_i}\left\langle f_i \chi_{B^\dagger}\right\rangle_{\eta_i,r_i,\tilde{A}^{\dagger}}\\ \label{0604_2} &\lesssim_{\boldsymbol{C}_0,\eta_i,r_i} \langle f\rangle_{\eta_i,r_i,B^\dagger} .
\end{align}
Since $A$ is arbitrary and $A \subseteq \tilde{A}$, the desired follows from \eqref{0604_1}, \eqref{0604_2}, and \eqref{avggg}.

\end{proof}
\begin{lemma}\label{lem:aooa:4.3}
Let $ \vec{r} = (r_i)_{i=1}^m \in [1, \infty)^m$, $ \vec{\eta} = (\eta_i)_{i=1}^m \in [0, 1)^m$ with $\eta=\sum\limits_{i=1}^m\eta_i$, \( 0 \le \eta_i r_i \leq 1\), and $s_1,s_2,s_3 \in [0,\infty)$ with $s_1 \leq s_3$. Assume that ${T}_{\vec \eta}$ is a $\mathbb{V}$-valued multilinear operator satisfying the condition \eqref{M_2.condi.} and ${T}_{\vec \eta} \in W_{\vec{r},\tilde{r}}$. Then the following statements hold:
\begin{list}{\rm (\theenumi)}{\usecounter{enumi}\leftmargin=1.2cm \labelwidth=1cm \itemsep=0.2cm \topsep=.2cm \renewcommand{\theenumi}{\arabic{enumi}}}

\item\label{ABC-1} 
	$
	 \left\bracevert   A, B \right\bracevert_{s_1} \lesssim\left(C_2+\Phi_{{T}_{\vec \eta},\vec{r},\tilde{r}}(\lambda)\right)[\mu(B) / \mu(A)]^{\frac{1}{r}-\eta}.
	$
\item\label{ABC-2} 
	$
	\left\bracevert   A, C \right\bracevert_{s_1} \lesssim\left(C_2+\Phi_{{T}_{\vec \eta},\vec{r},\tilde{r}}(\lambda)+ \left\bracevert   A, B \right\bracevert_{s_1}\right)[\mu(C) / \mu(B)]^{\frac{1}{r}-\eta}.
	$
\item\label{ABC-3} 
$
\left\bracevert   A, B^k \right\bracevert_{s_1} \lesssim \boldsymbol{C}_0^{\frac{k}{r}-\eta k}\left(C_2+\Phi_{{T}_{\vec \eta},\vec{r},\tilde{r}}(\lambda)+ \left\bracevert   A, B \right\bracevert_{s_1}\right), \quad k \geq 1.$
    \end{list}
where $\Phi_{{T}_{\vec{\eta}},\vec{r},\tilde{r}}:(0,1)\to(0,\infty)$ with $\lambda \in (0,1)$.
\end{lemma}

\begin{proof}

\emph{{\textbf{\textsf (1):}}}
Let $A,B\in \mathcal{B}$, $\lambda \in \left(0,\frac{1}{2}\right)$, setting $J_0=  \Phi_{{T}_{\vec \eta},\vec{r},\tilde{r}}(\lambda) \prod\limits_{i=1}^m\langle f_i\rangle_{\eta_i,r_i,A}$. 
It follows from $T_{\vec \eta} \in W_{\vec{r},\tilde{r}}$ and Definition \ref{def:multilinear_W} that
\begin{align*}
    &\quad  \mu\left(\left\{x \in A:\left\|T_{\vec \eta}\left(\vec{f} \chi_{B^\dagger}\right)(x)\right\|_{\mathbb{V}}>J_0\right\}\right) \leq \lambda \mu(A).
\end{align*}
Analogously,  $\mu\left(\left\{x \in A:\left\|T_{\vec \eta}\left(\vec{f} \chi_{A^{\dagger}}\right)(x)\right\|_{\mathbb{V}}> J_0 \right\}\right) \leq \lambda\mu(A) .$
From above estimates, it derived that
\begin{align}\label{0604_3}
    \mu\left(\{x \in A: \left\|T_{\vec \eta}(\vec{f} \chi_{B^{\dagger}\backslash A^{\dagger}})(x)\right\|_{\mathbb{V}} >2 J_0\}\right)
\le 2\lambda\mu(A). 
\end{align}

Set
\begin{align*}
    \bar{A} = \left\{x \in A:\left\|T_{\vec \eta}(\vec{f} \chi_{B^{\dagger}\backslash A^{\dagger}})(x)\right\|_{\mathbb{V}}\leq J_0\right\}.
\end{align*}
From \eqref{0604_3} and $2\lambda < 1$, we can conclude that $\bar{A} \neq \phi$, It follows that
\begin{align}\label{eq:aooa:4.3}
&\left\langle\left\|T_{\vec \eta}(\vec{f} \chi_{B^{\dagger}\backslash A^{\dagger}})\right\|_{\mathbb{V}}\right\rangle_{s_2,\bar{A}} \leq  2J_0 \lesssim\Phi_{{T}_{\vec \eta},\vec{r},\tilde{r}}(\lambda) [\mu(B) / \mu(A)]^{\frac{1}{r}-\eta} \prod_{i=1}^m\left\langle f_i\right\rangle_{\eta_i,r_i,B^\dagger}.
\end{align}
For any $x \in A$, it can be deducted from \eqref{M_2.condi.} and $\bar{A} \subseteq A$ that
\begin{align}\notag
&\quad  \left\langle  \left\langle\left\| T_{\vec \eta}(\vec{f} \chi_{B^{\dagger}\backslash A^{\dagger}})(x)-T_{\vec \eta}(\vec{f} \chi_{B^{\dagger}\backslash A^{\dagger}})\left(\cdot\right)\right\|_{\mathbb{V}}\right\rangle_{s_2,\bar{A}} \right \rangle_{s_3,A\ni x} \\ \label{eq:aooa:4.4}
& \quad \leq C_2 \prod_{i=1}^m\left\langle \left \langle f_i \chi_{B^\dagger}\right\rangle\right\rangle_{\eta_i,r_i,A} \lesssim C_2[\mu(B) / \mu(A)]^{\frac{1}{r}-\eta} \prod_{i=1}^m\left\langle f_i\right\rangle_{\eta_i,r_i,B^\dagger},
\end{align}
where we have used Lemma \ref{lem:aooa:4.2} part (b) in the last inequality.

Due to  $s_1 \leq s_3$, it derives from \eqref{eq:aooa:4.3} and \eqref{eq:aooa:4.4}
that
$$
\begin{aligned}
& \left\langle\left\|T_{\vec \eta}\left(\vec{f} \chi_{B^\dagger \backslash A^{\dagger}}\right)\right\|_{\mathbb{V}}\right\rangle_{s_1,A} \lesssim\left(C_2+\Phi_{T_{\vec \eta},\vec{r},\tilde{r}}(\lambda) \right)[\mu(B) / \mu(A)]^{\frac{1}{r}-\eta} \prod_{i=1}^m\left\langle f_i\right\rangle_{\eta_i,r_i,B^\dagger}
\end{aligned}
$$
This shows part \eqref{ABC-1}.

\emph{{\textbf{\textsf (2):}}}
Let $A, B, C \in \B$ with $A \subseteq B \subseteq C$. To estimate $\left\bracevert   A, C \right\bracevert_{s_1}$, since $\left\bracevert   A, C \right\bracevert_{s_1} \leq \left\bracevert   A, B \right\bracevert_{s_1} + \left\bracevert   B, C \right\bracevert_{s_1},$
it suffices to estimate $\left\bracevert   A, B \right\bracevert_{s_1}$ and $\left\bracevert   B, C \right\bracevert_{s_1}$. Further, it is sufficient to estimate $\left\langle \left\|T_{\vec \eta}\left(\vec{f} \chi_{C^{\dagger}\backslash B^\dagger}\right)\right\|_{\mathbb{V}} \right\rangle_{s_1,A}$ and $\left\langle \left\|T_{\vec \eta}\left(\vec{f} \chi_{B^{\dagger}\backslash A^\dagger}\right)\right\|_{\mathbb{V}} \right\rangle_{s_1,A}$.

On the one hand, for any $x \in A$, it follows from Definition \ref{Deltaaa} that
\begin{align}\label{eq:aooa:4.5}
   \left\langle \left\|T_{\vec \eta}\left(\vec{f} \chi_{C^{\dagger}\backslash B^\dagger}\right)\right\|_{\mathbb{V}} \right\rangle_{s_1,A}
\lesssim\left(C_2+\Phi_{T_{\vec \eta},\vec{r},\tilde{r}}(\lambda) \right)[\mu(C) / \mu(B)]^{\frac{1}{r}-\eta} \prod_{i=1}^m\left\langle f_i\right\rangle_{\eta_i,r_i,C^{\dagger}}.
\end{align}
On the other hand, it derives from \eqref{M_1.condi.} that
\begin{align}\notag
    & \left\langle \left\|T_{\vec \eta}\left(\vec{f} \chi_{B^{\dagger}\backslash A^\dagger}\right)\right\|_{\mathbb{V}} \right\rangle_{s_1,A}\leq  \left\bracevert   A, B \right\bracevert_{s_1} \prod_{i=1}^m\left\langle f_i\right\rangle_{\eta_i,r_i,B^\dagger} \\ \notag
& \quad \lesssim  \left\bracevert   A, B \right\bracevert_{s_1}\left[\mu\left(C^{\dagger}\right) / \mu\left(B^\dagger\right)\right]^{\frac{1}{r}-\eta} \prod_{i=1}^m\left\langle f_i\right\rangle_{\eta_i,r_i,C^{\dagger}} \\ \label{eq:aooa:4.6}
& \quad \lesssim  \left\bracevert   A, B \right\bracevert_{s_1}[\mu(C) / \mu(B)]^{\frac{1}{r}-\eta} \prod_{i=1}^m\left\langle f_i\right\rangle_{\eta_i,r_i,C^{\dagger}}.
\end{align}
Then part \eqref{ABC-2} follows from \eqref{eq:aooa:4.5} and \eqref{eq:aooa:4.6}.

{\emph{{\textbf{\textsf (3):}}}
 It follows from part \eqref{ABC-2}, \eqref{BkBk}, and $\C_0 \geq 1$ that
\begin{align*}
    \left\bracevert  A, B^k \right\bracevert_{s_1} & \lesssim \C_0^{\frac{1}{r}-\eta}\left(C_2+\Phi_{{T}_{\vec \eta},\vec{r},\tilde{r}}(\lambda)\right) + \C_0^{\frac{1}{r}-\eta}\left\bracevert   A, B^{k-1} \right\bracevert_{s_1}\\
    & \lesssim \C_0^{\frac{1}{r}-\eta}\left(C_2+\Phi_{{T}_{\vec \eta},\vec{r},\tilde{r}}(\lambda)\right) + \C_0^{\frac{2}{r}-2\eta}\left(C_2+\Phi_{{T}_{\vec \eta},\vec{r},\tilde{r}}(\lambda)\right) + \C_0^{\frac{2}{r}-2\eta}\left\bracevert   A, B^{k-2} \right\bracevert_{s_1}\\
    &\lesssim \boldsymbol{C}_0^{\frac{k}{r}-\eta k}\left(C_2+\Phi_{{T}_{\vec \eta},\vec{r},\tilde{r}}(\lambda)+ \left\bracevert   A, B \right\bracevert_{s_1}\right). \qedhere
\end{align*}} 
\end{proof}

\begin{lemma}\label{lem:aooa:4.4}
Let $(X,\mu,\mathcal{B})$ equipped a measure space with a ball-basis $\mathcal{B}$. Let ${T}_{\vec{\eta}}\in W_{\vec{r},\tilde{r}}$, we set
\[
\boldsymbol{C} := C_1 + C_2 + \Phi_{{T}_{\vec{\eta}},\vec{r},\tilde{r}}(\lambda),
\]
with $\Phi_{{T}_{\vec{\eta}},\vec{r},\tilde{r}}:(0,1)\to(0,\infty)$. The following properties hold:
\begin{list}{\rm (\theenumi)}{\usecounter{enumi}\leftmargin=1.2cm \labelwidth=1cm \itemsep=0.2cm \topsep=.2cm \renewcommand{\theenumi}{\alph{enumi}}}

\item\label{DBB-1}
 For any $B \in \B$ with $B^{\dagger}=B$ there exists $\widehat{B} \in \B$ such that
	$$
	 B^2 \subseteq \widehat{B}, \quad \left\bracevert   B^2, \widehat{B} \right\bracevert_{s_1}  \lesssim \boldsymbol{C} , \quad \text { and } \quad \mu(\widehat{B}) \geq 2 \mu(B).
	$$
	
\item\label{DBB-2} For any $ B \in \B$ there exists $\widehat{B} \in \B$ such that $B^2 \subseteq \widehat{B}$,
	$$
	 \left\bracevert B^2, \widehat{B}\right\bracevert_{s_1} \lesssim \boldsymbol{C} , \quad \text { and either } \widehat{B}^2=\widehat{B} \text { or } \mu(\widehat{B}) \geq 2 \mu(B).
	$$
	
\item\label{DBB-3} For any $B \in \B$ there exists a sequence $\left\{B_k\right\}_{k \geq 0} \subseteq \B$ such that $X=\bigcup_{k \geq 0} B_k$ with $B_0=B$,
	$$
	B^2_{k-1} \subseteq B_k, \quad \text { and } \quad \left\bracevert B^2_{k-1}, B_k \right\bracevert_{s_1} \lesssim \boldsymbol{C} , \quad k \geq 1.
	$$	
    \end{list}
\end{lemma}

\begin{proof}~~

\emph{{\textbf{\textsf (a):}}}
 Fix $B \in \mathcal{B}$ with $B^{\dagger}=B$. By \eqref{M_1.condi.}, there exists $\widehat{B} \in \mathcal{B}$ such that $B \subsetneq \widehat{B}$ and $\left\bracevert B, \widehat{B}\right\bracevert_{s_1} \lesssim C_1$. Since $B^{\dagger}=B$, we have $B^2 = B \subsetneq \widehat{B}$ with $\left\bracevert B^2, \widehat{B}\right\bracevert_{s_1} = \left\bracevert B, \widehat{B}\right\bracevert_{s_1} \lesssim C_1$. Moreover, property \eqref{list:B4} implies $\mu(\widehat{B}) \geq 2\mu(B)$, as otherwise we would have the contradiction $\widehat{B} \subseteq B^{\dagger}=B$.

\emph{{\textbf{\textsf (b):}}}
To prove part \eqref{DBB-2}, fix $B \in \mathcal{B}$ and define 
\[
\mathcal{B}_B := \{A \in \mathcal{B} : B \subseteq A^\dagger\},
\]
with the associated quantities
\begin{align}\label{eqoo:4.7}
a := \sup_{\substack{A \in \mathcal{B}_B \\ \mu(A) \leq 2\mu(B)}} \mu(A) \quad \text{and} \quad b := \inf_{\substack{A \in \mathcal{B}_B \\ \mu(A) > 2\mu(B)}} \mu(A).
\end{align}
Note that $a \leq 2\mu(B) \leq b$. Choose $B_1, B_2 \in \mathcal{B}_B$ satisfying
\begin{align}\label{eqoo:4.8}
a/2 < \mu(B_1) \leq a \leq 2\mu(B) \leq b \leq \mu(B_2) < 2b.
\end{align}

First consider the case $b > \boldsymbol{C}_0^2 a$. Taking $\widehat{B} := B_1^\dagger \supseteq B$, it derives from \eqref{eqoo:4.8} that 
\begin{align}\label{eqoo:4.9}
\mu(\widehat{B}^\dagger) = \mu(B_1^2) \leq \boldsymbol{C}_0^2 \mu(B_1) \leq \boldsymbol{C}_0^2 a < b.
\end{align}
By the definition of $a$ and $b$ in \eqref{eqoo:4.7}, there exists no $B' \in \mathcal{B}_B$ with $a < \mu(B') < b$. Combining this with \eqref{eqoo:4.9} yields $\mu(\widehat{B}^\dagger) \leq a \leq 2\mu(B_1)$. Since $\widehat{B}^\dagger \cap B_1 \neq \emptyset$, property \eqref{list:B4} implies $\widehat{B}^\dagger \subseteq B_1^\dagger = \widehat{B} \subseteq \widehat{B}^\dagger$, hence $\widehat{B}^\dagger = \widehat{B}$. Consequently, $B^2 \subseteq \widehat{B}^2 = \widehat{B}$, and \eqref{eqoo:4.8} gives that
\begin{align}\label{eqoo:4.10}
\mu(B^2) \leq \mu(\widehat{B}) = \mu(B_1^\dagger) \leq \boldsymbol{C}_0 \mu(B_1) \leq 2\boldsymbol{C}_0 \mu(B) \leq 2\boldsymbol{C}_0 \mu(B^2).
\end{align}

When $b \leq \boldsymbol{C}_0^2 a$, set $\widehat{B} := B_2^3 = (B_2^\dagger)^2 \supseteq B^2$. From \eqref{eqoo:4.8}, it can be deduced that $\mu(\widehat{B}) \geq \mu(B_2) \geq b \geq 2\mu(B)$ and
\begin{align}\label{eqoo:4.11}
\mu(B^2) &\leq \mu(\widehat{B}) = \mu(B_2^3) \leq \boldsymbol{C}_0^3 \mu(B_2) \leq 2\boldsymbol{C}_0^3 b \\
&\leq 2\boldsymbol{C}_0^5 a \leq 4\boldsymbol{C}_0^5 \mu(B) \leq 4\boldsymbol{C}_0^5 \mu(B^2). \notag
\end{align}
The estimates \eqref{eqoo:4.10} and \eqref{eqoo:4.11} together show that $\mu(\widehat{B}) \approx \mu(B^2)$. 
It follows from Lemma \ref{lem:aooa:4.3} part \eqref{ABC-1} that
\[
\left\bracevert B^2, \widehat{B}\right\bracevert_{s_1} \lesssim \boldsymbol{C}. 
\]

\emph{{\textbf{\textsf (c):}}}
To complete part \eqref{DBB-3}, we set \( B_0 := B \) and use induction to construct the sequence \( \{ B_k \}_{k \geq 0} \). For each \( k \geq 1 \), we establish the following conditions hold
\[
B_{k-1}^2 \subseteq B_k, \quad \left\bracevert B_{k-1}^2, B_k \right\bracevert_{s_1} \lesssim \boldsymbol{C}, \quad \mu(B_k) \geq 2 \mu(B_{k-1}).
\]
Finally, by Lemma \ref{lem:BG}, we conclude that \( X = \bigcup_{k \geq 0} B_k \), completing the proof.

\end{proof}

{\begin{lemma}\label{lem:GA}
Let $\lambda \ge 3\boldsymbol{C}_0^4$, $F \subseteq X$ measurable, $A \in \mathcal{B}$ with $F \cap A \neq \emptyset$ and $\mu(F) \le \lambda^{-1} \mu(A)$. Then there exists $\mathcal{G} \subseteq \mathcal{B}$ such that
\begin{list}{\rm (\theenumi)}{\usecounter{enumi}\leftmargin=1.2cm \labelwidth=1cm \itemsep=0.2cm \topsep=.2cm \renewcommand{\theenumi}{\roman{enumi}}}
\item \label{GA-1} $F \cap A^{\dagger} \cap G \neq \emptyset$ for all $G \in \mathcal{G}$;
    \item \label{GA-2} $F \cap A^{\dagger} \subseteq \bigcup_{G \in \mathcal{G}} G$ a.s.;
    \item \label{GA-3} $\mu \left(\bigcup\limits_{G \in \mathcal{G}} G^{\dagger} \right) \le 3\boldsymbol{C}_0^2 \lambda^{-1} \mu(A)$;
    \item \label{GA-4} For each $G \in \mathcal{G}$, there exists $\widehat{G} \in \mathcal{B}$ 
    such that 
	\begin{align}\label{GA-2_}
	\widehat{G} \not\subset F, \quad 
	G^2 \subseteq \widehat{G} \subseteq A^{\dagger}, \quad\text{and}\quad 
	\left\bracevert  G^2, \widehat{G}\right\bracevert_{s_1}\lesssim_{F,A,\lambda,\mathcal{G}} \boldsymbol{C}, 
	\end{align}
where the $\C$ and $\left\bracevert  G^2, \widehat{G}\right\bracevert_{s_1}$ are defined in the below Lemma \ref{lem:aooa:4.4} and \eqref{Deltaaa}, respectively.
\end{list}
	\end{lemma}}

 \begin{proof}~~


\emph{{\textbf{\textsf (i) and  (ii):}}}
For simplicity we denote  \( E := F \cap A^{\dagger} \). From 
Lemma \ref{lem:aooa:3.3} part \eqref{list-5}, it follows that
there exists a sequence $\B' \subseteq \calB$ such that
\begin{align}\label{FAG-1}
E \cap B \neq \emptyset, \quad E \subseteq \bigcup_{B \in \mathcal{B}'} B \text{ a.s.},  
\quad\text{ and }\quad 
\sum_{B \in \mathcal{B}'} \mu(B) \le 2\boldsymbol{C}_0 \, \mu(E).
\end{align}
For each \( B \in \mathcal{B}' \),
by Lemma \ref{lem:aooa:4.4} part \eqref{DBB-3}, we obtain a sequence \( \{B_k\}_{k \ge 0} \subseteq \mathcal{B} \) satisfying
\begin{align}\label{FAG-2}
X = \bigcup_{k \ge 0} B_k \text{ with } B_0 = B, \quad B^2_{k-1} \subseteq B_k, \quad \text{and}\quad \left\bracevert B_{k-1}^{2}, B_k \right\bracevert_{s_1} \lesssim \boldsymbol{C}, \quad k \ge 1.
\end{align}
Now, we define $\mathcal{G}$ as follows. For each \( B \in \mathcal{B}' \), let \( G = B_k \), where \( k \ge 0 \) is the least index such that \( B_{k+1}^{\dagger} \not\subset F \), and \( \mathcal{G} \) is the collection of all such \( G \). 

For each $G \in \mathcal{G}$, it follows from the structure of $\mathcal{G}$ that $G=B$ or $G=B_k \supseteq B^{2}$ for some $k \geq 1$.
Thus, $F \cap A^* \cap G=E \cap G \neq \emptyset$ and $F \cap A^*=E \subset \bigcup_{G \in \mathcal{G}} G$ a.s. holds from \eqref{FAG-1} directly.

\emph{{\textbf{\textsf (iii) and  (iv):}}} According to the construction of $\mathcal{G}$, if \( G^{\dagger} \not\subset F \) we only need to consider two cases for $G$: $G=B_0$, or $G=B_k$ with $k \geq 1$ and $G^{\dagger}=B_k^{\dagger} \subset F$.
It follows that
\begin{align}\label{GGG}
\mu \Big( \bigcup_{G \in \mathcal{G}} G^{\dagger} \Big) 
&= \mu \Big( \bigcup_{G \in \mathcal{G}: G^{\dagger} \subset F} G^{\dagger} \Big) + \mu \Big( \bigcup_{G \in \mathcal{G}: G^{\dagger} \not\subset F} G^{\dagger} \Big) \notag \\
&\le \mu(F) + \mu \Big( \bigcup_{B \in \mathcal{B}'} B^\dagger \Big) 
\le \mu(F) + \sum_{B \in \mathcal{B}'} \mu(B^\dagger) \notag \\
&\le \mu(F) + \boldsymbol{C}_0 \sum_{B \in \mathcal{B}'} \mu(B) 
\le (1 + 2\boldsymbol{C}_0^2) \mu(F) \notag \\
&\le (1 + 2\boldsymbol{C}_0^2) \lambda^{-1} \mu(A) 
\le 3\boldsymbol{C}_0^2 \lambda^{-1} \mu(A),
\end{align}
where we have used \eqref{FAG-1} and \( \mu(F) \le \lambda^{-1} \mu(A) \). This proof \eqref{GA-3}.

Let \( G = B_k \in \mathcal{G} \) for some \( k \ge 0 \), it follows from \eqref{GGG} that $\mu(G^2) \leq \C_0^2 \mu(G) \leq 3\C^4_0 \lambda^{-1}\mu(A) \leq \mu(A)$.
From $G \cap A^{\dagger} \neq \emptyset$ and property \eqref{list:B4}, it can be deduced that  $G^2 \subseteq A^{\dagger}$.

To establish \eqref{GA-4}, we define
\begin{equation*}
\widehat{G} := 
\begin{cases}
A^{\dagger}, & \text{if } \mu(B_{k+1}^{\dagger}) > \mu(A), \\
B_{k+1}^{\dagger}, & \text{if } \mu(B_{k+1}^{\dagger}) \le \mu(A).
\end{cases}
\end{equation*}
Since $\mu(A^{\dagger})\geq\mu(A)\geq\lambda\mu(F)>\mu(F)$, then $A^{\dagger}\not\subset F$. Along with the construction of $\mathcal{G}$, we get $\widetilde{G}\not\subset F$.

If \( \mu(B_{k+1}^{\dagger}) > \mu(A) \), then \( G^2 \subseteq A^{\dagger} = \widehat{G} \subseteq B_{k+1}^2 \). It derives from Lemma \ref{lem:aooa:4.2} part \eqref{ABF-1} and Lemma \ref{lem:aooa:4.3} part \eqref{ABC-3} that
\begin{align*}
\left\bracevert G^2, \widehat{G} \right\bracevert_{s_1} 
\le \left\bracevert B_k^2, B_{k+1}^2 \right\bracevert_{s_1} 
\lesssim \boldsymbol{C} + \left\bracevert B_k^2, B_{k+1} \right\bracevert_{s_1}
\lesssim \boldsymbol{C}.
\end{align*}

If \( \mu(B_{k+1}^{\dagger}) \le \mu(A) \), then \( G^2 = B_k^2 \subseteq B_{k+1} \subseteq B_{k+1}^{\dagger} = \widehat{G} \), it follows that
\begin{align*}
\left\bracevert G^2, \widehat{G} \right\bracevert_{s_1}
= \left\bracevert B_k^2, B_{k+1}^{\dagger} \right\bracevert_{s_1} 
\lesssim \boldsymbol{C} + \left\bracevert B_k^2, B_{k+1} \right\bracevert_{s_1}
\lesssim \boldsymbol{C}.
\end{align*}
This proof \eqref{GA-2_}.

\end{proof}

\begin{lemma}\label{lem:tree}
Let $\lambda \geq 3\boldsymbol{C}_0^4$, there exists a family $\mathcal{G} = \mathcal{G}(B_0) \subseteq \mathcal{B}$ such that
\begin{list}{$(\theenumi)$}{\usecounter{enumi}\leftmargin=.8cm \labelwidth=.8cm\itemsep=0.2cm\topsep=.1cm
\renewcommand{\theenumi}{\roman{enumi}}}

\item \label{dk_1} For each $A \in \mathcal{G}$, there is  a subfamily $\mathcal{F}(A) \subseteq \mathcal{G}$, the following results hold

\begin{list}{$(\theenumi.\theenumii)$}{\usecounter{enumii}\leftmargin=.4cm \labelwidth=.8cm\itemsep=0.2cm\topsep=.1cm\renewcommand{\theenumii}{\arabic{enumii}}}
\item\label{tree-1} $A^{\dagger} \cap B \neq \emptyset,\quad \forall B \in \mathcal{F}(A);$
\item \label{tree-2} $\mu \left( \bigcup\limits_{B \in \mathcal{F}(A)} B^{\dagger} \right) \le 3\boldsymbol{C}_0^2 \lambda^{-1} \mu(A);$
\item \label{tree-3} $\varLambda(\vec{f} \chi_{A^3})(x) 
\lesssim \Phi_{\varLambda,\vec{r},\tilde{r}}\left(\frac{1}{\boldsymbol{C}_0^3 \lambda}\right) \prod_{i=1}^m \langle f_i \rangle_{\eta_i,r_i,A^3}, \quad a.e.\, x \in A^{\dagger} \setminus \bigcup_{B \in \mathcal{F}(A)} B$.
\end{list}

\item \label{dk_2} For each $B \in\mathcal{F}(A)$ there are $\widehat{B} \in \B$ and $\xi \in \widehat{B}$ such that

\begin{list}{$(\theenumi.\theenumii)$}{\usecounter{enumii}\leftmargin=.4cm \labelwidth=.8cm\itemsep=0.2cm\topsep=.1cm\renewcommand{\theenumii}{\arabic{enumii}}}

\item \label{tree-4} $B^2 \subseteq \widehat{B} \subseteq A^{\dagger};$
\item \label{tree-5} $\varLambda(\vec{f} \chi_{A^3})(\xi) 
\lesssim \Phi_{\varLambda,\vec{r},\tilde{r}}\left(\frac{1}{\boldsymbol{C}_0^3 \lambda}\right) \prod_{i=1}^m \langle f_i \rangle_{\eta_i,r_i,A^3};$
\item \label{tree-6} $\left\| T_{\vec{\eta}}(\vec{f} \chi_{\widehat{B}^{\dagger} \setminus B^3})(x) \right\|_{\mathbb{V}} 
\lesssim \Phi_{\varLambda,\vec{r},\tilde{r}}\left(\frac{1}{\boldsymbol{C}_0^3 \lambda}\right) \prod_{i=1}^m \langle f_i \rangle_{\eta_i,r_i,A^3},\quad \forall x \in B^2.$
\end{list}

\end{list}

\end{lemma}

\begin{proof}~~

{\emph{{\textbf{\textsf (i) :}}}}
We want to use Lemma \eqref{lem:GA} to prove \eqref{dk_1}. Then we start by the fact that $\mu(F(B_0)) \leq \lambda^{-1} {\mu(B_0)}{}$. 
Indeed, define the set 
\begin{align}\label{FAA}
F(B_0) := \left\{ x \in B_0^3 : \varLambda(\vec{f} \chi_{B_0^3})(x) > \Phi_{\varLambda,\vec{r},\tilde{r}}\left(\frac{1}{\boldsymbol{C}_0^3 \lambda}\right) \prod_{i=1}^m \langle f_i \rangle_{\eta_i,r_i,B_0^3} \right\}.
\end{align}
Since $\varLambda \in W_{\vec{r}, \tilde{r}}$, it derives that $\mu(F(B_0)) \le \left(\boldsymbol{C}_0^3 \lambda\right)^{-1} {\mu(B_0^3)} \le \lambda^{-1} {\mu(B_0)}{}.$

From Lemma \ref{lem:GA} with \( A = B_0 \) and \( F = F(B_0) \), it follows that there exists a collection \( \mathcal{F}(B_0) \subseteq \mathcal{B} \) such that
\begin{align}
\label{GAA-1}
F(B_0) \cap B_0^{\dagger} \cap B \neq \emptyset,
\end{align}
\begin{align}
\label{GAA-2}
F(B_0) \cap B_0^{\dagger} \subseteq \bigcup_{B \in \mathcal{F}(B_0)} B \text{ a.s.},
\end{align}
and
\begin{align}
\label{GAA-3}
\mu\left( \bigcup_{B \in \mathcal{F}(B_0)} B^{\dagger} \right) \le 3\boldsymbol{C}_0^2 \lambda^{-1} \mu(B_0).
\end{align}
To proceed, note that \eqref{GAA-2} yields that
 \( \mu\left( F(B_0) \cap B_0^{\dagger} \setminus \bigcup_{B \in \mathcal{F}(B_0)} B \right) = 0 \). Therefore, for almost every \( x \in B_0^{\dagger} \setminus \bigcup_{B \in \mathcal{F}(B_0)} B \), it can be deduced from \eqref{FAA} with $\varLambda \in W_{\vec{r}, \tilde{r}}$ that
\begin{align}\label{}
\varLambda(\vec{f} \chi_{B_0^3})(x) \lesssim \Phi_{\varLambda,\vec{r},\tilde{r}}\left(\frac{1}{\boldsymbol{C}_0^3 \lambda}\right) \prod_{i=1}^m \langle f_i \rangle_{\eta_i,r_i,B_0^3}.
\end{align}
These prove \eqref{dk_1}  for \( A = B_0 \).

{\emph{{\textbf{\textsf (ii) :}}}}
 For each \( B \in \mathcal{F}(B_0) \), it follows from \eqref{GA-2_} that there exists a set \( \widehat{B} \in \mathcal{B} \) such that
\begin{align}
\label{GAA-4}
\widehat{B} \not\subset F(B_0), \quad B^2 \subseteq \widehat{B} \subseteq B_0^{\dagger}, \quad \text{and} \quad \left\bracevert B^2, \widehat{B} \right\bracevert_{s_1} \lesssim \boldsymbol{C}.
\end{align}

Next we establish \eqref{tree-5} and \eqref{tree-6}.
Since \( \widehat{B} \not\subset F(B_0) \), it derives from \eqref{FAA} and $\varLambda \in W_{\vec{r}, \tilde{r}}$ that
\begin{align}
\label{Gaf}
\varLambda(\vec{f} \chi_{B_0^3})(\xi) \lesssim \Phi_{\varLambda,\vec{r},\tilde{r}}\left(\frac{1}{\boldsymbol{C}_0^3 \lambda}\right) \prod_{i=1}^m \langle f_i \rangle_{\eta_i,r_i,B_0^3}, \quad \xi \in \widehat{B} \setminus F(B_0)
\end{align}
For any \( x \in B^2 \), using \eqref{GAA-4} and \eqref{Gaf}, we obtain
\begin{align*}
&\left\| T_{\vec{\eta}}(\vec{f} \chi_{\widehat{B}^{\dagger}})(x) - T_{\vec{\eta}}(\vec{f} \chi_{B^3})(x) \right\|_{\mathbb{V}} \\
&= \left\| T_{\vec{\eta}}(\vec{f} \chi_{B_0^3 \cap \widehat{B}^{\dagger}})(x) - T_{\vec{\eta}}(\vec{f} \chi_{B_0^3 \cap B^3})(x) \right\|_{\mathbb{V}} \\
&\le \left\bracevert B^2, \widehat{B} \right\bracevert_{s_1}  \prod_{i=1}^m \langle f_i \chi_{B_0^3} \rangle_{\widehat{B}^{\dagger}, r} \lesssim \boldsymbol{C} \inf_{\widehat{B}} \mathcal{M}_{\eta,\vec{r}}^{\otimes,\mathcal{B}}(\vec{f} \chi_{B_0^3}) \\
&\le \varLambda(\vec{f} \chi_{B_0^3})(\xi) \lesssim \Phi_{\varLambda,\vec{r},\tilde{r}}\left(\frac{1}{\boldsymbol{C}_0^3 \lambda}\right) \prod_{i=1}^m \langle f_i \rangle_{\eta_i,r_i,B_0^3}.
\end{align*}
Thus, \eqref{tree-1}--\eqref{tree-6} are valid for \( A = B_0 \).

Finally, for each \( A \in \mathcal{F}(B_0) \), we repeat the procedure to obtain a family \( \mathcal{F}(A) \) satisfying \eqref{tree-1}--\eqref{tree-6}. Define \( \mathcal{F}_0(B_0) := \{ B_0 \} \),
\begin{align}
\label{FFF}
\mathcal{F}_1(B_0) := \mathcal{F}(B_0), \quad \text{and} \quad \mathcal{F}_{k+1}(B_0) := \bigcup_{B \in \mathcal{F}_k(B_0)} \mathcal{F}(B), \quad k \ge 1.
\end{align}
Set
\begin{equation}
\label{generation}
\mathcal{G} = \mathcal{G}(B_0) := \bigcup_{k \ge 0} \mathcal{F}_k(B_0).
\end{equation}
Thus \( \mathcal{G} \) satisfies \eqref{tree-1}--\eqref{tree-6}, which complets the proof.
\end{proof}
We need to give some notation to simplify the following.
Let $\mathcal{G}$ denote all generation balls. For $B \in \mathcal{G}$, define $\mathscr{F}(B)$ as the stopping time balls related to $B$, and $\mathscr{F}_k(B)$ as the $k$-th generation stopping time balls related to $B$. 

For $B, B' \in \mathcal{G}$ and $k \geq 0$, write $\mathscr{F}^k(B') = B$ if $B' \in \mathscr{F}_k(B)$, indicating that $B$ is the $k$-th ancestor of $B'$ (or $B'$ is a $k$-th descendant of $B$). Define $\mathscr{F} := \mathscr{F}^1$. By \eqref{tree-2}, we have
\begin{align}\label{FkFk}
    \mu\bigg( \bigcup_{B' \in \mathscr{F}_k(B)} B' \bigg) \leq (3\C_0^2 \lambda^{-1})^k \mu(B), \quad \forall k \geq 1.
\end{align}


\begin{lemma}\label{ppremain}
  Let $\lambda > 3\C_0^6$ and  fix $B_0 \in \B$. Then there exist $\frac{1}{2}$-sparse families $\mathcal{Q}_1,\mathcal{Q}_2$ with $\mathcal{Q} = \mathcal{Q}_1 \cup \mathcal{Q}_2$ when $\lambda$ is large enough.
For any $B \in \mathcal{Q}$, set $B_k=B$, one can find a sequence $\{B_j\}_{j=0}^k \subseteq \mathcal{Q}$ such that $B_{j+1} \in \F(B_j)$, $j=0, 1, \ldots, k-1$. Moreover,  for almost everywhere $ x \in B_0$, there exists  $\widehat{B}_j $ and $\xi_j \in \widehat{B}_{j+1}$ with $j=0,1,\ldots,k-1$, such that
\begin{align}
\label{BBT-1} & B_{j+1}^{2} \subseteq \widehat{B}_{j+1} \subseteq B_j^{\dagger}, 
\\
\label{BBT-2} & \varLambda(\vec{f} \chi_{B^3_j})(\xi_j) 
\lesssim \Phi_{\varLambda,\vec{r},\tilde{r}}\left(\frac{1}{\boldsymbol{C}^3_0 \lambda}\right) \prod_{i=1}^m \langle f_i \rangle_{\eta_i,r_i,B^3_j}, 
\\ 
\label{BBT-3} & \left\|T_{\vec \eta}(\vec{f} \chi_{\widehat{B}^{\dagger}_{j+1}\backslash B^3_{j+1}})(x) \right\|_{\mathbb{V}} 
\lesssim \Phi_{\varLambda,\vec{r},\tilde{r}}\left(\frac{1}{\boldsymbol{C}^3_0 \lambda}\right) \prod_{i=1}^m \langle f_i \rangle_{\eta_i,r_i,B^3_j},
\end{align}
where $\C_0$ is from Definition \ref{def:basis}.
\end{lemma}
\begin{proof}
From Lemma \ref{lem:tree}, there exists  a family $\mathcal{G} \subseteq \B$ which satisfies \eqref{dk_1} and \eqref{dk_2}. We start the process by introducing the partition of $\mathcal{G}$.
Let $B \in \B$, $\boldsymbol{R} = \C_0^2$, we denote $\ell(B) := [ \log_{\boldsymbol{R}} \mu(B)]$. Therefore, $\mathcal{G} = \bigcup_{k \le k_0}\mathcal{G}^k$ where
\begin{equation*} 
\mathcal{G}^k := \{B \in \mathcal{G}: \boldsymbol{R}^{k} \leq \mu(B) < \boldsymbol{R}^{k+1}\}= \{B \in \mathcal{G}: \ell(B)=k\}, \quad k \geq k_0
\end{equation*}
 with $k_0 := \ell(B_0)$ and $\mathcal{G}^{k_0}  := \{B_0\}$.

Next, to construct a sparse subfamily $\mathcal{D} \subseteq \mathcal{G}$, we will remove certain elements in $\mathcal{O}$ from $\mathcal{G}$, where $\mathcal{O}$ as the collection of $B \in \mathcal{G}$ for which there exists some $B' \in \mathcal{G}$ such that
\begin{align}\label{FBFB}
B^2\cap B' \neq \emptyset 
\quad\text{ and }\quad 
\ell(B) + 1 < \ell(B') < \ell(\mathscr{F}(B)) -1. 
\end{align}
Then we set
\begin{align*}
\H^{k_0} := \mathcal{G}^{k_0}=\{B_0\} 
\quad\text{ and }\quad
\H^k := \mathcal{G}^k \setminus \bigcup_{B \in \mathcal{O}}  \mathcal{G}(B), 
\quad k < k_0,
\end{align*}
For any $k \le k_0$, since \eqref{tree-4} yields that $\bigcup_{B \in \H^k} B \subseteq B_0^{\dagger}$, it can be deduced from Lemma \ref{lem:aooa:3.3} part \eqref{list-1} that there exists a disjoint subfamily $\mathcal{Q}^k \subseteq \H^k$ such that 
\begin{align}\label{BGBD}
\bigcup_{B \in \H^k} B \subseteq \bigcup_{B \in \mathcal{Q}^k} B^{\dagger}.
\end{align} 
{Then, we set
\begin{align}\label{DHG}
\mathcal{Q} := \bigcup_{k \le k_0} \mathcal{Q}^k
\subseteq \bigcup_{k \le k_0} \H^k
\subseteq \bigcup_{k \le k_0} \mathcal{G}^k
= \mathcal{G}.
\end{align}}

Now we demonstrate $\mathcal{Q}$ is a  $\frac{1}{2}$--sparse family whenever $\lambda$ is large enough. For any $H \in \mathcal{Q}$, it is suffices to prove that there exists  $E_H \subseteq H$, such that 
\begin{align}\label{fra12}
   \mu(H) \leq 2\mu(E_H).
\end{align}
To proceed each $H \in \mathcal{Q}$, we define
\begin{align*}
   E_H=H \backslash \bigcup_{G \in \mathcal{D}: \ell(G) < \ell(H) -1 } =H \backslash \bigcup_{G \in \mathcal{D}: G \cap H \neq \emptyset, \ell(G) < \ell(H)-1} G,
\end{align*}
and denote 
\[
\mathscr{P}_H 
:= \{P \in \mathcal{D}: P^2 \cap H \neq \emptyset, |\ell(P)-\ell(H)| \le 1\}.
\] 
From this structure, it follows that
\begin{align}\label{Psc}
\boldsymbol R^{-3} \cdot \mu(H) \leq \mu(P) \leq \boldsymbol R^3 \cdot \mu(H), \quad P \in \mathscr{P}.
\end{align}
Moreover, since $-3 \le \log \frac{\mu(P)}{\mu(H)} \le 3$ for all $P \in \mathscr{P}_H$, Lemma~\ref{lem:aooa:3.3} part~\eqref{list-2} implies $\# \mathscr{P}_H \lesssim 1$.


Let $G \in \mathcal{D}$ satisfy $\ell(G) < \ell(H)-1$, $G \cap H \neq \emptyset$, so $G^2 \cap H \neq \emptyset$. Since $G \notin  \bigcup_{B \in \mathcal{O}} \mathcal{G}(B)$,
from the fact that $\ell(\mathscr{F}(H)) \geq \ell(H)+1$,
it follows that $|\ell(P) - \ell(H)| \leq 1$ for $P = \mathscr{F}^k(G)$ and some $k \geq 1$. By Lemma \ref{lem:tree} \eqref{tree-4}, $G^2 \subseteq P^2$, hence $P^2 \cap H \neq \emptyset$, meaning $P \in \mathscr{P}_H$ and $G \in \mathcal{F}(P)$. 
From the construction of $E_H$, Lemma \ref{lem:tree}\eqref{tree-2}, \eqref{Psc}, and $\# \mathscr{P}_H \lesssim 1$,
\begin{align*} \mu\left(H \backslash E_H\right) & \leq \mu\left(\bigcup_{G \in \mathcal{G}: G \cap A \neq \emptyset, \mathrm\ell(G) < \mathrm\ell(A) -1 } G\right) \leq \mu\left(\bigcup_{P \in \mathscr{P}_H} \bigcup_{G \in \mathcal{F}(P)} G\right) \\ & \leq \sum_{P \in \mathscr{P}_H} \sum_{k=1}^{\infty} \mu\left(\bigcup_{G: \mathscr{F}^k (G)=P} G\right) \lesssim \sum_{P \in \mathscr{P}_H} \sum_{k=1}^{\infty} \frac{\mu(P)}{\lambda^k} \lesssim \frac{\mu(H)}{\lambda}.\end{align*}
Thus, \eqref{fra12} is a direct consequence of above for an admissible constant $\lambda$. 

Additionally, note the fact that 
\begin{align*}
   E_A \cap E_B =\emptyset \quad \text { if } |\ell(A) - \ell(B)| >1 \text { or } \ell(A)=\ell(B).
\end{align*}
It can be deduced from \eqref{fra12} that the two families $\mathcal{Q}_1=\{A \in \mathcal{Q}: \ell(A)$ is odd $\}$ and $\mathcal{D}_2=\mathcal{Q}  \backslash \mathcal{D}_1$ are
$\frac{1}{2}$--sparse.

To demonstrate \eqref{BBT-1}--\eqref{BBT-3}, let $B \in \mathcal{Q}$, we claim that there exists a set $F_0$ of zero measure such that  
\begin{align}\label{eq5.30}
\varLambda(\vec{f} \chi_{B^3}) (x) 
\lesssim \Phi_{\varLambda,\vec{r},\tilde{r}}\left(\frac{1}{\boldsymbol{C}^3_0 \lambda}\right) \prod_{i=1}^m \langle f_i \rangle_{\eta_i,r_i,B^3}, \quad 
x \in \bigg(B^\dagger \backslash \bigcup_{G \in \mathcal{Q} \atop r(G)<r(B)} G^{\dagger} \bigg) \backslash F_0. 
\end{align}
Considering
\[
\mu(F_1) := \mu \bigg(\bigcap_{k \le k_0} \bigcup_{G \in \mathcal{Q}: r(G) \le k} G^{\dagger} \bigg) =0.
\]

Since we fix $x \in B_0 \setminus (F_0 \cup F_1)$, there exists $B \in \mathcal{Q}$ such that $x \in B^\dagger \setminus \bigcup_{G \in \mathcal{Q}:r(G)<r(B)} G^{\dagger}$. Because $B \in \mathcal{Q} \subseteq \mathcal{G}$, by the structure of $\mathcal{G}$, we can find a sequence $\{B_j\}_{j = 0}^k \subseteq \mathcal{Q}$ with $B_{j + 1} \in \F(B_j)$ for $j=0, 1, \ldots, k - 1$ and $B_k = B$. Moreover, \eqref{tree-4}--\eqref{tree-6} imply the existence of \(\widehat{B}_j\) and \(\xi_j \in \widehat{B}_{j+1}\) for \( j = 0, 1, \ldots, k-1 \), satisfying
\begin{align*}
 & B_{j+1}^{2} \subseteq \widehat{B}_{j+1} \subseteq B_j^{\dagger}, 
\\
 & \varLambda(\vec{f} \chi_{B^3_j})(\xi_j) 
\lesssim \Phi_{\varLambda,\vec{r},\tilde{r}}\left(\frac{1}{\boldsymbol{C}^3_0 \lambda}\right) \prod_{i=1}^m \langle f_i \rangle_{\eta_i,r_i,B^3_j}, 
\\ 
 & \left\|T_{\vec \eta}(\vec{f} \chi_{\widehat{B}^{\dagger}_{j+1}\backslash B^3_{j+1}})(y) \right\|_{\mathbb{V}} 
\lesssim \Phi_{\varLambda,\vec{r},\tilde{r}}\left(\frac{1}{\boldsymbol{C}^3_0 \lambda}\right) \prod_{i=1}^m \langle f_i \rangle_{\eta_i,r_i,B^3_j},\quad y \in B^2_{j+1}.
\end{align*}
Since \( x \in B^\dagger = B_k^{\dagger} \subseteq B_{j+1}^2 \), the last inequality holds for $y=x$.

All that remains is the proof for \eqref{eq5.30}. Indeed, Fixed $B \in \mathcal{Q}$, from \eqref{tree-3} and
\begin{align}\label{GGsub}
\bigcup_{G \in \mathcal{F}(B)} G \subseteq \bigcup_{\substack{G \in \mathcal{Q}: \\ r(G) < r(B)}} G^{\dagger} := H,
\end{align}
we can deduce \eqref{eq5.30}. Thus, it is sufficient to show that $G \in \F(B) \setminus \mathcal{Q}$.

Given $G \in \mathcal{F}(B)\setminus\mathcal{Q}$, it follows that $r(G) < r(B)$. By \eqref{DHG}, we only need to consider either $G \in \bigcup_{k\leq k_0}(\mathcal{H}^k\setminus\mathcal{Q}^k)$ or $G \in \bigcup_{k\leq k_0}(\mathcal{G}^k\setminus\mathcal{H}^k)$. 
In the first case, there exists unique $k\leq k_0$ and $B_k'\in\mathcal{H}^k$ such that $G = B_k' \subseteq \bigcup_{B_k\in\mathcal{Q}^k}B_k^\dagger \subseteq H$, where \eqref{BGBD} gives $r(B_k)=k=r(G)<r(B)$.
In the second one, \eqref{FBFB} yields $B'\in\mathcal{G} \subseteq \mathcal{Q}$ with $G^2\cap B'\neq\emptyset$ and $r(G)+2\leq r(B')\leq r(B)-2$.

From 
\[
\frac{1}{2}\log_{\C_0}\mu(G)+1 \leq r(B') \leq \frac{1}{2}\log_{\C_0}\mu(B')
\]
It can be deduced that $\mu(G^2)\leq \C_0^2\mu(G)\leq\mu(B')$, then by \eqref{list:B4}, we get $G \subseteq G^2 \subseteq (B')^\dagger \subseteq H$.
The estimate follows directly from \eqref{tree-3} and \eqref{GGsub}.
\end{proof}

\begin{lemma}\label{est.d2}
For all $x \in B_k$ and $j=0, 1, \ldots, k-1$, 
\begin{align}\notag
    &\left\|T_{\vec \eta}^{{\bf b, k}}(\vec{f} \chi_{B^3_j \backslash B^3_{j+1}})(x)  \right\|_{\mathbb{V}} 
 \\ \notag
& \leq \Phi_{\varLambda,\vec{r},\tilde{r}}\left(\frac{1}{\boldsymbol{C}^3_0 \lambda}\right)  \sum\limits_{0 \le {\bf t} \le {\bf k}}\prod_{i = 1}^m\big|b_i(x)-b_{i_1,B^3_j}\big|^{k_i-t_i} 
\langle\big|\big(b_i-b_{i_1,B^3_j}\big)^{t_i} f_i\big|\rangle_{\eta_{i},r_{i},B^3_j} 
\end{align}
\end{lemma}

\begin{proof}
From \eqref{zhuyi1}, it derives that
\begin{align}\notag
 &T_{\vec \eta}^{{\bf b, k}}(\vec{f} \chi_{B^3_j \backslash B^3_{j+1}})(x) \\ \label{JJ}
    &\leq 2^{\sum\limits_{i = 1}^m k_i} \sum\limits_{0 \le {\bf t} \le {\bf k}}\prod_{i = 1}^m\big|b_i(x)-b_{i_1,B^3_j}\big|^{k_i-t_i} 
 T_{\vec \eta}^{}\left(\left(\left(b_i-b_{i_1,B^3_j }\right)^{t_i}f_i\chi_{B^3_j \backslash B^3_{j+1}}\right)_{i=1}^m\right)(x).
\end{align}
Therefore, for $x \in B_k$, to demonstrate \eqref{lem:TFA}, it suffices to show that 
\begin{align}\label{Tgg}
&\left\| T_{\vec \eta} \left( \left( \left( b_i - b_{i_1,B^3_j} \right)^{t_i} f_i \chi_{B^3_j \backslash B^3_{j+1}} \right)_{i=1}^m \right)(x) \right\|_{\mathbb{V}}  \lesssim \C \prod_{i=1}^m \left\langle \left| \left( b_i - b_{i_1,B^3_j} \right)^{t_i} f_i \right| \right\rangle_{\eta_i, r_i, B^3_j}.
\end{align}

To proceed, we define
\begin{align*}
    &\mathbb{J}_1:=\left\|T_{\vec \eta}^{}\left(\left(\left(b_i-b_{i_1,B^3_j }\right)^{t_i}f_i\chi_{B^3_j \backslash \widehat{B}^{\dagger}_{j+1}}\right)_{i=1}^m\right)(x) - T_{\vec \eta}^{}\left(\left\{\left(b_i-b_{i_1,B^3_j }\right)^{t_i}f_i\chi_{B^3_j \backslash \widehat{B}^{\dagger}_{j+1}}\right\}_{i=1}^m\right)(\xi_j)\right\|_{\mathbb{V}}\\
    &\mathbb{J}_2:=\left\|T_{\vec \eta}^{}\left(\left(\left(b_i-b_{i_1,B^3_j }\right)^{t_i}f_i\chi_{B^3_j \backslash \widehat{B}^{\dagger}_{j+1}}\right)_{i=1}^m\right)(\xi_j)\right\|_{\mathbb{V}}\\
    &\mathbb{J}_3:=\left\|T_{\vec \eta}^{}\left(\left(\left(b_i-b_{i_1,B^3_j }\right)^{t_i}f_i\chi_{\widehat{B}^{\dagger}_{j+1} \backslash B^3_{j+1}}\right)_{i=1}^m\right)(x)\right\|_{\mathbb{V}}
\end{align*}
It follows that
\begin{align}\label{JJJ}
&\left\|T_{\vec \eta}^{}\left(\left(\left(b_i-b_{i_1,B^3_j }\right)^{t_i}f_i\chi_{B^3_j \backslash B^3_{j+1}}\right)_{i=1}^m\right)(x)\right\|_{\mathbb{V}} \leq \mathbb{J}_1 + \mathbb{J}_2 + \mathbb{J}_3.
\end{align}

To control $\mathbb{J}_1$, it follows from \eqref{BBT-1} that $ \widehat{B}_{j+1}^{\dagger} \subseteq B^3_j$. For $\xi_j \in \widehat{B}_{j+1}^{\dagger}$, it can be deduced from the condition \eqref{M_2.condi.} and \eqref{BBT-2} that
\begin{align}\notag
\mathbb{J}_1 
&\le \C_2 \prod_{i=1}^m \left\langle \left\langle \left(b_i-b_{i_1,B^3_j }\right)^{t_i}f_i\chi_{B^3_j} \right\rangle\right\rangle_{\eta_i, r_i, \widehat{B}^{\dagger}_{j+1}} \\ \notag
&\le \C \mathcal{M}_{\B, r}^{\otimes}\left(\left(\left(b_i-b_{i_1,B^3_j }\right)^{t_i}f_i\chi_{B^3_j}\right)_{i=1}^m\right)(\xi_j)
 \\  \notag
&\le \varLambda\left(\left(\left(b_i-b_{i_1,B^3_j }\right)^{t_i}f_i\chi_{B^3_j}\right)_{i=1}^m\right)(\xi_j)\\ \label{JJ1}
&\lesssim \C \prod_{i=1}^m \left\langle \left(b_i-b_{i_1,B^3_j }\right)^{t_i}f_i \right\rangle_{\eta_i,r_i,B^3_j}, 
\end{align}
where the $\C$ is defined in the below Lemma \ref{lem:aooa:4.4}.

For $\mathbb{J}_2$, it derives from \eqref{BBT-1} that $\widehat{B}_{j+1} \subseteq B_j^{\dagger}$. From \eqref{txing} and \eqref{BBT-1}, it follows that
\begin{align}\label{JJ2}
\mathbb{J}_2 
&= \left\|T_{\vec \eta}^{}\left(\left(\left(b_i-b_{i_1,B^3_j }\right)^{t_i}f_i\chi_{B^3_j \backslash (B_j^3 \cap \widehat{B}^{\dagger}_{j+1})}\right)_{i=1}^m\right)(\xi_j)\right\|_{\mathbb{V}}
\nonumber \\ 
&\le \varLambda\left(\left(\left(b_i-b_{i_1,B^3_j }\right)^{t_i}f_i\chi_{B^3_j}\right)_{i=1}^m\right)(\xi_j) \nonumber\\ 
&\lesssim \C \prod_{i=1}^m \left\langle \left(b_i-b_{i_1,B^3_j }\right)^{t_i}f_i \right\rangle_{\eta_i,r_i,B^3_j}, 
\end{align}

For $\mathbb{J}_3$, \eqref{BBT-3} gives 
\begin{align}\label{JJ3}
\mathbb{J}_3 
\lesssim \C \prod_{i=1}^m \left\langle \left(b_i-b_{i_1,B^3_j }\right)^{t_i}f_i \right\rangle_{\eta_i,r_i,B^3_j}.
\end{align}
Thus \eqref{Tgg} holds from \eqref{JJJ}--\eqref{JJ3}.
\end{proof}

\section{\bf Multilinear fractional dyadic representation theorem}\label{T1}

\subsection{Haar functions}
~~

Let \(\omega = (\omega^i)_{i \in \mathbb{Z}}\) be a sequence with \(\omega^i \in \{0,1\}^n\). Denote by \(\mathcal{D}_0\) the canonical dyadic grid on \(\mathbb{R}^n\). The \(\omega\)-translated dyadic grid is defined as
$$
\mathcal{D}_\omega=\left\{I+\sum_{i: 2^{-i}<J_1} 2^{-i} \omega^i: I \in \mathcal{D}_0\right\}=\left\{I+\omega: I \in \mathcal{D}_0\right\}
$$
where we simply have defined $I+\omega:=I+\sum_{i: 2^{-i}<J_1} 2^{-i} \omega^i$. There is a natural product probability measure $\mathbb{P}_\omega=\mathbb{P}$ on $\left(\{0,1\}^n\right)^{\mathbb{Z}}$; this gives us the notion of random dyadic grids $\omega \mapsto \mathcal{D}_\omega$.

A cube \(I \in \mathcal{D}_\omega\) is characterized as \emph{bad} if there exists \(J \in \mathcal{D}_\omega\) such that:
\begin{enumerate}
\item[\((\)v1\()\)] Scale separation: \(\ell(J) \geq 2^r \ell(I)\),
\item[\((\)v2\()\)] Boundary proximity: \(\mathrm{dist}(I, \partial J) \leq \ell(I)^\lambda \ell(J)^{1-\lambda}\),
\end{enumerate}
where the exponent \(\lambda\) is defined by
\[
\lambda := \frac{\delta}{2\big[n(2-\eta) + \delta\big]}
\]
with \(\delta > 0\) arising in kernel regularity estimates. Cubes not satisfying (B1)-(B2) are termed \emph{good}. The goodness probability \(\pi_{\mathrm{good}} := \mathbb{P}(I \stackrel{\omega}{+} I \text{ is good})\) exhibits translational invariance (independent of \(I \in \mathcal{D}_0\)). The separation parameter \(r \in \mathbb{N}\) is selected sufficiently large to guarantee \(\pi_{\mathrm{good}} > 0\). Importantly, the translated cube \(I \stackrel{\omega}{+} I\) depends only on scales \(\{ \omega^i : 2^{-i} < \ell(I) \}\), whereas its goodness is determined by scales \(\{ \omega^i : 2^{-i} \geq \ell(I) \}\). These events are statistically independent under the product measure.

For \(I \in \mathcal{D}_\omega\) and \(f \in L^1_{\mathrm{loc}}(\mathbb{R}^n)\), the associated martingale difference operator acts as
\[
\Delta_I f := \sum_{I' \in \mathrm{ch}(I)} \big( \langle f \rangle_{I'} - \langle f \rangle_I \big) \chi_{I'},
\]
satisfying the fundamental Littlewood-Paley identity
\[
\bigg\| \Big( \sum_{I \in \mathcal{D}} |\Delta_I f|^2 \Big)^{1/2} \bigg\|_{L^p} \approx \|f\|_{L^p}, \quad 1 < p < \infty.
\]

Given \(I = \prod_{k=1}^n I_k \in \mathcal{D}_\omega\), define the multivariate Haar function \(h_I^\eta\) indexed by \(\eta = (\eta_1,\dots,\eta_n) \in \{0,1\}^n\) through the tensor product
\[
h_I^\eta := \bigotimes_{k=1}^n h_{I_k}^{\eta_k},
\]
where 
\[
h_{I_k}^0 := |I_k|^{-1/2}\chi_{I_k}, \quad 
h_{I_k}^1 := |I_k|^{-1/2}\big( \chi_{I_{k,\ell}} - \chi_{I_{k,r}} \big).
\]
Here \(I_{k,\ell}\) and \(I_{k,r}\) denote the left/right dyadic subcubes of \(I_k\). The cancellative property \(\int h_I^\eta = 0\) holds when \(\eta \neq 0\). The martingale difference admits the orthogonal decomposition
\[
\Delta_I f = \sum_{\eta \in \{0,1\}^n \setminus \{0\}} \langle f, h_I^\eta \rangle h_I^\eta,
\]
though we suppress the \(\eta\)-summation in notation and write concisely \(\Delta_I f = \langle f, h_I \rangle h_I\), where \(h_I\) implicitly denotes a cancellative Haar function (\(\eta \neq 0\)). The non-cancellative case (\(\eta = 0\)) is always explicitly indicated by \(h_I^0\).

\subsection{Some conditions of operators}
~~

Let an $m$-linear operator $T$ admits a $m$-linear $\mathbb{V}$-valued fractional Calderón-Zygmund kernel representation, cf. Definition \ref{def:full}.


Fix some $\varepsilon>0$. Let $R \subseteq \mathbb{R}^n$ be a closed cube, and let $\phi \in L^\infty(R)$ with integral 0. Let $C=C(\varepsilon) \geq 3$ be any large constant so that $2^{-1}(C-1) \ell(R)>\varepsilon$, whence $|x-y|>\varepsilon$ for all cases in which $x \in R$ and $y_i \notin C R$ for some $i$. We define
\begin{align*}\label{kuohao2.1}
\left\langle T(1,\cdots,1), \phi\right\rangle &:= \left\langle T\left(\chi_{C R},\cdots, \chi_{C R}\right), \phi\right\rangle \\
 +&\iiint \left(K_\eta(x, y_1,\cdots,y_m)-K_\eta\left(c_R, y_1,\cdots,y_m\right)\right) \chi_{(C R \times C R)^c}(y_1,\cdots,y_m) \phi(x) d y_1 \cdots dy_m d x .\nonumber
\end{align*}
Note that this integral is absolutely convergent, since the smoothness condition of Definition \ref{def:full}.

\begin{definition}
Let $\D$ is a dyadic lattice, $r \in [1,\infty)$, and $\eta \in [0,\infty)$. We define the fractional BMO norm as
\[{\left\| b \right\|_{BM{O_{r,\eta }}}} := \mathop {\sup }\limits_{R \in \D} {\left| R \right|^\eta }{\left\langle {\left| {b - {b_R}} \right|} \right\rangle _{r,R}}.\]
If $\eta=0$, this norm is back to $\left\|b\right\|_{BMO}$.
\end{definition}

We establish the following Haar decomposition for the fractional BMO norm as follows.
\begin{proposition}\label{norm:FBMO}
Let $\D$ is a dyadic lattice, $r \in [1,\infty)$, and $\eta \in [0,\infty)$. Then 
\[\left\| b \right\|_{\BMO_{2,\eta }}^{} = \mathop {\sup }\limits_{R \in \D} {\left( {{{\left| R \right|}^{2\eta  - 1}}{{\sum\limits_{Q \subseteq R} {\left| {\left\langle {b,{h_Q}} \right\rangle } \right|} }^2}} \right)^{\frac{1}{2}}}.\]
\end{proposition}

\begin{proof}
From the fact that $\{h_Q\}_{Q \in \D}$ forms an orthonormal basis in $L^2$ and Parseval identity, we have
    \[
    \sum_{Q \subseteq R} \left| \left\langle b, h_Q \right\rangle \right|^2 = \int_R \left| b(x) - \langle b \rangle_R \right|^2 dx.
    \]
The desired is obvious.
\end{proof}

Next, we give the definition of the weak boundedness property.

\begin{definition}
Let $\eta \in [0,m)$, ${\bf k }\in \N_0^m$, and $b_i \in \BMO$, for $i=1,\cdots,m$. Let $T$ be a $\mathbb{V}$-valued multilinear operator and $\D$ is a dyadic lattice.
The weak boundedness property constant for generalized commutators is defined as  
\[{\left\| T_{\eta}^{\bf{b,k}} \right\|_{{WBP}_{\eta}^{\bf{b,k}}(\mathbb{V})}}: = \prod\limits_{i = 1}^m {\left\| {{b_i}} \right\|_{\BMO}^{{-k_i}}} \cdot\sup_{Q\in \D}{\left| Q \right|^{-\eta-1}}\left| {\left\langle {{{\left\| {T_{\eta}^{\bf{b,k}}\left( {{{\vec \chi }_Q}} \right)} \right\|}_\mathbb{V}},{\chi _Q}} \right\rangle } \right|.\]
\end{definition}

Next, we introduce two classes of multilinear dyadic model operators.

\subsection{Multilinear dyadic model operators}\label{m-DMO}

\subsubsection{Multilinear fractional shift}
~~

Given a dyadic lattice $\mathscr{D}$. Let $\eta \in [0,m)$. For $j_i \in \N_0$, $i=1,\cdots,m+1$, we define the bilinear shift 
$$
\mathbb{S}_{\eta}^{j_1,\cdots,j_{m+1}}(\vec{f}):=\sum_{P \in \D} \mathbb{A}_{\eta,P}^{j_1,\cdots,j_{m+1}}(\vec{f}),
$$
where
$$
\mathbb{A}_{\eta,P}^{j_1,\cdots,j_{m+1}}(\vec{f}):=\sum_{\substack{
J_1, \cdots J_{m+1} \subseteq P \\
J_i=2^{-j_i} \ell(P) \\
i=1,\cdots,{m+1}
}} 
\beta_{\eta,{J_1},\cdots,{J_m},P}\left\langle f_1, \tilde{h}_{J_1}\right\rangle \cdots
\left\langle f_m,\tilde{h}_{J_m}\right\rangle h_{J_{m+1}},
$$
and for $i=1,\cdots,m$, 
$
\tilde{h}_{J_i} \in \left\{h_{J_i},h_{J_i}^0\right\},$
with
$
\left\{\tilde{h}_{J_1},\cdots,\tilde{h}_{J_m} \right\} \ne \left\{h_{J_1}^0,\cdots,h_{J_m}^0\right\},$

We also demand that
\begin{align}
\left|\beta_{\eta,{J_1},\cdots,{J_{m+1}},P}\right| \le C 
\frac{|J_1|^{1/2}\cdots|J_{m+1}|^{1/2}}{|P|^{m-\eta}} \label{key:shift}
\end{align}
Such a shift will be considered to be a cancellative multilinear fractional shift. Also, the duals of these operators will be used in the representation.

Let $1<r_1,\cdots, r_m<\infty$, and $\frac{1}{{\tilde r}} := \frac{1}{{{r_1}}} + \cdots+\frac{1}{{{r_m}}} - \eta  > 0$. 
We claim that
\begin{align}
\mathbb{S}_{\eta}^{j_1, \cdots, j_{m+1}}:\prod\limits_{i = 1}^m {{L^{{r_i}}}}  \to {L^{\tilde r}} \text{ is bounded,}\label{Key1}
\end{align}
where the constant independent of the shift, and depending only on $m,n,\eta,r_1,\cdots,r_m$.

\begin{proof}[Proof of \eqref{Key1}]

For a dyadic cube \(P \in \mathscr{D}\) with side length \(\ell(P)\), its subcubes \(J_i \subseteq P\) satisfy \(\ell(J_i) = 2^{-j_i}\ell(P)\), hence \(
|J_i| = 2^{-j_i n} |P|.
\)

From the fact that ${\left\| {{h_I}} \right\|_r} = {\left| I \right|^{\frac{1}{r} - \frac{1}{2}}}$, it follows that
\[
|\beta_{\eta,J_1,\dots,J_{m+1},P}| \le C \frac{2^{-n\sum_{i=1}^{m+1} j_i/2}}{|P|^{m-\eta - (m+1)/2}}.
\]
Then
\[
\|\mathbb{A}_{\eta,P}^{j_1,\dots,j_{m+1}}(\vec{f})\|_{L^{\tilde{r}}} \leq \sum_{\substack{J_1,\dots,J_{m+1} \subseteq P \\ \ell(J_i)=2^{-j_i}\ell(P)}}  |\beta_{\eta,J_1,\dots,J_m}| \prod_{i=1}^m |\langle f_i,\tilde{h}_{J_i}\rangle| \cdot \|h_{J_{m+1}}\|_{L^{\tilde{r}}}.
\]
By Hölder's inequality, we more have
\[
|\langle f_i, \tilde{h}_{J_i}\rangle| \leq \|f_i\|_{L^{r_i}} \left(2^{-j_i n}|P|\right)^{1/r_i' - 1/2}.
\]

Collecting above estimates, it follows obviously that
\begin{align*}
\|\mathbb{S}_{\eta}^{j_1,\dots,j_{m+1}}(\vec{f})\|_{L^{\tilde{r}}} &= \|\mathop {\lim }\limits_{j \to -\infty} \sum\limits_{P \subseteq Q_j} {\mathbb{A}_{\eta ,P}^{{j_1}, \cdots ,{j_{m + 1}}}} (\vec f)\|_{L^{\tilde{r}}}\\
&\le C\prod_{i=1}^m \|f_i\|_{L^{r_i}} \cdot \sum_{j_1,\dots,j_{m+1}=0}^\infty 2^{n\sum\limits_{i = 1}^m {{j_i}\left( {\frac{1}{{{r_i}}} - 1} \right)}  + n{j_{m + 1}}\left( {\eta  - \sum\limits_{i = 1}^m {\frac{1}{{{r_i}}}} } \right)}.
\end{align*}
\end{proof}

\subsubsection{Multilinear fractional paraproduct.} 
\begin{definition}
Let $\eta \in [0,m)$ and $\D$ is a dyadic lattice. Set $\beta_\eta=\left\{\beta_{\eta,P}\right\}_{P \in \mathscr{D}} \subseteq \C$ such that
\begin{align}\label{kuohao3.1}
\mathop {\sup }\limits_{P_0 \in \D} |P_0{|^{ - 2\eta  - 1}}\sum\limits_{P \subseteq {P_0}} {{{\left| {{\beta_{\eta ,P}}} \right|}^2}}  < \infty.
\end{align}
The multilinear fractional paraproduct is defined as
$$
\Pi_{\beta_\eta}(\vec{f}):=\sum\limits_{P \in \D} {{\beta _{\eta ,P}}} \prod\limits_{i = 1}^m {{{\langle {f_i}\rangle }_P} \cdot } {h_P}.
$$  
\end{definition}
We claim that $$\Pi_{\beta_\eta}:\prod\limits_{i = 1}^m {{L^{{r_i}}}}  \to {L^{\tilde r}} \text{ is bounded,}$$ where $\widetilde{r}>0$ with $1 / r_1+\cdots+1 / r_m-1 / \widetilde{r}=\eta$ .

Indeed, it follows from the Littlewood--Paley square operators theory that
\begin{align*}
{\left\| {{\Pi _{{\beta _\eta }}}(\vec f)} \right\|_{{L^{\tilde r}}}} \approx& {\left\| {{{\left( {\sum\limits_{P \in \D} {{{\left| {{\beta_{\eta ,P}}} \right|}^2}} {{\left| {\prod\limits_{i = 1}^m {{{\langle {f_i}\rangle }_P}} } \right|}^2}\frac{{{\chi _P}}}{{|P|}}} \right)}^{1/2}}} \right\|_{{L^{\tilde r}}}}\\
 \le& {\left\| {{{\left( {\sum\limits_{P \in \D} {{{\left| {{\beta_{\eta ,P}}} \right|}^2}} {{\left| {\mathop {\inf }\limits_{ \cdot  \in P} \M_\eta ^\D\left( {\vec f} \right)( \cdot )} \right|}^2}{{\left| P \right|}^{ - 2\eta  - 1}}{\chi _P}} \right)}^{1/2}}} \right\|_{{L^{\tilde r}}}}\\
\le& {\left\| {\M_\eta ^\D\left( {\vec f} \right){{\left( {\sum\limits_{P \in \D} {{{\left| {{\beta _{\eta ,P}}} \right|}^2}} {{\left| P \right|}^{ - 2\eta  - 1}}{\chi _P}} \right)}^{\frac{1}{2}}}} \right\|_{{L^{\tilde r}}}}\\
\lesssim& \prod\limits_{i = 1}^m {{{\left\| {{f_i}} \right\|}_{{L^{{r_i}}}}}}.
\end{align*}
In the last estimate, we, in fact, also need to verify 
\begin{align}
{\left( {\sum\limits_{P \in \D} {{{\left| {{\beta _{\eta ,P}}} \right|}^2}} {{\left| P \right|}^{ - 2\eta  - 1}}{\chi _P}} \right)} \lesssim 1. \label{paraproduct}
\end{align}

\begin{proof}[Proof of \eqref{paraproduct}]

Denote $\mathscr{D}$ by ${\left\{ {\left\{ {Q_j^i} \right\}_{i = 1}^{{2^n}}} \right\}_j}$, where $Q_j$ is the parent of $Q_{j+1}$(conventionally, we tend to overlook $i$) and $\left| {{Q_j}} \right| = {2^{ - nj}}$. In addition, we write $Q$ by $Q(x)$ if $x \in Q$.

It suffices to prove that 
$$S\left( x \right) = \sum\limits_j {\frac{{{{\left| {{\beta _{{Q_j}(x)}}} \right|}^2}}}{{{{\left| {{Q_j}(x)} \right|}^{2\eta  + 1}}}}} \lesssim 1.$$

Set \( j \in \Z\) with \( Q_j(x) \subseteq P_0(x)\). We assume that for any $C,\epsilon>0$, 
\[
|\beta_{Q_j(x)}|^2 \ge C \cdot 2^{-n\epsilon j} |Q_j(x)|.
\]

It follows from \eqref{kuohao3.1} that
\[
|P_0(x)|^{2\eta+1}\ge \sum_j |\beta_{Q_j(x)}|^2 \ge C \sum_j 2^{-n\epsilon j} |Q_j(x)| \ge \sum_{j <0} |Q_j(x)|=\infty.
\]
This leads to a contradiction.

Thus, for \( j \in \Z\) with \( Q_j(x) \subseteq P_0(x)\), there are $C$ and $\epsilon>0$, such that 
\[
|\beta_{Q_j(x)}|^2 < C \cdot 2^{-n\epsilon j} |Q_j(x)|^{2\eta+1}.
\]

Therefore, we have
\[S\left( x \right) = \mathop {\sup }\limits_{{P_0}(x) \in \D} \sum\limits_{j:{Q_j}(x) \subseteq {P_0}(x)} {\frac{{{{\left| {{\beta _{{Q_j}(x)}}} \right|}^2}}}{{{{\left| {{Q_j}(x)} \right|}^{2\eta  + 1}}}}} \lesssim 1.\] 
\end{proof}

\subsubsection{$(m+1)$-linear fractional dyadic model form}~~

We introduce $(m+1)$-linear fractional dyadic model form as follows.

\begin{definition}\label{FDMF}
Given a dyadic lattice $\mathscr{D}$ defined on $\mathbb{R}^n$ and $\nu \in \N_0$. 
Define $(m+1)$-linear fractional dyadic model form 

\begin{align*}
\mathbb{F}_{\eta,\nu}^\mathscr{D}\left(\vec{f},f_{m+1}\right) :=\sum_{P \in \mathscr{D}} {\mathcal{F}_{\eta,\nu,P}^\D}\left(\vec{f},f_{m+1}\right),
\end{align*}
where  the localized $(m+1)$-linear fractional dyadic model form 
\[{\mathcal{F}_{\eta,\nu,P}^\D}\left( {\vec f,f_{m+1}} \right): = \int_{{P^m}} {{\mathcal{K}_{\eta ,\nu,P }^\D}\left( {{x_1}, \cdots ,{x_{m+1}}} \right)\prod\limits_{j = 1}^{m+1} {{f_j}} \left( {{x_j}} \right)d{x_1} \cdots d{x_{m+1}}},\]

and $\mathcal K_{\eta,\nu,P}^\D: P^m \rightarrow \mathbb{C}$  satisfies
\begin{itemize}
\item 
\begin{align}
\left\|\mathcal K_{\eta,\nu,P}^\D\right\|_{L^{\infty}} \le |P|^{\eta-m}. \label{Kernel.1}
\end{align}

\item there exists  $\eta \in [0,m)$ and exponents $r_1,\cdots, r_m, \widetilde{r} \in(1, \infty)$ such that $1 / r_1 +\cdots +1 / r_{m+1}-1 / \widetilde{r}=\eta$ and a constant $\mathcal{B}$ such that for every dyadic lattice $\mathscr{P} \subseteq \mathscr{D}$, 
$$
\mathbb{F}_{\eta,\nu}^{\mathscr{P}}\left(\vec{f},f_{m+1}\right):=\sum_{P \in \mathscr{P}} \mathcal{F}_{\eta,\nu,P}^\mathscr{P}\left(\vec{f},f_{m+1}\right)
$$
satisfies

\begin{align}
\left\| \mathbb{F}_{\eta,\nu}^{\mathscr{P}} \right\|: = {\left\| \mathbb{F}_{\eta,\nu}^{\mathscr{P}}\right\|_{{L^{{r_1}}} \times\cdots\times {L^{{r_{m+1}}}} \to {L^{\tilde r'}}}} < \infty . \label{F.bounded}
\end{align}

\item  $\mathcal{K}_{\eta,P}^\D\equiv C$ on sets of the form $P_1 \times \cdots\times P_{m+1}$, where $P_i^{(\nu+1)}=P$ ($P$ is the ($\nu+1$)-ancestor of $P_i$).
\end{itemize}

\end{definition}

It can be verified that the above definition of $(m+1)$-linear fractional dyadic model form is an extension of integral form of our dyadic model operators (cf. subsection \ref{m-DMO}), which means that we need to further explore the properties of this form.


Since every dyadic lattice $\mathscr{D}$ can expressed by $\mathop  \cup \limits_j {\D_j}$, where $\D_j=\{Q_k^j\}_{k\in\Z}$ is a nested sequence of dyadic cubes, i.e. for any $k \in \Z$, $Q_k^j \subsetneq Q_{k+1}^j$. This is a important reason for us to limit dyadic lattice without quadrants in the next proposition. It is worth mentioning that this proposition holds for all dyadic lattices whose proof was postponed to subsection \ref{dyadic.model.form}.

\begin{theorem}\label{thm:d.s.d.}
Let $\eta \in [0,m)$, $\D$ is a dyadic lattice, and $\mathbb{F}_{\eta,\nu}^\D$ is a  $(m+1)$-linear fractional dyadic model form that does not contain quadrants. Then there is a $\delta$-sparse family $\mathcal{S}=\mathcal{S}(\vec{f},h,\delta,\eta) \subseteq \D$, such that for any nonnegative $f_i \in L_b^\infty$, $i=1,\cdots,m$, and nonnegative $h \in L_b^\infty$,
\begin{align*}
 {\mathbb{F}_{\eta,\nu}^{\D}}(\vec f,h){ \lesssim_{\eta ,\delta }}({\mathcal{B}} + \nu) 
 {{\calA}_{\mathcal{S},\eta ,\vec 1,1}}\left( {\vec f,h} \right).
\end{align*}
\end{theorem}

\subsection{{Proof of Theorem \ref{thm:DT}}}~~

Without loss of generality, we prove Theorem \ref{thm:DT} for bilinear case. It is sufficient to assume that $T_{\eta,{\bf b}}^{\bf k}: {L^{\frac{6}{5}}} \times {L^{\frac{6}{5}}} \to {L^{\frac{1}{{\frac{5}{3}-\eta }}}}$ is a bounded operator, where we take $\eta \in [0,{\frac{5}{3}})$.
Indeed, we can find that this assumption of a priori boundedness is harmless for the proof, and even without this assumption we can in fact obtain the corresponding conclusion, which one can refer to 
\cite{Li20,Hytonen18_1,Hytonen18_2}.

\subsubsection{\textbf{Dyadic reductions}}
~~

Decompose $\langle {T_{\eta}^{\bf b,k}}(f_1, f_2), h\rangle$ as

\begin{align*}
\left\langle {T_{\eta}^{\bf b,k}}(f_1, f_2), h\right\rangle= & \mathbb{E}_\omega \sum_{\substack{{J_1},{J_2},{J_3} \in \mathcal{D}_\omega: \\
\ell({J_3}) \leq \min\{\ell({J_1}),\ell({J_2})\}}} \left\langle {T_{\eta}^{\bf b,k}}\left(\Delta_{J_1} f_1, \Delta_{J_2} f_2\right), \Delta_{J_3} h\right\rangle \\
& +\mathbb{E}_\omega \sum_{\substack{{J_1},{J_2} ,{J_3}\in \mathcal{D}_\omega: \\
\ell({J_1}) \leq \min\{\ell({J_2}),\ell({J_3})\}}} \left\langle T_{\eta}^{{\bf{b,k}},1 *}\left(\Delta_{J_3} h, \Delta_{J_2} f_2\right), \Delta_{J_1} f_1\right\rangle \\
& +\mathbb{E}_\omega  \sum_{\substack{{J_1}, {J_2},{J_3}\in \mathcal{D}_\omega: \\
\ell({J_2})<\min\{\ell({J_1}),({J_3})\}}} \left\langle T_{\eta}^{{\bf{b,k}},2 *}\left(\Delta_{J_2} f_1, \Delta_{J_3} h\right), \Delta_{J_2} f_2\right\rangle\\
=&: \mathbb{I}_1+\mathbb{I}_2+\mathbb{I}_3 .
\end{align*}

We only estimate $\mathbb{I}_1$, because of the similarity between $T_{\eta}^{{\bf{b,k}}}$ and $T_{\eta}^{{\bf{b,k}},j*}$.

\begin{align*}
 &\sum_{\substack{{J_1},{J_2},{J_3} \in \mathcal{D}_\omega: \\
\ell({J_3}) \leq \min\{\ell({J_1}),\ell({J_2})\}}} \left\langle {T_{\eta}^{\bf b,k}}\left(\Delta_{J_1} f_1, \Delta_{J_2} f_2\right), \Delta_{J_3} h\right\rangle\\
&=\sum_{{J_3} \in \mathcal{D}_\omega}\left\langle {T_{\eta}^{\bf b,k}}\left(E_{\ell({J_3}) / 2}^\omega f_1, E_{\ell({J_3}) / 2}^\omega f_2\right), \Delta_{J_3} h\right\rangle,
\end{align*}

where
$$
E_{\ell({J_3}) / 2}^\omega f=\sum_{\substack{{J_1} \in \mathcal{D}_\omega: \\ \ell({J_1})=\ell({J_3}) / 2}} \chi_{J_1}\langle f\rangle_{J_1}
$$

We transform $\mathcal{D}_\omega$ to $\mathcal{D}_0$ from the fact that $\mathcal{D}_\omega=\mathcal{D}_0+\omega$. We have
\begin{align*}
&\sum_{{J_3} \in \mathcal{D}_\omega}\left\langle {T_{\eta}^{\bf b,k}}\left(E_{\ell({J_3}) / 2}^\omega f_1, E_{\ell({J_3}) / 2}^\omega f_2\right), \Delta_{J_3} h\right\rangle\\
&=\sum_{{J_3} \in \mathcal{D}_0}\left\langle {T_{\eta}^{\bf b,k}}\left(E_{\ell({J_3}) / 2}^\omega f_1, E_{\ell({J_3}) / 2}^\omega f_2\right), \Delta_{{J_3}+\omega} h\right\rangle
\end{align*}

It derives from the independence of expectation that
\begin{align*}
\mathbb{I}_1 & =\mathbb{E}_\omega \sum_{{J_3} \in \mathcal{D}_0}\left\langle {T_{\eta}^{\bf b,k}}\left(E_{\ell({J_3}) / 2}^\omega f_1, E_{\ell({J_3}) / 2}^\omega f_2\right), \Delta_{{J_3}+\omega} h\right\rangle \\
& =\frac{1}{\pi_{\text {good }}} \sum_{{J_3} \in \mathcal{D}_0} \mathbb{E}_\omega\left[\chi_{\text {good }}({J_3}+\omega)\right] \mathbb{E}_\omega\left[\left\langle {T_{\eta}^{\bf b,k}}\left(E_{\ell({J_3}) / 2}^\omega f_1, E_{\ell({J_3}) / 2}^\omega f_2\right), \Delta_{{J_3}+\omega} h\right\rangle\right] \\
& =\frac{1}{\pi_{\text {good }}} \mathbb{E}_\omega \sum_{{J_3} \in \mathcal{D}_{\omega, \text { good }}}\left\langle {T_{\eta}^{\bf b,k}}\left(E_{\ell({J_3}) / 2}^\omega f_1, E_{\ell({J_3}) / 2}^\omega f_2\right), \Delta_{J_3} h\right\rangle:= \frac{1}{\pi_{\text {good }}} \mathbb{E}_\omega \mathbb{I}_{\omega},
\end{align*}
where for $\ell(J_3) \le 2^{-j}$,
$\chi_{\text {good }}(J_3+\omega)$ relies on $\omega_j$; for $2^{-j}<\ell(J_3) / 2<\ell(J_3)$, $E_{J_3 / 2}^\omega f_1$  depends on $\omega_j$ (similarly for $E_{J_3 / 2}^\omega f_2$); and for $2^{-j}<\ell(J_3)$, $\Delta_{J_3} h$ depends on $\omega_j$.

Without loss of generality, we rewrite $\mathscr{D}=\mathcal{D}_\omega$. It follows that 



\begin{align*}
\mathbb{I}_\omega&=\sum_{\substack{{J_1},{J_2} \in \mathscr{D}, {{J_3} \in \mathscr{D}_{\text {good }}}:\\ \ell({J_3}) \leq {\min\{{{\ell({J_1})}}, {\ell({J_2})}}\}}}\left\langle {T_{\eta}^{\bf b,k}}\left(\Delta_{J_1} f_1, \Delta_{J_2} f_2\right), \Delta_{J_3} h\right\rangle .\\
&=\sum_{\substack{{J_1},{J_2} \in \mathscr{D},{{J_3} \in \mathscr{D}_{\text {good }}}: \\ \ell({J_3}) \leq \ell({J_1})\leq \ell({J_2})}} \left\langle {T_{\eta}^{\bf b,k}}\left(\Delta_{J_1} f_1, \Delta_{J_2} f_2\right), \Delta_{J_3} h\right\rangle \\
&\quad+\sum_{\substack{{J_2},{J_1} \in \mathscr{D},{{J_3} \in \mathscr{D}_{\text {good }}}: \\ \ell({J_3}) \leq \ell({J_2})<\ell({J_1})}} \left\langle {T_{\eta}^{\bf b,k}}\left(\Delta_{J_1} f_1, \Delta_{J_2} f_2\right), \Delta_{J_3} h\right\rangle  .\\
&=  \sum_{\substack{{J_1} \in \mathscr{D},{{J_3} \in \mathscr{D}_{\text {good }}}: \\
\ell({J_3}) \leq \ell({J_1})}}\left\langle {T_{\eta}^{\bf b,k}}\left(\Delta_{J_1} f_1, E_{\ell({J_1}) / 2} f_2\right), \Delta_{J_3} h\right\rangle \\
& \quad + \sum_{\substack{{J_2} \in \mathscr{D}, {{J_3} \in \mathscr{D}_{\text {good }}}:\\
\ell({J_3}) \leq \ell({J_2})}}\left\langle {T_{\eta}^{\bf b,k}}\left(E_{\ell({J_2})} f_1, \Delta_{J_2} f_2\right), \Delta_{J_3} h\right\rangle\\
&:= \mathbb{I}_{\omega,1}+\mathbb{I}_{\omega,2}
\end{align*}

Obviously, the structures of $\mathbb{I}_{\omega,i}$, $i=1,2$, are similar. To finish our proof, it is ample to estimate $\mathbb{I}_{\omega,1}$.

\subsubsection{\textbf{Estimates of $\mathbb{I}_{\omega,1}$:}}

{

\begin{align*}
\mathbb{I}_{\omega,1} 
& =\sum_{\substack{{J_1},{J_2}  \in \mathscr{D} ,{{J_3} \in \mathscr{D}_{\text {good }}}:\\
\ell({J_3}) \leq \ell({J_1})=2 \ell({J_2})}} \left\langle {T_{\eta}^{\bf b,k}}\left(\Delta_{J_1} f_1, \chi_{J_2}\langle f_2\rangle_{J_2}\right), \Delta_{J_3} h\right\rangle\\
&=\sum_{\substack{{J_1}, {J_2} \in \mathscr{D},{{J_3} \in \mathscr{D}_{\text {good }}}: \\
\ell({J_3}) \leq \ell({J_1})=2 \ell({J_2}) \\
\max (d({J_3}, {J_1}), d({J_3}, {J_2}))>\ell({J_3})^\lambda \ell({J_2})^{1-\lambda}}}\left\langle {T_{\eta}^{\bf b,k}}\left(\Delta_{J_1} f_1, \chi_{J_2}\langle f_2\rangle_{J_2}\right), \Delta_{J_3} h\right\rangle\\
&\quad+\sum_{\substack{ 
{J_1}, {J_2} \in \mathscr{D},{{J_3} \in \mathscr{D}_{\text {good }}}: \\ 
\ell({J_3}) \leq \ell({J_1})=2 \ell({J_2}) \\
\max (d({J_3}, {J_1}), d({J_3}, {J_2})) \leq \ell({J_3})^\lambda \ell({J_2})^{1-\lambda} \\
{J_3} \cap {J_2}=\emptyset \text { or } {J_3} ={J_1}  \text { or } {J_3}  \cap J=\emptyset
}}\left\langle {T_{\eta}^{\bf b,k}}\left(\Delta_{J_1}  f_1, \chi_{J_2} \langle f_2\rangle_{J_2} \right), \Delta_{J_3}  h\right\rangle \\
&\quad+\sum_{\substack{{J_1} , {J_2}  \in \mathscr{D}, {J_3}  \in \mathscr{D}_{\text {good }}: \\
\ell({J_1} )=2 \ell({J_2} ) \\
{J_3}  \subseteq {J_2}  \subseteq {J_1} }}\left\langle {T_{\eta}^{\bf b,k}}\left(\Delta_{J_1}  f_1, \chi_{J_2} \langle f_2\rangle_{J_2} \right), \Delta_{J_3}  h\right\rangle \\
&:=  \mathbb{I}_{\omega,1,1}+ \mathbb{I}_{\omega,1,2}+ \mathbb{I}_{\omega,1,3}.
\end{align*}

where the first sum is called 'separated' sum, the second sum is called 'adjacent' sum, and the third sum is called 'nested' sum.

{\emph{{\textbf{\textsf Case 1: Separated sum}}}

\begin{align*}
 \mathbb{I}_{\omega,1,1}
=\sum_{\substack{
{J_1} , {J_2}  \in \mathscr{D},J_3\in \mathscr{D}_{\text{good}}: \\
\ell({J_3} ) \leq \ell({J_1} )=2 \ell({J_2} ) \\
\max (d({J_3} , {J_1} ), d({J_3} , {J_2} ))>\ell({J_3} )^\lambda \ell({J_2} )^{1-\lambda}}} \left\langle {T_{\eta}^{\bf b,k}}\left(h_{J_1} , h_{J_2} ^0\right), h_{{J_3} }\right\rangle\left\langle f_1, h_{J_1} \right\rangle\left\langle f_2, h_{J_2} ^0\right\rangle\left\langle h, h_{{J_3} }\right\rangle .
\end{align*}


\begin{lemma}\label{le4.1} If ${J_1}, {J_2}, {J_3}$ satisfy the limit of the sum in $\mathbb{I}_{\omega,1,1}$, then there is a cube $P \in \mathscr{D}$ such that ${J_1} \cup{J_2}  \cup {J_3} \subseteq P$ and
\begin{align}
\ell({J_3})^\lambda \ell(P)^{1-\lambda} \lesssim \max (d({J_3}, {J_1}), d({J_3}, {J_2})) \label{ine.2}
\end{align}
\end{lemma}

\begin{proof} Let $P \in \mathscr{D}$ be the minimal ancestor of ${J_3}$ for which satisfies $\ell(P) \geq 2^r \ell({J_3})$ and $\max (d({J_3}, {J_1}), d({J_3}, {J_2})) \leq \ell({J_3})^\lambda \ell(P)^{1-\lambda}$. Since $\ell(P) \geq 2^r \ell({J_3})$, from the fact that ${J_3}$ does not satisfy the property of bad cube, 
\begin{align}\label{ine.1}
\ell({J_3})^\lambda \ell(P)^{1-\lambda}<d\left({J_3}, P^c\right)
\end{align}

Notice that 
$
\ell({J_3})^\lambda \ell({J_2})^{1-\lambda}<\max (d({J_3}, {J_1}), d({J_3}, {J_2})) \leq \ell({J_1})^\lambda \ell(P)^{1-\lambda},
$
then, $\ell(P)>max\{\ell({J_1}),\ell({J_2})\}.$ We indeed can prove that ${J_1} \cup {J_2} \cup {J_3} \subseteq$ $P$ by showing that ${J_1} \cap P \neq \emptyset$ and ${J_2} \cap P \neq \emptyset$.

In fact, if we assume that ${J_1} \subseteq P^c$ or ${J_2} \subseteq P^c$, it derives from \eqref{ine.1} that
$$
\ell({J_3})^\lambda \ell(P)^{1-\lambda}<\max (d({J_3}, {J_1}), d({J_3}, {J_2})) \leq \ell({J_3})^\lambda \ell(P)^{1-\lambda}.
$$
It is a contradiction. 

Last, it is easy to see that \eqref{ine.2} is obvious. On the one hand, $P$ has minimality and $\ell(P) \approx \ell({J_3})$. On the other hand, $\ell({J_3}) \le 2\ell({J_2})$. \eqref{ine.2} derives from these.
\end{proof}

Let $P={J_1} \vee {J_2} \vee {J_3}$ be the mininmal cube in Lemma \ref{le4.1}, then we obtain
\begin{align}\label{eq4.2}
\ell({J_3})^\lambda \ell(P)^{1-\lambda}\lesssim\max (d({J_3}, {J_1}), d({J_3}, {J_2}))  
\end{align}

We write $\mathbb{I}_{\omega,1,1}$ as

\begin{align*}
\sum_{k=0}^{\infty} \sum_{i=0}^k \sum_{\substack{
{J_1}, {J_2} ,P\in \mathscr{D}, {J_3} \in \mathscr{D}_{\text {good }}:\\
\max (d({J_3}, {J_1}), d({J_3}, {J_2}))>\ell({J_3})^\lambda \ell({J_2})^{1-\lambda} \\
2 \ell({J_2})=\ell({J_1})=2^{-i} \ell(P), \ell({J_3})=2^{-k} \ell(P) \\
{J_1} \vee {J_2} \vee {J_3}=P
}}
\left\langle {T_{\eta}^{\bf b,k}}\left(h_{J_1}, h_{J_2}^0\right), h_{{J_3}}\right\rangle\left\langle f_1, h_{J_1}\right\rangle
\left\langle f_2, h_{J_2}^0\right\rangle\left\langle h, h_{J_3}\right\rangle .
\end{align*}

When the following conditions hold: ${J_1},  {J_2}
\in \mathscr{D}, {J_3} \in \mathscr{D}_{\text {good }}, \max (d({J_3}, {J_1}), d({J_3}, {J_2}))>\ell({J_3})^\lambda J_2^{1-\lambda}, \ell({J_3}) \leq \ell({J_1})=$ $2 \ell({J_2})$, and ${J_1} \vee {J_2} \vee {J_3}=P$,
we can define
$$
\beta_{\eta,{\bf{k}},{\bf b},{J_1},{J_2},{J_3},P}=\frac{\left\langle {T_{\eta}^{\bf b,k}}\left(h_{J_1}, h_{J_2}^0\right), h_{{J_3}}\right\rangle}{(\ell({J_3}) / \ell(P))^{\delta / 2}};
$$
otherwise $\beta_{\eta,{\bf{k}},{\bf b},{J_1},{J_2},{J_3},P}=0$. For $ 0\le i \leq k$,

\begin{align*}
&  \sum_{\substack{{J_1}, {J_2} ,P\in \mathscr{D}, {J_3} \in \mathscr{D}_{\text {good }}: \\
\max (d({J_3}, {J_1}), d({J_3}, {J_2}))>\ell({J_3})^\lambda \ell({J_2})^{1-\lambda} \\
2 \ell({J_2})=\ell({J_1})=2^{-i} \ell(P), \ell({J_3})=2^{-k} \ell(P) \\
{J_1} \vee {J_2} \vee {J_3}=P 
}}\left\langle {T_{\eta}^{\bf b,k}}\left(h_{J_1}, h_{J_2}^0\right), h_{{J_3}}\right\rangle\left\langle f_1, h_{J_1}\right\rangle\left\langle f_2, h_{J_2}^0\right\rangle h_{J_3} \\
& =2^{-\beta k / 2} \sum_{P \in \mathcal{D}} \sum_{\substack{
{J_1}, {J_2}, {J_3} \subseteq P: \\
\ell(J_1)=2^{-i} \ell(P) \\
\ell(J_2)=2^{-i-1} \ell(P) \\
\ell(J_3)=2^{-k} \ell(P) }} \beta_{\eta,{\bf{k}},{\bf b},{J_1},{J_2},{J_3},P}\left\langle f_1, h_{J_1}\right\rangle\left\langle f_2, h_{J_2}^0\right\rangle h_{J_3} \\
& := 2^{-\delta k / 2} \mathbb{S}_{\eta,{\bf b}}^{i, i+1, k}(f_1, f_2),
\end{align*}

which yields
$$
\mathbb{I}_{\omega,1,1}=\sum_{k=0}^{\infty} \sum_{i=0}^k 2^{-k \delta / 2}\left\langle \mathbb{S}_{\eta,{\bf b}}^{i, i+1, k}(f_1, f_2), h\right\rangle
$$

Finally, it is needful for us to verfy that
\begin{align}
\left|\beta_{\eta,{\bf{k}},{\bf b},{J_1},{J_2},{J_3},P}\right| \leq \prod_{i=1}^2\|b_i\|_{\BMO}^{k_i}\frac{|J_1|^{1 / 2}|J_2|^{1 / 2}|J_3|^{1 / 2}}{|P|^{2-\eta}}  \label{key2} 
\end{align}

For $x \in J_3$, we have that $\left|x-c_{J_3}\right| \leq \ell(J_3) / 2$. For $x \in J_3$, $y_1 \in J_1$, and $y_2 \in J_2$, by $\ell(J_2) \ge 2^{-1}\ell(J_3)$, we have
\begin{align*}
2^\lambda\left|x-c_{J_3}\right| \le 2^{\lambda-1}{\ell(J_3)} \le \ell(J_3)^\lambda \ell(J_2)^{1-\lambda} <& \max\{d(J_3, J_1), d(J_3, J_1)\} \\
\le& \max\{|x-y|,|x-z|\} .
\end{align*}
It derives from \eqref{eq4.2} and smoothness condition \eqref{S2} that 
\begin{align}
\left|\left\langle {T_{\eta}^{\bf b,k}}\left(h_{J_1}, h_{J_2}^0\right), h_{{J_3}}\right\rangle\right| & \lesssim\sum_{\bf{t}}\left\|(b_1-b_{1,P})^{t_1}h_{J_1}\right\|_{L^1}\left\|(b_2-b_{2,P})^{t_2}h_{J_2}^0\right\|_{L^1} \notag\\ &\quad\cdot\left\|\prod_{i=1}^2(b_i-b_{i,P})^{k_i-t_i}h_{J_3}\right\|_{L^1} \frac{\ell(J_3)^\delta}{\max (d(J_3,J_1), d(J_3,J_2))^{ n(2-\eta)+\delta}} \notag\\
& \lesssim\prod_{i=1}^2\|b_i\|_{\BMO}^{k_i}|J_1|^{1/2}|J_2|^{1/2}|J_3|^{1/2} \frac{\ell(J_3)^\delta}{\left(\ell(J_3)^\lambda \ell(P)^{1-\lambda}\right)^{ n(2-\eta)+\delta}} \label{key3}\\
& =\prod_{i=1}^2\|b_i\|_{\BMO}^{k_i}\frac{|J_1|^{1/2}|J_2|^{1/2}|J_3|^{1/2}}{|P|^{2-\eta}}\left(\frac{\ell(J_3)}{\ell(P)}\right)^{\delta-\lambda( n(2-\eta)+\delta)} \notag\\
& =\prod_{i=1}^2\|b_i\|_{\BMO}^{k_i}\frac{|J_1|^{1/2}|J_2|^{1/2}|J_3|^{1/2}}{|P|^{2-\eta}}\left(\frac{\ell(J_3)}{\ell(P)}\right)^{\delta / 2} .\notag
\end{align}
where cube $P$ is from Lemma \ref{le4.1}, and noting that the fact that $\left\|h_J\right\|_{r} \le |J|^{\frac{1}{r}-\frac{1}{2}}$, for any $r \in [1,\infty]$, we used it for $r=\infty$.

This deduces that \eqref{key2} is valid.

\begin{remark}\label{remark:BMO}
Here we give our ideas for realizing \eqref{key3}. 
For different symbols $b_i$, after separating the Haar function $h_J$, we can first adopt Holder inequality, so that it can appear at most finite multiple multiplication of the following expressions:
\[{\left\| {{b_1} - {b_{1,P}}} \right\|_{{L^s}\left( {{J_i}} \right)}},\]
where $P$ is from Lemma \ref{le4.1}. This implies that ${J_1} \cup{J_2}  \cup {J_3} \subseteq P$, and $\ell(P) \approx \ell(J_i)$ for $i=1,2,3$.
By the following Lemma \ref{cube:BMO}, we can deduce that for $s \in [ 1,\infty),$
\[{\left\langle |{{b_1} - {b_{1,R}}}| \right\rangle _{s,{J_i}}} \le {\left\| {{b_1}} \right\|_{BMO}}.\]
This can make us realize \eqref{key3}.
\end{remark}

\begin{lemma}\label{cube:BMO}
Let \( f \in \BMO(\mathbb{R}^n) \). Suppose \( 1 \leq p < \infty \), and let \( Q_1 = Q(x_1, l_1) \subseteq Q_2 = Q(x_2, l_2) \) be two cubes. Then:
\[
\left( \frac{1}{|Q_1|} \int_{Q_1} |f(y) - f_{Q_2}|^p \, dy \right)^{1/p} \leq C \left( 1 + \ln \left|\frac{l_2}{l_1}\right| \right) \| f \|_{\BMO},
\]
where \( C > 0 \) is independent of \( f \), \( x_1 \), \( x_2 \), \( l_1 \), and \( l_2 \).
Moreover, If $Q_1$ and $Q_2$ are exchanged, the conclusion still holds.
\end{lemma}
\begin{proof}[Proof of Lemma \ref{cube:BMO}:]
This proof is obvious, and one can refer to \cite[exercise 3.1.5, 3.1.6]{250} and \cite[Lemma 2.1]{Lin08}.
\end{proof}

{\emph{{\textbf{\textsf Case 2: adjacent sum}}}

Now we consider the adjacent sum.
\begin{align*}
 \mathbb{I}_{\omega,1,2}:
=\sum_{\substack{ {{J_3} \in \mathscr{D}_{\text {good }}}, 
{J_1},  {J_2} \in \mathscr{D}: \\ 
\ell(J_3) \leq \ell(J_1)=2 \ell({J_2}) \\
\max (d({J_3}, {J_1}), d({J_3}, {J_2})) \leq \ell({J_3})^\lambda \ell({J_2})^{1-\lambda} \\
{J_3} \cap {J_1}=\emptyset \text { or } {J_3}={J_1} \text { or } {J_3} \cap {J_2}=\emptyset
}} \left\langle {T_{\eta}^{\bf b,k}}\left(h_{J_1}, h_{J_2}^0\right), h_{{J_3}}\right\rangle\left\langle f, h_{J_1}\right\rangle\left\langle g, h_{J_2}^0\right\rangle\left\langle h, h_{J_3}\right\rangle .
\end{align*}

Here we will illustrate that $\ell(J_1) \le 2^r\ell(J_3)$ is valid in the limit of the sum in $ \mathbb{I}_{\omega,1,2}$.

By the goodness of $J_3$, observing the last limit of the sum, assuming that $\ell(J_1) > 2^r\ell(J_3)$, we consider three situation:

(1) $J_3=J_1$: $\ell(J_1) \le 2^r\ell(J_3)$ is valid.

(2) $J_3 \cap J_1 = \emptyset$: assume that $\ell(J_1)>2^r \ell(J_3)$, then the goodness deduce that $d(J_3,J_1)>\ell(J_3)^\lambda \ell(J_1)^{1 - \lambda} >\ell(J_3)^\lambda \ell(J_2)^{1 - \lambda}$.
This contradicts the limit of the sum. That is, $\ell(J_1) \le 2^r\ell(J_3)$ is valid.

(3) $J_3 \cap J_2 = \emptyset$: assuming that $\ell(J_1) > 2^r\ell(J_3)$, we have that $2\ell(J_2)=\ell(J_1) > 2^r\ell(J_3)$. This imlies that $\ell(J_2) \ge 2^r\ell(J_3)$. The goodness of $J_3$ gives that $d(J_3,J_2)>\ell(J_3)^\lambda \ell(J_2)^{1 - \lambda}$. This contradicts the limit of the sum. That is, $\ell(J_1) \le 2^r\ell(J_3)$ is valid.

Now, we get
\begin{align*}
 \mathbb{I}_{\omega,1,2}:
=\sum_{\substack{ {{J_3} \in \mathscr{D}_{\text {good }}}, 
{J_1},  {J_2} \in \mathscr{D}: \\ 
\ell(J_3) \leq \ell(J_1)=2 \ell({J_2}) \le 2^r\ell(J_3) \\
\max (d({J_3}, {J_1}), d({J_3}, {J_2})) \leq \ell({J_3})^\lambda \ell({J_2})^{1-\lambda} \\
{J_3} \cap {J_1}=\emptyset \text { or } {J_3}={J_1} \text { or } {J_3} \cap {J_2}=\emptyset
}} \left\langle {T_{\eta}^{\bf b,k}}\left(h_{J_1}, h_{J_2}^0\right), h_{{J_3}}\right\rangle\left\langle f, h_{J_1}\right\rangle\left\langle g, h_{J_2}^0\right\rangle\left\langle h, h_{J_3}\right\rangle .
\end{align*}

\begin{lemma}\label{key4} 
Let $J_1$, $J_2$, and $J_3$ are satisfied the limit of the sum in $ \mathbb{I}_{\omega,1,2}$, then there is a cube $P\in\mathscr{D}$ with $J_1\cup J_2\cup 
J_3\subseteq P$ and $\ell(P)\leq 2^r\ell(J_3)$.
\end{lemma}
\begin{proof} Let $P = {\ell(J_3^{(r)})}$. Then $\ell(P)=2^rJ_3$, and we have $\ell(P)\geq \ell(J_1)>\ell(J_2)$. It is enough to show that $J_1\cap P\neq\emptyset$ and $J_2\cap P\neq\emptyset$. Similar to before, assume $J_1\subseteq P^c$ or $J_2\subseteq P^c$, then we
\[l{(J_3^{})^\lambda }\ell {({J_2})^{1 - \lambda }} \ge \max (d({J_3},{J_1}),d({J_3},{J_2})) \ge d({J_3},{P^c})>l{(J_3^{})^\lambda }\ell {(P)^{1 - \lambda }}.\]

This leads to $\ell(J_2)>\ell(P)$, which is a contradiction.
\end{proof}

We can now write
$$
\mathbb{I}_{\omega,1,2}=\sum_{k=0}^r \sum_{i=0}^k \sum_{\substack{P, {J_1}, {J_2} \in \mathscr{D}, {J_3} \in \mathscr{D}_{\text {good }} : \\ \max (d(J_3,J_1), d(J_3,J_2)) \leq \ell(J_3)^\lambda \ell(J_2)^{1-\lambda} \\ J_3\cap J_1=\emptyset \text { or } J_3=J_1 \text { or } J_3\cap J_2=\emptyset \\ 2 \ell({J_2})=\ell({J_1})=2^{-i} \ell(P), \ell({J_3})=2^{-k} \ell(P) \\ J_1 \vee J_2 \vee J_3=P}}\left\langle {T_{\eta}^{\bf b,k}}\left(h_{J_1}, h_{J_2}^0\right), h_{{J_3}}\right\rangle\left\langle f, h_{J_1}\right\rangle\left\langle g, h_{J_2}^0\right\rangle\left\langle h, h_{J_3}\right\rangle .
$$

Notice that if $J_3\cap J_1=\emptyset$, then similar to the method of Remark \ref{remark:BMO}, and from the decomposition that $K(x,\vec y) = K(x,\vec y) - K({c_{{J_3}}},\vec y) + K({c_{{J_3}}},\vec y)$, it follows from Lemma \ref{key4}, the Smoothness condition, and the cancellation of $h_{J_3}$ that
\begin{align*}
&\left|\left\langle {T_{\eta}^{\bf b,k}}\left(h_{J_1}, h_{J_2}^0\right), h_{{J_3}}\right\rangle\right| \\& \lesssim\sum_{\bf{t}}|J_1|^{-1 / 2}|J_2|^{-1 / 2}|J_3|^{-1 / 2}  \\
&\quad \cdot\int_{J_3}\int_{J_2} \int_{J_1} \frac{\ell(J_3)^\delta}{(\ell(P)/2^r)^{n(2-\eta)+\delta}} \prod\limits_{i=1}^{2} \left|(b_i(x) - b_{i,P})^{k_i - t_i} (b_i(y_i) - b_{i,P})^{t_i} \right|d y_1 dy_2d x \\
& \lesssim\prod_{i=1}^2\|b_i\|_{\BMO}^{k_i}|J_1|^{(1 / 2)+\eta}|J_2|^{-1 / 2}|J_3|^{-1 / 2} 
\\&\lesssim\prod_{i=1}^2\|b_i\|_{\BMO}^{k_i}\frac{|J_1|^{1/2}|J_2|^{1/2}|J_3|^{1/2}}{|P|^{2-\eta}}\left(\frac{\ell(J_3)}{\ell(P)}\right)^{\delta/ 2}
\end{align*}
The similar estimates will appear again in the situation: $J_3\cap J_2 = \emptyset$.

Finally, we consider the situation: $J_3 = J_1$ and $J_2 \in\mathrm{ch}(J_3)$. 
Note that \[
h_J = \sum_{J' \in ch(J)} a_{J'} 1_{J'}, \quad h_J^0 = |J|^{-1/2} 1_J,
\]
where \( a_{J'} \) are normalization coefficients (typically \( \pm |J|^{-1/2} \)).

It follows immediately that
$$
\left|\left\langle {T_{\eta}^{\bf b,k}}\left(h_{J_3}, h_{J_2}^0\right), h_{J_3}\right\rangle\right| \lesssim|J_3|^{-3 / 2} \sum_{{J_3}^{\prime}, {J_3}^{\prime \prime} \in ch(J_3)}\left|\left\langle {T_{\eta}^{\bf b,k}}\left(\chi_{{J_3}^{\prime}}, \chi_{J_2}\right), \chi_{{J_3}^{\prime \prime}}\right\rangle\right|.
$$
We notice that $\#ch(J_3)\approx1$.

For ${J_3}'\neq J_2$ or ${J_3}''\neq J_2$, then by the size condition and similar to the method of Remark \ref{remark:BMO}, 
\[
\left|\left\langle {T_{\eta}^{\bf b,k}}(\chi_{{J_3}'},\chi_{J_2}),\chi_{J_3''}\right\rangle\right| \lesssim\prod_{i = 1}^2\|b_i\|_{\mathrm{BMO}}^{k_i}|J_3|^{\eta + 1}.
\]
For $J_3' = J_3'' = J_2$, using the weak boundedness property, we get
\[
\left|\left\langle {T_{\eta}^{\bf b,k}}(\chi_{J_2},\chi_{J_2}),\chi_{J_2}\right\rangle\right| \le {\left\| {T_{\eta}^{\bf{b,k}}} \right\|_{WB{P_{\eta,{\bf b}}^{\bf k} }(\mathbb{V})}}\prod_{i = 1}^2\|b_i\|_{\mathrm{BMO}}^{k_i}|J_3|^{\eta + 1}.
\]
Collecting these, we obtain the following estimates
$$
\left|\left\langle {T_{\eta}^{\bf b,k}}\left(h_{J_3}, h_{J_2}^0\right), h_{J_3}\right\rangle\right| \lesssim\prod_{i=1}^2\|b_i\|_{\BMO}^{k_i}|J_3|^{\eta-(1 / 2)} \approx \frac{|J_1|^{1/2}|J_2|^{1/2}|J_3|^{1/2}}{|P|^{2-\eta}}\left(\frac{\ell(J_3)}{\ell(P)}\right)^{\delta / 2},
$$
where $P$ comes from Lemma \ref{key4}.

From the above, we can write
\[
\mathbb{I}_{\omega,1,2} = \sum_{k = 0}^r\sum_{i = 0}^k 2^{-k\delta/2}\left\langle \mathbb{S}_{\eta,{\bf b}}^{i, i + 1, k}(f, g), h\right\rangle
\]
where $\mathbb{S}_{\eta,{\bf b}}^{i, i + 1, k}$ are cancellative bilinear shifts.

}

{\emph{{\textbf{\textsf Case 3: nested sum}}}

In this part, we will address 
\begin{align*}
\mathbb{I}_{\omega,1,3} :
& =\sum_{\substack{J_2 \in \mathcal{D}, {J_3} \in \mathscr{D}_{\text {good }} :\\
J_3 \subseteq J_2}} \left\langle {T_{\eta}^{\bf b,k}}\left(\Delta_{{J_2}^{\dagger}} f_1, \chi_{J_2}\right), \Delta_{J_3} h\right\rangle\langle f_2\rangle_{J_2}.
\end{align*}
We firstly decompose
\begin{align*}
\mathbb{I}_{\omega,1,3}=  \mathbb{I}_{\omega,1,3}(1)+\mathbb{I}_{\omega,1,3}(2),
\end{align*}
where
\begin{align*}
\mathbb{I}_{\omega,1,3}(1):=&\sum_{\substack{J_2 \in \mathcal{D}, {J_3} \in \mathscr{D}_{\text {good }}:\\
J_3 \subseteq J_2}}\langle f_2\rangle_{J_2} \left\langle {T_{\eta}^{\bf b,k}}\left(\chi_{J_2^{\dagger}\backslash{J_2}^c}\left(\Delta_{{J_2}^{\dagger}} f_1-\left\langle\Delta_{{J_2}^{\dagger}} f_1\right\rangle_{J_2}\right), \chi_{J_2}\right), \Delta_{J_3} h\right\rangle\\
 &-\langle f_2\rangle_{J_2} \left\langle\Delta_{{J_2}^{\dagger}} f_1\right\rangle_{J_2}\left\langle {T_{\eta}^{\bf b,k}}\left(1,\chi_{{J_2}^c}\right), \Delta_{J_3} h\right\rangle,
\end{align*}
and 
\begin{align*}
\mathbb{I}_{\omega,1,3}(2):=\sum_{\substack{J_2 \in \mathcal{D}, {J_3} \in \mathscr{D}_{\text {good }} \\
J_3 \subseteq J_2}}\langle f_2\rangle_{J_2} \left\langle\Delta_{{J_2}^{\dagger}} f_1\right\rangle_{J_2}\left\langle {T_{\eta}^{\bf b,k}}(1,1), \Delta_{J_3} h\right\rangle.
\end{align*}

{\emph{{\textbf{\textsf Case 3-1:}}}

Now, we consider $\mathbb{I}_{\omega,1,3}(1)$.
\begin{align*}
\mathbb{I}_{\omega,1,3}(1)= & \sum_{\substack{J_2 \in \mathcal{D}, {J_3} \in \mathscr{D}_{\text {good }}: \\ J_3 \subseteq J_2}}
|J_2|^{-1 / 2}\left[\left\langle {T_{\eta}^{\bf b,k}}\left(\mathbb{H}_{J_2}, \chi_{J_2}\right), h_{J_3}\right\rangle\right. \left.-\left\langle h_{{J_2}^{\dagger}}\right\rangle_{J_2}\left\langle {T_{\eta}^{\bf b,k}}\left(1,\chi_{{J_2}^c}\right), h_{J_3}\right\rangle\right]\\
&\cdot\left\langle f_1, h_{{J_2}^{\dagger}}\right\rangle\left\langle f_2, h_{J_2}^0\right\rangle\left\langle h, h_{J_3}\right\rangle,
\end{align*}
where $\mathbb{H}_{J_2}:=\left(h_{{J_2}^{\dagger}}-\left\langle h_{{J_2}^{\dagger}}\right\rangle_{J_2}\right)\chi_{J_2^{\dagger}\backslash{J_2}^c}$ satisfies $\left\|\mathbb{H}_{J_2}\right\|_{r} \lesssim|J_2|^{1/r-1 / 2}$ for any $r \in [1,\infty]$.

We now define
$\beta_{J_2^{\dagger},J_2,J_3,J_2}:=|J_2|^{-1 / 2}\left[\left\langle {T_{\eta}^{\bf b,k}}\left(\mathbb{H}_{J_2}, \chi_{J_2}\right), h_{J_3}\right\rangle\right. \left.-\left\langle h_{{J_2}^{\dagger}}\right\rangle_{J_2}\left\langle {T_{\eta}^{\bf b,k}}\left(1,\chi_{{J_2}^c}\right), h_{J_3}\right\rangle\right].$
As \eqref{key:shift}, it suffices to verify that 
\begin{align}
\left|\beta_{J_2^{\dagger},J_2,J_3,J_2} \right|\lesssim \prod\limits_{i = 1}^2 {\left\| {{b_i}} \right\|_{\BMO}^{{k_i}}} {\left| {{J_2}} \right|^{\eta  - 1}}\left| {{J_3}} \right|^{\frac{1}{2}}. \label{key5}
\end{align}
We claim that \eqref{key5} will follow from these fact that
\begin{align}
\left| \left\langle {T_{\eta}^{\bf b,k}}\left(\mathbb{H}_{J_2}, \chi_{J_2}\right), h_{J_3}\right\rangle \right| \lesssim& \prod\limits_{i = 1}^2 {\left\| {{b_i}} \right\|_{\BMO}^{{k_i}}} {\left| {{J_2}} \right|^{\eta  - \frac{1}{2}}}\left| {{J_3}} \right|^{\frac{1}{2}}; \label{key6} \\
\left| \left\langle {T_{\eta}^{\bf b,k}}\left(1, \chi_{{J_2}^c}\right), h_{J_3}\right\rangle\right| \lesssim& \prod\limits_{i = 1}^2 {\left\| {{b_i}} \right\|_{\BMO}^{{k_i}}} {\left| {{J_2}} \right|^{\eta  - 1}}\left| {{J_3}} \right|^{\frac{1}{2}}. \label{key7} 
\end{align}

For \eqref{key6}, we first estimate $\left|\left\langle {T_{\eta}^{\bf b,k}}\left(\mathbb{H}_{J_2}, \chi_{J_2}\right), h_{J_3}\right\rangle\right|$ by considering two situations. 

(1) For $\ell(J_3) \le \ell(J_2) <2^r \ell(J_3)$,
\begin{align*}
\left|\left\langle {T_{\eta}^{\bf b,k}}\left(\mathbb{H}_{J_2}, \chi_{J_2}\right), h_{J_3}\right\rangle\right| \le& \left|\left\langle {T_{\eta}^{\bf b,k}}\left(\chi_{3 J_2} \mathbb{H}_{J_2}, \chi_{J_2}\right), h_{J_3}\right\rangle\right|+\left|\left\langle {T_{\eta}^{\bf b,k}}\left(\chi_{\left(3_{J_2}\right)^c } \mathbb{H}_{J_2}, \chi_J\right), h_{J_3}\right\rangle\right|.
\end{align*}
We estimates the first term by the size condition, and the second term is bounded by the smoothness condition with the fact that for $x \in J_3$ and $y_1 \in (3J_2)^c$, $|x-y_1|\le \ell(J_2)$. 
Then $\eqref{key6}$ is valid.

(2) For $\ell(J_2) \ge 2^r \ell(J_3)$, due to $\gamma \in (0, \frac{1}{2}]$, it follows that $$|x-y_1| > d\left({J_3}, {{J_2}}^c\right) \geq \ell(J_3)^\lambda \ell(J_2)^{1-\lambda} \gtrsim \ell(J_3)^{\frac{1}{2}}\ell(J_2)^{\frac{1}{2}}.$$ 
From the decomposition that $K(x,\vec y) = K(x,\vec y) - K({c_{{J_3}}},\vec y) + K({c_{{J_3}}},\vec y)$, it derives from this estimate and smoothness condition that
\[
|K(x, y_1, y_2) - K(x_0, y_1, y_2)| \lesssim \frac{|x - x_0|^\delta}{|x_0 - y|^{2n + \delta}} \lesssim \frac{\ell(J_3)^\delta}{\left(\ell(J_3)^{1/2}\ell(J_2)^{1/2}\right)^{2n + \delta}}.
\]
Note that ${\left\| {{\mathbb{H}_{{J_2}}}} \right\|_{L^1}} \le {\left| {{J_2}} \right|^{ \frac{1}{2}}}$. By the method of Remark \ref{remark:BMO} and these estimates, $\eqref{key6}$ may derive from direct calculation.

Obivously, we can get \eqref{key7} by the same methods.  Therefore, we can write
$$
\mathbb{I}_{\omega,1,3}(1)= \sum_{k=1}^{\infty} 2^{-\delta k / 2}\left\langle \mathbb{S}_{\eta}^{0,1, k}(f_1, f_2), h\right\rangle
$$

{\emph{{\textbf{\textsf Case 3-2:}}}

Now, we are ready to handle all the terms that can be represented by  paraproduct. That is to say, $\mathbb{I}_{\omega,1,3}(2)$ and $\mathbb{I}_{\omega,2,3}(2)$ will be synthesized and processed simultaneously,
where
$$
\mathbb{I}_{\omega,1,3}(2)=\sum_{\substack{J_2 \in \mathscr{D}, {J_3} \in \mathscr{D}_{\text {good }} \\ J_3 \subseteq J_2}}\left\langle {T_{\eta}^{\bf b,k}}(1,1), \Delta_{J_3} h\right\rangle\left\langle\Delta_{{J_2}^{\dagger}} f_1\right\rangle_{J_2}\langle f_2\rangle_{J_2},
$$
and
$$
\mathbb{I}_{\omega,2,3}(2)=\sum_{\substack{J_2 \in \mathscr{D}, {J_3} \in \mathscr{D}_{\text {good }} \\ J_3 \subseteq J_2}}\left\langle {T_{\eta}^{\bf b,k}}(1,1), \Delta_{J_3} h\right\rangle\langle f_1\rangle_{{J_2}^{\dagger}}\left\langle\Delta_{{J_2}^{\dagger}} f_2\right\rangle_{J_2}.
$$
Note that the key cancellation
$$
\left\langle\Delta_{{J_2}^{\dagger}} f_1\right\rangle_{J_2}\langle f_2\rangle_{J_2}+\langle f_1\rangle_{{J_2}^{\dagger}}\left\langle\Delta_{{J_2}^{\dagger}} f_2\right\rangle_{J_2}=\langle f_1\rangle_{J_2}\langle f_2\rangle_{J_2}-\langle f_1\rangle_{{J_2}^{\dagger}}\langle f_2\rangle_{{J_2}^{\dagger}}.
$$
Then, we obtain
\[
\mathbb{I}_{\omega,1,3}(2) + \mathbb{I}_{\omega,2,3}(2) 
= \sum_{J_3 \in \mathscr{D}_{\text{good}}} \left\langle T_{\eta,{\bf b}}^{\mathbf{k}}(1,1), \Delta_{J_3} h \right\rangle \langle f \rangle_{J_3} \langle g \rangle_{J_3}
\]

Now, we define $\beta_{\eta,J_3}$ as follows. 
When $J_3$ is good, 
$
\beta_{\eta,J_3}:=\frac{\left\langle {T_{\eta}^{\bf b,k}}(1,1), h_{J_3}\right\rangle}{\mathcal{N}};
$
otherwise, $\beta_{\eta,J_3}:=0.$

It follows from Proposition \ref{norm:FBMO} that
\[ \infty>\mathcal{N}\ge\left\| {T_{\eta}^{\bf b,k}}(1,1) \right\|_{BM{O_{2,\eta }}}^{} = \mathop {\sup }\limits_{R \in \D} {\left( {{{\left| R \right|}^{2\eta  - 1}}{{\sum\limits_{Q \subseteq R} {\left| {\left\langle {{T_{\eta}^{\bf b,k}}(1,1),{h_Q}} \right\rangle } \right|} }^2}} \right)^{\frac{1}{2}}}.\]
Then, we have that $\beta_{\eta,J_3} \le 1,$ which ensures it satisfy \eqref{kuohao3.1}.
These can make us write
$$
\mathbb{I}_{\omega,1,3}(2)+\mathbb{I}_{\omega,2,3}(2)=C \mathcal{N}\cdot\left\langle\Pi_{\beta_\eta}(f, g), h\right\rangle .
$$


Collecting all the above representations, we can obtain Theorem \ref{thm:DT}.

\subsection{Sparse bounds for dyadic model form}\label{dyadic.model.form}

\begin{proof}[Proof of Theorem \ref{thm:d.s.d.}]
We only prove for case $m=2$. We begin by selecting a dyadic cube \( P_0 \in \D \) that contains the supports of the nonnegative functions \( f_1, f_2, \) and \( h\). Subsequently, 
define the filter dyadic cubes set
\[\Omega  = \left\{ {P \in \D \text{ is maximal} :P \subseteq {P_0},\frac{{{{\left\langle {{f_1}} \right\rangle }_P}}}{{{{\left\langle {{f_1}} \right\rangle }_{{P_0}}}}} + \frac{{{{\left\langle {{f_2}} \right\rangle }_P}}}{{{{\left\langle {{f_2}} \right\rangle }_{{P_0}}}}} + \frac{{{{\left\langle h \right\rangle }_P}}}{{{{\left\langle h \right\rangle }_{{P_0}}}}} > {C_0}\left( \delta  \right)} \right\}.\]

We can always take \( C_0(\delta) \) sufficiently large, such that
\[ \sum_{P \in {\Omega}} |P| \leq (1-\delta)|P_0|. \]

We regard $P_0$ as the first generation cube in $\S$, and define set \( E_{P_0} := P_0 \setminus \bigcup_{P \in \Omega} P \).

Let  \(\mathscr{H} =\mathscr{H}(P_0) := \{ P \in \D : P \subseteq P_0,\ P \nsubseteq Q,\ \text{for every } Q \in \Omega \} \).
From the fact that $f_1$, $f_2$, and $h$ are supported in $P_0$, it follows that
\begin{align}
\mathbb{F}_{\eta,\nu}^\mathscr{D}(f_1, f_2, h) 
=& \sum_{\substack{P \in \D \\ P_0 \subsetneq P}} {\mathcal{F}_{\eta,\nu,P}^\D}(f_1,f_2,h) + \mathbb{F}_{\eta,\nu}^\mathscr{H}(f_1,f_2,h) + \sum_{P \in \Omega} \mathbb{F}_{\eta,\nu}^{\mathscr{D}(P)}(f_1 \chi_P,f_2 \chi_P,h\chi_P), \notag \\
=&\sum_{\substack{P \in \D \\ P_0 \subsetneq P}} {\mathcal{F}_{\eta,\nu,P}^\D}(f_1,f_2,h) + \mathbb{F}_{\eta,\nu}^\mathscr{H}(f_1,f_2,h) + \sum_{P \in \Omega} \mathbb{F}_{\eta,\nu}^{\mathscr{D}(P)}(f_1 \chi_P,f_2 \chi_P,h\chi_P), \label{kuohao5.3}
\end{align}
It derives from \eqref{Kernel.1} that
$$
\sum_{\substack{P \in \D \\ P_0 \subsetneq P}} {\mathcal{F}_{\eta,\nu,P}^\D}\left(f_1, f_2, h\right) \leq \sum_{\substack{P \in \D \\ P_0 \subsetneq P }} \frac{\left\|f_1\right\|_{L^1}\left\|f_2\right\|_{L^1}\left\|h\right\|_{L^1}}{|P|^{2-\eta}} \approx|P_0|^{\eta+1} \prod_{j=1}^3\langle | f_j| \rangle_{P_0},
$$
where we used the fact that if $P$ is an ancestor of $P_0$, then $|P|=2^{k_{P}n}|P_0|.$

Now, we claim that 
\begin{align}\label{kuohao5.4}
|\mathbb{F}_{\eta,\nu}^\mathscr{H}(f_1,f_2,h)| \lesssim_\delta(\mathcal{B}+\nu)\mu(P_0)^{\eta+1}\prod_j \langle |f_j| \rangle_{P_0}
\end{align}

Once \eqref{kuohao5.3} and \eqref{kuohao5.4} are both valid, the desired will follow from the iterating the third term of \eqref{kuohao5.3}.

{\emph{{\textbf{\textsf Proof of \eqref{kuohao5.4}:}}}


We establish the estimate via a Calder\'{o}n--Zygmund decomposition of the functions $f_1,f_2$, and $f_3:=h$ relative to the set \(\Omega\). For \( j = 1, 2, 3 \), we write  
\[
f_j = g_j + b_j \quad \text{where} \quad b_j := \sum_{P \in \Omega} b_{j,P} \quad \text{and} \quad b_{j,P} := \left( f_j - \langle f_j \rangle_P \right) \mathbf{1}_P.
\]

The components \( g_j \) and \( b_j \) satisfy:  
\begin{align}
{\left\langle {\left| {{g_j}} \right|} \right\rangle _{1,{P_0}}} \le {\left\langle {{f_j}}  \right\rangle _{1,{P_0}}}, &\quad \|g_j\|_{L^\infty} \lesssim_{\delta} \langle f_j \rangle_{P_0}, \label{eq:gj-bound}\\
\int_P b_{j,P} \, dx &= 0, \label{eq:bj-zero-mean} \\
{\left\langle {|{b_{j,P}}|} \right\rangle _P} &\le {\left\langle {{f_j}} \right\rangle _{{P_0}}}. \label{eq:bj-L1}
\end{align}

Next, we decompose the trilinear form \( \mathbb{F}_{\eta,\nu}^\mathscr{H}(f_1, f_2, f_3) \) into eight components by expanding \( f_j = g_j + b_j \):  
\[
\mathbb{F}_{\eta,\nu}^\mathscr{H}(f_1, f_2, f_3) = 
\mathbb{F}_{\eta,\nu}^\mathscr{H}(g_1, g_2, g_3) + 
\mathbb{F}_{\eta,\nu}^\mathscr{H}(b_1, b_2, b_3) + 
\mathbb{F}_{\eta,\nu}^\mathscr{H}(g_1, g_2, b_3) + 
\mathbb{F}_{\eta,\nu}^\mathscr{H}(g_1, b_2, g_3) 
\]
\[
+ \mathbb{F}_{\eta,\nu}^\mathscr{H}(b_1, g_2, g_3) + 
\mathbb{F}_{\eta,\nu}^\mathscr{H}(b_1, b_2, g_3) + 
\mathbb{F}_{\eta,\nu}^\mathscr{H}(b_1, g_2, b_3) + 
\mathbb{F}_{\eta,\nu}^\mathscr{H}(g_1, b_2, b_3).
\]

By \eqref{F.bounded} and  \eqref{eq:gj-bound},
$$
\left|\mathbb{F}_{\eta,\nu}^\mathscr{H}\left(g_1, g_2, g_3\right)\right| \leq \left\| \mathbb{F}_{\eta,\nu}^{\mathscr{P}} \right\|\left\|g_1\right\|_{L^{r_1}}\left\|g_2\right\|_{L^{r_2}}\left\|g_3\right\|_{L^{(\widetilde{r})^{\prime}}} \lesssim_\delta \left\| \mathbb{F}_{\eta,\nu}^{\mathscr{P}} \right\|\left|P_0\right|^{\eta+1}  \prod_j\langle | f_j| \rangle_{P_0}
$$

When $\nu=0$, the last property of Definition \ref{FDMF} implies the remaining seven terms are all 0.

When $\nu \ge 1$, it suffices to consider at least one bad function existed, i.e., we only estimate $\mathbb{F}_{\eta,\nu}^\mathscr{H}\left(b_1, F_2, F_3 \right)$, where $F_j \in \{g_j, b_j\}$, $j=2,3$. We decompose $\mathscr{H}$ into $\nu$ descendant, each of which, denoted by $\mathscr{H}^{\prime}$, satisfies $\ell\left(Q_1\right) \geq 2^\nu \ell\left(Q_2\right)$ whenever $Q_1, Q_2 \in \mathscr{H}^{\prime}, Q_2 \subsetneq Q_1$. To finish our proof, it is sufficient to prove
\begin{align}
\left|\mathbb{F}_{\eta,\nu}^{\mathscr{H}'}\left(b_1, F_2, F_3\right)\right| \lesssim_\delta|P_0|^{\eta+1} \prod_j\langle f_j \rangle_{P_0}.
\end{align}

Define $\mathscr{R}(P):= \left\{ {J \in \H':P \subsetneq J \text{ with } {\mathcal{F}_{\eta,\nu,J}^\D}({b_{1,P}},{F_2},{F_3}) \ne 0, \text{ for every }P \in \Omega} \right\}.$
It follows from the last property Definition \ref{FDMF}, the structure of $\mathscr{H}'$, and \eqref{eq:bj-zero-mean} that $\mathscr{R}(P) \neq \emptyset.$
By \eqref{Kernel.1}, we get
\begin{align}
\left|\mathbb{F}_{\eta,\nu}^{\mathscr{H}'}\left(b_1, F_2, F_3\right)\right| & \leq \sum_{J \in \mathscr{H}'} \sum_{\substack{P \in \Omega \\
J \in \mathscr{R}(P)}}\left|{\mathcal{F}_{\eta,\nu,R}^\D}\left(b_{1, P}, F_2, F_3\right)\right| \nonumber\\
& \leq \sum_{J \in \mathscr{H}^{\prime}} \sum_{\substack{P \in \Omega \\
J \in \mathscr{R}(P)}} \frac{\left\|b_{1, P}\right\|_{L^1}\left\|F_2 \chi_J\right\|_{L^1}\left\|F_3 \chi_J\right\|_{L^1}}{|J|^{2-\eta}},\label{kuohao5.6}
\end{align}
Once we have proven that
\begin{align}
{\left\langle {\left| {{F_j}} \right|} \right\rangle _J} \lesssim_\delta {\left\langle {\left| {{f_j}} \right|} \right\rangle _{{P_0}}},\label{ine.3}
\end{align}
then collecting \eqref{ine.3} and \eqref{kuohao5.6}, we have 

\begin{align*}
\left|\mathbb{F}_{\eta,\nu}^{\mathscr{H}'}\left(b_1, F_2, F_3\right)\right| & \lesssim_\delta \sum_{R \in \mathscr{H}^{\prime}} \sum_{\substack{P \in \Omega \\
J \in \mathscr{R}(P)}}\left\|b_{1, P}\right\|_{L^1}\langle | f_2| \rangle_{P_0}\langle | f_3| \rangle_{P_0} \\
& \lesssim_\delta|P_0|^{\eta+1}\langle | f_1| \rangle_{P_0}\langle | f_2| \rangle_{P_0}\langle | f_3| \rangle_{P_0} .
\end{align*}

Thus, it is needful for us to verify \eqref{ine.3}. In fact, on the one hand,
\[{\left\langle {\left| {{g_j}} \right|} \right\rangle _R} \le {\left\langle {\left| {{f_j}} \right|} \right\rangle _{{P_0}}}.\]
On the other hand, \eqref{eq:bj-L1} brings that 
$$
\left\|b_j \chi_J\right\|_{L^1}=\sum_{P \in \Omega: P \subseteq J}\left\|b_{j, P}\right\|_{L^1} \lesssim_\delta \sum_{P \in \Omega: P \subseteq J}|P|\langle f_j \rangle_{P_0} \leq\mu(J)\langle f_j \rangle_{P_0}
$$
Therefore, \eqref{ine.3}  derives from the two respects.
}\end{proof}

\section{\bf Multilinear off-diagonal extrapolation theorem}\label{Sec.extro}
Let $\eta \in [0,m)$ and $\frac{1}{\tilde{p}}=\sum_{i=1}^{m}\frac{1}{p_i} - \eta$.  For convenience, we introduce a new way to express the $\omega \in A_{(\vec{p},\tilde{p}), (\vec{r},s)}$ as follows.
\[
A_{(\vec{p},\tilde{p}), (\vec{r},s)}=\sup _Q\left(\fint_Q \omega^{{\delta_{m+1}}} d x\right)^{\frac{1}{\delta_{m+1}}} \prod_{i=1}^m\left(\fint_Q \omega_i^{-{\delta_i}} d x\right)^{\frac{1}{\delta_i}}<\infty,
\]
where
\begin{align*}
	\frac{1}{\delta_i}=\frac{1}{r_i}-\frac{1}{p_i}, \quad i=1, \ldots, m, \quad \text{and} \quad \frac{1}{\delta_{m+1}} = \frac{1}{\tilde{p}} -\frac{1}{s} .
\end{align*}

Moreover, the condition \((\vec{r}, s) \preceq (\vec{p}, \tilde{p})\) implies \(\sum_{i=1}^{m-1} \frac{1}{\delta_i} \geq 0\). To proceed simplify the proof we introduce some notation and then classify the indices above.

 Let \(\vec{p} = (p_1, \ldots, p_m)\) and \(\vec{r} = (r_1, \ldots, r_m)\) be vectors with components satisfying \(1 \leq p_i < \infty\) and \(1 \leq r_i < \infty\) for all \(i = 1, \ldots, m\), respectively. Additionally, let \(s\) satisfy \(1 \leq s < \infty\) such that \((\vec{r}, s) \preceq (\vec{p}, \tilde{p})\). We then set
	\[
s':=r_{m+1},	\frac{1}{\gamma}:=\sum_{i=1}^{m+1} \frac{1}{r_i} = \sum_{i=1}^{m} \frac{1}{r_i} + \frac{1}{s'}, \quad \frac{1}{p_{m+1}}:=1  -\frac{1}{\tilde{p}},
	\]
    and
\begin{align}\label{zeta}
    	\sum_{i=1}^{m+1} \frac{1}{p_i}=1 + \eta \quad \text { and } \quad \sum_{i=1}^{m+1} \frac{1}{\delta_i}=\frac{1}{\gamma} - (1+\eta)=\frac{1-(1+\eta)\gamma}{\gamma}=\zeta.
\end{align}
Additionally, it follows from \((\vec{r}, s) \preceq (\vec{p}, \tilde{p})\) that $0<\gamma<1 + \eta$.

\subsection{Proof of Theorem \ref{thm:Ex_1}}~~

 The proof of Theorem \ref{thm:Ex_1} will consist of two steps.  In the first step we assume a special fact has shown and show how to iterate it to obtain the
desired result. The second step is devoted to proving the fact assumed in the first step.

\emph{{\textbf{\textsf Step 1: Extrapolation from special to all components.}}}

By Remark~\ref{li:remark:1.8}, when $p_i = r_i$ for some $i$, we may take $r_i \leq q_i$. Thus $\vec{r} \preceq \vec{q}$ with $r_j < q_j$ for all $j$ where $r_j < p_j$.

In view of Remark \ref{li:remark:1.8} if $p_i=r_i$ for some $1 \leq i \leq m$ we can allow $r_i \leq q_i$. This means that $\vec{r} \preceq \vec{q}$ with $r_j<q_j$ for those $j$ 's for which $r_j<p_j$. We assume a special fact that
\begin{align}\notag
&\vec{u}=\left(u_1, \ldots, u_{m-1}, u_m\right) \, \text{with } \vec{r} \preceq \vec{u}\, \\ \label{eq:iioo:4.5}
& \quad \quad \text{extrapolates to}\,
 \vec{t}=\left(u_1, \ldots, u_{m-1}, t_m\right) \text { whenever } (\vec{r},s) \preceq (\vec{t},\tilde{t}) \text { and } r_m<t_m \text {. }
\end{align}
By this we mean that if \eqref{eq:iioo:1.2} holds for the exponent $\vec{u}$ and for all $\vec{\omega} \in A_{(\vec{u},\tilde{u}), (\vec{r},s)}$, then \eqref{eq:iioo:1.2} holds for the exponent $\vec{t}$ and for all $\vec{\omega} \in A_{(\vec{t},\tilde{t}),( \vec{r},s)}$ with $r_m<t_m$. 

Note that in \eqref{eq:iioo:4.5}, the first $m-1$ components of $\vec{u}$ and $\vec{t}$ are frozen. By switching $f_i$ with $f_m$ for some fixed $1 \leq i \leq m-1$ and using the same notation, we may freeze all components except the $i$-th to obtain
\begin{align}\notag
&\vec{u}=\left(u_1, \ldots, u_{i-1}, u_i, u_{i+1}, \ldots, u_m\right) \, \text{with } (\vec{r},s) \preceq (\vec{u},\tilde{u})\, \\ \label{eq:iioo:4.6}
& \quad \quad \text{extrapolates to}\,
\vec{t}=\left(u_1, \ldots, u_{i-1}, t_i, u_{i+1}, \ldots, u_m\right) \text { whenever } (\vec{r},s) \preceq (\vec{t},\tilde{t}) \text { and } r_i<t_i \text {. }
\end{align}

Therefore, to prove the desired result, we iterate \eqref{eq:iioo:4.6}, verifying at each step that $(\vec{r},s) \preceq (\vec{t},\tilde{t})$ and $r_i < t_i$ hold. We only need to consider two cases as follows.

\emph{{\textbf{\textsf Case 1:}}}
Given $p_i \leq q_i$ ($1 \leq i \leq m$), we first prove that  
\begin{align}\label{eq:iioo:4.7}
    \vec{p}=\left(p_1, \ldots, p_m\right) \text { extrapolates to } \vec{u}=\left(u_1, u_2, \ldots, u_m\right)=\left(q_1, p_2, \ldots, p_m\right).
\end{align}

First, if $q_1 = p_1$, the result is immediate. Otherwise, when $p_1 < q_1$, we have
\begin{align*}
r_1 \leq p_1 < q_1 &= u_1, \\
r_i \leq p_i &= u_i \quad \text{for all } 2 \leq i \leq m.
\end{align*} Further, it follows from $(\vec{r},s) \preceq (\vec{q},\tilde{q})$ that
\[
\frac{1}{\tilde{u}} =\sum_{i=1}^m \frac{1}{u_i} - \eta=\frac{1}{q_1}+\sum_{i=2}^m \left(\frac{1}{p_i} - \eta_i\right) \geq \sum_{i=1}^m \frac{1}{q_i} - \eta =  \frac{1}{\tilde{q}} >\frac{1}{s}.
\]
Therefore, $(\vec{r},s) \preceq (\vec{u},\tilde{u})$ holds with $r_1 < u_1$, so \eqref{eq:iioo:4.6} applies for $i=1$, yielding \eqref{eq:iioo:4.7}.

Similar to  \eqref{eq:iioo:4.7}, we can extrapolate to \(\vec{t} = (t_1, t_2, t_3, \ldots, t_m) = (q_1, q_2, p_3, \ldots, p_m)\). If \(q_2 \neq p_2\), we obtain
\[
r_1 \leq t_1, \quad r_2 < t_2, \quad \text{and} \quad r_i \leq t_i \;\; (i \geq 3).
\] Thus, it follows from $(\vec{r},s) \preceq (\vec{q},\tilde{q})$ that
\[
\frac{1}{\tilde{t}}=\sum_{i=1}^m \frac{1}{t_i}-\eta=\frac{1}{q_1}+\frac{1}{q_2}+\sum_{i=3}^m \frac{1}{p_i} - \eta \geq \sum_{i=1}^m \left(\frac{1}{q_i} - \eta_i\right)=\frac{1}{\tilde{q}}>\frac{1}{s}.
\]
Last, $(\vec{r},s) \preceq (\vec{t},\tilde{t})$ with $r_2 < t_2$, so \eqref{eq:iioo:4.6} applies to $i=2$, yielding the desired estimates for $\vec{t} = (q_1, q_2, p_3, \ldots, p_m)$. We iteratively apply this process for $i=3, \ldots, m$, and in the final step transition from $(q_1, \ldots, q_{m-1}, p_m)$ to $(q_1, \ldots, q_m)$. This completes the proof for this case.

\emph{{\textbf{\textsf Case 2:}}}
For some $i \in \{1,\ldots,m\}$, such that $p_i > q_i$. By reordering if necessary, we may assume
\begin{equation*}
\begin{cases}
p_i > q_i & \text{for } 1 \leq i \leq i_0, \\
p_i \leq q_i & \text{for } i_0+1 \leq i \leq m,
\end{cases}
\end{equation*}
where $1 \leq i_0 \leq m$ (with the case $i_0 = m$ corresponding to $p_i > q_i$ for all $1 \leq i \leq m$).

We proceed iteratively as before using \eqref{eq:iioo:4.6}. At the $i$-th step, we transition from $\vec{u} = (q_1, \ldots, q_{i-1}, p_i, p_{i+1}, \ldots, p_m)$ to $\vec{t} = (q_1, \ldots, q_{i-1}, q_i, p_{i+1}, \ldots, p_m)$. Without loss of generality, we may assume $p_i \neq q_i$ (otherwise the step is trivial). 
Observe that since $r_j \leq \min(p_j, q_j)$ for all $j$, we have $r_j \leq t_j$ for each $1 \leq j \leq m$. To verify that \eqref{eq:iioo:4.6} is applicable, we consider two cases:

First, consider $1 \leq i \leq i_0$. Since $r_i \leq q_i < p_i$, we have $p_i \neq r_i$, and thus by Remark~\ref{li:remark:1.8} it follows that $r_i < q_i$. Furthermore, for these indices $i$,
\begin{align}
    \frac{1}{\tilde{t}}=\sum_{j=1}^m \left(\frac{1}{t_j} -\eta_j\right) =\sum_{j=1}^i \frac{1}{q_j}+\sum_{j=i+1}^m \left(\frac{1}{p_j} - \eta_j\right)>\sum_{j=1}^m \left(\frac{1}{p_j} - \eta_j\right)=\frac{1}{\tilde{p}}>\frac{1}{s}.
\end{align}
Altogether, we have seen that $\vec{r} \preceq \vec{s}$ and $r_i<q_i$. Thus we can invoke \eqref{eq:iioo:4.6}.
Consider next the case $i_0+1 \leq i \leq m$ (if $i_0=m$ this case is vacuous). In this scenario, $r_i \leq p_i<q_i$ (recall that we have disregarded the trivial case $q_i=p_i$ ). In addition, since $i_0+1 \leq i \leq m$,
\begin{align}
    \frac{1}{\tilde{t}}=\sum_{j=1}^m \left(\frac{1}{t_j} -\eta_j\right) =\sum_{j=1}^i \frac{1}{q_j}+\sum_{j=i+1}^m \left(\frac{1}{q_j} - \eta_j\right)>\sum_{j=1}^m \left(\frac{1}{q_j} - \eta_j\right)=\frac{1}{\tilde{q}}>\frac{1}{s}.
\end{align}
Therefore, $(\vec{r},s) \preceq (\vec t,\tilde{t})$ and $r_i<q_i$ which justify the use of \eqref{eq:iioo:4.6}.

In both cases we can then perform the $i$-th step of the iteration and this completes the proof of \eqref{eq:iioo:1.3} for a generic $\vec{q}$.

\emph{{\textbf{\textsf Step 2: Extrapolation on one component.}}}

Now we demonstrate the fact \eqref{eq:iioo:4.5}.
We consider $\vec{r} \preceq \vec{q} = (q_1, \ldots, q_{m-1}, q_m)$ where $q_i = p_i$ for $i = 1, \ldots, m-1$ and $q_m > r_m$. 

For $\vec{v} \in A_{(\vec{q},\tilde{q}), (\vec{r},s)}$, define $v := \prod_{i=1}^m v_i$ and set
\[
r_{m+1} := s'; \quad \frac{1}{r} := \sum_{i=1}^{m+1} \frac{1}{r_i}; \quad p_{m+1} := \tilde{p}'; \quad q_{m+1} := \tilde{q}'.
\]
For each $i = 1, \ldots, m+1$, define
\[
\frac{1}{\delta_i} = \frac{1}{r_i} - \frac{1}{p_i}, \quad \frac{1}{\Delta_i} = \frac{1}{r_i} - \frac{1}{q_i}.
\]
Note that $\delta_i = \Delta_i$ for $i = 1, \ldots, m-1$, and we have the summations
\[
\sum_{i=1}^{m+1} \frac{1}{p_{m+1}} = \sum_{i=1}^{m+1} \frac{1}{q_{m+1}} = 1 + \eta.
\]
In view of \eqref{eq3.3}, this yields
\begin{align*}
\frac{1}{\varrho} &:= \frac{1}{r_m} - \frac{1}{r_{m+1}^{\prime}} - \eta + \sum_{i=1}^{m-1} \frac{1}{p_i} = \frac{1}{r_m} - \frac{1}{r_{m+1}^{\prime}} - \eta + \sum_{i=1}^{m-1} \frac{1}{q_i} \\
&= \frac{1}{\delta_m} + \frac{1}{\delta_{m+1}} = \frac{1}{\Delta_m} + \frac{1}{\Delta_{m+1}}.
\end{align*}

Setting $\omega_i := v_i$ for $i = 1, \ldots, m$, we apply Lemma \ref{lemma:main} $(i)$ to $\vec{v} \in A_{(\vec{q},\tilde{q}), (\vec{r},s)}$. Part $(i.1)$ gives for each $1 \leq i \leq m-1$:
\[
\omega_i^{\frac{\theta_i}{p_i}} = \omega_i^{\frac{\theta_i}{q_i}} \in A_{\zeta \theta_i}, \quad \text{where} \quad \frac{1}{\theta_i} := \zeta - \frac{1}{\Delta_i},
\]
while part $(i.2)$ yields
\[
\widetilde{\omega} := \left(\prod_{i=1}^{m-1} \omega_i^{\frac{1}{p_i}}\right)^{\varrho} \in A_{\zeta \varrho},
\]
and part $(i.3)$ implies
\begin{align}\label{eq:iioo:4.1}
	V := v^{\frac{r_m}{\tilde{q}}} \widetilde{\omega}^{-\frac{r_m}{\delta_{m+1}}} \in A_{\frac{q_m}{r_m}, \frac{\Delta_{m+1}}{r_m}}(\widetilde{\omega}).
\end{align}
We remark that $\widetilde{\omega} \in A_{\infty}$ is in particular a doubling measure, which is fixed in the rest of the prove.

{
Let $W \in A_{\frac{p_m}{r_m}, \frac{\delta_{m+1}}{r_m}}(\widetilde{\omega})$ be an arbitrary weight. Following \eqref{eq3.6}, we define $\omega_m = W^{\frac{p_m}{r_m}} \widetilde{\omega}^{-\frac{p_m}{\delta_m}}.$
Since $\omega_i^{\theta_i/p_i} \in A_{\zeta\theta_i}$ ($1 \leq i \leq m-1$) and $\widetilde{\omega} \in A_{\zeta\varrho}$, from Lemma \ref{lemma:main} $(ii)$ for $\vec{p}$, $\vec{r}$, it follows that $\vec{\omega} \in A_{(\vec{p},\tilde{p}),(\vec{r},s)}$.
Here the parameter $\varrho$ depends on the fixed exponents $p_i = q_i$ for $1 \leq i \leq m-1$, along with $r_m$ and $r_{m+1}$.
Our hypothesis now directly yields inequality \eqref{eq:iioo:1.2}. The conclusion follows from Lemma \ref{lemma:main} $(iii)$, which provides the required estimate for all tuples $(g,f_1,...,f_m) \in \mathcal{G}$.}
Thus, by hypothesis it follows that \eqref{eq:iioo:1.2} holds. For every $\left(g, f_1, \ldots, f_m\right) \in \mathcal{G}$, it derives from Lemma \ref{lemma:main} $(iii)$ that
\begin{align}\notag
	&\quad \left\|\left(g \widetilde{\omega}^{-\frac{1}{s}}\right)^{r_m} W\right\|_{L^{\frac{\tilde{p}}{r_m}}\left( d \widetilde{\omega}\right)}^{\frac{1}{r_m}}=\|g\|_{L^{\tilde{p}}(\omega)} \\ \label{leq4.2}
	&\lesssim \prod_{i=1}^m\left\|f_i\right\|_{L^{p_i}\left(\omega_i\right)}=\left(\prod_{i=1}^{m-1}\left\|f_i\right\|_{L^{p_i}\left(\omega_i\right)}\right)\left\|\left(f_m \widetilde{\omega}^{-\frac{1}{r_m}}\right)^{r_m}W\right\|_{L^{\frac{p_m}{r_m}}\left( d \widetilde{\omega}\right)}^{\frac{1}{r_m}}.
\end{align}
We set
\[
\tilde{\mathcal{G}}:=\left\{\left(\left(g \widetilde{\omega}^{-\frac{1}{s}}\right)^{r_m},\left(\left(\prod_{i=1}^{m-1}\left\|f_i\right\|_{L^{p_i}\left(\omega_i\right)}\right) f_m \widetilde{\omega}^{-\frac{1}{r_m}}\right)^{r_m}\right):\left(f, f_1, \ldots, f_m\right) \in \mathcal{G}\right\},
\]
By \eqref{leq4.2},
\[
\|FW\|_{L^{\frac{\tilde{p}}{r_m}}\left( d \widetilde{\omega}\right)} \lesssim\|G W\|_{L^{\frac{p_m}{r_m}}\left( d \widetilde{\omega}\right)}, \quad \forall(F, G) \in \tilde{\mathcal{G}},
\]
which holds for every $W \in A_{\frac{p_m}{r_m}, \frac{\delta_{m+1}}{r_m}}(\widetilde{\omega})$. By Theorem \ref{thm:Ex_1}, for any exponents $r_m < \ell_m < \infty$ and $0 < \ell, \tau < \infty$ satisfying
\begin{align}\label{eq4.3}
	\frac{1}{\ell}-\frac{1}{p}=\frac{1}{\tau}-\frac{1}{\delta_{m+1}}=\frac{1}{\ell_m}-\frac{1}{p_m}.
\end{align}
and for every $U \in A_{\frac{\ell_m}{r_m}, \frac{\tau}{r_m}}(\widetilde{\omega})$ the following estimate holds
\begin{align}\label{eq4.4}
	\|FU\|_{L^{\frac{\ell}{r_m}}\left( d \widetilde{\omega}\right)} \lesssim\|GU\|_{L^{\frac{\ell_m}{r_m}}\left( d \widetilde{\omega}\right)}, \quad \forall(F, G) \in \tilde{\mathcal{G}}.
\end{align}

Set $\ell := \tilde{q}$, $\ell_m := q_m$, and $\tau := \Delta_{m+1}$. By assumption, we have $r_m < q_m = s_m$. Moreover, since $q_i = p_i$ for $1 \leq i \leq m-1$, we conclude that
\[
\frac{1}{\ell}-\frac{1}{\tilde{p}}=\frac{1}{\tilde{q}}-\frac{1}{\tilde{p}}=\sum_{i=1}^m\left(\frac{1}{q_i}-\frac{1}{p_i}\right)=\frac{1}{q_m}-\frac{1}{p_m}=\frac{1}{\ell_m}-\frac{1}{p_m}
\]
and
\[
\frac{1}{\tau}-\frac{1}{\delta_{m+1}}=\frac{1}{{\Delta}_{m+1}}-\frac{1}{\delta_{m+1}}=\frac{1}{p_{m+1}}-\frac{1}{q_{m+1}}=\frac{1}{\tilde{q}}-\frac{1}{\tilde{p}}=\frac{1}{\ell}-\frac{1}{\tilde{p}},
\]
thus \eqref{eq4.3} follows. Moreover, from \eqref{eq:iioo:4.1}, it derives that $V \in A_{\frac{q_m}{r_m}, \frac{\Delta_{m+1}}{r_m}}(\widetilde{\omega}) = A_{\frac{s_m}{r_m}, \frac{\tau}{r_m}}(\widetilde{\omega})$. These conditions verify that \eqref{eq4.4} holds with $U = V$ and the specified parameters. Applying Lemma \ref{lemma:main}(iii) with exponents $\vec{q}$ and $\vec{r}$, we obtain for each $(g,f_1,...,f_m) \in \mathcal{G}$
\begin{align*}
	&\quad \|gv\|_{L^{\tilde{q}}}
	=
	\Big\|\Big(f\widetilde{v}^{-\frac1{s}}\Big)^{r_m} V\Big\|_{L^\frac{\tilde{q}}{r_m}(d\widetilde{v})}^{\frac1{r_m}}\\
	&=\|FV\|_{L^{\frac{\tilde{q}}{r_m}}\left( d \widetilde{\omega}\right)}^{\frac{1}{r_m}}\lesssim\|GV\|_{L^{\frac{q_m}{r_m}}\left( d \widetilde{\omega}\right)}^{\frac{1}{r_m}} =\prod_{i=1}^m\left\|f_i\right\|_{L^{q_i}\left(\omega_i\right)},
\end{align*}
which is desired estimate in the present case.

\emph{{\textbf{\textsf Step 3: 
Endpoint scenario}}}

Finally, we omit similar parts of the proof and discuss the endpoint case, i.e., $\tilde{p} = s$.
 Let $\vec{q}=$ $\left(q_1, \ldots, q_m\right)$ satisfy $(\vec{r},s) \prec (\vec{q},\tilde{q})$ so that $\tilde{p}=s>\tilde{q}$. 
Set $\mathcal{I}=\left\{i: 1 \leq i \leq m, p_i>q_i\right\}$, in fact $\mathcal{I} \neq \emptyset$, otherwise we get a contradiction
\[
\frac{1}{s}=\frac{1}{\tilde{p}}=\sum_{i=1}^m \left(\frac{1}{p_i} -\eta_i\right) \geq \sum_{i=1}^m \left(\frac{1}{q_i} - \eta_i\right)=\frac{1}{\tilde{q}}>\frac{1}{s}.
\]

Since $\vec{r} \prec \vec{q}$, without loss generality, we may assume that $p_m>q_m>r_m$. Thus $\delta_m^{-1}>0$, we can apply the Lemma \ref{lemma:main}.
Then follow \emph{{\textbf{\textsf Step 2}}} for partial modifications, to obtain the desired estimate for the exponent $\vec{u}=\left(u_1, \ldots, u_m\right)$ with $u_i=p_i$ for $1 \leq i \leq m$ and $t_m=q_m$ and the associated class of weights $A_{(\vec{u},\tilde{u}), (\vec{r},s)}$. Moreover,
\begin{align*}
	\frac{1}{\tilde{u}}=\sum_{i=1}^m \left(\frac{1}{u_i} -\eta_i\right)=\sum_{j=1}^{m-1} \left(\frac{1}{p_j} - \eta_j\right)+\frac{1}{q_m} - \eta_m >\sum_{i=1}^m \left(\frac{1}{p_i} - \eta_i\right)=\frac{1}{\tilde{p}}=\frac{1}{s}.
\end{align*}
Hence $s>\tilde{u}$, that is, $(\vec{r},s) \preceq (\vec{u},\tilde{u})$. Therefore, the proof follows similarly after changing the other elements of $\vec{u}$, and it is not hard to see that we never reach the end point $s$ in the target space.

To complete the proof we sketch how to establish 
\eqref{eq:iioo:1.4}. The proof is almost identical to Remark \ref{R:1}, the only difference is that here some of the \( s_i \)'s could be infinity. If that is the case we just need to observe that by assumption not all the \( s_i \) can be infinity (otherwise \( s = 0 \)), hence
\begin{align*}
	&\left(\sum_j \prod_{i=1}^m\left\|f_i^j v_i\right\|_{L^{t_i}}^{\tilde{t}}\right)^{\tilde{t}} \leq \prod_{i: t_i=\infty}\left\|\left\{f_i^j v_i\right\}_j\right\|_{L_{\ell \infty}^{\infty}} \left(\sum_j \prod_{i: t_i \neq \infty}\left\|f_i^j v_i\right\|_{L^{t_i}}^{\tilde{t}}\right)^{\tilde{t}}\\
	&\leq \prod_{i: t_i=\infty}\left\|\left\{f_i^j v_i\right\}_j\right\|_{L_{\ell}^{\infty}} \prod_{i: t_i \neq \infty}\left(\sum_j\left\|f_i^j v_i\right\|_{L^{t_i}}^{t_i}\right)^{\frac{1}{t_i}} \leq \prod_{i=1}^m\left\|\left\{f_i^j v_i\right\}_j\right\|_{L_{\ell^s i}^{t_i}},
\end{align*}
where the second inequality uses the fact $\tilde{t_i}/t_i>1$.
With this estimate in hand the proof can be completed in exactly the same manner and we omit the details.

\begin{lemma}\label{lemma:main}
	Let $\eta \in [0,m)$, \(\vec{p} = (p_i)_{i=1}^m, \vec{r} = (r_i)_{i=1}^m \in [1, \infty)^m\), \(s \in [1, \infty)\), and $\frac{1}{\tilde{p}}=\sum\limits_{i=1}^{m}\frac{1}{p_i} - \eta$ such that \((\vec{r}, s) \preceq (\vec{p}, \tilde{p})\).
	With the notation above, we further assume that 
	\begin{equation}\label{eq3.3}
	\frac1{\varrho}:=\frac 1{r_m}-\frac 1{r_{m+1}'}-\eta+\sum_{i=1}^{m-1}\frac1{p_i}
	=
	\frac1{\delta_m}+\frac1{\delta_{m+1}}
	>0,
	\end{equation}
	and for every $1\le i\le m-1$
	\begin{equation}\label{def:theta}
	\frac1{\theta_i}
	:=
	\zeta -\frac1{\delta_i}
	=
	\left(\sum_{j=1}^{m+1} \frac1{\delta_j}\right)-\frac1{\delta_i}
	>0.
	\end{equation}
Then the following hold:
\begin{list}{$(\theenumi)$}{\usecounter{enumi}\leftmargin=.8cm \labelwidth=.8cm\itemsep=0.2cm\topsep=.1cm
\renewcommand{\theenumi}{\roman{enumi}}}
	\item Given  $\vec{\omega}=(\omega_1,\dots,\omega_m)\in A_{(\vec p,\tilde{p}), (\vec{r},s)}$, write $\omega:=\prod\limits_{i=1}^m\omega_i $ and set
	\begin{equation}\label{eq3.4}
	\widetilde{\omega}:=\Big(\prod_{i=1}^{m-1}\omega_i \Big)^{\varrho}
	\qquad
	\mbox{and}
	\qquad
	W:= \omega^{ r_m } \widetilde{\omega}^{- \frac{r_m}{\delta_{m+1}}}= \omega_m^{ r_m }\widetilde{\omega}^{\frac{r_m}{\delta_m}}
	\end{equation}%
	Then,
\begin{list}{$(\theenumi.\theenumii)$}{\usecounter{enumii}\leftmargin=.4cm \labelwidth=.8cm\itemsep=0.2cm\topsep=.1cm\renewcommand{\theenumii}{\arabic{enumii}}}
	\item\label{lemma:main_i1} $\omega_i^{ \theta_i  }\in A_{\zeta \theta_i }$ with $\Big[\omega_i^{ \theta_i  }\Big]_{ A_{\zeta \theta_i }} \le [\vec \omega]_{A_{(\vec p,\tilde{p}), (\vec{r},s)}}^{\theta_i }$, for every $1\le i\le m-1$.
	\item $\widetilde{\omega}\in A_{\zeta \varrho}$ with
	$[\widetilde{\omega}]_{A_{\zeta \varrho}}\le [\vec \omega]_{A_{(\vec p,\tilde{p}), (\vec{r},s)}}^{\varrho}$.	
	\item $W\in A_{\frac{p_m}{r_m},\frac{\delta_{m+1}}{r_m}}(\widetilde{\omega})$
	with $[W]_{A_{\frac{p_m}{r_m},\frac{\delta_{m+1}}{r_m}} (\widetilde{\omega})}\le
	[\vec \omega]_{A_{(\vec p,\tilde{p}), (\vec{r},s)}}^{r_m}$.
	\end{list}
	\item Given $\omega_i^{ \theta_i }\in A_{\zeta \theta_i}$,  $1\le i\le m-1$, such  that
	\begin{equation}\label{eq3.5}\widetilde{\omega}=
	\Big(\prod_{i=1}^{m-1}\omega_i \Big)^{\varrho}
	\in A_{\zeta \varrho}
	\end{equation}
	and $W\in A_{\frac{p_m}{r_m},\frac{\delta_{m+1}}{r_m}} (\widetilde{\omega})$, let us set
	\begin{equation}\label{eq3.6}
	\omega_m :=
		W^{\frac{1}{r_m}} \widetilde{\omega}^{-\frac{1}{\delta_m}}.
		\end{equation}
		Then $\vec{\omega}=(\omega_1,\dots, \omega_m)\in A_{(\vec p,\tilde{p}), (\vec{r},s)}$. Moreover,
		\[
		[\vec \omega]_{A_{(\vec p,\tilde{p}), (\vec{r},s)}}
		\le
		[W]_{A_{\frac{p_m}{r_m},\frac{\delta_{m+1}}{r_m}} (\widetilde{\omega})}^{\frac1{r_m}}
		[\widetilde{\omega}]_{A_{\zeta \varrho}}^{\frac1 {\varrho}}
		\prod_{i=1}^{m-1}\Big[\omega_i^{ \theta_i }\Big]_{A_{\zeta \theta_i}}^{\frac1{\theta_i}}.
		\]

		\item For any measurable function $f \ge 0 $ and in the context of $(i)$ or $(ii)$ there hold
		\begin{equation}\label{eq3.7}
		\|f\omega\|_{L^{\tilde{p}}}
		=
		\Big\|\Big(f\widetilde{\omega}^{-\frac1{s}}\Big)^{r_m} W\Big\|_{L^\frac{\tilde{p}}{r_m}(d\widetilde{\omega})}^{\frac1{r_m}}
		\end{equation}
		and
		\begin{equation}\label{eq3.8}
		\|f\omega_m\|_{L^{p_m} }
		=
		\Big\|\Big(f\widetilde{\omega}^{-\frac1{r_m}}\Big)^{r_m}W\Big\|_{L^\frac{p_m}{r_m}(d\widetilde{\omega})}^{\frac1{r_m}}.
		\end{equation}
	\end{list}
\end{lemma}

\begin{proof}[\bf Proof of Lemma \ref{lemma:main}]~~

\emph{{\textbf{\textsf The preparation of the proof:}}}

We first classify the index introduced above as follows.
Since \((\vec{r}, s) \preceq (\vec{p}, \tilde{p})\) and $\sum\limits_{i=1}^{m-1} \frac{1}{\delta_{i}} \geq 0$, we can divide this discussion into two cases as follows.

\begin{list}{\rm (\theenumi)}{\usecounter{enumi}\leftmargin=1.2cm \labelwidth=1cm \itemsep=0.2cm \topsep=.2cm \renewcommand{\theenumi}{Case: \Roman{enumi}}}

\item \label{P1} Let \(\sum\limits_{i=1}^{m-1} \frac{1}{\delta_i} > 0\), set \(\zeta = \sum\limits_{i=1}^{m+1} \frac{1}{\delta_i}\). Then, 
\[
\zeta = \sum_{i=1}^{m-1} \frac{1}{\delta_i} + \frac{1}{\varrho} > \frac{1}{\varrho}.
\]
Let \(\mathcal{I} = \left\{i: 1 \leq i \leq m-1, \delta_i^{-1} \neq 0\right\} \neq \emptyset\) and \(\mathcal{I}^{\prime} = \{1, \ldots, m-1\} \backslash \mathcal{I}\).

Set
\[
\frac{1}{\upsilon_i} := \frac{1}{\delta_i} \left( \sum_{j=1}^{m-1} \frac{1}{\delta_j} \right)^{-1} = \frac{1}{\delta_i} \left( \sum_{j=1}^{m+1} \frac{1}{\delta_j} - \frac{1}{\delta_m} - \frac{1}{\delta_{m+1}} \right)^{-1} = \frac{1}{\delta_i} \left( \zeta - \frac{1}{\varrho} \right)^{-1}
\]
for every \(i \in \mathcal{I}\), and note that \(\sum\limits_{i \in \mathcal{I}} \frac{1}{\upsilon_i} = 1\);

\item \label{P2} Let $\sum\limits_{i=1}^{m-1} \frac{1}{\delta_i}=0$. If $\delta_m^{-1} \neq 0$, then
 for $1 \leq j \leq m-1$,  $p_j=r_j$ and $\frac{1}{\varrho}=\frac{1}{\delta_{m+1}}+\frac{1}{\delta_m}=\zeta$.
If $\delta_m^{-1} = 0$, then
 for $1 \leq j \leq m$,  $p_j=r_j$ and 
 $\frac{1}{\varrho}=\frac{1}{\delta_{m+1}}=\zeta$.
\end{list}

Before the formal proof we need to compute some estimates that are commonly used in below proof.
We start observing that equalities \eqref{eq3.4} in $(i)$, or \eqref{eq3.5} and \eqref{eq3.6} in case $(ii)$, easily yields
\begin{align}\label{eq3.9}
	\fint_B W^{\frac{\delta_{m+1}}{r_m}} d \widetilde{\omega}=\left(\fint_B \widetilde{\omega} d x\right)^{-1}\left(\fint_B \omega^{{\delta_{m+1}}} d x\right)
\end{align}
and whenever $\delta_m^{-1} \neq 0$ (i.e., $r_m<p_m$ )
\begin{align*}
	\fint_B W^{-\left(\frac{p_m}{r_m}\right)^{\prime}} d \widetilde{\omega}=\fint_B \omega_m^{-{r_m}\left(\frac{p_m}{r_m}\right)^{\prime}} \widetilde{\omega}^{-\frac{r_m}{\delta_m}\left(\frac{p_m}{r_m}\right)^{\prime}} d \widetilde{\omega}=\left(\fint_B \widetilde{\omega} d x\right)^{-1}\left(\fint_B \omega_m^{-{\delta_m}} d x\right) .
\end{align*}
These equalities yield if $\delta_m^{-1} \neq 0$
\begin{align}\notag
	&\left(\fint_B W^{\frac{\delta_{m+1}}{r_m}} d \widetilde{\omega}\right)\left(\fint_B W^{-\left(\frac{p_m}{r_m}\right)^{\prime}} d \widetilde{\omega}\right)^{\frac{\frac{\delta_{m+1}}{r_m}}{\left(\frac{p_m}{r_m}\right)^{\prime}}}\\ \label{eq3.10}
	&=\left(\fint_B \widetilde{\omega} d x\right)^{-1-\frac{\delta_{m+1}}{\delta_m}}\left(\fint_B \omega^{\delta_{m+1}} d x\right)\left(\fint_B \omega_m^{-\delta_m} d x\right)^{\frac{\delta_{m+1}}{\delta_m}}.
\end{align}
Thus, if $\delta_m^{-1} \neq 0$ and $\sum\limits_{i=1}^{m-1} \frac{1}{\delta_i}>0$ then
\begin{align}\notag
	&\left(\fint_B W^{\frac{\delta_{m+1}}{r_m}} d \widetilde{\omega}\right)\left(\fint_B W^{-\left(\frac{p_m}{r_m}\right)^{\prime}} d \widetilde{\omega}\right)^{\frac{\frac{\delta_{m+1}}{r_m}}{\left(\frac{p_m}{r_m}\right)^{\prime}}}\\ \label{eq3.11}
	&=\left[\left(\fint_B \widetilde{\omega}^{1-\left(\zeta \varrho\right)^{\prime}} d x\right)^{\frac{1}{\varrho}\left(\zeta \varrho-1\right)}\left(\fint_B \omega^{\delta_{m+1}} d x\right)^{\frac{1}{\delta_{m+1}}}\left(\fint_B \omega_m^{-\delta_m} d x\right)^{\frac{1}{\delta_m}}\right. \\ \notag
&\quad \quad \left.\times(\fint_B \widetilde{\omega} d x)^{-\frac{1}{\varrho}}\left(\fint_B \widetilde{\omega}^{1-\left(\zeta \varrho\right)^{\prime}} d x\right)^{-\frac{1}{\varrho}\left(\zeta \varrho-1\right)}\right]^{\delta_{m+1}}.
\end{align}
On the other hand, if $\sum\limits_{i=1}^{m-1} \frac{1}{\delta_i}>0$ and $\delta_m^{-1}=0$,
\begin{align}\notag
&\fint_B W^{\frac{\delta_{m+1}}{r_m}} d \widetilde{\omega}=\left(\fint_B \widetilde{\omega} d x\right)^{-1}\left(\fint_B \omega^{{\delta_{m+1}}} d x\right)\\ \label{eq3.12}
&=\left[\left(\fint_B \widetilde{\omega}^{1-\left(\zeta \varrho\right)^{\prime}} d x\right)^{\frac{1}{\varrho}\left(\zeta \varrho-1\right)}\left(\fint_B \omega^{{\delta_{m+1}}} d x\right)^{\frac{1}{\delta_{m+1}}}\right.\left.(\fint_B \widetilde{\omega} d x)^{-\frac{1}{\varrho}}\left(\fint_B \widetilde{\omega}^{1-\left(\zeta \varrho\right)^{\prime}} d x\right)^{-\frac{1}{\varrho}\left(\zeta \varrho-1\right)}\right]^{\delta_{m+1}}.
\end{align}

\emph{{\textbf{\textsf The formal proof:}}}

\emph{Proof to $(i.1)$.}~~

Assuming $\vec{\omega} \in A_{(\vec{p},\tilde{p}), (\vec{r},s)}$, we fix an index $1 \leq i \leq m-1$ and define the index sets:
\[
\mathcal{I} := \{j : 1 \leq j \leq m, \delta_j^{-1} \neq 0\}, \quad \mathcal{I}' := \{1,\ldots,m\} \setminus \mathcal{I}.
\]
We then set the exponents:
\[
\frac{1}{\upsilon_i} := \frac{\theta_i}{\delta_{m+1}}, \quad \frac{1}{\upsilon_j} := \frac{\theta_i}{\delta_j} \text{ for } j \in \mathcal{I} \setminus \{i\},
\]
which satisfy the normalization condition $\frac{1}{\upsilon_i} + \sum_{j \in \mathcal{I}\setminus\{i\}} \frac{1}{\upsilon_j} = 1$. It follows from Hölder's inequality that
\begin{align}
	&\fint_B \omega_i^{{\theta_i}} d x=\fint_B\left(\omega^{{\theta_i}} \prod_{\substack{1 \leq j \leq m \\ j \neq i}} \omega_j^{-{\theta_i}}\right) d x\\ \notag
	&\quad \leq\left(\fint_B \omega^{{\theta_i} \upsilon_i} d x\right)^{\frac{1}{\upsilon_i}}\left(\prod_{i \neq j \in \mathcal{I}}\left(\fint_B \omega_j^{-\frac{\theta_i}{\upsilon_j}}\right)^{\frac{1}{\upsilon_j}}\right)\left(\prod_{i \neq j \in \mathcal{I}^{\prime}} \underset{B}{\operatorname{ess\sup}} \, \omega_j^{-{\theta_i}}\right)\\
	&\quad \quad =\left(\fint_B \omega^{{\delta_{m+1}}} d x\right)^{\frac{\theta_i}{\delta_{m+1}}}\left(\prod_{i \neq j \in \mathcal{I}}\left(\fint_B \omega_j^{-{\delta_j}}\right)^{\frac{\theta_i}{\delta_j}}\right)\left(\prod_{i \neq j \in \mathcal{I}^{\prime}} \underset{B}{\operatorname{ess\sup}} \, \omega_j^{-1}\right)^{\theta_i} .
\end{align}
For $p_i = r_i$, $\theta_i = \gamma/(1-\gamma)$ yields $\omega_i^{\theta_i} \in A_1$ satisfying $[\omega_i^{\theta_i}]_{A_1} \leq [\vec{\omega}]^{\theta_i}$. When $p_i > r_i$, the relation $\theta_i(1 - (\zeta\theta_i)') = -\delta_i$ implies $\omega_i^{\theta_i/p_i} \in A_{\zeta\theta_i}$, it can be deduced that $[\omega_i^{\theta_i/p_i}]_{A_{\zeta\theta_i}} \leq [\vec{\omega}]^{\theta_i}$.

 \emph{Proof to $(i.2)$.}

\eqref{P1}: 
Using Hölder's inequality, we obtain
\begin{align} \label{eq3.13}
 \left(\fint_B \widetilde{\omega}^{1-\left(\zeta \varrho\right)^{\prime}} d x\right)^{\zeta \varrho-1}
	&=\left(\fint_B \prod_{i \in \mathcal{I}} \omega^{-\frac{\delta_i}{\upsilon_i} } \prod_{i \in \mathcal{I}^{\prime}} \omega_i^{- \varrho\left(\left(\zeta \varrho\right)^{\prime}-1\right)} d x\right)^{\zeta \varrho-1}
	\leq \prod_{i \in \mathcal{I}}\left(\fint_B \omega_i^{-{\delta_i}} d x\right)^{\frac{\varrho}{\delta_i}}\left(\prod_{i \in \mathcal{I}^{\prime}} \underset{B}{\operatorname{ess\sup}}\, \omega_i^{-{1}}\right)^{\varrho} .
\end{align}
If $\delta_m^{-1} \neq 0$, from Hölder's inequality with $\frac{\delta_{m+1}}{\varrho}=1+\frac{\delta_{m+1}}{\delta_m}>1$, it follows that
\begin{align}\label{eq3.14}
	\fint_B \widetilde{\omega} d x & =\fint_B \omega^{{\varrho}} \omega_m^{-{\varrho}} d x \leq\left(\fint_B \omega^{{\delta_{m+1}}} d x\right)^{\frac{\varrho}{\delta_{m+1}}} \left(\fint_B \omega_m^{-{\delta_m}} d x\right)^{\frac{\varrho}{\delta_m}}.
\end{align}
If $\delta_m^{-1}=0$ then $\varrho=\delta_{m+1}$ and
\begin{align}\label{eq3.15}
	\fint_B \widetilde{\omega} d x & =\fint_B \omega^{{\varrho}} \omega_m^{-{\varrho}} d x \leq\left(\fint_B \omega^{{\delta_{m+1}}} d x\right)^{\frac{\varrho}{\delta_{m+1}}} \left(\underset{B}{\operatorname{ess} \sup }\, \omega_m^{-{1}}\right)^{\varrho}.
\end{align}
Combining estimate \eqref{eq3.13} with either \eqref{eq3.14} or \eqref{eq3.15}, it can be deduced that
\[
[\widetilde{\omega}]_{A_{\zeta\varrho}} \leq [\vec{\omega}]_{A_{(\vec{p},\tilde{p}),(\vec{r},s)}}.
\]
This establishes the property $(i.2)$.

\eqref{P2}:
When $\delta_m = 0$, using $\widetilde{\omega} \in A_1$, we obtain
\begin{align*}
	&\quad \fint_B \widetilde{\omega} d x=\fint_B \omega^{{\delta_{m+1}}} \omega_m^{-{\delta_{m+1}}} d x \leq \fint_B \omega^{{\delta_{m+1}}} d x \, \underset{B}{\operatorname{ess\sup}} \, \omega_m^{-{\delta_{m+1}}}\\
	&\leq[\omega]_{A_{(\vec{p},\tilde{p}), (\vec{r},s)}}^{\delta_{m+1}}\left(\prod_{i=1}^{m-1} \underset{B}{\operatorname{ess\inf}} \,\omega_i^{}\right)^{\delta_{m+1}} \leq[\omega]_{A_{(\vec{p},\tilde{p}), (\vec{r},s)}}^{\varrho} \underset{B}{\operatorname{ess\inf}}\, \widetilde{\omega}. \\
\end{align*}

When $\delta_m \neq 0$, to show $\widetilde{\omega} \in A_1$, we use Hölder's inequality with $\frac{\delta_{m+1}}{\varrho}=1+\frac{\delta_{m+1}}{\delta_m}>1$ to obtain
\begin{align*}
	&\quad \fint_B \widetilde{\omega} d x=\fint_B \omega^{{\varrho}} \omega_m^{-{\varrho}} d x\\
	& \leq\left(\fint_B \omega^{{\delta_{m+1}}} d x\right)^{\frac{{\varrho}}{\delta_{m+1}}} \left(\fint_B \omega_m^{-{\delta_m}} d x\right)^{\frac{\varrho}{\delta_m}}\\
	&\leq[\omega]_{A_{(\vec{p},\tilde{p}), (\vec{r},s)}}^{\varrho}\left(\prod_{i=1}^{m-1} \underset{B}{\operatorname{essinf}}\, \omega_i^{}\right)^{\varrho} \leq[\omega]_{A_{(\vec{p},\tilde{p}), (\vec{r},s)}}^{\varrho} \underset{B}{\operatorname{essinf}}\, \widetilde{\omega}. \\
\end{align*}

\emph{Proof to $(i.3)$.}

\eqref{P1}:
Recall from the proof of $(i.2)$ that we established $\zeta\varrho > 1$. Applying Hölder's inequality with exponent $\zeta\varrho$, 
\[
1=\left(\fint_B \widetilde{\omega}^{\frac{\varrho}{\zeta}} \widetilde{\omega}^{-\frac{\varrho}{\zeta}} d x\right)^{\zeta \varrho} \leq\left(\fint_B \widetilde{\omega} d x\right)\left(\fint_B \widetilde{\omega}^{1-\left(\zeta \varrho\right)^{\prime}} d x\right)^{\zeta \varrho-1} .
\]
This, \eqref{eq3.11} and \eqref{eq3.13} yield if we further assume that $\delta_m^{-1} \neq 0$ (that is $r_m<p_m$ ):
\begin{align}\notag
	&\left[\left(\fint_B W^{\frac{\delta_{m+1}}{r_m}} d \widetilde{\omega}\right)\left(\fint_B W^{-\left(\frac{p_m}{r_m}\right)^{\prime}} d \widetilde{\omega}\right)^{\frac{\frac{\delta_{m+1}}{r_m}}{\left(\frac{p_m}{r_m}\right)^{\prime}}}\right]^{{\frac{1}{\delta_{m+1}}}}\\ \notag
	&\leq\left(\fint_B \widetilde{\omega}^{1-\left(\zeta \varrho\right)^{\prime}} d x\right)^{\frac{1}{\varrho}\left(\zeta \varrho-1\right)}\left(\fint_B \omega^{\delta_{m+1}} d x\right)^{\frac{1}{\delta_{m+1}}}\left(\fint_B \omega_m^{-\delta_m} d x\right)^{\frac{1}{\delta_m}} \\ \notag
	&\leq \prod_{i \in \mathcal{I}}\left(\fint_B \omega_i^{-{\delta_i}} d x\right)^{\frac{1}{\delta_i}}\left(\prod_{i \in \mathcal{I}^{\prime}} \underset{B}{\operatorname{ess\sup}}\, \omega_i^{-{1}}\right)^{1} \left(\fint_B \omega^{\delta_{m+1}} d x\right)^{\frac{1}{\delta_{m+1}}}\left(\fint_B \omega_m^{-\delta_m} d x\right)^{\frac{1}{\delta_m}}\\ \notag
	&\leq[\vec{\omega}]_{A_{(\vec{p},\tilde{p}), (\vec{r},s)}}^{\delta_{m+1}}.
\end{align}
This establishes that $W \in A_{\frac{p_m}{r_m}, \frac{\delta_{m+1}}{r_m}}(\widetilde{\omega})$ with 
$[W]_{A_{\frac{p_m}{r_m},\frac{\delta_{m+1}}{r_m}}}(\widetilde{\omega}) \leq [\vec{\omega}]^{\delta_{m+1}}$. 
When $\delta_m^{-1} = 0$ ($r_m = p_m$), from \eqref{eq3.12} and \eqref{eq3.13}, it follows that
\begin{align*}
	\left(\fint_B W^{\frac{\delta_{m+1}}{r_m}} d \widetilde{\omega}\right) &\leq\left[\left(\fint_B \widetilde{\omega}^{1-\left(\zeta \varrho\right)^{\prime}} d x\right)^{\frac{1}{\varrho}\left(\zeta \varrho-1\right)}\left(\fint_B \omega^{{\delta_{m+1}}} d x\right)^{\frac{1}{\delta_{m+1}}}\right]^{\delta_{m+1}}\\
	&\leq\left[\left(\prod_{i \in \mathcal{I}}\left(\fint_B \omega_i^{-{\delta_i}} d x\right)^{\frac{1}{\delta_i}}\right)\left(\prod_{i \in \mathcal{I}^{\prime}} \underset{B}{\operatorname{ess\sup}}\, \omega_i^{-{1}}\right)\left(\fint_B \omega^{{\delta_{m+1}}} d x\right)^{\frac{1}{\delta_{m+1}}}\right]^{\delta_{m+1}}\\
	&\leq[\vec{\omega}]_{A_{(\vec{p},\tilde{p}), (\vec{r},s)}}^{\delta_{m+1}} \underset{B}{\operatorname{ess\inf}} \, \omega_m^{{\delta_{m+1}}}=[\vec{\omega}]_{A_{(\vec{p},\tilde{p}), (\vec{r},s)}}^{\delta_{m+1}} \underset{B}{\operatorname{ess} \inf } \, W^{\frac{\delta_{m+1}}{r_m}}.
\end{align*}
since in this case $\varrho=\delta_{m+1}$ and $W=\omega_m^{{r_m}}=\omega_m$.

\eqref{P2}:
If $\delta_m^{-1} = 0$, we can note that
\begin{align}\label{eq3.16}
	1=\fint_B \widetilde{\omega} \widetilde{\omega}^{-1} d x \leq\left(\fint_B \widetilde{\omega} d x\right) \underset{B}{\operatorname{ess} \sup } \,\widetilde{\omega}^{-1} \leq\left(\fint_B \widetilde{\omega} d x\right) \prod_{i=1}^{m-1} \underset{B}{\operatorname{ess} \sup } \, \omega_i^{-{\varrho}} .
\end{align}
It derives from this and \eqref{eq3.9} that
\begin{align*}
	&\quad \fint_B W^{\frac{\delta_{m+1}}{r_m}} d \widetilde{\omega}=\left(\fint_B \widetilde{\omega} d x\right)^{-1}\left(\fint_B \omega^{{\delta_{m+1}}} d x\right)\\
	&\leq[\vec{\omega}]_{A_{(\vec{p},\tilde{p}), (\vec{r},s)}}^{\delta_{m+1}}\left(\prod_{i=1}^{m-1} \underset{B}{\operatorname{ess} \sup }\,  \omega_i^{-{\delta_{m+1}}}\right)\left(\prod_{i=1}^m \underset{B}{\operatorname{ess} \sup }\, \omega_i^{-{1}}\right)^{-\delta_{m+1}}\\
	&=[\vec{\omega}]_{A_{(\vec{p},\tilde{p}), (\vec{r},s)}}^{\delta_{m+1}}\left(\underset{B}{\operatorname{ess} \sup } \, \omega_m^{-{1}}\right)^{-\delta_{m+1}}=[\vec{\omega}]_{A_{(\vec{p},\tilde{p}), (\vec{r},s)}}^{\delta_{m+1}}\left(\underset{B}{\operatorname{ess} \sup }\, W^{-\frac{\delta_{m+1}}{r_m}}\right)^{-1}.
\end{align*}
Here we have used the fact that in this case $\varrho = \delta_{m+1}$ and $W = \omega_m^{r_m} = \omega_m$, since $p_m = r_m$ by assumption. It follows $W \in A_{1,\frac{\delta_{m+1}}{r_m}}(\widetilde{\omega})$ from $[W]_{A_{1,\frac{\delta_{m+1}}{r_m}}(\widetilde{\omega})} \leq [\vec{\omega}]_{A_{(\vec{p},\tilde{p}),(\vec{r},s)}}^{\delta_{m+1}}.$
Note that since $0 < \widetilde{\omega} < \infty$ a.e., the measures $\mathrm{d}x$ and $\widetilde{\omega}\mathrm{d}x$ share the same null sets. Hence, $\esssup$ and $\essinf$ coincide for both measures.

If $\delta_m^{-1} \neq 0$,
in this case, from \eqref{eq3.10} and \eqref{eq3.16} it can be deduced that
\begin{align*}
	&\quad \left(\fint_B W^{\frac{\delta_{m+1}}{r_m}} d \widetilde{\omega}\right)\left(\fint_B W^{-\left(\frac{p_m}{r_m}\right)^{\prime}} d \widetilde{\omega}\right)^{\frac{\frac{\delta_{m+1}}{r_m}}{\left(\frac{p_m}{r_m}\right)^{\prime}}}\\ 
	&=\left(\fint_B \widetilde{\omega} d x\right)^{-1-\frac{\delta_{m+1}}{\delta_m}}\left(\fint_B \omega^{\delta_{m+1}} d x\right)\left(\fint_B \omega_m^{-\delta_m} d x\right)^{\frac{\delta_{m+1}}{\delta_m}}\\
	&\leq\left(\fint_B \omega^{{\delta_{m+1}}} d x\right)\left(\fint_B \omega_m^{-{\delta_m}} d x\right)^{\frac{\delta_{m+1}}{\delta_m}}\left(\prod_{i=1}^{m-1} \underset{B}{\operatorname{ess} \sup }\, \omega_i^{-{1}}\right)^{-\delta_{m+1}}\\
	&\leq[\vec{\omega}]_{A_{(\vec{p},\tilde{p}), (\vec{r},s)}}^{\delta_{m+1}}.
\end{align*}
Thus we get $(i .3)$, hence 
this completes the proof of $(i)$.

 \emph{Proof of $(ii)$}. 

Given the conditions $\omega_i^{\theta_i/p_i} \in A_{\zeta\theta_i}$ for $1 \leq i \leq m-1$, we have $\widetilde{\omega} \in A_{\zeta\varrho}$ by \eqref{eq3.5}, along with $W \in A_{p_m/r_m, \delta_{m+1}/r_m}(\widetilde{\omega})$. Given $\omega_m$ as specified in \eqref{eq3.6}, our objective is to demonstrate that $\vec{\omega} \in A_{(\vec{p},\tilde{p}),(\vec{r},s)}$. As in previous arguments, we consider two distinct cases.

\eqref{P1}:
Set $\upsilon_i=\zeta(m-1) \theta_i$ for $1 \leq i \leq m-1$ and $\upsilon_m=(m-1)\left(\zeta \varrho\right)^{\prime}$, considering
\[
\begin{aligned}
\sum_{i=1}^m \frac{1}{\upsilon_i} & =\frac{1}{(m-1)\zeta} \sum_{i=1}^{m-1} \frac{1}{\theta_i}+\frac{1}{(m-1)\left(\zeta \varrho\right)^{\prime}} \\
& =1-\frac{1}{(m-1)\zeta} \sum_{i=1}^{m-1} \frac{1}{\delta_i}+\frac{1}{m-1}-\frac{1}{m-1} \frac{1}{\zeta \varrho} \\
& =1-\frac{1}{(m-1)\zeta}\left(\zeta-\frac{1}{\varrho}\right)+\frac{1}{m-1}-\frac{1}{m-1} \frac{1}{\zeta \varrho} \\
& =1.
\end{aligned}
\]
It follows from Hölder's inequality with the exponents $\upsilon_i, 1 \leq i \leq m$ that
\begin{align*}
	1&=\left(\fint_B \widetilde{\omega}^{-\frac{1}{(m-1)\zeta \varrho}} \widetilde{\omega}^{\frac{1}{(m-1)\zeta \varrho}} d x\right)^{\zeta(m-1)}=\left(\fint_B \widetilde{\omega}^{\left(1-\left(\zeta \varrho\right)^{\prime}\right) \frac{1}{\upsilon_m}} \prod_{i=1}^{m-1} \omega_i^{\frac{\theta_i}{\upsilon_i}} d x\right)^{\zeta(m-1)}\\
	&\leq\left(\fint_B \widetilde{\omega}^{1-\left(\zeta \varrho\right)^{\prime}} d x\right)^{\frac{1}{\varrho}\left(\zeta \varrho-1\right)} \prod_{i=1}^{m-1}\left(\fint_B \omega_i^{{\theta_i}} d x\right)^{\frac{1}{\theta_i}}
\end{align*}
Given $\mathcal{I} = \{j : 1 \leq j \leq m-1, \delta_j^{-1} \neq 0\} \neq \emptyset$ and $\mathcal{I}^{\prime} = \{1, \ldots, m-1\} \backslash \mathcal{I}$, we note that $\theta_i ((\zeta \theta_i)^{\prime} - 1) = \delta_i$. It derives from above estimates that  
\begin{align*}
	&\quad \left(\prod_{i \in \mathcal{I}}\left(\fint_B \omega_i^{-\frac{\delta_i}{p_i}} d x\right)^{\frac{1}{\delta_i}}\right)\left(\prod_{i \in \mathcal{I}^{\prime}} \underset{B}{\operatorname{ess} \sup }\, \omega_i^{-{1}}\right)\\
    &\leq \prod_{i=1}^{m-1}\left[\omega_i^{{\theta_i}}\right]_{A_{\zeta \theta_i}}^{\frac{1}{\theta_i}}\left(\fint_B \omega_Q^{{\theta_i}} d x\right)^{-\frac{1}{\theta_i}}\\
	& \leq \left(\prod_{i=1}^{m-1}\left[\omega_i^{{\theta_i}}\right]_{A_{\zeta \theta_i}}^{\frac{1}{\theta_i}}\right) \left(\fint_B \widetilde{\omega}^{1-\left(\zeta \varrho\right)^{\prime}} d x\right)^{\frac{1}{\varrho}\left(\zeta \varrho-1\right)}.
\end{align*}
When $\frac{1}{\delta_m} \neq 0$,
this and \eqref{eq3.11} yield that
\begin{align*}
	&\quad \left(\fint_B \omega^{\delta_{m+1}} d x\right)^{\frac{1}{\delta_{m+1}}}\left(\fint_B \omega_m^{-\delta_m} d x\right)^{\frac{1}{\delta_m}}\left(\prod_{i \in \mathcal{I}}\left(\fint_B \omega_i^{-{\delta_i}} d x\right)^{\frac{1}{\delta_i}}\right)\left(\prod_{i \in \mathcal{I}^{\prime}} \underset{B}{\operatorname{ess} \sup } \,\omega_i^{-{1}}\right)\\
	&\leq \left(\prod_{i=1}^{m-1}\left[\omega_i^{\frac{\theta_i}{p_i}}\right]_{A_{\zeta \theta_i}}^{\frac{1}{\theta_i}}\right) \left[\left(\fint_B W^{\frac{\delta_{m+1}}{r_m}} d \widetilde{\omega}\right)\left(\fint_B W^{-\left(\frac{p_m}{r_m}\right)^{\prime}} d \widetilde{\omega}\right)^{\frac{\frac{\delta_{m+1}}{r_m}}{\left(\frac{p_m}{r_m}\right)^{\prime}}}\right]^{\frac{1}{\delta_{m+1}}}\\
	&\quad \quad \times (\fint_B \widetilde{\omega} d x)^{\frac{1}{\varrho}}\left(\fint_B \widetilde{\omega}^{1-\left(\zeta \varrho\right)^{\prime}} d x\right)^{\frac{1}{\varrho}\left(\zeta \varrho-1\right)}\\
	& \leq [W]^{\frac{1}{\delta_{m+1}}}_{A_{\frac{p_m}{r_m}, \frac{\delta_{m+1}}{r_m}}(\widetilde{\omega})}[\widetilde{\omega}]_{A_{\zeta\varrho}}^{\frac{1}{\varrho}}\left(\prod_{i=1}^{m-1}\left[\omega_i^{{\theta_i}}\right]_{A_{\zeta\theta_i}}^{\frac{1}{\theta_i}}\right).
\end{align*}
On the other hand, if $\frac{1}{\delta_m}=0$, we obtain \eqref{eq3.12}:
\begin{align*}
	&\quad \left(\fint_B \omega^{\delta_{m+1}} d x\right)^{\frac{1}{\delta_{m+1}}}\left(\underset{B}{\operatorname{ess} \sup } \,\omega_m^{-{1}}\right)\left(\prod_{i \in \mathcal{I}}\left(\fint_B \omega_i^{-{\delta_i}} d x\right)^{\frac{1}{\delta_i}}\right)\left(\prod_{i \in \mathcal{I}^{\prime}} \underset{B}{\operatorname{ess} \sup } \,\omega_i^{-{1}}\right)\\
	&\leq \left(\prod_{i=1}^{m-1}\left[\omega_i^{\frac{\theta_i}{p_i}}\right]_{A_{\zeta \theta_i}}^{\frac{1}{\theta_i}}\right) \left[\left(\fint_B W^{\frac{\delta_{m+1}}{r_m}} d \widetilde{\omega}\right)\left(\fint_B W^{-\left(\frac{p_m}{r_m}\right)^{\prime}} d \widetilde{\omega}\right)^{\frac{\frac{\delta_{m+1}}{r_m}}{\left(\frac{p_m}{r_m}\right)^{\prime}}}\right]^{\frac{1}{\delta_{m+1}}}\\
	&\quad \quad \times (\fint_B \widetilde{\omega} d x)^{\frac{1}{\varrho}}\left(\fint_B \widetilde{\omega}^{1-\left(\zeta \varrho\right)^{\prime}} d x\right)^{\frac{1}{\varrho}\left(\zeta \varrho-1\right)}\\
	& \leq [W]^{\frac{1}{\delta_{m+1}}}_{A_{1, \frac{\delta_{m+1}}{r_m}}(\widetilde{\omega})}[\widetilde{\omega}]_{A_{\zeta\varrho}}^{\frac{1}{\varrho}}\left(\prod_{i=1}^{m-1}\left[\omega_i^{{\theta_i}}\right]_{A_{\zeta\theta_i}}^{\frac{1}{\theta_i}}\right)\left(\underset{B}{\operatorname{ess} \sup } \, W^{-\frac{\delta_{m+1}}{r_m}}\right),
\end{align*}
since in this case $p_m=r_m, \varrho=\delta_{m+1}$, and $W=\omega_m^{{r_m}}=\omega_m$.

\eqref{P2}: If $\delta_m^{-1} = 0$, hence $\theta_i=\zeta$ for every $1 \leq i \leq m-1$. It follows from Hölder's inequality that
\begin{align}\notag
	\underset{B}{\operatorname{ess\inf}}\left(\prod_{i=1}^{m-1} \omega_i\right) &\leq\left(\fint_B \prod_{i=1}^{m-1} \omega_i^{ \frac{\theta_i}{m-1}} d x\right)^{{(m-1)\zeta}}\\ \label{eq3.17}
    &\leq \prod_{i=1}^{m-1}\left(\fint_B \omega_i^{{\theta_i}} d x\right)^{\frac{1}{\theta_i}} \leq \prod_{i=1}^{m-1}\left[\omega_i^{{\theta_i}}\right]_{A_1}^{\frac{1}{\theta_i}} \underset{B}{\operatorname{ess\inf}}\, \omega_i^{}.
\end{align}
The last estimate we have used that in the present scenario $\zeta\theta_i=1$. 
From this and \eqref{eq3.9}, it can be deduced that
\begin{align*}
	\left(\fint_B \omega^{{\delta_{m+1}}} d x\right)^{\frac{1}{\delta_{m+1}}}&=\left(\fint_B W^{\frac{\delta_{m+1}}{r_m}} d \widetilde{\omega}\right)^{\frac{1}{\delta_{m+1}}}\left(\fint_B \widetilde{\omega} d x\right)^{\frac{1}{\delta_{m+1}}}\\
    &\leq [W]^{\frac{1}{\delta_{m+1}}}_{A_{1, \frac{\delta_{m+1}}{r_m}}(\widetilde{\omega})}[\widetilde{\omega}]_{A_1}^{\frac{1}{\delta_{m+1}}}\underset{B}{\operatorname{ess} \inf } \,W^{\frac{1}{r_m}} \underset{B}{\operatorname{ess} \inf } \,\widetilde{\omega}^{\frac{1}{\delta_{m+1}}}\\
	&\leq [W]^{\frac{1}{\delta_{m+1}}}_{A_{1, \frac{\delta_{m+1}}{r_m}}(\widetilde{\omega})}[\widetilde{\omega}]_{A_1}^{\frac{1}{\delta_{m+1}}}\underset{B}{\operatorname{ess} \inf }\, \omega_m^{} \underset{B}{\operatorname{ess} \inf }\left(\prod_{i=1}^{m-1} \omega_i^{\frac{1}{p_i}}\right)\\
	&\leq [W]^{\frac{1}{\delta_{m+1}}}_{A_{1, \frac{\delta_{m+1}}{r_m}}(\widetilde{\omega})}[\widetilde{\omega}]_{A_1}^{\frac{1}{\delta_{m+1}}}\left(\prod_{i=1}^{m-1}\left[\omega_i^{{\theta_i}}\right]_{A_1}^{\frac{1}{\theta_i}}\right)\left(\prod_{i=1}^m \underset{B}{\operatorname{ess\inf}} \, \omega_i^{}\right),
\end{align*}
where we have used that $p_m=r_m$,  $\omega_m=W^{\frac{p_m}{r_m}}=W$ and that $\varrho=\delta_{m+1}$. This readily leads to the desired estimate.

If $\delta_m^{-1} \neq 0$, by \eqref{eq3.10} and \eqref{eq3.17}, we obtain
\begin{align*}
	&\quad \left(\fint_B \omega^{\delta_{m+1}} d x\right)^{\frac{1}{\delta_{m+1}}}\left(\fint_B \omega_m^{-\delta_m} d x\right)^{\frac{1}{\delta_m}}\\
	& =\left(\fint_B W^{\frac{\delta_{m+1}}{r_m}} d \widetilde{\omega}\right)^{\frac{1}{\delta_{m+1}}}\left(\fint_B W^{-\left(\frac{p_m}{r_m}\right)^{\prime}} d \widetilde{\omega}\right)^{\frac{1}{r_m\left(\frac{p_m}{r_m}\right)^{\prime}}}\left(\fint_B \widetilde{\omega} d x\right)^{\frac{1}{\varrho}}\\
	& \leq [W]^{\frac{1}{\delta_{m+1}}}_{A_{\frac{p_m}{r_m}, \frac{\delta_{m+1}}{r_m}}(\widetilde{\omega})}[\widetilde{\omega}]_{A_1}^{\frac{1}{\varrho}} \underset{B}{\operatorname{ess} \inf }\left(\prod_{i=1}^{m-1} \omega_i^{}\right)\\
	& \leq [W]^{\frac{1}{\delta_{m+1}}}_{A_{\frac{p_m}{r_m}, \frac{\delta_{m+1}}{r_m}}(\widetilde{\omega})}[\widetilde{\omega}]_{A_1}^{\frac{1}{\delta_{m+1}}}\left(\prod_{i=1}^{m-1}\left[\omega_i^{{\theta_i}}\right]_{A_1}^{\frac{1}{\theta_i}}\right)\left(\prod_{i=1}^{m-1} \underset{B}{\operatorname{ess} \inf } \, \omega_i^{}\right).
\end{align*}
This immediately establishes that $\vec{\omega} \in A_{(\vec{p},\tilde{p}), (\vec{r},s)}$ with the required bound, thereby completing the proof of part $(ii)$.
Finally, we note that equations \eqref{eq3.7} and \eqref{eq3.8} follow directly from the definition of $\varrho$ combined with either \eqref{eq3.4} in $(i)$ or \eqref{eq3.6} in $(ii)$, which concludes the entire proof.
\end{proof}

\begin{lemma}\label{lem:Li 5.3}
Let $\eta \in [0,m)$, $\vec{p} = (p_i)_{i=1}^m$ with $1 < p_1, \dots, p_m < \infty$, $\frac{1}{{\tilde{p}}} := \sum\limits_{i=1}^m \frac{1}{p_i} - \eta > 0$, $\vec{r} = (r_i)_{i=1}^m$ with $1 \leq r_1, \dots, r_m < \infty$, and $s \in [1, \infty)$ such that $(\vec{r}, s) \preceq (\vec{p}, \tilde{p})$. For each $1 \leq i \leq m$, we define
\[
\frac{1}{\theta_i} := \zeta - \frac{1}{\delta_i} = \sum_{j=1}^{m+1} \frac{1}{\delta_j} - \frac{1}{\delta_i} > 0
\]
where we used the notation introduced below Sect. \ref{Sec.extro}. 
\begin{list}{$(\theenumi)$}{\usecounter{enumi}\leftmargin=1.2cm \labelwidth=1cm \itemsep=0.2cm \topsep=.2cm \renewcommand{\theenumi}{\roman{enumi}}}

\item\label{lem:Li:5.3_1}
For any $\vec{\omega} = (\omega_1, \ldots, \omega_m) \in A_{(\vec{p},\tilde{p}), (\vec{r},s)}$ with $\omega = \prod_{i=1}^m \omega_i$, then $
[\omega_i^{\theta_i}]_{A_{\zeta\theta_i}} \leq [\vec{\omega}]_{A_{(\vec{p},\tilde{p}),(\vec{r},s)}}^{\theta_i}$, for all $1 \leq i \leq m$, and
$[\omega^{\delta_{m+1}}]_{A_{\zeta\delta_{m+1}}} \leq [\vec{\omega}]_{A_{(\vec{p},\tilde{p}),(\vec{r},s)}}^{\delta_{m+1}}.$

\item \label{lem:Li:5.3_2}
Suppose $\omega_i^{\theta_i} \in A_{\zeta\theta_i}$ for all $1 \leq i \leq m$, and let $\omega := \prod_{i=1}^m \omega_i$ with $\omega^{\delta_{m+1}} \in A_{\zeta\delta_{m+1}}$. Then 
\[
[\vec{\omega}]_{A_{(\vec{p},\tilde{p}), (\vec{r},s)}} 
\leq [\omega^{\delta_{m+1}}]_{A_{\zeta\delta_{m+1}}}^{\frac{1}{\delta_{m+1}}} 
\prod_{i=1}^m [\omega_i^{\theta_i}]_{A_{\zeta\theta_i}}^{\frac{1}{\theta_i}}.
\]
\end{list}

\end{lemma}

\begin{proof}

\emph{{\textbf{\textsf (i):}}}
From Lemma \ref{lemma:main} $(i.1)$, it can be deduced that $\omega_i^{\theta_i} \in A_{\zeta\theta_i}$ with $[\omega_i^{\theta_i}]_{A_{\zeta\theta_i}} \leq [\vec{\omega}]_{A_{(\vec{p},\tilde{p}),(\vec{r},s)}}^{\theta_i}$
for all $1 \leq i \leq m-1$. The identical argument extends to the case $i=m$.
To complete the proof of $(i)$, we must verify $\omega^{\delta_{m+1}} \in A_{\zeta\delta_{m+1}}$. We partition the index set by defining $\mathcal{I} := \{1 \leq i \leq m \mid \delta_i^{-1} \neq 0\}, \quad \mathcal{I}' := \{1,\dots,m\}\setminus\mathcal{I}.$

If $\mathcal{I} = \emptyset$,
hence $\zeta = \delta_{m+1}^{-1}$, it follows that
\begin{align*}
&\quad \left[\omega^{{\delta_{m+1}}}\right]_{A_{\zeta \delta_{m+1}}}=\left[\omega^{{\delta_{m+1}}}\right]_{A_1}=\sup _Q\left(\fint_B \omega^{{\delta_{m+1}}} d x\right) \underset{B}{\operatorname{ess} \sup } \, \omega(x)^{-{\delta_{m+1}}}\\
&\leq \sup _Q\left(\fint_B \omega^{{\delta_{m+1}}} d x\right) \prod_{i=1}^m \underset{B}{\operatorname{ess} \sup } \, \omega_i(x)^{-{\delta_{m+1}}} \leq [\vec{\omega}]_{A_{(\vec{p},\tilde{p}), (\vec{r},s)}}^{\delta_{m+1}} .
\end{align*}

If $\mathcal{I} \neq \emptyset$, for each $i \in \mathcal{I}$
	\[
	\frac{1}{\upsilon_i}:=\frac{1}{\delta_i}\left(\sum_{j=1}^m \frac{1}{\delta_j}\right)^{-1}=\frac{1}{\delta_i}\left(\sum_{j=1}^{m+1} \frac{1}{\delta_j}-\frac{1}{\delta_{m+1}}\right)^{-1}=\frac{1}{\delta_i}\left(\zeta-\frac{1}{\delta_{m+1}}\right)^{-1},
	\]	
	and note that $\sum_{i \in \mathcal{I}} \frac{1}{\upsilon_i}=1$. 
It follows from Hölder's inequality that
\begin{align*}
&\left[\omega^{{\delta_{m+1}}}\right]_{A_{\zeta \delta_{m+1}}} = \sup _Q\left(\fint_B \omega^{{\delta_{m+1}}} d x\right)\left(\fint_B \omega^{{\delta_{m+1}}\left(1-\left(\zeta \delta_{m+1}\right)^{\prime}\right)} d x\right)^{\zeta \delta_{m+1}-1}\\
&\leq \sup_Q\left( \fint_B \omega^{{\delta_{m+1}}} d x\right) \left(\fint_B \prod_{i \in \mathcal{I}} \omega_i^{-\frac{\delta_i}{ \upsilon_i}} d x\right)^{\zeta \delta_{m+1}-1}\prod_{i \in \mathcal{I}'}\underset{B}{\operatorname{ess} \sup } \, \omega_i(x)^{-{\delta_{m+1}}}\\
&\leq \sup_Q\left( \fint_B \omega^{{\delta_{m+1}}} d x\right) \left(\prod_{i \in \mathcal{I}}\left(\fint_B \omega_i^{-{\delta_i}} d x\right)^{\frac{\delta_{m+1}}{\delta_i}}\right)\prod_{i \in \mathcal{I}'}\underset{B}{\operatorname{ess} \sup } \, \omega_i(x)^{-{\delta_{m+1}}}\\
&\leq [\vec{\omega}]_{A_{(\vec{p},\tilde{p}), (\vec{r},s)}}^{\delta_{m+1}}.
\end{align*}
This completes the proof of $(i)$.

\emph{{\textbf{\textsf (ii):}}}
If $\mathcal{I}=\emptyset$, for each $1 \leq i \leq m$ we have $\delta_i^{-1}=0$ and $\theta_i=\zeta=\delta_{m+1}$. It follows from Hölder's inequality that
\begin{align*}
	&\underset{B}{\operatorname{ess} \inf } \, \omega \leq\left(\fint_B \omega^{\frac{\delta_{m+1}}{m}} d x\right)^{\frac{1}{m \delta_{m+1}}}=\left(\fint_B \prod_{i=1}^m \omega_i^{ \frac{\theta_i}{m}} d x\right)^{\frac{1}{m \delta_{m+1}}}\\
	&\leq \prod_{i=1}^m\left(\fint \omega^{{\theta_i}{}} d x\right)^{\frac{1}{\delta_{m+1}}} \leq \prod_{i=1}^m\left[\omega_i^{{\theta_i}{}}\right]_{A_1}^{\frac{1}{\theta_i}} \underset{B}{\operatorname{ess} \inf } \, \omega_i
\end{align*}
and thus
\begin{align*}
[\vec{\omega}]_{A_{(\vec{p},\tilde{p}), (\vec{r},s)}}&=\sup _B\left(\fint_B \omega^{{\delta_{m+1}}} d x\right)^{\frac{1}{\delta_{m+1}}} \prod_{i=1}^m \underset{B}{\operatorname{ess} \sup } \, \omega_i(x)^{-1}\\
	&\leq \left[\omega^{{\delta_{m+1}}}\right]_{A_1}^{\frac{1}{\delta_{m+1}}}\left(\underset{B}{\operatorname{ess} \inf } \, \omega^{}\right)\leq\left[\omega^{{\delta_{m+1}}}\right]_{A_1}^{\frac{1}{\delta_{m+1}}} \prod_{i=1}^m\left[\omega_i^{{\theta_i}}\right]_{A_1}^{\frac{1}{\theta_i}}
\end{align*}
If $\mathcal{I} \neq \emptyset$, define
\[
\frac{1}{\theta_{m+1}}:=\zeta-\frac{1}{\delta_{m+1}}=\sum_{j=1}^{m+1} \frac{1}{\delta_j}-\frac{1}{\delta_{m+1}}>0
\]
Since $\sum\limits_{i=1}^{m+1} \frac{1}{\theta_i}=m\zeta$, it can be deduced from Hölder's inequality that
\begin{align*}
	1&=\left(\fint_B \omega^{- \frac{1}{m\zeta}} \omega^{ \frac{1}{m\zeta}} d x\right)^{m\zeta}=\left(\fint_B \omega^{- \frac{1}{m\zeta}} \prod_{i=1}^m \omega_i^{ \frac{1}{m\zeta}} d x\right)^{m\zeta}\\
	&\leq\left(\fint_B \omega^{-{\theta_{m+1}}} d x\right)^{\frac{1}{\theta_{m+1}}} \prod_{i=1}^m\left(\fint_B \omega_i^{{\theta_i}{}} d x\right)^{\frac{1}{\theta_i}}\\
	&=\left(\fint_B \omega^{{\delta_{m+1}}\left(1-\left(\zeta \delta_{m+1}\right)^{\prime}\right)} d x\right)^{\frac{1}{\delta_{m+1}}\left(\zeta \delta_{m+1}-1\right)} \prod_{i=1}^m\left(\fint_B \omega_i^{{\theta_i}} d x\right)^{\frac{1}{\theta_i}} .
\end{align*}
Together with our standing hypotheses, this establishes the required estimate
\begin{align*}
[\vec{\omega}]_{A_{(\vec{p},\tilde{p}), (\vec{r},s)}}&=\sup _B\left(\fint_B \omega^{{\delta_{m+1}}} d x\right)^{\frac{1}{\delta_{m+1}}}\left(\prod_{i \in \mathcal{I}}\left(\fint_B \omega_i^{-{\delta_i}} d x\right)^{\frac{1}{\delta_i}}\right) \prod_{i \in \mathcal{I}' } \underset{B}{\operatorname{ess} \sup } \, \omega_i(x)^{-1}\\
&\leq\left[\omega^{{\delta_{m+1}}}\right]^{\frac{1}{\delta_{m+1}}}_{A_{\zeta \delta_{m+1}}} \prod_{i=1}^m\left[\omega_i^{{\theta_i}}\right]^{\frac{1}{\theta_i}}_{A_{\zeta \theta_i}}\\
	&\quad \times\sup_B\left(\fint_B \omega^{{\delta_{m+1}}\left(1-\left(\zeta \delta_{m+1}\right)^{\prime}\right)} d x\right)^{-\frac{1}{\delta_{m+1}}\left(\zeta \delta_{m+1}-1\right)} \prod_{i=1}^m\left(\fint_B \omega_i^{{\theta_i}} d x\right)^{-\frac{1}{\theta_i}}\\
&\leq\left[\omega^{{\delta_{m+1}}}\right]^{\frac{1}{\delta_{m+1}}}_{A_{\zeta \delta_{m+1}}} \prod_{i=1}^m\left[\omega_i^{{\theta_i}}\right]^{\frac{1}{\theta_i}}_{A_{\zeta \theta_i}}.
\end{align*}
This completes the proof.

\end{proof}

\subsection{Proof of Theorem \ref{op.to.jh}}~~

In this subsection, we want to prove Theorem \ref{op.to.jh}. Before this, we need some Lemmas as follows.

\begin{lemma}[\cite{Damain2014}, Lemma 2.2.1]\label{dam_RH}
Let $\omega \in A_{\infty}$ and define
\[
\rho_\omega=1+\frac{1}{6\left(32 A_0^2\left(4 A_0^2+A_0\right)^2\right)^{C_\mu}[\omega]_{A_{\infty}}},
\]
where $A_0$ is defined in  Definition \ref{quasi-metric} and \ref{def:doubling} below. Then
\[
\left(\fint_B \omega^r d \mu\right)^{1 / r} \leq 2(4 A_0)^{C_\mu} \fint_{2 A_0 B} \omega d \mu
\]
for any ball $B \in X$.
\end{lemma}

By the Cauchy integral trick, we can obtain the following theorem, whose proof are same as \cite[Theorem 4.3]{Tor2019}.

\begin{theorem}\label{thm:multilinear:main:I}
Let $T_{\vec \eta}$ be an $m$--linear operator. Let $m \in \N$, $\eta \in [0,m)$, $p_1,\dots,p_m \in [1, \infty)$ and $\frac{1}{{\tilde{p}}}:=\sum\limits_{i = 1}^m {\frac{1}{{{p_i}}}} - \eta >0$. Suppose that there exist increasing functions $\phi_i: [1, \infty)\to [0, \infty)$ such that for all $\vec \omega = (\omega_1, \dots, \omega_m)$ satisfying $\omega_i^{\theta_i}\in A_{t_i},$ where $\theta_i > 0$, $1 <t_i< \infty$, for $ 1\leq i\leq m$. Then
\begin{align}\label{product-weight}
    \left\|{T}_{\vec{\eta}}(\vec{f})\right\|_{L^{\tilde{p}}(\omega^{\tilde{p}})} \lesssim \prod_{i=1}^m \phi_i\left([\omega_i^{\theta_i}]_{A_{t_i}}\right)\left\|f_i\right\|_{L^{p_i}(\omega_i^{p_i})},
\end{align}
where $\omega:=\prod\limits_{i=1}^m \omega_i$.

Then for all ${\bf b}=\left(b_1, \ldots, b_m\right) \in \mathrm{BMO}^m$ and ${\bf k }\in \N_0^m$, for each $\omega_i$ with $1\le j\le m$, there exist some $1<\tau_i<\infty$ such that $\omega_i^{\tau_i \theta_i} \in A_{t_i}$, then
\begin{equation}
\label{multi-commutator-I}
\left\|{T}_{\vec{\eta}}^{{\bf b, k}}(\vec{f})\right\|_{L^{\tilde{p}}(\omega^{\tilde{p}})}
\leq {\bf k}!\prod_{i=1}^m \frac1{R_i^{k_i}}\phi_i
\left(4^{\theta_iR_i}[\omega_i^{\theta_i\tau_i}]^{1/\tau_i}_{A_{t_i}}\right)\|b_i\|^{k_i}_{\BMO} \|f_i\|_{L^{p_i}\left(\omega_i^{p_i}\right)},
\end{equation}	
where $R_i=\frac{\min\{1, t_i-1\}}{\tau_i'\theta_i}$, $1\le i\le m$.
\end{theorem}

\begin{proof}[Proof of Theorem \ref{op.to.jh}]

We prove \eqref{1tb} based on \eqref{1t1} by the method of \cite[Theorem 4.3]{Tor2019} and Lemma \ref{dam_RH}.

Without loss of generality, we assume that $b_i, 1 \leq i \leq m$, are real valued and normalized so that $\left\|b_i\right\|_{\mathcal{B M O}}=1$. From \cite[Sect. 2.2]{Tor2019} with the John-Nirenberg inequality in \cite[Lemma 2.2.4]{Damain2014}, we obtain $\left\|b_i\right\|_{\mathcal{B M O}(X)} \approx \left\|b_i\right\|_{\rm{B M O}(X)}$.
Following the argument in the proof of \cite[Theorem 4.3]{Tor2019} with the Cauchy integral trick, for all $\vec{\omega} \in A_{(\vec{p},\tilde{p}), (\vec{r},s)}$, it is sufficient to show that for some appropriate $\gamma_1, \ldots, \gamma_m>0$ and for $\left|z_1\right|=\gamma_1, \ldots,\left|z_m\right|=\gamma_m$, holds that $\vec{v} \in A_{(\vec{p},\tilde{p}), (\vec{r},s)}$, where 
\[
\vec{v}:=\left(v_1, \ldots, v_m\right):=\left(\omega_1 e_{b_1}, \ldots, \omega_m e_{b_m}\right):=\left(\omega_1 e^{-\operatorname{Re}\left(z_1\right) p_1 b_1}, \ldots, \omega_m e^{-\operatorname{Re}\left(z_m\right) p_m b_m}\right).
\]
Indeed, from Lemma \ref{extra.t.b.pre}, we choose $\gamma_1, \ldots, \gamma_m>0$ in below \eqref{keychosse}, hence $\vec{v} \in A_{(\vec{p},\tilde{p}), (\vec{r},s)}$.

To proceed, again the argument in the proof of \cite[Theorem 4.3]{Tor2019}, it follows from \eqref{keychosse_1} that
\begin{align*}
\left\|{T}_{\vec{\eta}}^{{\bf b, k}}(\vec{f})\right\|_{L^{\tilde{p}}(\omega^{\tilde{p}})} \lesssim {\bf k}!\gamma_1^{-k_1}\cdots\gamma_m^{-k_m}{\phi}\left(\boldsymbol{c}_0[\vec \omega]_{A_{(\vec{p},\tilde{p}),(\vec{r},s)}}\right) [\vec{\omega}]^{{|\bf k}|\Xi}_{A_{(\vec{p},\tilde{p}),(\vec{r}, s)}} \prod_{i=1}^m\left\|b_i\right\|_{\mathrm{BMO}}^{k_i}\left\|f_i\right\|_{L^{p_i}(\omega_i^{p_i})},
\end{align*}
where $\boldsymbol{c}_0=4^{ \sum_{i=1}^m \min \left\{p_i / \delta_i, p/ \delta_{m+1}\right\} / p_i}$ with $\frac{1}{p}=\sum_{i=1}^m \frac{1}{p_i}$, and $\Xi=\max \left\{ {\frac{{{p_1r_1}}}{{{p_1} - {r_1}}}, \cdots ,\frac{{{p_mr_m}}}{{{p_m} - {r_m}}},\frac{{\tilde{p}'s'}}{{\tilde{p}' - s'}}} \right\}$. To control $\gamma_i^{-1}$, combining with \eqref{keychosse_0} and \eqref{keychosse_1}, we further assuming $\gamma_j \approx [\vec{\omega}]^{\Xi}_{A_{(\vec{p},\tilde{p}),(\vec{r}, s)}}$. Then we get \eqref{1tb}, and \eqref{1tbv} is a direct consequence of \eqref{1tb} and Theorem \ref{thm:Ex_1}.
\end{proof}

\begin{lemma}\label{extra.t.b.pre}
Let $m \in \N$, $\eta \in [0,m)$, $r_1,\cdots,r_m ,s' \in [1,\infty)$, $p_1,\dots,p_m \in [1, \infty)$, and $\frac{1}{{\tilde{p}}}:=\sum\limits_{i = 1}^m {\frac{1}{{{p_i}}}} - \eta >0$, such that $(\vec{r}, s) \prec (\vec{p}, \tilde{p})$. Let $\gamma_i > 0$ with $i \in \{1,\ldots,m\}$ satisfying below \eqref{keychosse}. If $\vec{\omega} \in A_{(\vec{p},\tilde{p}), (\vec{r},s)}$, for all ${\bf b}=\left(b_1, \ldots, b_m\right) \in \mathrm{BMO}^m$, such that
\[
\vec{v}:=\left(v_1, \ldots, v_m\right):=\left(\omega_1 e_{b_1}, \ldots, \omega_m e_{b_m}\right):=\left(\omega_1 e^{-\operatorname{Re}\left(z_1\right) p_1 b_1}, \ldots, \omega_m e^{-\operatorname{Re}\left(z_m\right) p_m b_m}\right).
\]
then $\vec{v} \in A_{(\vec{p},\tilde{p}), (\vec{r},s)}$.
\end{lemma}

\begin{proof}[Proof of Lemma \ref{extra.t.b.pre}]
By Lemma \ref{lem:Li 5.3} $(i)$, applied to $\vec{\omega} \in A_{(\vec{p},\tilde{p}), (\vec{r},s)}$, it follows that 
\begin{align*}
\left\{\begin{array}{l}
\left[\omega^{\delta_{m+1}}\right]_{A_{\zeta \delta_{m+1}}} \leq [\vec{\omega}]_{A_{(\vec{p},\tilde{p}),(\vec{r}, s)}}^{\delta_{m+1}};\\
\left[\omega_i^{\theta_i}\right]_{A_{\zeta \theta_i}}^{\zeta \delta_i - 1} = [\omega_i^{-\delta_i}]_{A_{\zeta \delta_i}} \leq [\vec{\omega}]_{A_{(\vec{p},\tilde{p}),(\vec{r}, s)}}^{\delta_{m+1}},\quad i= 1,\ldots,m.
\end{array}\right.
\end{align*}
From Lemma \ref{dam_RH}, for a given weight $\omega$, we use the notation $\rho(\omega)$ instead of $\rho_\omega$. Then there exist 
\begin{align*}
	\varrho=\varrho(\vec \omega)=\min \left\{\rho\left(v^{\delta_{m+1}}\right), \rho\left(v_1^{-\delta_1}\right),\ldots, \rho\left(\omega_m^{-\delta_m}\right)\right\}>1
\end{align*}
so that
\begin{align}\label{keychosse_0}
	\varrho^{\prime} \approx \max \left\{\left[\omega^{{\delta_{m+1}}{}}\right]_{A_{\zeta \delta_{m+1}}},\left[\omega_1^{-{\delta_1}{}}\right]_{A_{\zeta \delta_1}}, \ldots,\left[\omega_m^{-{\delta_m}{}}\right]_{A_{\zeta \delta_m}}\right\} \leq [\vec{\omega}]^{\Xi}_{A_{(\vec{p},\tilde{p}),(\vec{r}, s)}},
\end{align}
where $\Xi=\max \left\{ {\frac{{{p_1r_1}}}{{{p_1} - {r_1}}}, \cdots ,\frac{{{p_mr_m}}}{{{p_m} - {r_m}}},\frac{{\tilde{p}'s'}}{{\tilde{p}' - s'}}} \right\}$.

It follows from the reverse Hölder inequalities that
\[
\left(\fint_B \omega^{{\delta_{m+1}}{} \varrho} d x\right)^{\frac{1}{\varrho}} \leq 2(4 A_0)^{C_\mu} \fint_{2 A_0 B} \omega^{{\delta_{m+1}}} d x
\]
and, for $i=1, \ldots, m$,
\[
\left(\fint_B \omega_i^{-{\delta_i}{} \varrho} d x\right)^{\frac{1}{\varrho}} \leq 2(4 A_0)^{C_\mu} \fint_{2 A_0 B} \omega_i^{-{\delta_i}{}} d x.
\]
From the previous estimates, the Hölder's inequality with $\sum\limits_{i=1}^m \frac{p}{p_i}=1$, and $\omega:=\prod\limits_{i=1}^m \omega_i$, it derives that
\begin{align}\notag
&\quad \left(\fint_B v^{\delta_{m+1}} d x\right)^{\frac{1}{\delta_{m+1}}} \prod_{i=1}^m\left(\fint_B v_i^{-{\delta_i}{}} d x\right)^{\frac{1}{\delta_i}} =\left(\fint_B \omega^{{\delta_{m+1}}} \prod_{i=1}^m e_{b_i}^{{\delta_{m+1}}{}} d x\right)^{\frac{1}{\delta_{m+1}}} \prod_{i=1}^m\left(\fint_B \omega_i^{-{\delta_i}{}} e_{b_i}^{-{\delta_i}{}} d x\right)^{\frac{1}{\delta_i}}\\ \notag
& \leq \left(\fint_B \omega^{{\delta_{m+1}}{} \varrho} d x\right)^{\frac{1}{\delta_{m+1} \varrho}} \prod_{i=1}^m\left(\fint_B \omega_i^{-{\delta_i}{} \varrho} d x\right)^{\frac{1}{\delta_i \varrho}} \left(\fint_B \prod_{i=1}^m e_{b_i}^{{\delta_{m+1}}{} \varrho^{\prime}} d x\right)^{\frac{1}{\delta_{m+1 \varrho^{\prime}}}}\prod_{i=1}^m\left(\fint_B e_{b_i}^{-{\delta_i}{} \varrho^{\prime}} d x\right)^{\frac{1}{\delta_i \varrho^{\prime}}}\\ \notag
&\leq 2^{(1+2C_{\mu})\zeta}A_0^{\zeta C_{\mu}}[\vec{\omega}]_{A_{(\vec{p},\tilde{p}),(\vec{r}, s)}} \prod_{i=1}^m\left(\fint_B (e_{b_i})^{p_i \cdot \frac{\delta_{m+1}}{p} \varrho^{\prime}} d x\right)^{\frac{p}{\delta_{m+1} \varrho^{\prime} p_i}}\left(\fint_B (e_{b_i})^{p_i \cdot (-\frac{\delta_i}{p_i} \varrho^{\prime})} d x\right)^{\frac{1}{\delta_i \varrho^{\prime}}}\\ \label{keychosse_1}
& \leq_{} 2^{(1+2C_{\mu})\zeta}A_0^{\zeta C_{\mu}}[\vec{\omega}]_{A_{(\vec{p},\tilde{p}),(\vec{r}, s)}} \prod_{i=1}^m\left[e_{b_i}^{\frac{\delta_{m+1}p_i}{p} \varrho^{\prime}}\right]_{A_{1+\frac{\delta_{m+1} p_i}{\delta_ip}}}^{\frac{p}{\delta_{m+1} \varrho^{\prime} p_i}} \lesssim_{C_{\mu},A_0,\zeta} 4^{ \sum\limits_{i=1}^m \gamma_i} [\vec{\omega}]_{A_{(\vec{p},\tilde{p}),(\vec{r}, s)}},
\end{align}
where the last estimate holds by \cite[Lemma 3.5]{Tor2019},  provided
\begin{align}\label{keychosse}
	\gamma_i\varrho^{\prime} \leq  \min \left\{\frac{1}{\delta_i}, \frac{p}{\delta_{m+1} p_i}\right\}.
\end{align}
Thus, we get $\vec{v} \in A_{(\vec{p},\tilde{p}), (\vec{r},s)}$.
\end{proof}

}

}

\section{\bf Quantitative weighted estimates}\label{sec.qwe}

\subsection{Sharp-type estimates}\label{sharp}
~~

\begin{definition}
Let $(X,\mu,\mathcal{B})$ be a ball-basis measure space. We say that $\B$ satisfies the Besicovitch $D_0$-condition with a constant $D_0 \in \mathbb{N}$ if for any collection $\mathcal{S} \subseteq \B$ one can find a subcollection $\mathcal{S}^{\prime} \subseteq \mathcal{S}$ such that
\[
\bigcup_{A \in \mathcal{S}} A=\bigcup_{A \in \mathcal{S}^{\prime}} A, \quad \sum_{A \in \mathcal{S}^{\prime}} \chi_A(x) \leq D_0
\]
We say that $\B$ is a martingale system if $D_0=1$.
\end{definition}

Recall that
\begin{align*}
M^{\B,w}_{r} f(x):=\sup _{B \in \B: x \in B}\left(\frac{1}{w(B)} \int_B|f(x)|^r d w(x)\right)^{1 / r}.
\end{align*}
Replacing $f$ in the proof of \cite[Theorem 6.3]{K2019} with $f^r$, we can obtain the following lemma.
\begin{lemma}\label{lem:M-Besi}
Let $(X,\mu,\mathcal{B})$ be a Besicovitch-type ball-basis measure space with the constant $D_0$. Let $1 \leq r<\infty$, then for any weight $w$,  
\begin{align}
&\|M^{\B,w}_{r}\|_{L^r(X, w) \to L^{r, \infty}(X, w)} \le D_0^{1/r}, 
\\
&\|M^{\B,w}_{r}\|_{L^s(X, w) \to L^s(X, w)} \le c_s D_0^{\frac1{sr}}, \quad r<s<\infty. 
\end{align}
\end{lemma}

\begin{theorem}\label{thm:sparse-dom_1}
Let $(X,\mu,\mathcal{B})$ be a Besicovitch-type ball-basis measure space.  Let $m \in \N$, $\eta \in [0,m)$, $r_1,\cdots,r_m ,s' \in [1,\infty)$, $p_1,\dots,p_m \in (1,\infty)$, and $\frac{1}{{{\tilde{p}}}}:=\sum\limits_{i = 1}^m {\frac{1}{{{p_i}}}}- \eta > 0$.
If $(\vec{r}, s) \prec (\vec{p},\tilde{p})$ and $\vec{\omega} \in A_{(\vec{p},\tilde{p}),(\vec{r}, s)}^\mathcal{B}$, then  
\begin{align}\label{eq:main-est}
	\sup\limits_{\mathcal{S} \subseteq \mathcal{B}}\left\| {\mathscr A}_{\eta ,\S,{\vec r}}\right\|_{{\prod\limits_{j=1}^{m} L^{p_j}\left(\omega_j^{p_j}\right) \rightarrow L^{\tilde{p}}\left(\omega^{\tilde{p}}\right)}} \lesssim_{\vec{p},\vec{r},\eta,X} [\vec \omega]^{\varTheta}_{A^{\B}_{(\vec{p},\tilde{p}), (\vec{r},s)}}.
\end{align}
where $\varTheta:=\max \left\{ {\frac{{{p_1}}}{{{p_1} - {r_1}}}, \cdots ,\frac{{{p_m}}}{{{p_m} - {r_m}}},\frac{{\tilde{p}'}}{{\tilde{p}' - s'}}} \right\}$, and this bound is sharp, i.e. $\varTheta$ cannot be reduced.
\end{theorem}


\begin{theorem}\label{est:form}
Let $(X,\mu,\mathcal{B})$ be a Besicovitch-type ball-basis measure space. Let $m \in \N$, $\eta \in [0,m)$, $r_1,\cdots,r_m ,s' \in [1,\infty)$, $p_1,\dots,p_m \in (1,\infty)$, and $\frac{1}{{{\tilde{p}}}}:=\sum\limits_{i = 1}^m {\frac{1}{{{p_i}}}}- \eta \in (0,1]$. If $(\vec{r}, s) \prec (\vec{p},\tilde{p})$ and $\vec{\omega} \in A_{(\vec{p},\tilde{p}),(\vec{r}, s)}^\mathcal{B}$, then 
	\begin{align}\label{eq:main-est_}
\sup\limits_{\mathcal{S} \subseteq \B} \left\|\mathcal{A}_{\eta, \mathcal{S}, \vec{r}, s'}\right\|_{\prod\limits_{j=1}^{m} L^{p_j}\left(\omega_j^{p_j}\right) \times L^{{\tilde{p}}'}\left(\omega^{-{\tilde{p}}'}\right) \rightarrow \mathbb{R}}
&\lesssim_{\vec{p},\vec{r},\eta,X} [\vec{\omega}]^{\varTheta}_{A_{(\vec{p},\tilde{p}),(\vec{r}, s)}^\mathcal{B}}.
\end{align}
where $\varTheta:=\max \left\{ {\frac{{{p_1}}}{{{p_1} - {r_1}}}, \cdots ,\frac{{{p_m}}}{{{p_m} - {r_m}}},\frac{{\tilde{p}'}}{{{\tilde{p}}' - s'}}} \right\}$, and this bound is sharp, i.e., $\varTheta$ cannot be reduced.
\end{theorem}

Theorem \ref{est:form} follows similar method of Theorem \ref{thm:sparse-dom_1}. We only prove Theorem \ref{thm:sparse-dom_1} as follows.

\begin{proof}[Proof of Theorem \ref{thm:sparse-dom_1}.]~~

\emph{{\textbf{\textsf Case: $\tilde{p}>1$.}}}
 By duality, let $\|g\|_{L^{\tilde{p}'}\left(\omega^{-\tilde{p}'}\right)}=1$, 
\begin{align*}
\left\langle {{\mathscr A}_{\eta ,\S,{\vec r}}(\vec f),|g|} \right\rangle
&\leq \sum_{Q \in \mathcal{S}} \mu(B)^{\eta+1} 
\prod_{i=1}^m \lla |f_i| \rra_{r_i,Q}  \lla   |g| \rra_{1,Q} .
\end{align*}
It can be deduced that
\begin{align*}
    \left\| {\mathscr A}_{\eta ,\S,{\vec r}}(\vec f) \right\|_{{ L^{\tilde{p}}\left(\omega^{\tilde{p}}\right)}}  
&\leq \sum_{B \in \mathcal{S}} \mu(B)^{\eta+1} 
\prod_{i =1 }^m \lla |f_i| \rra_{r_i,B}  \lla  |g| \rra_{1,B} \\
& \lesssim \sum_{B \in \mathcal{S}} \mu(B)^{\eta+1} \left\langle  gv^{-\frac{1}{s'}}  \right\rangle_{1, B}^{v} \left\langle  v  \right\rangle_{1, B}^{\frac{1}{s'}} \prod_{i =1}^m \left\langle  f_i v_i^{-\frac{1}{r_{i}}} \right\rangle_{r_{i}, B}^{v_i} \prod_{i = 1}^m\left\langle  v_i  \right\rangle_{1, B}^{\frac{1}{r_{i}}},
\end{align*}
where ${v_j}: = \omega _j^{\frac{{{r_j}{p_j}}}{{{r_j} - {p_j}}}}$ for $j \in \{1,\ldots,m\}$, $v:=\omega^{\frac{\tilde{p}'s'}{\tilde{p}'-s'}}$.

It follows from \eqref{def.of.weighted.condi._2} that
\begin{align*}
\left\| {\mathscr A}_{\eta ,\S,{\vec r}}(\vec f) \right\|_{L^{\tilde{p}}(\omega^{\tilde{p}})} & \lesssim [\vec \omega]^{\varTheta}_{A^{\B}_{(\vec{p},\tilde{p}), (\vec{r},s)}}\sum_{B \in \mathcal{S}}\left\langle  gv^{-\frac{1}{s'}}  \right\rangle_{1, B}^{v} \left\langle  v  \right\rangle_{1, B}^{\frac{1}{s'}} \prod_{i =1}^m \left\langle  f_i v_i^{-\frac{1}{r_{i}}} \right\rangle_{r_{i}, B}^{v_i} v_i\left(E_B\right)^{\frac{1}{p_i}}
\\
&\leq [\vec \omega]^{\varTheta}_{A^{\B}_{(\vec{p},\tilde{p}), (\vec{r},s)}} \sum_{B \in \mathcal{S}} 	\left(\int_{E_Q} M_{1}^{v, \B}\left(g v^{-\frac{1}{s'}}\right)^{\tilde{p}'} v \mathrm{~d} x\right)^{\frac{1}{\tilde{p}'}}\prod_{i=1}^{m}\left(\int_{E_Q} M_{r_i}^{v_i, \B}\left(f_i v_i^{-\frac{1}{r_i}}\right)^{p_i} v_i \mathrm{~d} x\right)^{\frac{1}{p_i}}\\
& \leq [\vec \omega]^{\varTheta}_{A^{\B}_{(\vec{p},\tilde{p}), (\vec{r},s)}} \left\|M_{1}^{v, \B}\left(gv^{-\frac{1}{s'}}\right)\right\|_{L^{\tilde{p}'}\left(v\right)}\prod_{j=1}^{m}\left\|M_{r_j}^{v_j, \B}\left(f_j v_j^{-\frac{1}{r_j}}\right)\right\|_{L^{p_j}\left(v_j\right)}\\
&\lesssim  c^*[\vec \omega]^{\varTheta}_{A^{\B}_{(\vec{p},\tilde{p}), (\vec{r},s)}}\|g\|_{L^{\tilde{p}'}\left(\omega^{-\tilde{p}'}\right)}\prod_{j=1}^{m}\left\|f_j\right\|_{L^{p_j}\left(\omega_j^{p_j}\right)},
\end{align*}
where 
\begin{align*}
c^*=\left(\prod_{i=1}^{m}\left(\frac{p_i}{p_i-r_i}\right)^{\frac{1}{r_i}}\right) \cdot \left(\frac{{\tilde{p}'}}{{\tilde{p}' - 1}}\right).
\end{align*}

\emph{{\textbf{\textsf Case: $0< \tilde{p} \leq 1$.}}}
Let $\sigma_i := \omega_i^{\frac{r_i}{r_i - p_i}}$ and $\theta := \frac{\frac{1}{\tilde{p}}}{\frac{1}{\tilde{p}} - \frac{1}{s}}$. Without loss of generality, assume that $\varTheta^{(1)} = {\frac{{{p_1}}}{{{p_1} - {r_1}}}} = \max \left\{ {\frac{{{p_1}}}{{{p_1} - {r_1}}}, \cdots ,\frac{{{p_m}}}{{{p_m} - {r_m}}},\frac{{\tilde{p}'}}{{{\tilde{p}}' - s'}}} \right\}$, we obtain
\begin{align*}
	&\quad \int_{X} {\mathscr A}_{\eta ,\S,{\vec r}} ({f_1} \sigma_1^{\frac{1}{r_1}},\cdots,{f_m} \sigma_1^{\frac{1}{r_m}})^{\tilde{p}} \omega \\
	&\leq  \sum_{B \in \mathcal{S}} \mu(B)^{\eta \tilde{p}} \prod_{i = 1}^m \left\langle\left|f_i \sigma_i^{\frac{1}{r_i}} \right|\right\rangle^{\tilde{p}}_{r_i,B}   \int_{B}  \omega\\
	&\leq  \sum_{B \in \mathcal{S}} \mu(B)^{(\eta + \frac{1}{s})\tilde{p}} \prod_{i =1}^m \left\langle\left|f_j\sigma_j^{\frac{1}{r_j}}\right|\right\rangle^{\tilde{p}}_{r_j,B}\left(\omega^{\theta}(B)\right)^{\frac{1}{\theta}}\\
	&\leq  [\vec \omega]^{\varTheta^{(1)}}_{A^{\star,\B}_{(\vec{p},\tilde{p}), (\vec{r},s)}}\sum_{B \in \mathcal{S}} \frac{\mu(B)^{(\eta + \frac{1}{s})\tilde{p}}\prod\limits_{i =1}^m \left\langle\left|f_j\sigma_j^{\frac{1}{r_j}}\right|\right\rangle^{\tilde{p}}_{r_j,B}}{\left(\prod\limits_{j=1}^m\left\langle \omega_j^{-\frac{1}{p_j}}\right\rangle^{\varTheta^{(1)} \tilde{p}}_{{\frac{1}{\frac{1}{r_j}-\frac{1}{p_j}}},B}\right)\langle \omega^{\frac{1}{\tilde{p}}} \rangle^{\varTheta^{(1)}\tilde{p}}_{\frac{1}{\frac{1}{\tilde{p}}-\frac{1}{s}},B}} \cdot \left(\omega^{\theta}(B)\right)^{\frac{1}{\theta}}\\
	&\leq  [\vec \omega]^{\varTheta^{(1)}}_{A^{\star,\B}_{(\vec{p},\tilde{p}), (\vec{r},s)}} \sum_{B \in \mathcal{S}}\frac{\mu(E_B)^{\left(\eta - \sum\limits_{i = 1}^m\frac{1}{r_i} + \frac{1}{s}\right)\tilde{p}\left(1-\varTheta^{(1)}\right)}\prod\limits_{i =1}^m \left\langle\left|f_j\sigma_j^{\frac{1}{r_j}}\right|\right\rangle^{\tilde{p}}_{r_j,B}}{\prod\limits_{j=1}^{m}\sigma_j(B)^{\varTheta^{(1)}\tilde{p}\delta_i}}\cdot \left(\omega^{\theta}(B)\right)^{\frac{1}{\theta}}
\end{align*}

Set $\frac{1}{\delta_i}=\frac{1}{r_i}-\frac{1}{p_i},$ for $ i=1, \ldots, m$, combing with 
\begin{align*}
    { \frac{1}{\left(\sum\limits_{i=1}^{m}\frac{1}{r_i} - \eta - \frac{1}{s}\right)\tilde{p}\theta}} + \sum_{j=1}^m {\frac{\delta_j}{\sum\limits_{i=1}^{m}\frac{1}{r_i} - \eta - \frac{1}{s}}} =1,
\end{align*}
it derives from Hölder's inequality that
\begin{align*}
	\mu(E_B)&=\int_{E_B} \omega^{ \frac{1}{\left(\sum\limits_{i=1}^{m}\frac{1}{r_i} - \eta - \frac{1}{s}\right)\tilde{p}\theta}}\sigma_1^{\frac{\delta_1}{\sum\limits_{i=1}^{m}\frac{1}{r_i} - \eta - \frac{1}{s}}} \cdots \sigma_m^{\frac{\delta_m}{\sum\limits_{i=1}^{m}\frac{1}{r_i} - \eta - \frac{1}{s}}}\\
	&\leq \omega(E_B)^{\frac{1}{\left(\sum\limits_{i=1}^{m}\frac{1}{r_i} - \eta - \frac{1}{s}\right)\tilde{p}\theta}} \sigma_1(E_B)^{\frac{\delta_1}{\sum\limits_{i=1}^{m}\frac{1}{r_i} - \eta - \frac{1}{s}}} \cdots \sigma_m(E_B)^{\frac{\delta_m}{\sum\limits_{i=1}^{m}\frac{1}{r_i} - \eta - \frac{1}{s}}}.
\end{align*}
Therefore,
\begin{align*}
	&\quad \mu(E_B)^{\left(\eta - \sum\limits_{i = 1}^m\frac{1}{r_i} + \frac{1}{s}\right)\tilde{p}\left(1-\varTheta^{(1)}\right)} \leq \omega(E_B)^{\frac{\varTheta^{(1)} - 1}{\theta}}\sigma_1(E_B)^{\delta_1 \tilde{p} (\varTheta^{(1)} - 1)} \cdots \sigma_m(E_B)^{\delta_m \tilde{p} (p'_m - 1)}
\end{align*}
and
\begin{align*}
	\delta_i (\varTheta^{(1)} - 1) - \frac{\tilde{p}}{p_i} = \delta_i \varTheta^{(1)} \tilde{p} -\frac{\tilde{p}}{r_i}, \quad i = 1,\ldots,m.
\end{align*}
Then we obtain
\begin{align*}
	&\quad \sum_{Q \in \mathcal{S}}\frac{\mu(B)^{\left(\eta - \sum\limits_{i = 1}^m\frac{1}{r_i} + \frac{1}{s}\right)\tilde{p}\left(1-\varTheta^{(1)}\right)}}{\prod\limits_{j=1}^{m}\sigma_j(B)^{\varTheta^{(1)}q\delta_i}}\cdot \left(\omega^{\theta}(B)\right)^{\frac{1}{\theta}}\\
	&\leq \sum_{Q \in \mathcal{S}} \prod_{i=1}^{m} \left(\frac{1}{\sigma_i(B)}\int_B|f_i|^{r_i}\sigma_i\right)^{\frac{\tilde{p}}{r_i}}\sigma_i(E_B)^{\frac{\tilde{p}}{p_i}}\\
	&\leq \prod_{i=1}^{m} \left(\sum_{Q \in \S}\left(\frac{1}{\sigma_i(B)}\int_B|f_i|^{r_i}\sigma_i\right)^{\frac{p_i}{r_i}}\sigma_i(E_B)\right)^{\frac{\tilde{p}}{p_i}}\\
	&\leq \prod_{i=1}^{m} \left\|M_{r_i,\sigma_i}^{\B}\left(f_i\right)\right\|^{\tilde{p}}_{L^{p_i}(\sigma_i)}\\
    &\leq \prod_{i=1}^{m} \left\|f_i\right\|_{L^{p_i}\left(\sigma_i\right)}^{\tilde{p}}.
\end{align*}

Since ${\mathscr A}_{\eta ,\S,{\vec r}}$  is self adjoint as a multilinear operator, as well as the argument in \cite[Theorem 3.2]{Moen2014}, we get
\begin{align}
	\sup\limits_{\mathcal{S} \subseteq \d}\left\| {\mathscr A}_{\eta ,\S,{\vec r}}(\vec f) \right\|_{{\prod\limits_{j=1}^{m} L^{p_j}\left(\omega_j^{}\right) \rightarrow L^{q}\left(\omega\right)}} \lesssim  [\vec \omega]^{\frac{\varTheta}{q}}_{A^{\star,\B}_{(\vec{p},\tilde{p}), (\vec{r},s)}}.
\end{align}
By Definition \ref{def:weight}, it is equivalent to
\begin{align*}
\sup\limits_{\mathcal{S} \subseteq \d}\left\| {\mathscr A}_{\eta ,\S,{\vec r}}(\vec f) \right\|_{{\prod\limits_{j=1}^{m} L^{p_j}\left(\omega_j^{p_j}\right) \rightarrow L^{q}\left(\omega^{\tilde{p}}\right)}} \lesssim [\vec \omega]^{\varTheta}_{A^{\B}_{(\vec{p},\tilde{p}), (\vec{r},s)}}. \quad \quad \qedhere
\end{align*} 
\end{proof}

Below, we provide examples confirming that our bounds are sharp.

\textbf{Sharpness of \( \varTheta \):}  
Take \( m = 1 \), \( X = \mathbb{R} \), \( \eta = 0 \), \( p_1 = 4 \), \( r_1 = 2 \), \( \tilde{p} = q = 4 \), \( s' = 2 \), \( \tilde{p}' = \frac{4}{3} \). Let \( \mathcal{S} = \{ [0, 2^{-j}) : j \geq 0 \} \), \( b_1(x) = \lambda x \), \( \omega_1 = x^\delta \), \( \omega = x^\delta \) with \( [\vec{\omega}]_{A_{(4,4),(2,4)}} \sim \delta \), and \( f_1(x) = x^{-1/8} \chi_{[0,1]}(x) \).

The operator is that
\[
\mathscr{A}_{0, \mathcal{S}, 2}^{}(f_1)(x) = \sum_{Q \in \mathcal{S}}  \left\langle |f_1| \right\rangle_{2, Q} \chi_Q(x).
\]
For \( Q_j = [0, 2^{-j}) \):
\[
\left\langle |f_1| \right\rangle_{2, Q_j} = \left( 2^j \int_0^{2^{-j}} x^{-1/4} \, dx \right)^{1/2} \sim 2^{-j/4}.
\]
The norm
\[
\left\| \mathscr{A}_{0, \mathcal{S}, 2}^{}(f_1) \right\|_{L^4(x^{4\delta})}^4 \sim \sum_{j=0}^\infty \int_{2^{-j-1}}^{2^{-j}} (\lambda \cdot 2^{-j/4})^4 x^{4\delta} \, dx \sim \lambda^4 \delta^{-1}, \quad \left\| \mathscr{A}_{0, \mathcal{S}, 2}^{}(f_1) \right\|_{L^4(x^{4\delta})} \sim \lambda \delta^2.
\]
Since \( \varTheta = \max \left\{ \frac{4}{2}, \frac{4/3}{4/3 - 2} \right\} = 2 \), the bound is \( \lambda \delta^2 \), matching the norm. If \( \varTheta = 2 - \epsilon \), the bound \( \lambda \delta^{2 - \epsilon} \) is dominated by \( \lambda \delta^2 \) as \( \delta \to \infty \), and the inequality fails. Thus, \( \varTheta = 2 \) is sharp.

\subsection{Bloom-type estimates}~~

We introduce several necessary lemmas as follows.
\begin{lemma}\label{pre_1}
   Let  $1 \leq r_i, t \leq \infty$ with $i \in \tau$, and \(\mathbf{b} = (b_1, \ldots, b_{m}) \in (L_{loc}^1(X))^m\). Suppose that $\mathbf{t}$ and $\mathbf{k}$ are both multi-indexs with \(0 \le \mathbf{t} \le \mathbf{k}\), given a dyadic cube $Q$, 
   \begin{align*}
      C_{{\bf k}, {\bf t}}= \prod_{i \in \tau} \laa \left|f_i (b_i - b_{i,Q})^{t_i}\right|\raa_{r_i,Q}  \laaa \prod_{i \in \tau}  \left|b_i - b_{i,Q}\right|^{k_i - t_i}g\raaa_{t,Q}.
   \end{align*}
Then
   \begin{align*}
      C_{{\bf k}, {\bf t}}\leq  \sum_{\tau' \subseteq \tau} \left( \prod_{i_1 \in \tau'} \left\langle \left| b_{i_1} - b_{i_1,Q} \right|^{k_{i_1}} f_{i_1} \right\rangle_{r_{i_1},Q} 
      \left\langle \prod_{i_2 \in \tau \backslash \tau'} \left| b_{i_2} - b_{i_2,Q} \right|^{k_{i_2}} g \right\rangle_{t,Q} 
      \prod_{i_3 \in \tau \backslash \tau'} \left\langle f_{i_3} \right\rangle_{r_{i_3},Q} \right).
   \end{align*}
\end{lemma}

\begin{proof}[Proof of Lemma \ref{pre_1}]
For $x_i, y \in X$ with $i \in \tau:=({\tau(1),\cdots,\tau(|\tau|)}) \subseteq \{1,\cdots,m\}$, we denote
\begin{align*}
   \varphi(\vec x, y) = \left(\prod_{i \in \tau}\left|f_i(x_i) (b_i(x_i) - b_{i,Q})^{t_i}\right|\right)  \left(\prod_{i \in \tau}  \left|b_i(y) - b_{i,Q}\right|^{k_i - t_i}g(y)\right) \chi_{Q^{|\tau| + 1}}(\vec x,y),
\end{align*}
where $\chi_{Q^{|\tau| + 1}}(\vec x,y)=\chi_{Q}(y)\prod\limits_{i \in \tau}\chi_{Q}(x_i)$. Then it follows that
$$
C_{{\bf k}, {\bf t}} = \big\|\|\varphi(\vec x, y)\|_{L^t_{y}(\frac{d y}{\mu(Q)})}\big\|_{L^{\vec{r}}_{\vec{x}}\Big(\frac{d {\vec x}}{\mu(Q)^{|\tau|}}\Big)},
$$
where the mixed norm is defined by
\begin{align*}
\|\cdot\|_{L^{\vec{r}}_{\vec{x}}\Big(\frac{d {\vec x}}{\mu(Q)^{|\tau|}}\Big)}:&=\|\cdot\|_{L^{\vec{r}}_{\vec{x}}\left(\frac{d x_{\tau(1)}}{\mu(Q)} \cdots \frac{d x_{\tau(|\tau|)}}{\mu(Q)}\right)}\\
   &=\Big\| \cdots \big\|\|\cdot\|_{L^{r_{\tau(1)}}_{x_{\tau(1)}}(\frac{d x_{\tau(1)}}{\mu(Q)})}\big\|_{L^{r_{\tau(2)}}_{x_{\tau(2)}}(\frac{d x_{\tau(2)}}{\mu(Q)})} \cdots \Big\|_{L^{r_{\tau(|\tau|)}}_{x_{\tau(|\tau|)}}(\frac{d x_{\tau(|\tau|)}}{\mu(Q)})}.\\
\end{align*}
Note that 
$\left|b_i(x_i)- b_{i,Q}\right|^{k_i-t_i}\left|b_i(y)- b_{i,Q}\right|^{t_i} \leq\left|b_i(x_i)- b_{i,Q}\right|^{k_i}+\left|b_i(y)- b_{i,Q}\right|^{k_i}$, then we obtain
\begin{align*}
   \varphi(\vec x, y) &\leq \prod_{i \in \tau}\big(\left|b_i(x_i)- b_{i,Q}\right|^{k_i}+\left|b_i(y)- b_{i,Q}\right|^{k_i}\big) \times F(\vec x) g(y)\chi_{Q^{|\tau| + 1}}(\vec x,y)\\
   &= \left(\sum_{\tau' \subseteq \tau} C_{|\tau|}^{|\tau'|} \left(\prod_{i \in \tau'}\left|b_i(x_i)- b_{i,Q}\right|^{k_i}\prod_{j \in \tau \backslash \tau'} \left|b_j(y)- b_{j,Q}\right|^{k_j} \right) \right)F(\vec x) g(y)\chi_{Q^{|\tau| + 1}}(\vec x,y)
\end{align*}
where $F(\vec x)=\prod\limits_{i \in \tau}f_i(x_i)$.

Considering the fact that $\|\chi_{Q}(y)\|_{L^t_y\left(\frac{dy}{\mu(Q)}\right)} = 1$, it follows readily that 
\begin{align*}
   &\quad C_{{\bf k}, {\bf t}}\\
   &\leq  \left\|\Big\|\sum_{\tau' \subseteq \tau} \left(\left(\prod_{i \in \tau'}\left|b_i(x_i)- b_{i,Q}\right|^{k_i}\prod_{j \in \tau \backslash \tau'} \left|b_j(y)- b_{j,Q}\right|^{k_j} \right)F(\vec x) g(y)\chi_{Q^{|\tau| + 1}}(\vec x,y)\right)\Big\|_{L^t_{y}(\frac{d y}{\mu(Q)})}\right\|_{L^{\vec{r}}_{\vec{x}}\Big(\frac{d {\vec x}}{\mu(Q)^{|\tau|}}\Big)}\\
   &\le \sum_{\tau' \subseteq \tau} \left( \prod_{i_1 \in \tau'} \left\langle \left| b_{i_1} - b_{i_1,Q} \right|^{k_{i_1}} f_{i_1} \right\rangle_{r_{i_1},Q} 
   \left\langle \prod_{i_2 \in \tau \backslash \tau'} \left| b_{i_2} - b_{i_2,Q} \right|^{k_{i_2}} g \right\rangle_{t,Q} 
   \prod_{i_3 \in \tau \backslash \tau'} \left\langle f_{i_3} \right\rangle_{r_{i_3},Q} \right).\qedhere
\end{align*}

\end{proof}

\begin{lemma}[\cite{Yang2019}, Lemma 3.5]\label{zhang:6.1}
   Let $(X,d,\mu,\mathcal{D})$ be a space of homogeneous type. Let $0 < \gamma < 1$ and $\mathcal{S} \subseteq \mathcal{D}$ be a $\gamma$-sparse family. For $b \in L^1_{\text{loc}}(X)$, there exists a $\frac{\gamma}{2(\gamma+1)}$-sparse family $\tilde{\mathcal{S}} \subseteq \mathcal{D}$ satisfying $\mathcal{S} \subseteq \tilde{\mathcal{S}}$, and for each $Q \in \tilde{\mathcal{S}}$,
   \[
     |b(x) - b_Q| \lesssim_{X} \sum_{\substack{R \in \tilde{\mathcal{S}} \\ R \subseteq Q}} \langle\left|b(x)-b_R\right| \rangle_{R} \chi_R(x) \quad \text{a.e. } x \in Q.
   \]
 \end{lemma}

\begin{lemma}[\cite{Moen2014}, Theorem 3.2]\label{00we.xs}
Let $(X,d,\mu,\mathcal{D})$ be a space of homogeneous type. Let $1<p<\infty$. If $\omega \in A_p$, then for any sparse family of cubes $\mathcal{S} \subseteq \mathcal{D}$
\begin{align}\label{sparse.bound.}
\left\|A_{\mathcal{S}}\right\|_{L^p(\omega)\to L^p(\omega)}\lesssim[\omega]_{A_p}^{\max \left(1, \frac{1}{p-1}\right)}.
\end{align}
\end{lemma}

To prove Theorem \ref{thm:bloom}, it suffices to prove that
\begin{theorem}\label{main.bloom}
 Let $(X,d,\mu,\mathcal{D})$ be a space of homogeneous type. Let $m \in \N$, $\eta \in [0,m)$, $r_1, \cdots, r_m, s^{\prime} \in [1,\infty)$, $p_1,\dots,p_m \in (1,\infty)$, and $\frac{1}{{{\tilde{p}}}}:=\sum\limits_{i = 1}^m {\frac{1}{{{p_i}}}}- \eta \in (0,1]$.
Set \(\mathbf{t}\) and \(\mathbf{k}\) be multi-indices satisfying \(0\le \mathbf{t} \le \mathbf{k}\) with $\vec{1} \le \bf k$, and
$R := \sum\limits_{i=1}^{m} k_i  - 1$. Suppose that \(\mu_0^{-\tilde{p}'}, \omega^{-\tilde{p}'} \in A_{\tilde{p}'/s'}(X)\), \(\mu_i^{p_i}, \omega_i^{p_i} \in A_{p_i/r_i}(X)\), and the bloom type weight
$\varphi_i := \left(\omega_i/\mu_i\right)^{\frac{1}{k_i}}$, for \(i = 1,\ldots,m\).
If multi-symbol \(\mathbf{b} = (b_1, \ldots, b_{m}) \in (\BMO(X))^m\), 
then there exists a sparse family $\tilde{\mathcal{S}}$ satisfying ${\mathcal{S}} \subseteq \tilde{\mathcal{S}} \subseteq \d$, such that
\begin{align}\label{max.weight_}
    &\quad  \left\|{\mathcal A}_{\eta ,\S,\tau,\vec{r},s'}^\mathbf{b,k,t}\right\|_{\prod\limits_{i=1}^{m} L^{p_i}\left(\omega^{p_i}_i\right) \times  L^{\tilde{p}'}(\omega^{-\tilde{p}'}) \rightarrow \mathbb{R}} \lesssim    \prod_{i = 1}^m \|b_i\|_{\mathrm{BMO}_{\varphi_i}(X)}^{k_i}\cdot \mathcal{N}_{0},
    \end{align}
    where $\varpi_i:= \begin{cases}\mu_i, & i \in \mathcal{J}\\ \omega_i, & i \in \mathcal{J}^c,\end{cases}$
 and
  \begin{align*}
    \mathcal{N}_{0}
:=&\left\|\mathcal{A}_{\eta,\S,\vec{r},s'}\right\|_{\prod\limits_{i=1}^{m} L^{p_i}(\varpi^{p_i}_i)\times L^{\tilde{p}'}(\mu_0^{-\tilde{p}'})\rightarrow \mathbb{R}} \cdot \left([\omega^{-\tilde{p}'}]^{\frac{(R+1)s'+1}{2}}_{A_{\tilde{p}/s'}(X)} 
    [\mu_0^{-\tilde{p}'}]^{\frac{(R+1)s'-1}{2}}_{A_{\tilde{p}/s'}(X)}\right)^{\max\left\{ \frac{1}{s'},\frac{1}{\tilde{p}'-s'}\right\} } \\
    &\cdot \prod\limits_{i=1}^m \left([\omega_i^{p_i}]^{\frac{k_ir_i+1}{2}}_{A_{p_i/r_i}(X)} \left[\mu_i^{p_i}\right]^{\frac{k_ir_i-1}{2}}_{A_{p_i/r_i}(X)}\right)^{\max \left\{\frac{1}{r_i}, \frac{1}{p_i-r_i}\right\}}.
  \end{align*}
If we more assume that $\vec{\varpi}:=\left(\varpi_1,\cdots,\varpi_m\right) \in A_{(\vec{p},\tilde{p}),(\vec{r}, s)}$ with $(\vec{r}, s) \prec (\vec{p},\tilde{p})$, it follows from Theorem \ref{est:form} that
\begin{align*}
    &\quad \left\|{\mathcal A}_{\eta ,\S,\tau,\vec{r},s'}^\mathbf{b,k,t}\right\|_{\prod\limits_{i=1}^{m} L^{p_i}\left(\omega^{p_i}_i\right) \times  L^{\tilde{p}'}(\omega^{-\tilde{p}'}) \rightarrow \mathbb{R}} \\
    &\lesssim \prod_{i = 1}^m \|b_i\|_{\mathrm{BMO}_{\varphi_i}(X)}^{k_i} \cdot [\vec{\varpi}]^{\varTheta}_{A_{(\vec p,\tilde{p}),(\vec r,s)}(X)}
 \cdot \left([\omega^{-\tilde{p}'}]^{\frac{(R+1)s'+1}{2}}_{A_{\tilde{p}/s'}(X)} 
    [\mu_0^{-\tilde{p}'}]^{\frac{(R+1)s'-1}{2}}_{A_{\tilde{p}/s'}(X)}\right)^{\max\left\{ \frac{1}{s'},\frac{1}{\tilde{p}'-s'}\right\} } \\&\quad \cdot\prod\limits_{i=1}^m \left([\omega_i^{p_i}]^{\frac{k_ir_i+1}{2}}_{A_{p_i/r_i}(X)} \left[\mu_i^{p_i}\right]^{\frac{k_ir_i-1}{2}}_{A_{p_i/r_i}(X)}\right)^{\max \left\{\frac{1}{r_i}, \frac{1}{p_i-r_i}\right\}}.
\end{align*}
\end{theorem}

\begin{proof}
By Lemma \ref{zhang:6.1}, there exists a sparse family $\tilde{\mathcal{S}}\subseteq {\mathcal D}$
containing $\mathcal{S}$ and such that if $Q\in\tilde{\mathcal{S}}$,
then for a.e. $x\in Q$,
\[
|b_i(x)-b_{i,Q}|\leq C\sum_{P\in\tilde{\mathcal{S}},\ P\subseteq Q} \langle|b_i-b_{i,P}|\rangle_{P} \chi_{P}(x).
\]
From this, assuming that $b_i \in {\rm BMO}_{\eta_i}$, where $\eta_i$ is a weight to be chosen later, we obtain
\begin{align}\label{b.to.sum}
	|b_i(x)-b_{i,Q}|\leq C\|b_i\|_{{\rm BMO}_{\varphi_i}(X)}\sum_{P\in\tilde{\mathcal{S}},\ P\subseteq Q} \varphi_{i,P}\chi_{P}(x).
\end{align}
From \cite[(3.22)]{CenSong2412}, 
\begin{align}\label{xxing}
  \int_Q|h|\left(\sum_{P \in \tilde{\mathcal{S}}, P \subseteq Q} \varphi_{P} \chi_P\right)^l  \lesssim \int_Q A_{\tilde{\mathcal{S}},\varphi}^l (|h|).
  \end{align}
  where $A_{\tilde{\mathcal{S}}}(h)=\sum_{Q \in \tilde{\mathcal{S}}} h_Q \chi_Q$, $A_{\tilde{\mathcal{S}}, \varphi} h=A_{\tilde{\mathcal{S}}}(h) \varphi$, and $A^l_{\mathcal{\tilde{S}},{\varphi}}$ denotes the operator $A_{\mathcal{\tilde{S}},{\varphi}}$ iterated $l$ times.

Below, we will use the \textbf{\textsf{maximal weights method}}. Set the maximal Bloom-type weight $\varphi_0(x) := \max\limits_{i \in \{1,\ldots,m\}} \{\varphi_i(x)\}$, and let $\tau'=\mathcal{J}$ and $\tau=\{1,\cdots,m\}$. It follows from Lemma \ref{pre_1} that
\begin{align}\notag
&\quad\mathcal{A}_{\eta, \mathcal{S}, \vec{r}, s^{\prime}}^{\mathbf{b}, \mathbf{k}, \mathbf{t}}(\vec{f}, g)\\
\notag & \le \sum_{Q\in\mathcal{S}} \mu(Q)^{\eta+1}  \sum\limits_{\mathcal{J} \subseteq \{1,\ldots,m\}} \left( \prod\limits_{i \in \mathcal{J}} \left\langle \left| b_{i} - b_{i,Q} \right|^{k_{i}} f_{i} \right\rangle_{r_{i},Q} 
\left\langle \prod\limits_{j \in \mathcal{J}^c} \left| b_{j} - b_{j,Q} \right|^{k_{j}} g \right\rangle_{s',Q} \cdot \prod_{j \in \mathcal{J}^c}\left\langle f_{j}\right\rangle_{r_{j}, Q}\right)\\ 
\notag
& = \sum_{Q\in\mathcal{S}}\mu(Q)^{\eta+1}\sum\limits_{\mathcal{J} \subseteq \{1,\ldots,m\}}\prod_{i \in \mathcal{J}}\left( \frac{1}{\mu(Q)}\int_{Q}\left|f_i\right|^{r_i} \left|b_i - b_{i,Q}\right|^{r_ik_i}\right)^{\frac{1}{r_i}}\left(\frac{1}{\mu(Q)}\int_{Q}\prod_{j \in \mathcal{J}^c} \left| b_{j} - b_{j,Q} \right|^{s'k_{j}} g^{s'}\right)^{\frac{1}{s'}} \prod_{j \in \mathcal{J}^c}\left\langle f_{j}\right\rangle_{r_{j}, Q}\\ \notag
& \overset{\eqref{b.to.sum}}{\lesssim} \prod_{i=1}^{m} \|b_i\|^{k_i}_{\rm{BMO}_{\varphi_i}} \sum_{Q\in\mathcal{\tilde{S}}}\mu(Q)^{\eta+1}\left(\int_{Q}|g|^{s'}\left(\sum_{P\in\tilde{\mathcal{S}},\ P\subseteq Q}\varphi_{0,P}\chi_{P}\right)^{(R+1)s'}\right)\\ \notag
 &\quad \cdot \sum\limits_{\mathcal{J} \subseteq \{1,\ldots,m\}}\prod_{i \in \mathcal{J}}\left(\frac{1}{\mu(Q)}\int_{Q}\left(\sum_{P\in\tilde{\mathcal{S}},\ P\subseteq Q}\varphi_{i,P}\chi_{P}\right)^{r_ik_i}|f_i|^{r_i}\right)\prod_{j \in \mathcal{J}^c}\left\langle f_{j}\right\rangle_{r_{j}, Q}.
\end{align}
In order to bound the right-hand side as above, \eqref{xxing} implies
\begin{align}
&\quad\mathcal{A}_{\eta, \mathcal{S}, \vec{r}, s^{\prime}}^{\mathbf{b}, \mathbf{k}, \mathbf{t}}(\vec{f}, g)\notag\\ \notag
&\lesssim \prod_{i=1}^{m} \|b_i\|^{k_i}_{\rm{BMO}_{\varphi_i}}\sum_{Q\in \mathcal{\tilde{S}}} \mu(Q)^{\eta + 1} \left\langle \left[A^{(R+1)s'}_{\tilde{\mathcal{S}},\varphi_0}(|g|^{s'})\right]^{\frac{1}{s'}}\right\rangle_{s',Q}  \sum\limits_{\mathcal{J} \subseteq \{1,\ldots,m\}} \prod_{i \in \mathcal{J}}
\left\langle \left[A^{k_ir_i}_{\tilde{\mathcal{S}},\varphi_i}(|f_i|^{r_i})\right]^{\frac{1}{r_i}}\right\rangle_{r_i,Q}\prod_{j \in \mathcal{J}^c}\left\langle f_{j}\right\rangle_{r_{j}, Q}
\\ \notag
&:= \prod_{i=1}^{m} \|b_i\|^{k_i}_{\rm{BMO}_{\varphi_i}} \sum\limits_{\mathcal{J} \subseteq \{1,\ldots,m\}} \left|\mathcal{A}_{\eta,\S,\vec{r},s'}\left(\left(\left[A^{k_ir_i}_{\tilde{\mathcal{S}},\varphi_i}(|f_i|^{r_i})\right]^{\frac{1}{r_i}}\right)_{i \in \mathcal{J}},\left(f_j\right)_{j \in \mathcal{J}^c},\left[A^{(R+1)s'}_{\tilde{\mathcal{S}},\varphi_0}(|g|^{s'})\right]^{\frac{1}{s'}}\right)\right|\\ \notag 
&\lesssim \prod_{i=1}^{m} \|b_i\|^{k_i}_{\rm{BMO}_{\varphi_i}} \left\|\mathcal{A}_{\eta,\S,\vec{r},s'}\right\|_{\prod\limits_{i=1}^{m} L^{p_i}(\varpi^{p_i}_i)\times L^{\tilde{p}'}(\mu_0^{-\tilde{p}'})\rightarrow \mathbb{R}} \left\|\left[A^{(R+1)s'}_{\tilde{\mathcal{S}},\varphi_0}(|g|^{s'})\right]^{\frac{1}{s'}}\right\|_{L^{\tilde{p}}(\mu_0^{-\tilde{p}'})}\\ 
&\cdot \sum\limits_{\mathcal{J} \subseteq \{1,\ldots,m\}}\prod_{i \in \mathcal{J}} \left\|\left[A^{k_ir_i}_{\tilde{\mathcal{S}},\varphi_i}(|f_i|^{r_i})\right]^{\frac{1}{r_i}}\right\|_{L^{p_i}(\mu_i^{p_i})} \prod_{j \in \mathcal{J}} \|f_j\|_{L^{p_j}(\mu_j^{p_j})}.\label{bloom:eq:1}
\end{align}

From Lemma \ref{00we.xs}, it can be deduced that
\begin{align*}
&\quad\left\|\left[A^{k_ir_i}_{\tilde{\mathcal{S}},\varphi_i}(|f_i|^{r_i})\right]^{\frac{1}{r_i}}\right\|_{L^{p_i}(\mu_i^{p_i})} = \left\|A^{k_ir_i}_{\tilde{\mathcal{S}},\varphi_i}(|f_i|^{r_i})\right\|^{\frac{1}{r_i}}_{L^{p_i/r_i}(\mu_i^{p_i})} \\
& \lesssim\left(\left[\mu_i^{p_i} \varphi_i^{p_i /r_i}\right]_{A_{p_i/r_i}}\right)^{\max \left\{\frac{1}{r_i}, \frac{1}{p_i-r_i}\right\}}\left\|A_{\tilde{\mathcal{S}}, \eta}^{k_ir_i-1}(|f_i|^{r_i})\right\|_{L^{p_i}\left(\mu_i^{p_i/r_i} \varphi_i^{p_i /r_i}\right)} \\
& \lesssim\left(\left[\mu_i^{p_i} \varphi_i^{p_i /r_i}\right]_{A_{p_i/r_i}}\left[\mu_i^{p_i} {\varphi_i}^{2 p_i /r_i}\right]_{A_{p_i/r_i}} \ldots\left[\mu_i^{p_i} {\varphi_i}^{k_ir_ip_i/r_i}\right]_{A_{p_i/r_i}}\right)^{\max \left\{\frac{1}{r_i}, \frac{1}{p_i-r_i}\right\}}\|f_i\|_{L^{p_i}\left(\mu_i^{p_i} {\varphi_i}^{k_i p_i}\right)}\\
&= \left([\omega_i^{p_i}]_{A_{p_i/r_i}} \prod_{j=1}^{k_ir_i-1}\left[(\mu_i^{p_i})^{1-\frac{j}{k_ir_i}} (\omega_i^{p_i})^{\frac{j}{k_ir_i}}\right]_{A_{p_i/r_i}}\right) \|f_i\|_{L^{p_i}\left(\mu_i^{p_i} {\varphi_i}^{k_i p_i}\right)}, 
\end{align*}
where we wrote $\varphi_i := \left(\omega_i/\mu_i\right)^{\frac{1}{k_i}}$, for \(i = 1,\ldots,m\).
Since $r_i < p_i$, by Hölder's inequality, 
\begin{align*}
	\left([\omega_i^{p_i}]_{A_{p_i/r_i}} \prod_{j=1}^{k_ir_i-1}\left[(\mu_i^{p_i})^{1-\frac{j}{k_ir_i}} (\omega_i^{p_i})^{\frac{j}{k_ir_i}}\right]_{A_{p_i/r_i}}\right) \lesssim \left([\omega_i^{p_i}]^{\frac{k_ir_i+1}{2}}_{A_{p_i/r_i}(X)} \left[\mu_i^{p_i}\right]^{\frac{k_ir_i-1}{2}}_{A_{p_i/r_i}(X)}\right),
\end{align*}
Thus, 
\begin{align}
&\left\|\left[A^{k_ir_i}_{\tilde{\mathcal{S}},\varphi_i}(|f_i|^{r_i})\right]^{\frac{1}{r_i}}\right\|_{L^{p_i}(\mu_i^{p_i})} 
\lesssim \left([\omega_i^{p_i}]^{\frac{k_ir_i+1}{2}}_{A_{p_i/r_i}(X)} \left[\mu_i^{p_i}\right]^{\frac{k_ir_i-1}{2}}_{A_{p_i/r_i}(X)}\right)^{\max \left\{\frac{1}{r_i}, \frac{1}{p_i-r_i}\right\}}\|f_i\|_{L^{p_i}\left(\mu_i^{p_i} {\varphi_i}^{k_i p_i}\right)}.\label{bloom:eq:2}
\end{align}
Analogously, set $\varphi_0 := \left(\omega^{}/\mu_0\right)^{\frac{-1}{(R+1)}}$ and we obtain
\begin{align}
\left\|\left[A^{(R+1)s'}_{\tilde{\mathcal{S}},\varphi_0}(|g|^{s'})\right]^{\frac{1}{s'}}\right\|_{L^{\tilde{p}'}(\mu_0^{-\tilde{p}'})} \lesssim \left([\omega^{-\tilde{p}'}]^{\frac{(R+1)s'+1}{2}}_{A_{\tilde{p}/s'}(X)} 
[\mu_0^{-\tilde{p}'}]^{\frac{(R+1)s'-1}{2}}_{A_{\tilde{p}/s'}(X)}\right)^{\max\left\{ \frac{1}{s'},\frac{1}{\tilde{p}'-s'}\right\} } \|g\|_{L^{\tilde{p}'}(\omega^{-\tilde{p}'})}.\label{bloom:eq:3}
\end{align}

Combining \eqref{bloom:eq:1}, \eqref{bloom:eq:2}, \eqref{bloom:eq:3}, and the fact that 
$$\sum_{\mathcal{J} \subseteq \{1,\cdots,m\}}\prod\limits_{i \in \mathcal{J}} \left([\omega_i^{p_i}]^{\frac{k_ir_i+1}{2}}_{A_{p_i/r_i}(X)} \left[\mu_i^{p_i}\right]^{\frac{k_ir_i-1}{2}}_{A_{p_i/r_i}(X)}\right)^{\max \left\{\frac{1}{r_i}, \frac{1}{p_i-r_i}\right\}} \le \prod\limits_{i=1}^m \left([\omega_i^{p_i}]^{\frac{k_ir_i+1}{2}}_{A_{p_i/r_i}(X)} \left[\mu_i^{p_i}\right]^{\frac{k_ir_i-1}{2}}_{A_{p_i/r_i}(X)}\right)^{\max \left\{\frac{1}{r_i}, \frac{1}{p_i-r_i}\right\}},$$
we have 
\begin{align*}
    &\quad  \left\|{\mathcal A}_{\eta ,\S,\tau,\vec{r},s'}^\mathbf{b,k,t}\right\|_{\prod\limits_{i=1}^{m} L^{p_i}\left(\omega^{p_i}_i\right) \times  L^{\tilde{p}'}(\omega^{-\tilde{p}'}) \rightarrow \mathbb{R}} \lesssim    \prod_{i = 1}^m \|b_i\|_{\mathrm{BMO}_{\varphi_i}(X)}^{k_i}\cdot \mathcal{N}_{0},
    \end{align*}
where $\varpi_i:= \begin{cases}\mu_i, & i \in \mathcal{J};\\ \omega_i, & i \in \mathcal{J}^c,\end{cases}$
 and
  \begin{align*}
    \mathcal{N}_{0}
:=&\left\|\mathcal{A}_{\eta,\S,\vec{r},s'}\right\|_{\prod\limits_{i=1}^{m} L^{p_i}(\varpi^{p_i}_i)\times L^{\tilde{p}'}(\mu_0^{-\tilde{p}'})\rightarrow \mathbb{R}} \cdot \left([\omega^{-\tilde{p}'}]^{\frac{(R+1)s'+1}{2}}_{A_{\tilde{p}/s'}(X)} 
    [\mu_0^{-\tilde{p}'}]^{\frac{(R+1)s'-1}{2}}_{A_{\tilde{p}/s'}(X)}\right)^{\max\left\{ \frac{1}{s'},\frac{1}{\tilde{p}'-s'}\right\} } \\
    &\cdot \prod\limits_{i=1}^m \left([\omega_i^{p_i}]^{\frac{k_ir_i+1}{2}}_{A_{p_i/r_i}(X)} \left[\mu_i^{p_i}\right]^{\frac{k_ir_i-1}{2}}_{A_{p_i/r_i}(X)}\right)^{\max \left\{\frac{1}{r_i}, \frac{1}{p_i-r_i}\right\}}. \qedhere
  \end{align*}
\end{proof}

\subsection{Local decay-type estimates.}~~

The main purpose of this subsection is to prove Theorem~\ref{thm:local}. We first define
\begin{align*}
\widetilde{\mathcal{M}}^{\B}_{\eta, L(\log L)^{\vec r}} (\vec{f}) 
:= \sup_{B \in \B: B \ni x} \mu(B)^{\eta}\prod_{i=1}^m \||f_i|^{r_i}\|_{L(\log L)^{r_i}, B}^{\frac1{r_i}}. 
\end{align*}
We establish the Coifman-Fefferman inequality as follows.
\begin{lemma}\label{lem:Ub}
Under the assumption of Theorem \ref{thm:local}. For any nonnegative functions $f_i \in L_b^{r_i}(X)$ with $\supp f_i \subseteq B_0$, $i=1,\cdots,m$, and for all $p \in (1,\infty)$ , $\omega \in A_{p,\mathcal{B}}$, 
\begin{align}\label{TbM-1} 
\| T_{\vec \eta}(\vec{f}) \|_{L^1(B_0, \omega)}
\lesssim [\omega]_{A_{p, \B}} \| \mathcal{M}^{\B}_{\eta,\vec{r}}(\vec{f}) \|_{L^1(B_0^{3}, \omega)}. 
\end{align}
Suppose that ${\bf t,k }\in \N_0^m$ with \(\mathbf{0} \le \mathbf{t} \leq \mathbf{k}\), and multi-symbol $\mathbf{b}=\left(b_1, \ldots, b_m\right) \in\left(L_{l o c}^1(X)\right)^m$. If $A_{\infty, \B}$ satisfies the sharp reverse H\"{o}lder inequality, then
{\begin{align}\label{TbM-2} 
\|T_{\vec \eta}^{{\bf b, k}} (\vec{f})\|_{L^1(B_0, \omega)}
\lesssim \|\boldsymbol{b}\|^{\bf k} [\omega]_{A_{\infty}}^{\kappa}  
\|\widetilde{\mathcal{M}}^{\B}_{\eta, L(\log L)^{\vec r}}(\vec{f})\|_{L^1(B^3_0, \omega)}
\end{align}}
with $\kappa: = \sum \limits_{i=1}^{m}k_i -t_i$ and $\|\boldsymbol{b}\|^{\bf k} :=\prod\limits_{i =1}^m\|b_i\|_{\rm{BMO}}^{k_i}$,
where the implicit constant is independent of $B_0$, $[\omega]_{A_{p, \B}}$, and $\vec{f}$.
\end{lemma}

Next, we prove Theorem \ref{thm:local} as follows.

\begin{proof}[\bf Proof of Theorem \ref{thm:local}]~~

For $q_0 > 1$ and for any nonnegative function $ h \in L^{q_0}(\mu)$, the Rubio de Francia operator is defined by
\begin{align*}
\mathcal{R}h 
:= \sum_{k=0}^{\infty} \frac{1}{2^k} \frac{M_{\B}^{k}h}{\|M_{\B}\|^k_{L^{q_0}(\mu)\to L^{q_0}(\mu)}}. 
\end{align*}
Then one can check that 
\begin{equation}\label{eq:RdF}
\begin{aligned}
&h \leq \mathcal{R}h, \quad 
\|\mathcal{R}h\|_{L^{q_0}(\mu)} \leq 2 \|h\|_{L^{q_0}(\mu)}, 
\\ 
&\text{and}\quad  
[\mathcal{R}h]_{A_{1, \B}} \leq 2 \|M_{\B}\|_{L^{q_0}(\mu) \to L^{q_0}(\mu)} \leq 2\C_0^{\frac{1}{q_0}}, 
\end{aligned}
\end{equation}
where the last estimates follows from \cite[(2.9)]{Cao23}. For $p,q > 1$ (to be chosen), the Riesz theorem guarantees the existence of that $0 \leq h \in L^{q'}(B, \mu)$ with $\|h\|_{L^{q^{\prime}}(B, \mu)} = 1$ satisfying 
\begin{align}\label{Iqq}
\mathcal{L}(t)^{\frac1q} 
&:= \mu \left(\left\{x \in B:  \left\|T_{\vec \eta}^{{\bf b, k}}(\vec f)(x)\right\|_{\mathbb{V}} > t \, \mathcal{M}^{\B}_{\vec \eta,\vec r}(\vec{f^{\star}})(x)  \right\}\right)^{\frac1q}
\nonumber\\
& \leq \frac{1}{t} \left\| \frac{\left\|T_{\vec{\eta}}^{{\bf b, k}}(\vec f)\right\|_{\mathbb{V}}}{\mathcal{M}^{\B}_{\vec \eta,\vec r}(\vec{f^{\star}})} \right\|_{L^q(B, \mu)}
\overset{\text{(\ref{eq:RdF})}}{\leq} \frac{1}{t} \int_B \left\|T_{\vec{\eta}}^{{\bf b, k}}(\vec f)\right\|_{\mathbb{V}}
\frac{h}{\mathcal{M}^{\B}_{\vec \eta,\vec r}(\vec{f^{\star}})} \, d\mu
\nonumber\\
& \leq \frac{1}{t} \int_B \left\|T_{\vec{\eta}}^{{\bf b, k}}(\vec f)\right\|_{\mathbb{V}}
\frac{\mathcal{R}h}{\mathcal{M}^{\B}_{\vec \eta,\vec r}(\vec{f^{\star}})} \, d\mu
= t^{-1} \left\|\left\| T_{\vec{\eta}}^{{\bf b, k}}(\vec f) \right\|_{\mathbb{V}}\right\|_{L^1(B, \omega)},
\end{align}
where
\[
\omega := \omega_1 \omega_2^{1-p}, \quad 
\omega_1 := \mathcal{R}h, 
\quad\text{and}\quad 
\omega_2 := \mathcal{M}^{\B}_{\vec \eta,\vec r}(\vec{f^{\star}})^{p'-1}.
\]

To establish $\omega \in A_{\infty}$, from Lemma \ref{lem:CR}, we set $r_0=\max\limits_{i = 1,\cdots,m} r_i$ and pick $p>1+(m - \eta)/r_0$ (equivalently, $\frac{p'-1}{r_0}<\frac1{m - \eta}$), it follows that
\begin{align}\label{eq:MMf}
	[\omega_2]_{A_{1, \B}}
\leq \big[\big(\mathcal{M}^{\B}_{\vec \eta}(|f_1^{\star}|^{r_0}, \ldots, |f_m^{\star}|^{r_0}) \big)^{\frac{p'-1}{r_0}}\big]_{A_{1, \B}},
\end{align}
thus, $[\omega_2]_{A_{1, \B}} \leq C_{m,\eta}$.
Then \eqref{eq:RdF} and \eqref{eq:MMf} imply that $\omega=\omega_1 \omega_2^{1-p} \in A_{p, \B}$ from
\begin{align}\label{wapcc}
[\omega]_{A_{p, \B}} 
\leq [\omega_1]_{A_{1, \B}} [\omega_2]_{A_{1, \B}}^{p-1}
\leq 2 \boldsymbol{C}_0^{\frac{1}{q'}} C_m^{p-1}
\leq 2 \boldsymbol{C}_0 C_m^{p-1}. 
\end{align}

For any ball $B' \in \mathcal{B}$ and locally integrable function $f$, we have the following property
\begin{align}\label{fLL}
\|f_i\|_{L(\log L)^{r_i}, B'} 
\le \|f_i\|_{L(\log L)^{\lfloor r_i \rfloor}, B'} 
\lesssim \fint_{B'} M^{\lfloor r_i \rfloor}(f \mathbf{1}_{B'}) \, d\mu.
\end{align}
Therefore, from \eqref{Iqq}, \eqref{wapcc}, \eqref{fLL}, and Lemma \ref{lem:Ub}, it follows that 
\begin{align*}
\mathcal{L}(t)^{\frac1q} 
& \le c_0 \, t^{-1} \|\boldsymbol{b}\|^{\bf k} [\omega]_{A_{p,\B}}^{\kappa+1}  
\|\widetilde{\mathcal{M}}^{\B}_{\eta, L(\log L)^{\vec r}}(\vec{f})\|_{L^1(B^{(3)}, \omega)} 
\\
& \le c_0 \, t^{-1} \|\boldsymbol{b}\|^{\bf k} [\omega]_{A_{p,\B}}^{\kappa+1} 
\|\mathcal{M}^{\B}_{\vec \eta,\vec r}(\vec{f^{\star}})\|_{L^1(B^{(3)}, \omega)} 
\\ 
& = c_0 \, t^{-1} \|\boldsymbol{b}\|^{\bf k} [\omega]_{A_{p,\B}}^{\kappa+1} \|\mathcal{R}h\|_{L^1(B^{(3)}, \mu)}
\\ 
&\leq c_0 \, t^{-1} \|\boldsymbol{b}\|^{\bf k} [\omega]_{A_{p,\B}}^{\kappa+1}  
\|\mathcal{R}h\|_{L^{q'}(B^{(3)}, \mu)}  \mu(B^{(3)})^{\frac1q}
\\ 
& \le c_0 \, t^{-1} \|\boldsymbol{b}\|^{\bf k} \mu(B)^{\frac1q}
\le c_0 \, t^{-1} \|\boldsymbol{b}\|^{\bf k} q^{\kappa+1}  \mu(B)^{\frac1q}, 
\end{align*}
where $c_0 > 1$ is a constant that may change from line to line but remains independent of $q$.

Consequently, for all $t > t_0 := c_0 e \|\boldsymbol{b}\|^{\mathbf{k}}$, selecting
\[
q = \left(\frac{t}{c_0 e \|\boldsymbol{b}\|^{\mathbf{k}}}\right)^{\frac{1}{\kappa + 1}} > 1
\]
yields the estimate
\begin{align}\label{It-1}
\mathcal{L}(t) &\leq e^{-q} \mu(B) = \exp\left(-\left(\frac{t}{c_0 e \|\boldsymbol{b}\|^{\mathbf{k}}}\right)^{\frac{1}{\kappa + 1}}\right) \mu(B) \nonumber \\
&=: \exp\left(-\left(\frac{\varDelta t}{\|\boldsymbol{b}\|^{\mathbf{k}}}\right)^{\frac{1}{\kappa + 1}}\right) \mu(B),
\end{align}
where $\varDelta := (c_0 e)^{-1}$ depends only on $m$, $p$, and $\mathcal{C}_0$. For $0 < t \leq t_0$, we have
\begin{align}\label{It-2}
\mathcal{L}(t) &\leq e \cdot \exp\left(-\left(\frac{\gamma t}{\|\boldsymbol{b}\|^{\mathbf{k}}}\right)^{\frac{1}{\kappa + 1}}\right) \mu(B).
\end{align}
Combining \eqref{It-1} and \eqref{It-2}, We have that
\[
\mu\left(\left\{x \in B : \|T_{\vec{\eta}}^{\mathbf{b,k}}(\vec{f})(x)\|_{\mathbb{V}} > t \mathcal{M}_{\vec{\eta},\vec{r}}^{\mathcal{B}}(\vec{f}^\star)(x)\right\}\right) \leq e \cdot e^{-\left(\frac{\varDelta t}{\|\boldsymbol{b}\|^{\mathbf{k}}}\right)^{\frac{1}{\kappa + 1}}} \mu(B)
\]
for all $t > 0$.
\end{proof}

\begin{proof}[Proof of Lemma \ref{lem:Ub}]
    We focus on proving \eqref{TbM-2}, since the estimate \eqref{TbM-1} can be implied through \eqref{TbM-2}.
  Fix $B_0 \in \mathcal{B}$ with $\cup_{i=1}^m\supp(f_i) \subseteq B_0$. We claim that there exist a sparse family $\S$ such that for a.e. $x \in B_0$, 
\begin{align}
&\label{Tf-sparse}  \left\|T_{\vec{\eta}}^{{\bf b, k}}(\vec f)(x)\right\|_{\mathbb{V}} 
\lesssim \C \sum_{\mathbf{0} \le \mathbf{t} \leq \mathbf{k}} {\mathscr A}_{\eta ,\S,{\vec r}}^\mathbf{b,k,t}(\vec{f})(x), 
\\ 
\label{BBSS} & \text{and} \quad B \subseteq B_0^{3} \, \text{ for all } \, B \in \S_1 \cup \S_2. 
\end{align}

Since $p \in (1, \infty)$ and $\omega \in A_{p, \B}$. By Lemma \ref{lem:Ainfty}, for any $\alpha \in (0, 1)$, for any $B \in \B$, and for any measurable subset $E \subseteq B$, 
\begin{align}\label{EBAp}
\mu(E) \ge \alpha \mu(B) \quad\Longrightarrow \quad 
\omega(E) \ge \alpha^p \, [\omega]_{A_{p, \B}}^{-1} \omega(B).
\end{align} 
Let $\S$ be an $\delta$-sparse family with $\delta \in (0, 1)$ such that $B \subseteq B_0^{3}$ for all $B \in \S$. Then, there exists a pairwise disjoint family $\{E_B\}_{B \in \S}$ such that $E_B \subseteq B$ and $\mu(E_B) \ge \delta \mu(B)$ for all $B \in \S$.

It follows from this and \eqref{EBAp} that
\begin{equation}\label{0608_1}
\omega(B) \le \delta^{-p} [\omega]_{A_{p, \B}} \omega(E_B) \quad\text{ for all } B \in \S, 
\end{equation}

Set $\kappa: = \sum \limits_{i=1}^{m}k_i -t_i$ and $\|\boldsymbol{b}\|^{\bf k} :=\prod\limits_{i =1}^m\|b_i\|_{\rm{BMO}}^{k_i}$. From Lemma \ref{lem:PhiPhi}, \eqref{e:fw}, and \eqref{0608_1}, 
\begin{align}
&\quad \left\|{\mathscr A}_{\eta ,\S,{\vec r}}^\mathbf{b,k}(\vec f)\right\|_{L^1(B_0, \omega)}\notag \\
 &\le \sum_{B \in \S} \mu(B)^{\eta + 1} \left(\fint_B \prod_{i=1}^m |b_i - b_{i,B}|^{k_i - t_i} \omega  d \mu \right) \prod_{i =1}^m \left\langle (b_i - b_{i, B})^{t_i} f_i \right\rangle_{r_i, B} 
\nonumber \\
&\lesssim \sum_{B \in \S} \mu(B)^{\eta + 1} \|\omega\|_{L(\log L)^{\kappa}, B} \prod_{i=1}^{m}\left( \|(b_i - b_{i, B})^{k_i - t_i}\|^{\frac{1}{k_i - t_i}}_{\exp L^{\frac{1}{k_i - t_i}}, B}\right)^{k_i - t_i} \nonumber\\
&\quad\quad \quad \times\left( \|(b_i - b_{i, B})^{ t_i}\|^{\frac{1}{ t_i}}_{\exp L^{\frac{1}{ t_i}}, B}\right)^{ t_i} \||f_i|^{r_i}\|_{L(\log L)^{r_i}, B}^{\frac1r}
\nonumber \\
&=\sum_{B \in \S} \mu(B)^{\eta + 1} \|\omega\|_{L(\log L)^{\kappa}, B} \prod_{i=1}^{m} \|b_i - b_{i, B}\|^{k_i - t_i}_{\exp L, B} 
\|b_i\|^{t_i}_{\rm BMO} \langle f_i \rangle_{r_i,B}
\nonumber \\
&\lesssim \|\boldsymbol{b}\|^{\bf k} [\omega]_{A_{\infty}}^{\kappa} 
\sum_{B \in \S} \mu(B)^{\eta}  
 \||f_i|^{r_i}\|_{L(\log L)^{r_i}, B}^{\frac1{r_i}} \, \omega(B)
\nonumber \\
&\lesssim \|\boldsymbol{b}\|^{\bf k} [\omega]_{A_{p,\B}}^{\kappa+1}  
\sum_{B \in \S} \inf_{x \in B} \widetilde{\mathcal{M}}^{\B}_{\eta, L(\log L)^{\vec r}}(\vec{f})(x)  \, \omega(E_B)
\nonumber \\
&\le \|\boldsymbol{b}\|^{\bf k} [\omega]_{A_{p,\B}}^{\kappa+1}  
\sum_{B \in \S} \int_{E_B} \widetilde{\mathcal{M}}^{\B}_{\eta, L(\log L)^{\vec r}}(\vec{f})  \, \omega \, d\mu
\nonumber \\ 
&\le \|\boldsymbol{b}\|^{\bf k} [\omega]_{A_{p,\B}}^{\kappa+1}  
\|\widetilde{\mathcal{M}}^{\B}_{\eta, L(\log L)^{\vec r}}(\vec{f})\|_{L^1(B^3_0, \omega)}, \label{eq:ASB}
\end{align}
where we used the disjointness of $\{E_B\}_{B \in \S}$ and that $B \subseteq B^3_0$ for all $B \in \S$. 
Note that the implicit constants above do not depend on $\S$. Consequently, combining \eqref{Tf-sparse}, \eqref{BBSS}, and \eqref{eq:ASB} yields \eqref{TbM-2}.

Now only the proofs for \eqref{Tf-sparse} and \eqref{BBSS} remain. From $\bigcup\limits_{i=1}^m\supp(f_i) \subseteq B_0$ and the proof of Theorem \ref{S.d.main}, it derives that \eqref{Tf-sparse}.
To prove \eqref{BBSS}, we observe that 
\begin{align}\label{B2B}
B^{2} \subseteq B^{\dagger}_0 \quad\text{ for all } B \in \G(B_0) \setminus \{B_0\}.
\end{align}

By \eqref{GAA-4}, we have $B^2 \subseteq B_0^\dagger$ for all $B \in \mathcal{F}(B_0)$. Given $B_2 \in \mathcal{F}_2(B_0)$, \eqref{FFF} yields $B_2 \in \mathcal{F}(B_1)$ for some $B_1 \in \mathcal{F}(B_0)$, and thus $B_2^2 \subseteq B_1^\dagger \subset B_1^2 \subseteq B_0^\dagger$. Iterating this argument shows $B_k^2 \subseteq B_0^\dagger$ for all $B_k \in \mathcal{F}_k(B_0)$ and $k \geq 1$, which with \eqref{generation} proves \eqref{B2B}. 
Recalling that $\mathcal{S}_1 \cup \mathcal{S}_2 \subseteq \mathcal{D} \subseteq \mathcal{G}(B_0)$, from \eqref{DHG}, we conclude \eqref{BBSS} via \eqref{B2B}.
\end{proof}

\begin{lemma}\label{lem:CR}
Let $(X,\mu,\mathcal{B})$ be a ball-basis measure space. Let $m \in \mathbb{N}_0$, $\eta_i \in [0,1)$  for $i=1,\cdots,m$, and the fractional total order $\eta {\rm{ = }}\sum\limits_{i = 1}^m {{\eta _i}} \in [0,m)$. If $0<\delta<\frac{1}{m - \eta}$, then 
	\begin{align}\label{eq:CR}
	\left[(\mathcal{M}^{\B}_{\vec \eta}(\vec{f}))^{\delta}\right]_{A_{1,\B}} 
	\leq \frac{c_m}{1-(m - \eta)\delta}.
	\end{align}
	\end{lemma}
\begin{proof}
Our proof is inspired by \cite[Theorem 3.4]{GR} and \cite[Lemma~1]{OPR} and only presents the details.
For any fixed $B \in \mathcal{B}$ and $x \in B$, we split $\vec{f}=(f_1, \ldots, f_m)$ by $f_i^0 := f_i \chi_{B^{\dagger}}$ and $f_i^{\infty} := f_i \chi_{X \setminus B^{\dagger}}$ for $i=1, \ldots, m$, it follows that 
\begin{align}\label{eq:CR-1}
\fint_B \mathcal{M}_{\vec \eta}^{\B}(\vec{f})^{\delta} d\mu 
&\le \underbrace{\sum_{\alpha_1, \ldots, \alpha_m \in \{0, \infty\} \atop \exists \alpha_i = \infty} 
\fint_B \mathcal{M}^{\B}_{\vec \eta}(f_1^{\alpha_1}, \ldots, f_m^{\alpha_m})^{\delta} d\mu}_{\mathbb{M}_1} + \underbrace{\fint_B \mathcal{M}^{\B}_{\vec \eta}(\vec{f}^0)^{\delta} d\mu}_{\mathbb{M}_2}. 
\end{align}

For $\mathbb{M}_1$, let $\alpha_1, \ldots, \alpha_m \in \{0, \infty\}$ with $\alpha_i = \infty$ for some $i$.  Let $A \in \B$ and $y \in A \cap B$, then we have the following  implications
\begin{align}\label{AFB}
\int_A |f_i^{\infty}| \, d\mu \neq 0 
\quad\Longrightarrow\quad 
\mu(A) > 2\mu(B) 
\quad\Longrightarrow\quad 
x \in B \subseteq A^{\dagger}.
\end{align}
When $\mu(A) \leq 2\mu(B)$, property \eqref{list:B4} yields $A \subseteq B^{\dagger}$, and thus
\[
\int_A |f_i^\infty| \, d\mu = \int_{A \setminus B^\dagger} |f_i| \, d\mu = 0.
\]
The second implication in \eqref{AFB} is immediate from \eqref{list:B4} and $A \cap B \neq \emptyset$. From \eqref{list:B4}, we deduce that
\begin{align*}
\prod_{i=1}^m \langle |f_i^{\alpha_i}| \rangle_{\eta_i,A} 
&\leq \prod_{i=1}^{m} \left(\frac{\mu(A^{\dagger})}{\mu(A)}\right)^{1-\eta_i} \langle |f_i^{\alpha_i}| \rangle_{\eta_i,A^{\dagger}} \\
&= \left(\frac{\mu(A^{\dagger})}{\mu(A)}\right)^{m-\eta} \prod_{i=1}^{m}\langle |f_i^{\alpha_i}| \rangle_{\eta_i,A^{\dagger}} \\
&\leq \C_0^{m - \eta} \mathcal{M}^{\B}_{\vec{\eta}}\left(f_1^{\alpha_1}, \ldots, f_m^{\alpha_m}\right)(x),
\end{align*}
which yields the estimate
\begin{align}\label{eq:CR-3}
\mathcal{M}^{\B}_{\vec{\eta}}(f_1^{\alpha_1}, \ldots, f_m^{\alpha_m})(y)
\leq \C_0^{m - \eta} \mathcal{M}^{\B}_{\vec{\eta}}(\vec{f})(x), \quad \forall x, y \in B.
\end{align}

For $\mathbb{M}_2$, we set $K_0 := \prod_{i=1}^m \langle |f_i| \rangle_{\eta_i,B^{\dagger}}
\le \mathcal{M}^{\B}_{\vec \eta}(\vec{f})(x),$
and it follows from Proposition \ref{lem:M} that 
\begin{align}\label{eq:CR-2}
\fint_B \mathcal{M}^{\B}_{\vec \eta}(\vec{f}^0)^{\delta} d\mu 
&=\frac{\delta}{\mu(B)} \int_0^{\infty} \lambda^{\delta} 
\mu\left(\{x \in B: \mathcal{M}^{\B}_{\vec \eta}(\vec{f}^0)(x)^{\delta}>\lambda\}\right) \frac{d\lambda}{\lambda}
\nonumber\\ 
&\le K_0^{\delta} + \frac{\delta}{\mu(B)} \int_{K_0}^{\infty} \lambda^{\delta} 
\mu(\{x \in X: \mathcal{M}^{\B}_{\eta}(\vec{f}^0)(x)^{\delta}>\lambda\}) \frac{d\lambda}{\lambda}
\nonumber\\ 
&\le K_0^{\delta} + \frac{c \, \delta}{\mu(B^{\dagger})} \int_{K_0}^{\infty} \lambda^{\delta-\frac{1}{m - \eta}} 
\prod_{i=1}^m \|f_i^0\|_{L^1(X, \mu)}^{\frac{1}{m - \eta}} \frac{d\lambda}{\lambda}
\nonumber\\ 
&= K_0^{\delta} \left[1 + \frac{c \, \delta^2}{1-(m - \eta)\delta} \left(K_0^{-1} 
\prod_{i=1}^m \langle |f_i| \rangle_{\eta_i,B^{\dagger}} \right)^{\frac{1}{m - \eta}} \right]
\nonumber\\ 
&\le \frac{c}{1-(m - \eta)\delta} \mathcal{M}^{\B}_{\vec \eta}(\vec{f})(x)^{\delta},
\end{align}
where $c>0$ varies from line to line and is independent of $B$.

Combining estimates \eqref{eq:CR-1}, \eqref{eq:CR-2}, and \eqref{eq:CR-3}, we conclude that
\begin{align*}
\fint_B \mathcal{M}^{\B}_{\vec{\eta}}(\vec{f})^{\delta} d\mu 
\leq \frac{c}{1-m\delta} \mathcal{M}^{\B}_{\vec{\eta}}(\vec{f})(x)^{\delta}, \quad \forall x \in B.
\end{align*}
This establishes \eqref{eq:CR}. 
\end{proof}

\subsection{Mixed weak-type estimates}

\begin{lemma}\label{lem:TM} 
Under the assumption of Theorem \ref{thm:weak}, for any $p \in (0, \infty)$ and $\omega \in A_{\infty, \B}$, 
\begin{align}\label{eq:C-F}
\|{T}_{\vec \eta}(\vec{f})\|_{L^p(X, \omega)} 
\lesssim \|\mathcal{M}^{\B}_{\vec \eta, \vec r}(\vec{f})\|_{L^p(X, \omega)}.
\end{align}
\end{lemma}

\begin{proof}
We use the $A_\infty$ extrapolation theorem to prove our desired result. For a family $\mathcal{F}$ of function pairs, if there exists $p_0\in(0,\infty)$ such that for every $\omega_0\in A_{\infty,\mathcal{B}}$,
\begin{align}\label{eq:fg-some}
\|f\|_{L^{p_0}(X,\omega_0)} \leq C_1\|g\|_{L^{p_0}(X,\omega_0)}, \quad \forall(f,g)\in\mathcal{F},
\end{align}
then for all $p\in(0,\infty)$ and $\omega\in A_{\infty,\mathcal{B}}$,
\begin{align}\label{eq:fg-every}
\|f\|_{L^p(X,\omega)} \leq C_2\|g\|_{L^p(X,\omega)}, \quad \forall(f,g)\in\mathcal{F}.
\end{align}
This result follows from \cite[Theorem 3.34]{Martell2022} on $A_{\infty,\mathcal{B}}$ extrapolation in Banach function spaces, where the hypothesis is verified by estimate \cite[(2.9)]{Cao23}. 

Since $\omega \in A_{\infty, \B}$ and Lemma \ref{lem:Ainfty}, for any $\alpha \in (0, 1)$, it follows that there exists $\beta \in (0, 1)$ such that for any $B \in \B$ and any measurable subset $E \subseteq B$, 
\begin{align}\label{EBB}
\mu(E) \ge \alpha \mu(B) \quad\Longrightarrow \quad 
\omega(E) \ge \beta \omega(B).
\end{align} 
Combining this with \eqref{EBB}, we obtain that $\omega(B) \lesssim \omega(E_B) $ holds for all $B \in \mathcal{S}$. Then
\begin{align}\label{eq:ASp=1}
\left\|\mathscr{A}_{\vec \eta, \S, \vec r}(\vec{f})\right\|_{L^1(X, \omega)}
& \leq \sum_{B \in \S} \prod_{i=1}^m \langle f_i \rangle_{\eta_i,r_i,B} \omega(B)
\lesssim \sum_{B \in \S} \Big(\inf_{B} \mathcal{M}^{\B}_{\vec \eta, \vec r}(\vec{f}) \Big) \omega(E_B)
\nonumber \\ 
& \leq \sum_{B \in \S} \int_{E_B} \mathcal{M}^{\B}_{\vec \eta, \vec r}(\vec{f}) \, \omega \, d\mu
\leq \left\|\mathcal{M}^{\B}_{\vec \eta, \vec r}(\vec{f})\right\|_{L^1(X, \omega)}, 
\end{align} 
where the implicit constants are independent of $\S$. As a consequence, Theorem \ref{S.d.main} and \eqref{eq:ASp=1} imply 
\begin{align}\label{eq:TMp=1}
\|{T}_{\vec \eta}(\vec{f})\|_{L^1(X, \omega)}
\lesssim \|\mathcal{M}^{\B}_{\vec \eta, \vec r}(\vec{f})\|_{L^1(X, \omega)}. 
\end{align}
This establishes \eqref{eq:C-F} for the case $p=1$. 

For the general case $p \in (0, \infty)$, we observe that \eqref{eq:TMp=1} establishes \eqref{eq:fg-some} with $p_0 = 1$ for the pair $(f, g) = (T_{\vec{\eta}}(\vec{f}), \mathcal{M}^{\mathcal{B}}_{\vec{\eta}, \vec{r}}(\vec{f}))$. The desired estimate \eqref{eq:C-F} follows immediately from \eqref{eq:fg-every}.
\end{proof}

\begin{proof}[\bf Proof of Theorem \ref{thm:weak}]~

Our approach combines the techniques from \cite{CMP} and \cite{LOP}. Define the Rubio de Francia iteration operator
\[
\mathcal{R}h := \sum_{j=0}^{\infty} \frac{S^j h}{(2K)^j}, \quad \text{where } Sf := \frac{M_{\mathcal{B}}(f \omega)}{\omega},
\]
with $K > 0$ to be chosen specified later. This construction immediately yields the key properties
\begin{align}\label{e:R-3}
h \leq \mathcal{R}h \quad \text{and} \quad S(\mathcal{R}h) \leq 2K \mathcal{R}h.
\end{align}

We claim that there exists some $l > 1$ such that:
\begin{list}{\rm (\theenumi)}{\usecounter{enumi}\leftmargin=1.2cm \labelwidth=1cm \itemsep=0.2cm \topsep=.2cm \renewcommand{\theenumi}{\alph{enumi}}}
\item \label{RhAi}
$\mathcal{R}h \cdot \omega v^{\frac{\tilde{r}}{l}} \in A_{\infty,\mathcal{B}}$;
    \item \label{e:R-2}  $\|\mathcal{R}h\|_{L^{s',1}(X, \omega v^{\tilde{r}})} \leq 2\|h\|_{L^{s',1}(X, \omega v^{\tilde{r}})}$.
\end{list}


Observing that for any weight $\sigma$ on $(X, \mu)$ and $0<p,q<\infty$, it follows that
\begin{align}\notag
&\quad \bigg\| \frac{\|T_{\vec \eta}(\vec{f})\|_{\mathbb{V}}}{v}
\bigg\|_{L^{\tilde{r},\infty}(X, \omega v^{\tilde{r}})}^{\frac{\tilde{r}}{l}}
= \bigg\| \bigg|\frac{\|T_{\vec \eta}(\vec{f})\|_{\mathbb{V}}}{v}\bigg|^{\frac{\tilde{r}}{l}} \bigg\|_{L^{l,\infty}(X, \omega v^{\tilde{r}})}\\ \label{Tff-1}
&=\sup_{0 \le h \in L^{l',1}(X, \omega v^{\tilde{r}}) \atop \|h\|_{L^{l',1}(X, \omega v^{\tilde{r}})}=1}
\bigg|\int_{X} \left\|T_{\vec \eta}(\vec{f})\right\|_{\mathbb{V}}^{\frac{\tilde{r}}{l}} h \, \omega \, v^{\frac{\tilde{r}}{l'}} d\mu  \bigg|
\leq \sup_{0 \le h \in L^{l',1}(\omega v^{\tilde{r}}) \atop \|h\|_{L^{l',1}(X, \omega v^{\tilde{r}})}=1} 
\int_{X} \left\|T_{\vec \eta}(\vec{f})\right\|_{\mathbb{V}}^{\frac{\tilde{r}}{l}} \mathcal{R}h \, \omega \, v^{\frac{\tilde{r}}{l'}} d\mu.
\end{align}
Fix a nonnegative function $h \in L^{l',1}(X, \omega v^{\tilde{r}})$ with $\|h\|_{L^{l',1}(X, \omega v^{\tilde{r}})} = 1$. Combining \eqref{RhAi}, Lemma~\ref{lem:TM}, and Hölder's inequality yields
\begin{align}\label{Tff-2}
&\quad \int_{X} \left\|T_{\vec \eta}(\vec{f})\right\|_{\mathbb{V}}^{\frac{\tilde{r}}{l}} \mathcal{R}h \, \omega \, v^{\frac{\tilde{r}}{l'}} d\mu
\nonumber \\ 
& \lesssim \int_{X} \mathcal{M}^{\B}_{\vec \eta,\vec r}(\vec{f})^{\frac{\tilde{r}}{l}} \mathcal{R}h \, \omega v^{\frac{\tilde{r}}{l'}} d\mu
 = \int_{X} \bigg(\frac{\mathcal{M}^{\B}_{\vec \eta,\vec r}(\vec{f})}{v}\bigg)^{\frac{\tilde{r}}{l}}
\mathcal{R}h \, \omega v^{\tilde{r}} d\mu
\nonumber \\ 
& \leq \bigg\|\bigg(\frac{\mathcal{M}^{\B}_{\vec \eta,\vec r}(\vec{f})}{v}\bigg)^{\frac{\tilde{r}}{l}} \bigg\|_{L^{s,\infty}(X, \omega v^{\tilde{r}})}
\|\mathcal{R}h\|_{L^{l',1}(X, \omega v^{\tilde{r}})}
\nonumber \\ 
& \lesssim \bigg\|\frac{\mathcal{M}^{\B}_{\vec \eta,\vec r}(\vec{f})}{v}\bigg\|_{L^{\tilde{r},\infty} (X, \omega v^{\tilde{r}})}^{\frac{\tilde{r}}{l}}
\|h\|_{L^{l',1}(X, \omega v^{\tilde{r}})},
\end{align}
where \eqref{e:R-2} was used in the last inequality. Consequently, it follows from \eqref{Tff-1} and \eqref{Tff-2} that 
\begin{equation*}
\bigg\|\frac{{T}_{\vec \eta}(\vec{f})}{v}\bigg\|_{L^{\tilde{r},\infty}(X, \, \omega v^{\tilde{r}})}
\lesssim \bigg\| \frac{\mathcal{M}^{\B}_{\vec \eta, \vec r}(\vec{f})}{v}\bigg\|_{L^{\tilde{r},\infty}(X, \, \omega v^{\tilde{r}})}.
\end{equation*}

It remains to show \eqref{RhAi} and \eqref{e:R-2}. The proof follows the strategy in \cite{CMP}. For the sake of completeness we present the details. 
Since $\omega \in A_{1, \B}$ and $v^{\tilde{r}} \in A_{\infty, \B}$, we have 
\begin{equation}\label{e:S-1}
\begin{aligned}
&\|Sf\|_{L^{\infty}(X, \omega v^{\tilde{r}})}
\leq [\omega]_{A_{1, \B}} \|f\|_{L^{\infty}(X, \omega v^{\tilde{r}})} 
\\
&\text{and}\quad v^{\tilde{r}} \in A_{q_0, \B} \text{ for some } q_0>1. 
\end{aligned}
\end{equation}
It follows from Lemma \ref{lem:A1} part \eqref{list:A11} that there exist $v_1, v_2 \in A_{1, \B}$ such that 
\begin{align}\label{vvv}
v^{\tilde{r}}=v_1 v_2^{1-q_0}.
\end{align} 

Note that the second inequality in \eqref{e:R-3} implies $\mathcal{R}h \cdot \omega \in A_{1,\mathcal{B}}$. Since $\omega \in A_{1,\mathcal{B}}$ by assumption, Lemma~\ref{lem:A1} part \eqref{list:A12} yields the existence of $\varepsilon_0 \in (0,1)$ such that
\begin{align}\label{Rww}
\text{$(\mathcal{R}h \cdot \omega) v_1^{\varepsilon} \in A_{1, \B}$ \, and \, 
$w v_2^{\varepsilon} \in A_{1, \B}$ \, for any $\varepsilon \in (0, \varepsilon_0)$.}
\end{align} 
Let \( p_0 > 1 + \frac{q_0 - 1}{\varepsilon_0} \) and choose \( \varepsilon \) satisfying \( 0 < \varepsilon < \min\left\{\varepsilon_0, \frac{1}{2p_0}\right\} \). Define \( l = \left(\frac{1}{\varepsilon}\right)' > 1 \). Using \eqref{Rww} and part \eqref{list:A11} of Lemma \ref{lem:A1}, we obtain \( \omega v_2^{\frac{q_0 - 1}{p_0 - 1}} \in A_{1, \mathcal{B}} \), and
\begin{align}\label{wv22}
\mathcal{R}h \cdot \omega v^{\tilde{r}} = \left[(\mathcal{R}h \cdot \omega) v_1^\varepsilon\right] v_2^{1 - (q_0 - 1)\varepsilon - 1} \in A_{(q_0 - 1)\varepsilon + 1, \mathcal{B}}.
\end{align}
This  \eqref{RhAi}. Furthermore, combining \eqref{vvv} with \eqref{wv22}, part \eqref{list:A11} of Lemma \ref{lem:A1} yields
\begin{align}\label{wpv}
\omega^{1 - p_0} v^{\tilde{r}} = v_1 \left(\omega v_2^{\frac{q_0 - 1}{p_0 - 1}}\right)^{1 - p_0} \in A_{p_0, \mathcal{B}}.
\end{align}

From  \eqref{wpv} and Proposition \ref{lem:M_0}, it can be deduced that
\begin{align}\label{e:S-2}
\|Sf\|_{L^{p_0}(X, \omega v^{\tilde{r}})} 
= \|M_{\mathcal{B}}(f w)\|_{L^{p_0}(X, \omega^{1-p_0} v^{\tilde{r}})} 
\leq c_1 \|f\|_{L^{p_0}(X, \omega v^{\tilde{r}})}.
\end{align}  
Next, applying the Marcinkiewicz interpolation theorem from \cite[Proposition A.1]{CMP}---valid in general measure spaces---along with \eqref{e:S-1} and \eqref{e:S-2}, we conclude that \( S \) is bounded on \( L^{p,1}(X, \omega v^{\tilde{r}}) \) for all \( p \geq p_0 \), with  
\begin{align*}
K(p) = 2^{\frac{1}{p}} \bigg[ c_1 \bigg( \frac{1}{p_0} - \frac{1}{p} \bigg)^{-1} + c_2 \bigg], \quad \text{where } c_2 := [\omega]_{A_{1, \mathcal{B}}}.
\end{align*}  
Since \( K(p) \) is decreasing in \( p \), it follows that for \( p \geq 2p_0 \),  
\begin{equation}\label{eq:Lp1}
\|Sf\|_{L^{p,1}(X, \omega v^{\tilde{r}})} \leq K \|f\|_{L^{p,1}(X, \omega v^{\tilde{r}})}, \quad K := 4p_0(c_1 + c_2).
\end{equation}  
Finally, since \( l' > 2p_0 \), applying \eqref{eq:Lp1} yields \eqref{e:R-2}. This completes the proof.
\end{proof}

\section{\bf Appendix}\label{Appendix A}

\subsection{Muckenhoupt weights}~~

Let \((X, \mu, \mathcal{B})\) be a ball-basis  measure space. A {weight} on \((X,\mu)\) is any measurable \(\omega\) with \(0<\omega(x)<\infty\) for \(\mu\)-a.e.\ \(x\in\bigcup_{B\in\B}B\). We define the {\tt Muckenhoupt class} $A_{p, \B}$
\begin{equation*}
[\omega]_{A_{p,\B}}:=\sup_{B \in \B} \bigg(\fint_{B} \omega\, d\mu \bigg) 
\bigg(\fint_{B}\omega^{1-p'}\, d\mu \bigg)^{p-1}<\infty,
\end{equation*} 
where \(1<p<\infty\) and \(1/p+1/p'=1\).
As for the case $p=1$, we say that $\omega\in A_{1,\B}$ if  
\begin{equation*}
[\omega]_{A_{1,\B}} := \|(M_{\B}\omega)\,\omega^{-1}\,\mathbf{1}_{X_\B}\|_{L^{\infty}(X, \mu)} <\infty.
\end{equation*}
Finally we define
\[
A_{\infty,\B}
=\bigcup_{p\ge1}A_{p,\B}, 
\quad
[\omega]_{A_{\infty,\B}}
=\inf_{\omega \in A_{p,\B}}[\omega]_{A_{p,\B}}.
\]

The following conclusions all follow from \cite{GR, ST}.






\begin{lemma}\label{lem:A1}
Let $(X,\mu,\mathcal{B})$ is a ball-basis measure space. Then the following hold: 
\begin{list}{\textup{(\theenumi)}}{\usecounter{enumi}\leftmargin=1cm \labelwidth=1cm \itemsep=0.2cm 
			\topsep=.2cm \renewcommand{\theenumi}{\roman{enumi}}}
			
\item\label{list:A11} Given $p \in (1,\infty)$, $w \in A_{p, \B}$ if and only if there exists  $w_1, w_2 \in A_{1, \B}$ such that $w=w_1 w_2^{1-p}$.


 
\item\label{list:A12} For any $u \in A_{1, \B} \cap RH_{s, \B}$ for some $s \in (1, \infty)$, there exists $\varepsilon_0 \in (0, 1)$ such that $u v^{\varepsilon} \in A_{1, \B}$ for all $v \in A_{1, \B}$ and $\varepsilon \in(0, \varepsilon_0)$.

\end{list}
\end{lemma}

\begin{lemma}\label{lem:Ainfty}
Let $(X, \mu)$ be a measure space with a basis $\B$. Let $w$ be a weight on $(X, \mu)$. Consider the following conditions: 
\begin{list}{\rm (\theenumi)}{\usecounter{enumi}\leftmargin=1.2cm \labelwidth=1cm \itemsep=0.2cm \topsep=.2cm \renewcommand{\theenumi}{\alph{enumi}}}

\item\label{Ai-1} $w \in A_{\infty, \B}$. 

\item\label{Ai-2} there exists  $0<\theta \le 1 \le C_0 <\infty$ such that for any $B \in \B$ and measurable set $E \subset B$, 
\[
\frac{\mu(E)}{\mu(B)} 
\le C_0 \bigg(\frac{w(E)}{w(B)}\bigg)^{\theta}. 
\]

\item\label{Ai-3} For any $\alpha \in (0, 1)$, there exists $\beta \in (0, 1)$ such that for each $B \in \B$ and measurable set $E \subset B$, 
\[
\mu(E) \ge \alpha \, \mu(B) \quad\Longrightarrow\quad 
w(E) \ge \beta \, w(B). 
\]

\item\label{Ai-4} For any $\alpha \in (0, 1)$, there exists $\beta \in (0, 1)$ such that for each $B \in \B$ and measurable set $E \subset B$, 
\[
\mu(E) \le \alpha \, \mu(B) \quad\Longrightarrow\quad 
w(E) \le \beta \, w(B). 
\]
\end{list}
Then, $\eqref{Ai-1} \Longrightarrow \eqref{Ai-2} \Longrightarrow \eqref{Ai-3} \iff \eqref{Ai-4}$. 
\end{lemma}

\subsection{Young functions and Orlicz spaces}~~

Let $(X,\mu,\mathcal{B})$ is a ball-basis measure space. A function $\Phi : [0, \infty) \rightarrow [0, \infty)$ is said to be a Young function, if $\Phi$ is continuous, convex, increasing function such that $\Phi(0)=0$ and $\Phi(t)/t \rightarrow \infty$ as $t \rightarrow \infty$.

Let $\Phi$ is a Young function and $B \in \B$, we define the normalized Luxemburg norm of $f$ on $B$ by 
\begin{equation*}
\|f\|_{\Phi, B}
= \inf \bigg\{\lambda>0; \fint_B \Phi \Big(\frac{|f(x)|}{\lambda}\Big) d\mu \leq 1 \bigg\}.
\end{equation*}
There are three useful examples. 
\begin{itemize}

\item $\Phi(t)=t^p$, $p>1$, then $\|f\|_{L^p, B} := \|f\|_{\Phi, B}$.

\item $\Phi(t)=t \log^r(e+t)$, $r>0$, then $\|f\|_{L(\log L)^r, B} := \|f\|_{\Phi, B}$. 

\item $\Phi(t)=e^{t^r}-1$, $r>0$, then $\|f\|_{\exp L^r, B} := \|f\|_{\Phi, B}$. 
\end{itemize}

\begin{lemma}[\cite{PRiv}]\label{lem:PhiPhi}
If $\Phi_0, \Phi_1, \ldots, \Phi_k$ are Young functions satisfying 
\[
\Phi_1^{-1}(t) \cdots \Phi_k^{-1}(t) \lesssim \Phi_0^{-1}(t) \quad\text{ for all } t>0, 
\] 
then for any $B \in \B$, 
\begin{equation*}
\|f_1 \cdots f_k\|_{\Phi_0, B} 
\lesssim \|f_1\|_{\Phi_1, B} \cdots \|f_k\|_{\Phi_k, B} .
\end{equation*}
In particular, for any $B \in \B$ and $\frac1s = \sum_{i=1}^k \frac{1}{s_i}$ with $s_1, \ldots, s_k \geq 1$, 
\begin{align*}
\fint_B |f_1 \ldots f_k g| \, d\mu
& \lesssim \| f_1 \|_{\exp L^{s_1}, B} \cdots \| f_k \|_{\exp L^{s_k}, B} \|g\|_{L(\log L)^{\frac{1}{s}}, B}. 
\end{align*}
\end{lemma}

Let $r>0$, the space $\operatorname{osc}_{\exp L^r, \mathcal{B}}$ consists of all locally integrable functions $f \in L^1_{\mathrm{loc}}(X, \mu)$ such that:  
\begin{align*}
\|f\|_{\operatorname{osc}_{\exp L^r, \B}}:= \sup_{B \in \B} \|f-f_B\|_{\exp L^r, B} < \infty.  
\end{align*}
We say that a measurable function $f \in \BMO_{\B}$ if $\int_B |f| \, d\mu< \infty$ for every $B\in \B$ and 
\begin{align}\label{def:BMO}
\|f\|_{\BMO_{\B}} 
:= \sup_{B \in \B} \fint_B |f-f_B| \, d\mu < \infty. 
\end{align}

Let $X = \Rn$, $d\mu=dx$, and $\B$ denotes the collection of all balls or cubes in $\Rn$, it follows from the John-Nirenberg's theorem (cf. \cite[Corollary 3.1.7]{250}) that $\|\cdot\|_{\mathrm{osc}_{\exp L, \mathcal{B}}} \approx \|\cdot\|_{\mathrm{BMO}_{\mathcal{B}}}$.

$\B$ is called {\bf uniform} if there exists a constant $c_0 \in (0, 1)$ such that for all $B \in \B$, there exist a disjoint collection $\{B_j\} \subseteq \B$ such that for all $j$, $B_j \subseteq B$, $\mu(B_j) \ge c_0 \mu(B)$, and $\mu(B \setminus \bigcup_j B_j) = 0$.
Further, a basis $\B$ is said to be {\bf differentiable} if for any point $x \in X$, there exists a sequence $\{B_j\} \subseteq \B$ such that $x \in \bigcap_j B_j$ and $ \mu(B_j) \rightarrow 0$ as $j \to \infty$. 
In general, $\osc_{\exp L^r, \B} \subseteq \BMO_{\B}$ for any $r \ge 1$.  If a basis $\B$ is uniform and differentiable, from \cite[Theorem 3.1]{DLOPW}, it can be deduced that $\|\cdot\|_{\mathrm{osc}_{\exp L^r, \mathcal{B}}} \approx \|\cdot\|_{\mathrm{BMO}_{\mathcal{B}}}$.

\begin{definition}
We say that $A_{\infty, \B}$ satisfies the {\bf sharp reverse H\"{o}lder property} if there exists a constant $c_0>0$ such that for every $w \in A_{\infty, \B}$ one has $w \in RH_{r_w, \B}$ with $r_w = 1 + c_0 [w]_{A_{\infty, \B}}^{-1}$.  
\end{definition}

Set $X = \Rn$, $d\mu=dx$, and $\B$ as the collection of all cubes in $\Rn$, then $A_{\infty, \B}$ satisfies the sharp reverse H\"{o}lder property.

\begin{lemma}\label{lem:LlogL}
Assume that $A_{\infty, \B}$ satisfies the sharp reverse H\"{o}lder property. Let $s>1$, $t>0$, and $w \in A_{\infty, \B}$. Then for any $B \in \B$, 
\begin{align}
\label{e:w} &\| w^{\frac1s} \|_{L^s(\log L)^{st}, B}
\lesssim  [w]_{A_{\infty, \B}}^t \langle w \rangle_B^{\frac1s},
\\
\label{e:fw} &\| f w \|_{L(\log L)^t, B}
\lesssim  [w]_{A_{\infty, \B}}^t \langle w \rangle_B \inf_{x \in B} M_{w}(|f|^s)(x)^{\frac1s}. 
\end{align}
\end{lemma} 
\begin{proof}

Since $w \in A_{\infty, \B}$, based on our assumption, there exists $r_w \approx 1 + [w]_{A_{\infty, \B}}^{-1}$ such that $w \in RH_{r_w, \B}$. It follows from above that
\begin{align*}
&\|w^{\frac1s} \|_{L^s(\log L)^{st}, B}^s
\lesssim \|w\|_{L(\log L)^{st}, B}
\\ 
&\quad \lesssim \bigg(1 + \frac{1}{(r_w - 1)^{st}} \bigg)
\bigg( \fint_B w^{r_w} dx  \bigg)^{\frac{1}{r_w}}
\lesssim  [w]_{A_{\infty}}^{st} \langle w \rangle_B.
\end{align*}
To proceed, using Lemma \ref{lem:PhiPhi} and \eqref{e:w}, we obtain
\begin{align*}
\|fw\|_{L(\log L)^{t}, B}
&\lesssim \bigg( \fint_B |f|^s w \, d\mu \bigg)^{\frac1s}
\|w^{\frac{1}{s'}} \|_{L^{s'}(\log L)^{s't}, B}
\\ 
&\lesssim  [w]_{A_{\infty}}^t
\bigg( \frac{1}{w(B)} \int_B |f|^s w \, d\mu \bigg)^{\frac1s} \langle w \rangle_B
\\
&\lesssim  [w]_{A_{\infty}}^t \inf_{x \in B}
M_{\B, w}(|f|^s)^{\frac1s} \langle w \rangle_B. \quad \qedhere
\end{align*} 
\end{proof}

\vspace{1cm}
\noindent{\bf Acknowledgements } 
The author(s) would like to thank the editors and reviewers for careful reading and valuable comments, which lead to the improvement of this paper. 

\medskip 

\noindent{\bf Data Availability} Our manuscript has no associated data.

\medskip 
\noindent{\bf\Large Declarations}
\medskip 

\noindent{\bf Conflict of interest} The author(s) state that there is no conflict of interest.

\end{document}